\documentclass[11pt,a4paper]{amsart}
\usepackage[english]{babel}
\usepackage[utf8]{inputenc}
\usepackage[T1]{fontenc}
\usepackage{csquotes} 
\usepackage{amssymb}
\usepackage{float} 
\usepackage{amsthm}
\usepackage[top=2cm, bottom=2cm, left=2cm, right=2cm,twoside=false]{geometry} %%% small margins
\usepackage{mathtools} 
\usepackage[usenames,x11names]{xcolor}
\usepackage[all]{xy} %commutative diagrams \xymatrix
\usepackage[normalem]{ulem}%\uline{...}, \sout{...}
\usepackage{multirow} %multiple rows in tables

\usepackage[shortlabels]{enumitem}% more options for enumerated lists
\setlist[1]{leftmargin=*}
\setlist[enumerate,1]{label=(\alph*)}
\setlist[enumerate,2]{label=(\roman*), ref=(\alph{enumi}.\roman*), leftmargin=.5em}
\newlist{enumerate-alt}{enumerate}{1}
\setlist[enumerate-alt]{leftmargin=*, label=(\roman*)}
\newlist{parlist}{enumerate}{1}
\setlist[parlist]{leftmargin=0cm, itemindent=2\parindent, label=(\alph*), ref=(\alph*), itemsep=0.2em}%, topsep=0.3em}
\newlist{steps}{enumerate}{1}
\setlist[steps]{leftmargin=0cm, itemindent=2\parindent, itemsep=.5em, label={\textbf{\arabic*.}},  ref=\textbf{\arabic*}}
\usepackage{caption}
\usepackage{subcaption} %convenient subfigures: \subcaptionbox{caption \label{label}}[width]{source}
\usepackage{tikz}
\usepackage{tikz-cd}
\usepackage{pgfplots}

\usepackage{hyperref} 
\hypersetup{
	linkcolor=DodgerBlue3, %internal refernces (x11color family) 
	citecolor  = teal,
	urlcolor   = DodgerBlue4,
	colorlinks = true, 
	hyperfootnotes =false,
	bookmarksdepth=3,
	linktoc=all
}
\usepackage{multicol}

\makeatletter
\renewcommand\subsection{\@startsection{subsection}{2}%
	\z@{-.5\linespacing\@plus-.7\linespacing}{.5\linespacing}%
	{\bfseries\itshape}}
\renewcommand\subsubsection{\@startsection{subsubsection}{3}%
	\z@{.5\linespacing\@plus.7\linespacing}{-.5\linespacing}%
	{\normalfont\bfseries}}
\renewcommand\paragraph{\@startsection{paragraph}{4}%
	\z@{.5\linespacing\@plus.7\linespacing}{-.5\linespacing}%
	{\normalfont\bfseries}}
\makeatother

\usepackage{indentfirst}

\makeatletter \renewenvironment{proof}[1][\proofname]{
	\par\pushQED{\qed}\normalfont
	\topsep6\p@\@plus6\p@\relax
	\trivlist\item[\hskip\labelsep\bfseries#1\@addpunct{.}]
	\ignorespaces}{
	\popQED\endtrivlist\@endpefalse} \makeatother
\theoremstyle{plain}
\newtheorem{theo}{Theorem}[section]
\newtheorem*{theo*}{Theorem}
\newtheorem{question}[theo]{Question}

\theoremstyle{definition}
\newtheorem{definition}[theo]{Definition}
\newtheorem{lema}[theo]{Lemma}
\newtheorem{prop}[theo]{Proposition}

\newtheorem{example}[theo]{Example}

\newtheorem{notation}[theo]{Notation}
\newtheorem*{notation*}{Notation}
\newtheorem{remark}[theo]{Remark}

\newtheorem{setting}[theo]{Setting}
\newtheorem{cor}[theo]{Corollary}
\newtheorem*{remark*}{Remark} 

\theoremstyle{remark$(i)$}

\newtheorem*{claim*}{Claim}

\newtheorem*{case*}{Case}
\newcommand{\CC}{{\mathbb{C}}}
\newcommand{\NN}{{\mathbb{N}}}

\newcommand{\RR}{{\mathbb{R}}}

\newcommand{\ZZ}{{\mathbb{Z}}}

\newcommand{\comp}{{\circ}}

\let\sec\S % \sec - paragraph character; \S=\mathbb{S} - sphere 

\newcommand{\B}{\mathbb{B}}
\newcommand{\C}{\mathbb{C}}
\newcommand{\D}{\mathbb{D}}
\newcommand{\F}{\mathbb{F}}
\newcommand{\N}{\mathbb{N}}
\renewcommand{\P}{\mathbb{P}}
\newcommand{\Q}{\mathbb{Q}}
\newcommand{\R}{\mathbb{R}}
\renewcommand{\S}{\mathbb{S}}
\newcommand{\Z}{\mathbb{Z}}

\newcommand{\cA}{\mathcal{A}}
\newcommand{\cB}{\mathcal{B}}
\newcommand{\cC}{\mathcal{C}}
\newcommand{\cE}{\mathcal{E}}
\newcommand{\cF}{\mathcal{F}}

\newcommand{\cI}{\mathcal{I}}
\newcommand{\cJ}{\mathcal{J}}

\newcommand{\cM}{\mathcal{M}}

\newcommand{\cO}{\mathcal{O}}
\newcommand{\cP}{\mathcal{P}}
\newcommand{\cQ}{\mathcal{Q}}
\newcommand{\cR}{\mathcal{R}}
\newcommand{\cS}{\mathcal{S}}
\newcommand{\cT}{\mathcal{T}}
\newcommand{\cU}{\mathcal{U}}
\newcommand{\cV}{\mathcal{V}}
\newcommand{\cW}{\mathcal{W}}
\newcommand{\cX}{\mathcal{X}}
\newcommand{\cZ}{\mathcal{Z}}

\renewcommand{\epsilon}{\varepsilon}
\renewcommand{\tilde}{\widetilde}
\renewcommand{\hat}{\widehat}
\renewcommand{\bar}{\overline}

\newcommand{\sdot}{\, \cdot \, } %free entry

\makeatletter
\@namedef{subjclassname@2020}{%
	\textup{2020} Mathematics Subject Classification}
\makeatother

\renewcommand{\leq}{\leqslant}
\renewcommand{\geq}{\geqslant}

\renewcommand{\to}{\longrightarrow}
\newcommand{\onto}{\twoheadrightarrow}
\newcommand{\into}{\hookrightarrow}
\newcommand{\trans}{\pitchfork}

\newcommand{\Sing}{\operatorname{Sing}}
\newcommand{\Exc}{\operatorname{Exc}}

\newcommand{\Hom}{\operatorname{Hom}}
\newcommand{\reg}{^{\mathrm{reg}}}
\newcommand{\Int}{\operatorname{Int}} %interior
\newcommand{\redd}{_{\mathrm{red}}} %reduced structure
\newcommand{\sspan}{\operatorname{span}}

\newcommand{\Sp}{\operatorname{Sp}}
\newcommand{\Fr}{\operatorname{Fr}}
\newcommand{\CZ}{\operatorname{CZ}}
\newcommand{\Ind}{\operatorname{Ind}}
\newcommand{\Fix}{\operatorname{Fix}}

\newcommand{\lbr}{[\![}
\newcommand{\rbr}{]\!]}

\newcommand{\HF}{\operatorname{HF}}

\newcommand{\CFpt}{\operatorname{CF}^{\textnormal{pt}}}
\newcommand{\CFlp}{\operatorname{CF}^{\textnormal{loop}}}
\newcommand{\std}{_{\mathrm{std}}}
\newcommand{\pt}{\mathrm{pt}}
\newcommand{\gp}{^{\mathrm{gp}}}

\newcommand{\rest}{\vartheta}

\newcommand{\hor}{_{\mathrm{hor}}}
\renewcommand{\vert}{_{\mathrm{vert}}}

\renewcommand{\d}{\partial}
\newcommand{\id}{\mathrm{id}}

\newcommand{\de}{\coloneqq} %:=

\usepackage{xspace}

\newcommand{\AC}{_{A}}

\newcommand{\limn}{\lim_{\nu\rightarrow \infty}}

\newcommand{\pr}{\mathrm{pr}}

\usepackage[scr=boondoxo]{mathalfa}
\newcommand{\ac}{\mathscr{a}}
\newcommand{\kk}{\mathscr{k}}

\newcommand{\btau}{\boldsymbol{\tau}}

	\usepgfplotslibrary{fillbetween}
	\usetikzlibrary{snakes,patterns,calligraphy}

\begin{document}

\title
[Symplectic monodromy at radius zero]
{Symplectic monodromy at radius zero \\ and equimultiplicity of $\mu$-constant families}
\author{Javier Fern\'andez de Bobadilla}
\address{Javier Fern\'andez de Bobadilla:  
	(1) IKERBASQUE, Basque Foundation for Science, Euskadi Plaza, 5, 48009 Bilbao, Basque Country, Spain; 
	(2) BCAM,  Basque Center for Applied Mathematics, Mazarredo 14, 48009 Bilbao, Basque Country, Spain; 
	(3) Academic Colaborator at UPV/EHU.} 
\email{jbobadilla@bcamath.org}

\thanks{J.F.B was supported by the Spanish Ministry of Science, Innovation and Universities, project reference MTM2016-76868-C2-1-P (UCM), J.F.B. and T. P. were supported by the Basque Government through the BERC 2018-2021 program, and Gobierno Vasco Grant IT1094-16, 
	by the Spanish Ministry of Science, Innovation and Universities BCAM Severo Ochoa accreditation SEV-2017-0718, by the Spanish Ministry of Science, Innovation and Universities, project reference PID2020-114750GB-C33, and by CIRM Jean Morlet Semester \emph{Singularity Theory from Modern Perspectives}. 
}

\author{Tomasz Pe{\l}ka}
\address{Tomasz Pe{\l}ka:
	(1) BCAM,  Basque Center for Applied Mathematics, Mazarredo 14, 48009 Bilbao, Basque Country, Spain; 
	(2) Institute of Mathematics, University of Warsaw, Banacha 2, 02-097 Warsaw, Poland}
\email{tpelka@mimuw.edu.pl}

\begin{abstract}
We show that every family of isolated hypersurface singularities with constant Milnor number has constant multiplicity. To achieve this, we endow the A'Campo model of \enquote{radius zero} monodromy with a symplectic structure. This new approach allows to generalize a spectral sequence of McLean converging to fixed point Floer homology of iterates of the monodromy to a more general setting which is well suited to study $\mu$-constant families. 
\end{abstract}

\maketitle 
\setcounter{tocdepth}{1}
\tableofcontents

\section{Introduction}\label{sec:intro}

Let $f\in\CC\lbr z_1,...,z_n\rbr $ be a formal power series. We write $f=\sum_{\iota\in\NN^n}a_{\iota}z^\iota$ with the usual multi-index notation. The \emph{multiplicity} of $f$ is the biggest number $\nu(f)$ such that $f\in \mathfrak{m}^{\nu(f)}$, where $\mathfrak{m}$ is the maximal ideal of $\CC\lbr x_1,...,x_n\rbr $. The Milnor number of $f$ is $\mu(f)=\dim_{\C}\C \lbr z_1,\dots,z_n \rbr/\langle \tfrac{\d f}{\d z_1},\dots, \tfrac{\d f}{\d z_{n}}\rangle$, see  \cite[\sec 7]{Milnor}. A \emph{continuous family of power series} parametrized by $t\in [0,1]$ assigns to each $t$ a  power series $f_t=\sum_{\iota\in\NN^n}a_{\iota}z^\iota$ in such a way that each coefficient $a_{\iota}(t)$ is continuous in $t$. 
\smallskip

The main result of this article is the following.

\begin{theo}\label{theo:Zariski}
	Let $(f_t)$ be a continuous family of power series. If the Milnor number $\mu(f_t)$ is independent of $t$ and finite then the multiplicity $\nu(f_t)$ is also independent of $t$.
\end{theo}

Theorem \ref{theo:Zariski} implies a positive answer to Zariski Multiplicity Question for the case of families of isolated singularities, see Corollary \ref{cor:Zariski}. A brief survey on Zariski Question is given in Section \ref{sec:history}: for more extensive ones we refer to \cite{Eyral-survey,Eyral-book} or \cite{Bobadilla_survey}.
\smallskip

\subsection{Summary of our approach}\label{sec:intro_summary}

The starting point of our approach towards proving Theorem \ref{theo:Zariski} is a result of McLean \cite[Corollary 1.4]{McLean}, which characterizes the multiplicity of an isolated singularity as the minimal integer $m\geq 1$ such that a version $\HF^*(\phi_f^m,+)$ of the fixed point Floer cohomology of the $m$-th iterate of the symplectic monodromy $\phi_f\colon \mathbb{F}_f\to \mathbb{F}_f$ of the Milnor fibration does not vanish. Here, $\mathbb{F}_f$ denotes the symplectic Milnor fiber (the precise definition of $(\mathbb{F}_f,\phi_f)$ requires some care, but we skip this subtlety for the sake of clarity at the introduction). To prove Theorem \ref{theo:Zariski}, we compare the groups  $\HF^*(\phi_{f_t}^m,+)$, for different members $f_t$ of  a $\mu$-constant family. We will now outline what is the key difficulty of this approach, why the setting of \cite{McLean} is not sufficient to cover it, and how our new method allows to overcome this obstacle.
\smallskip

Let $f$ be a germ of an isolated hypersurface singularity. It is known that the pair $(\mathbb{F}_{f},\phi_f)$ is a Liouville domain with a compactly supported exact symplectomorphism \cite{McLean}. Any two such pairs which are symplectically isotopic (in a certain sense)  give rise to the same Floer cohomology, see \cite[Lemma B.3]{McLean} or Proposition \ref{prop:isotopy_invariance}. So, the natural approach to prove Theorem \ref{theo:Zariski} would be to prove that, for small $t>0$, the pair $(\mathbb{F}_{f_{t}},\phi_{f_t})$ is symplectically isotopic to $(\mathbb{F}_{f_{0}},\phi_{f_{0}})$. However, even proving that $\mathbb{F}_{f_{t}}$ and $\mathbb{F}_{f_{0}}$ are isotopic as Liouville domains seems hard. The main reason is that it is not known whether the Milnor radius $\epsilon_t$ for $f_t$ can be chosen independent of $t$: in principle, we could have $\epsilon_t\rightarrow 0$ as $t\rightarrow 0$. This obstructs the direct comparison of Milnor fibrations for $f_t$ and $f_0$.

In \cite{McLean}, McLean uses an equivalence relation which is weaker than isotopy. Namely, to each pair $(\mathbb{F},\phi)$ as above, a construction of Giroux \cite{Giroux_original,Giroux_cool-gadget} associates a pair $L_{\phi}\subseteq S_{\phi}$ of contact manifolds. Now, \cite[Appendix B]{McLean}, roughly speaking, shows that $\HF^{*}(\phi^m,+)$ depends only on this associated \emph{contact pair}. For a symplectic monodromy of the Milnor fiber, \cite{McLean} shows that this pair is the link $f^{-1}(0)\cap \S^n_{\epsilon}\subseteq \S^n_{\epsilon}$, where $0<\epsilon\ll 1$ is a Milnor radius, with the standard contact structure introduced by Varchenko \cite{Varchenko}.  This provides the basis for the main theorem of \cite{McLean}, asserting that the multiplicity %(and the log-canonical threshold too)
only depends on the embedded contact type of the link. %of the contactomorphism type of the pair $(\S_{\epsilon_f},\S_{\epsilon_f}\cap f^{-1}(0))$.

However, the above mentioned McLean's main result cannot be easily applied in our setting, due to the radius degeneration problem explained in the paragraph above. Indeed, let $(f_{t})_{t\in [0,1]}$ be a $\mu$-constant family, and let $\epsilon_{t}$ be a Milnor radius of $f_{t}$. Fix $0<t\ll 1$ and choose $\epsilon_t<\epsilon_0$.  Then
\begin{equation}
	\label{eq:cobordismintro}
	(\overline{\B_{\epsilon_{0}}\setminus \B_{\epsilon_t}},\overline{\B_{\epsilon_{0}}\setminus \B_{\epsilon_t}}\cap f_{t}^{-1}(0))
\end{equation}
defines a symplectic cobordism between the contact pair of $f_t$ and a contact pair which is contactomorphic to the one for $f_{0}$. This cobordism is known to be topologically trivial if $n\geq 4$ \cite{Le-Ramanujam}, but proving its symplectic triviality seems hard.
\smallskip

In order to explain our approach let us briefly recall known facts about the Milnor fibration. Let $f\colon (\CC^n,0)\to(\CC,0)$ be a holomorphic function germ. A classical result  \cite{Milnor} asserts that  there is  $\epsilon_0>0$ such that for any $0<\epsilon<\epsilon_0$ and any $0<\delta\ll \epsilon$ the restrictions
\begin{equation}
	\label{eq:milfibsphintro}
	\frac{f}{|f|}\colon \S_\epsilon \setminus f^{-1}(0) \to\S^1
\end{equation}
and
\begin{equation}
	\label{eq:milfibtbintro}
	f\colon \B_\epsilon\cap f^{-1}(\D_\delta^*)\to \D_\delta^*
\end{equation}
are locally trivial fibrations. The first one (called Milnor fibration \emph{in the sphere}) is smoothly equivalent to the restriction of the second one (called Milnor fibration \emph{in the tube}) to $\Int\B_{\epsilon}\cap f^{-1}(\d\D_{\delta})$. The numbers $\epsilon,\delta$ are called \emph{Milnor radii} for the ball and the disc, respectively.  %Here $\D_\delta$ and $\D_\delta^*$ denote the disc and the punctured disc of radius $\delta$ respectively.
If $f$ has an isolated singularity at the origin then the monodromy of the above fibrations can be chosen compactly supported.

The closures of the fibers of both \eqref{eq:milfibsphintro} and \eqref{eq:milfibtbintro} are Liouville domains for the standard Liouville form $-\frac{\pi}{2}d^{c}||z||^2$. For the fibration~\eqref{eq:milfibsphintro} McLean associates a monodromy $\phi_f:\mathbb{F}_f\to\mathbb{F}_f$ which is a compactly supported exact symplectomorphism and studies the Floer cohomology of the pair $(\mathbb{F}_f,\phi_f)$. 

In turn, we note that the fibration \eqref{eq:milfibtbintro} will not be, in general, locally trivial in a symplectic sense, since there is no reason why the volume of the fibers should be constant in $z\in\D^*_\delta$. Nonetheless, in Example \ref{ex:Milnor-fibration} we fix this issue by slightly modifying the shape of the ball $\B_{\epsilon}$.%, so we can speak of the symplectic monodromy of \eqref{eq:milfibtbintro}, too; though we do not claim that \eqref{eq:milfibsphintro} and \eqref{eq:milfibtbintro} are symplectically equivalent.}

In \cite[Theorem 1.2]{McLean} McLean provides a spectral sequence converging to $\HF^{*}(\phi_f^m,+)$, whose first page allows to compute the multiplicity of $f$. This first page is closely related to the A'Campo formula \cite{A'Campo} for the Lefschetz numbers and zeta function of $\phi_{f}^m$, which we now recall.

Let $h\colon X\to\CC^n$ be a log resolution of $f$ that only modifies $\C^n$ over the origin. Write $(f\comp h)^{-1}(0)=\sum_{i=0}^Nm_i D_i$, where $D_0$ is the strict transform of $f^{-1}(0)$.  Then the function $f\comp h$ can be expressed in local coordinates $(z_{1},...,z_n)$ as $\prod_{i=1}^kz_{i}^{m_{i}}$.
Denote by $\tau:\D_{\log}\to\D_\delta$ the real oriented blow up at the origin. Then $\d\D_{\log}:=\tau^{-1}(0)$ can be thought as a circle \enquote{at radius 0}. Exploiting the local structure of $f\comp h$, A'Campo constructed a topological locally trivial fibration
\begin{equation}
	\label{eq:ACampofibintro}
	f_A\colon A\to\D_{\log}
\end{equation}
whose restriction over $\D_{\delta}^{*}$ coincides with~\eqref{eq:milfibtbintro} and whose restriction
\begin{equation}
	\label{eq:ACampofibintrorad0}
	f_A|_{f_A^{-1}(\d \D_{\log})}\colon f_A^{-1}(\d \D_{\log})\to \d\D_{\log}
\end{equation}
over the radius $0$ circle admits a monodromy $\phi_A$ such that the fixed point set of $\phi_A^m$ is equal to
\begin{equation}
	\label{eq:fixedsetintro}
	\bigsqcup_{m_i|m}B_i,
\end{equation}
where $B_i$ is a cyclic covering of degree $m_i$ of $D_i\setminus (\bigcup_{j\neq i} D_i)$ determined by $f$, cf.\ \cite[\sec 2.3]{Denef_Loeser-Lefshetz_numbers}. Therefore, we have the equality of Lefschetz numbers
\begin{equation}\label{eq:ACampo-formula}
	\Lambda(\phi_f^m)=\sum_{m_i|m}\chi(B_i),
\end{equation}
where $\chi(\sdot)$ denotes the topological Euler characteristic. Notice that the summand $\chi(B_0)$ appears for every $m$, but it is equal to $0$, since $B_0$ is homeomorphic to $(f^{-1}(0)\cap\S_{\epsilon_0})\times (0,\epsilon_0]$ by the conical structure theorem \cite[Theorem 2.10]{Milnor}, and $\chi(f^{-1}(0)\cap\S_{\epsilon_0})=0$ since $f^{-1}(0)\cap\S_{\epsilon_0}$ is an odd-dimensional, smooth manifold. At this stage, we would also like to remark that the space $A$ constructed by A'Campo is a manifold {\em with corners}, and $f_A$ has no further structure than the topological locally trivial fibration. 

As we will explain below, the main step of our proof is to endow $A$ with a smooth structure, and define a fiberwise symplectic form on $A$ such that $\phi_{A}$ is realized as the symplectic lift of the angular vector field on $\d\D_{\log}$. In particular, the flow of the symplectic lift of the radial vector field on $\D_{\log}$ provides a \emph{symplectic isotopy} between $\phi_{A}$ and the monodromy of the Milnor fibration \eqref{eq:milfibtbintro}.%, whose symplectic version is introintroduced in a precise way in Example \ref{ex:Milnor-fibration}.}

In turn, McLean in  \cite[Theorem 1.2]{McLean} constructs a certain Liouville domain $\mathbb{F}$ with a compactly supported exact symplectomorphism $\phi$, whose dynamics coincides with the one of $\phi_A$; and whose associated contact pair is contactomorphic to the link of $f$. Hence even though $(\F_{f},\phi_{f})$ and $(\F,\phi)$ might not be  symplectically isotopic, one still has that $\HF^{*}(\phi^m,+)=\HF^{*}(\phi_{f}^m,+)$ is the required Floer cohomology. Now, given such $(\mathbb{F},\phi)$, the first step to define $\HF^*(\phi,+)$ is to deform $\phi$ slightly by a Hamiltonian near $\d\mathbb{F}$ so that the deformed $\check\phi$ has no fixed points in a collar near $\d\mathbb{F}$. Therefore, the fixed point set of the deformed McLean monodromy $\check{\phi}^m$ is of the form % McLean~\cite{McLean} associates with $f$ a pair $(\mathbb{F},\phi)$ and a Hamiltonian deformation $\check\phi$ such that the contact pair associated with $(\mathbb{F},\phi)$ by Giroux construction coincides with the contact pair associated with $(\mathbb{F}_f,\phi_f)$. The point is that the fixed point set of $\check\phi^m$ is of the form
\begin{equation}
	\label{eq:mcleandecomp}
	\bigsqcup_{m_i|m,\ i\neq 0}B_i
\end{equation}
where $B_i$ have the same properties as above. Note that the term $B_0$ got removed by the deformation near $\d \mathbb{F}$. This allows to construct a spectral sequence converging to $\HF^*(\phi,+)$ whose $E_1$-pieces are the ordinary homologies of the components $B_i$ for $m_i|m$, $i\neq 0$, see \cite[Theorem 1.2]{McLean}. Since $\chi(\HF^*(\phi^m,+))=\Lambda(\phi^m)$, this result may be regarded as a lifting  of the A'Campo formula \eqref{eq:ACampo-formula} to the setting of Floer cohomology.

Since the component $B_0$ is {\em not present} in \eqref{eq:mcleandecomp}, the $E_1$-page is empty for $m$ is strictly smaller than the multiplicity $\nu(f)$. It is also easy to show that $\HF^*(\phi^m,+)\neq 0$ for $m=\nu(f)$.

To have a grasp of why the deformation $\check\phi$ can be performed so that the component $B_0$ is not present, observe that, exploiting the topological product structure $B_0\cong (f^{-1}(0)\cap\S_{\epsilon})\times (0,\epsilon]$, where $\epsilon$ is the Milnor radius of $f$, the monodromy homeomorphism $\phi_A$ constructed by A'Campo can be isotoped to a homeomorphism $\check{\phi}_A$ that eliminates the fixed point component $B_0$ and leaves the other ones. Now for a sufficiently small Milnor radius $\epsilon>0$, $B_0$ becomes symplectically trivial, that is, a piece of a symplectization $(f^{-1}(0)\cap\S_{\epsilon}) \times (0,\epsilon)$ of the contact link. This allows to realize the above isotopy in the symplectic category, see \cite[p.\ 1023]{McLean}.

\smallskip

The Milnor fibration in the tube \eqref{eq:milfibtbintro} is better suited to study families, in particular $\mu$-constant ones. Indeed, let $F:\CC^n\times\D_\xi\to\CC\times\D_\xi$ be holomorphic of the form $F(x,t)=(f_t(x),t)$ such that $f_t$ has a  $\mu$-constant isolated singularity at the origin. Then, in~\cite{Le-Ramanujam} it is proved that there exits Milnor radii $\epsilon_0,\delta_0$ for $f_0$ such that, after shrinking $\xi>0$, the restriction
\begin{equation}
	\label{eq:auxfibintro}
	F\colon (\B_{\epsilon_0}\times\D_\xi)\cap F^{-1}(\D_{\delta_0}^*\times\D_\xi)\to \D_{\delta_0}^*\times\D_\xi
\end{equation}
is a locally trivial fibration. So, fixed any $t\in \D^*_\xi$ and any $\delta_t<\delta_0$, the restriction
\begin{equation}
	\label{eq:auxfibintro2}
	f_t\colon \B_{\epsilon_0}\cap f_t^{-1}(\D_{\delta_t}^*)\to \D_{\delta_t}^*
\end{equation}
is diffeomorphic to the Milnor fibration in the tube for $f_0$. Now we pick $\epsilon_t<\epsilon_0$ and $\delta_t<\delta_0$, Milnor radii for $f_t$. Then, for this particular choice of radii, the Milnor fibration in the tube for $f_t$ is embedded in the fibration~\eqref{eq:auxfibintro2}, and according with~\cite{Le-Ramanujam}, for any $z\in \D_{\delta_{t}}$ the difference between the fiber of~\eqref{eq:auxfibintro2} over $z$ and the Milnor fiber $f_t(z)\cap\B_{\epsilon_t}$ of $f_t$ over $z$ is a cobordism diffeomorphic to~\eqref{eq:cobordismintro}. Moreover, the monodromy of~\eqref{eq:auxfibintro2} can be chosen with support at the interior of $f_t^{-1}(z)\cap\B_{\epsilon_t}$.  Thus in our case, the component $B_0$ of the fixed point set \eqref{eq:fixedsetintro} of $\phi_A$ corresponds to the  cobordism \eqref{eq:cobordismintro}, which is not known to be symplectically trivial. So, $B_0$ cannot be eliminated by a small Hamiltonian perturbation. 
\smallskip 

Our proof of Theorem \ref{theo:Zariski} 
can be summarized in the following steps:
\begin{steps}
	\item\label{step:red-balls} We modify the fibration~\eqref{eq:auxfibintro} so that it becomes symplectically locally trivial and maintains the topological properties stated above. This is a matter of modifying the shape of the domain $(\B_{\epsilon_0}\times\D_\xi)\cap F^{-1}(\D_{\delta_0}^*\times\D_\xi)$ in such a way that the fibers can be symplectomorphic, in particular have constant volume. As a consequence, for each $t\in\D_\xi$ we define a pair $(\mathbb{F}_t,\phi_t)$ formed by a Liouville domain and an exact symplectomorphism, whose topology coincides with the monodromy of the fibration~\eqref{eq:auxfibintro2}, and such that for any small $t>0$, the pairs $(\mathbb{F}_{t},\phi_{t}^m)$ and $(\mathbb{F}_{0},\phi_{0}^m)$ are symplectically isotopic. In particular, $\HF^*(\phi_{t}^m,+)\cong \HF^*(\phi_{0}^m,+)$ for every $m$.
	\item\label{step:A'Campo} For each fixed $t$, we produce a symplectic version of A'Campo construction. That is, we construct a symplectically locally trivial fibration~\eqref{eq:ACampofibintro} satisfying the following two properties.
	\begin{enumerate}[(i)]
		\item\label{item:intro-i} The restrictions of $f_A$ and of the fibration~\eqref{eq:auxfibintro} coincide symplectically over $\D_\delta\setminus\D_{\delta'}\subset\D_{\log}$ for a certain $0<\delta'<\delta$,
		\item\label{item:intro_ii} The monodromy of~\eqref{eq:ACampofibintrorad0} induced by the lifting of the unit vector field in $\d\D_{\log}$ has exactly the same dynamical properties as the topological monodromy constructed by A'Campo.
	\end{enumerate}   
	We call this monodromy  the \emph{symplectic monodromy at radius zero} and denote it by $(\mathbb{F}^0_{t},\phi_{t}^0)$. Property \ref{item:intro-i} gives an isotopy between $(\mathbb{F}_{t},\phi_{t})$ and $(\mathbb{F}^0_{t},\phi_{t}^0)$, so we have $\HF^*(\phi_{t}^m,+)\cong \HF^*((\phi_{t}^0)^m,+)$ for any $m$. 
	
	As we explain below, this step gets rather involved, and creates a technique that may be useful for other degeneration problems in algebraic geometry. Our construction is, in fact, performed in a much broader generality, covering e.g.\ the case when $f$ is a function from Stein space $Y$, and both $Y$ and $f^{-1}(0)$ can have many isolated singularities.
	\item\label{step:sequence} %It is important to notice that in Step \ref{step:A'Campo} above, the trick performed by McLean to eliminate the component $B_0$ from the set of fixed points of $\phi_{t}^0^m$ is hard to be generalized, mainly because the cobordism~\eqref{eq:cobordismintro} is not known to be symplectically trivial. So, w
	We slightly generalize McLean's spectral sequence in such a way that the component $B_0$, corresponding to the cobordism \eqref{eq:cobordismintro}, is taken into account. The new terms appearing in the first page of the spectral sequence are summands of the Borel-Moore homology $H_*^{BM}(B_0,\ZZ_2)$. They vanish because the cobordism~\eqref{eq:cobordismintro} is homologically trivial. Thus we get the same $E_1$ page as \cite[Theorem 1.2]{McLean}, hence the same characterization of multiplicity -- now for the Milnor radius $\epsilon_0$ fixed within the family. This result can also be stated in more generality indicated in Step \ref{step:A'Campo} above, see Propositions \ref{prop:spectral-sequence-monodromy},  \ref{prop:monodromy}.
\end{steps}

Our proof does not use Giroux construction \cite{Giroux_original,Giroux_cool-gadget} or McLean result \cite[(HF2)]{McLean} asserting that Floer homology depends only on the associated contact link. We directly deform the natural symplectic monodromy arising from the Milnor fibration in the tube~\eqref{eq:milfibtbintro} to a radius zero monodromy. %Namely, we extend the symplectic structure of the resolution directly to \enquote{radius zero}, i.e.\ to the fiber $\mathbb{F}_{A}$ of the A'Campo space $A$. 
This way, for \emph{any} fixed radius we get a model $(\mathbb{F}_A,\phi_A)$ which is actually isotopic to the fibration $(\mathbb{F}_{f},\phi_f)$ at that fixed radius, and has dynamical properties which allow to derive McLean spectral sequence. This is done without relying on most constructions in~\cite{McLean}.

As a byproduct of our construction, in Corollary \ref{cor:no-fixed-poins} we produce, for any {\em singular} hypersurface germ, a representative of its monodromy in the symplectic mapping class group of its Milnor fiber, which has no fixed points. We also produce one whose fixed point set is precisely the boundary of the Milnor fiber. This extends McLean results, who produces models with the same properties, but instead of being symplectically isotopic to the monodromy, they just share the same contact pairs.

\subsection{Outline}
We will now indicate the sections of the article where each of the above steps is realized.
\smallskip

The technical core of the proof lies in Step \ref{step:A'Campo}, that is, in the construction of the symplectic structure on the A'Campo space $A$. It is carried out in Sections \ref{sec:ACampo} and \ref{sec:symplectic}. Definition \ref{def:AX} of the space $A$ can be explained in two steps. As before, let $h\colon X\to \C^n$ be a log resolution of $f$, and let $D=\sum_{i}m_{i}D_{i}$ be the total transform of $f^{-1}(0)$. First, we perform a real oriented blowup $X_{\log}\to X$ which extends the fibration over $\D_{\delta}^{*}$ to one over $\D_{\log}$, see \eqref{eq:ACampofibintro}. This can be phrased conveniently in terms of Kato--Nakayama construction \cite{Kato_Nakayama} in log geometry, see \cite{Cauwbergs,CFP_Motivic-logarithmic-topological} for applications in this context. The resulting monodromy has the required behavior over each $D_{i}$, but is not continuous over the intersections, see Example \ref{ex:monodromy}. To remedy this problem, we multiply the preimage of each intersection $\bigcap_{i\in I} D_{i}$ by a corresponding face $\Delta_{I}$ of the dual complex of $D$. This face is parametrized by \emph{tropical coordinates} $w_{i}$, see \eqref{eq:def-w_i-u_i}, which measure the relative speed of convergence toward each $D_i$. The idea of such a \emph{hybrid} construction, combining classical and tropical coordinates, goes back to \cite{Bergman}, and has been successfully applied to various degeneration problems,  see e.g.\ \cite{BJ_tropical,EM_Lagrangian-fibrations}. For our application, we need to push this idea further.

Indeed, the space $A$ constructed above is a manifold with corners. We need to choose its smoothing in a very careful way,  which allows to introduce a fiberwise symplectic form satisfying restrictive properties \ref{item:intro-i}--\ref{item:intro_ii} above. To do this, we introduce smooth charts \eqref{eq:AC-chart} in such a way that both radial coordinates $r_i=|z_i|$ of $X$; and tropical coordinates $w_i$ of the simplices in $\d A$, decay exponentially 
towards their zero loci. This way, the local coordinate $r_i$ corresponding to a component $D_i$ gets replaced by a global smooth function $\bar{v}_{i}\colon A\to \R$ measuring, in a sense, a \enquote{hybrid} distance to $D_i$. Using these functions in \eqref{eq:omegaE} we provide an explicit formula for the symplectic form. The resulting monodromy $\phi_A|_{\d A}$ at radius zero  is then a time one flow of the vector field given by the formula \eqref{eq:monodromy-vector-field}.

The guiding principle of the technical proofs in Sections \ref{sec:ACampo}--\ref{sec:symplectic} is that all choices made along the construction become irrelevant at radius zero. Hence a local computation can be used to prove that at radius zero, $\phi_{A}$ is a rotation by $\tfrac{2\pi}{m_i}$ about $D_i$, and an interpolating Dehn twist in between, see Figure \ref{fig:monodromy}. This way, we get a symplectic representative of the monodromy as required in \ref{item:intro-i} and \ref{item:intro_ii}.
\smallskip

Step \ref{step:red-balls} is addressed in Section \ref{sec:acobs}, where we recall the necessary notions from the theory of Liouville domains and their fibrations \cite{Giroux_cool-gadget,Seidel_Fukaya-categories}. In particular, in Definition \ref{def:acob} we recall McLean's notion of graded abstract contact open book \cite{McLean}, i.e.\ pair $(\mathbb{F},\phi)$ as above, for which $\HF^{*}(\phi,+)$ is defined. Then, we describe a procedure which makes a symplectic monodromy into a graded abstract contact open book, see Definition \ref{def:monodromy_acob}. Our setting is closely related to \emph{Liouville fibrations with singularities} considered in \cite{Seidel_Fukaya-categories}, see Remark \ref{rem:Seidel}.
\smallskip

In Section \ref{sec:spectral-sequence} we recall the definition of Floer cohomology groups $\HF^{*}(\phi,+)$. To be precise,  we will use Floer \emph{homology} introduced by Uljarevic in \cite{Uljarevic}, cf.\ \cite{Seidel_more}: in Section \ref{sec:floerequivalence} we explain how to compare it with the setting of \cite{McLean}. In Proposition \ref{prop:spectral-sequence}, we prove a slight generalization of McLean spectral sequence \cite[(HF3)]{McLean} to the setting which includes the fixed-point component $B_0$. In Proposition \ref{prop:spectral-sequence-monodromy}, we apply this spectral sequence to the monodromy $\phi_{A}$ constructed in Step \ref{step:A'Campo}, using its dynamical properties listed in Proposition \ref{prop:monodromy}. This way, we get a spectral sequence similar to \cite[Theorem 1.2]{McLean}, but with term $B_0$ taken into account. Theorem \ref{theo:Zariski} is deduced in Section \ref{sec:proof}.

In order to make the article more accessible to readers who are not specialists in symplectic geometry, we try to be reasonably self-contained at various points.

\subsection{A brief history of the Zariski multiplicity conjecture}
\label{sec:history}

In his seminal paper \cite{Zariski_conjecture}, Zariski posed eight questions which had a major impact on the current shape of singularity theory. Out of these questions, the only one which remains open is the following.

\begin{question}\label{q:Zariski}
	Let $f,g\colon (\C^n,0)\to (\C,0)$ be holomorphic germs of the same topological type. Is it true that $f$ and $g$ have the same multiplicity?
\end{question}

Here, $f,g$ are said to be \emph{of the same topological type} if there exists a germ of a homeomorphism $\Phi\colon (\CC^n,0)\to (\CC^n,0)$ such that $\Phi(V(f))=V(g)$, where $V(f)$ denotes the zero set of $f$.
\smallskip 

If the above homeomorphism $\Phi$ is assumed to be $\cC^{1}$, then the affirmative answer to Question \ref{q:Zariski} was given by  Ephraim~\cite{Ephraim_C1}, and, independently, by Trotman~\cite{Trotman-thesis}. The answer is also known to be positive for $n=2$, i.e.\ for plane curve singularities \cite{Zar0}, and if one of the germs is smooth \cite{A'Campo,Le1}: in fact, the latter result follows from the A'Campo formula \eqref{eq:ACampo-formula} for the Lefschetz number. It is also true if $n=3$ and one of the germs has multiplicity two, see Navarro-Aznar~\cite{Navarro_Aznar}. If $n=3$ and the homeomorphism $\Phi$ is additionally bilipschitz, then a positive answer to Question \ref{q:Zariski} was given in \cite[Theorem 3.4(1)]{FdBFS} by  Fernandes, Sampaio and the first author.

Assume that $n=3$, and both $f$ and $g$ are germs of isolated surface singularities. Then, the answer to Question \ref{q:Zariski} is known to be positive if both $f,g$ are quasihomogeneous (Xu--Yau~\cite{Xu_Yau,Yau_top-types}); if one has arithmetic genus at most two (Yau~\cite{Yau_n=3_pa<=2}), or geometric genus at most $3$, or  Milnor number at most 26 (Yau--Zuo~\cite{Yau_Zuo-lower_bound}). It is also true for suspensions of irreducible plane curve singularities (Mendris--N\'emethi~\cite{MN_suspensions}). For $n>3$, a positive answer is known if one of the germs has Milnor number at most $2^{n}-1$ \cite{Yau_Zuo-lower_bound}.

In general, if $f$ and $g$ are topologically right-left equivalent by bilipschitz homeomorphisms, Risler--Trotman~\cite{Risler-Trotman_bilipschitz} proved that they have the same multiplicity.%, so again Question \ref{q:Zariski} has a positive answer.
\smallskip

An important special case of Question \ref{q:Zariski}, which was also widely studied since \cite{Zariski_conjecture}, concerns families of topologically equivalent germs. Theorem \ref{theo:Zariski} gives a positive answer to this question in case of isolated singularities. More precisely, we have the following.

\begin{cor}\label{cor:Zariski}
	Let $f_t\colon (\C^n,0)\to (\C,0)$ be a continuous family of holomorphic germs of isolated singularities, parametrized by $t\in [0,1]$. Assume that the topological type of $f_t$ is independent of $t$. Then the multiplicity $\nu(f_t)$ is independent of $t$, too.
\end{cor}

Indeed, families of isolated singularities with constant topological type have constant Milnor number $\mu$, because $\mu$ is a topological invariant \cite[7.2]{Milnor}, see \cite{Le1,Te:1980}. Thus Corollary \ref{cor:Zariski} follows from Theorem \ref{theo:Zariski}.  We remark that the question whether every $\mu$-constant family has constant topological type has positive answer if $n\neq 3$ \cite{Le-Ramanujam}, and is open if $n=3$, i.e.\ for families of surface singularities. 

In Example \ref{ex:not-in-a-family}, we show that there do exist isolated singularities of the same topological type, which are not connected by a $\mu$-constant family. Therefore, while Corollary \ref{cor:Zariski} gives a positive answer to Question \ref{q:Zariski} for families of isolated singularities, the non-family version remains open. The \emph{non-isolated} case of Question \ref{q:Zariski}, even for families, remains open, too. Nonetheless, David Massey observed that our results imply a positive answer for a particular class of \emph{aligned} singularities, see Corollary \ref{cor:non-isolated}.
\smallskip

Up to now, Theorem \ref{theo:Zariski} and Corollary \ref{cor:Zariski} were established only in certain special cases. Gabrielov and Kushnirenko~\cite{GaKou} proved Theorem \ref{theo:Zariski} for $f_0$ homogeneous. Greuel~\cite{Greuel_quasihomogeneous} and O'Shea~\cite{OShea-quasihomogeneous} proved it for $\mu$-constant deformations of functions $f_0$, with the only hypothesis that $f_0$ is quasi-homogeneous. The case when $f_0$ is Newton-nondegenerate was settled by Abderrahmane \cite{Abderrahmane_Newton-nondegenerate}, cf.\ \cite{Saia_Tomazella}. Other particular cases were studied e.g.\ by Greuel--Pfister~\cite[\sec 3]{GP_Advances}, Massey \cite[7.9]{Massey_Lecture-notes} or Plenat--Trotman~\cite{Plenat_Trotman}. Let us also remark that classical results of Hironaka \cite{Hironaka_normal-cones} imply that Whitney equisingular families are equimultiple. In \cite{TvS_weak-Whitney-implies-equimultiple}, Trotman and van Straten proved that \emph{weak} Whitney equisingularity implies equimultiplicity, too. 
	\smallskip

As a final remark, we note that Theorem \ref{theo:Zariski} is easily reduced to the case when the family $(f_t)$ is algebraic, i.e.\ each coefficient $a_{\iota}$ is a polynomial in $t$. This way, Theorem \ref{theo:Zariski} becomes a purely algebraic statement, which can be formulated over any field. Nonetheless, even in the curve case ($n=2$), no purely algebraic proof of this statement is known. In fact, the only proof over the complex numbers consists in proving that a $\mu$-constant family is topologically trivial, and then applying the affirmative answer to Question \ref{q:Zariski}, known for $n=2$. For families of functions defined over an algebraically closed field of characteristic zero, an application of Lefschetz Principle gives the same conclusion, see Remark \ref{rem:algebraic_Zariski}. Answering this question was one of the main motivations for development of the computer algebra program \textsc{Singular}~\cite{Sing}, see \cite{GP_history-of-singular} for a historical account.

Recently, Nero Budur, L\^e Quy Thuong, Honc Duc Nguyen and the first author have conjectured that the fixed point Floer cohomology of the $m$-th iterate of the monodromy coincides with the compactly supported cohomology of the $m$-th restricted contact locus \cite[Conjecture 1.5]{BBLN_contact-loci}. This might suggest an algebraic approach to the proof of Theorem \ref{theo:Zariski}, based on studying variation of contact loci in $\mu$-constant families. However, this problem does not seem easy at all.
%
%The reader may consult Eyral's survey~\cite{Eyral-survey} and book~\cite{Eyral-book}, or the survey of the first author \cite{Bobadilla_survey} for further information on Zariski's multiplicity conjecture.

\paragraph{Acknowledgments}
Part of this project was realized during the semester \emph{Singularity Theory from Modern Perspectives} at Centre International de Rencontres Mathématiques in Marseille. We would like to thank CIRM for their hospitality and great working conditions. We would also like to thank Norbert A'Campo, Enrique Artal, Nero Budur,  Ailsa Keating, Ignacio Luengo, David Massey, Andr\'as N\'emethi, Adam Parusiński, Mar\'ia Pe Pereira, Patrick Popescu-Pampu, Edson Sampaio, and Duco van Straten for helpful discussions. We also thank the referees for valuable comments.

\section{Preliminaries}\label{sec:prelim}

In this section, we review some standard notation which used throughout the article; and recall those basic definitions from symplectic geometry which will be needed to introduce the fiberwise symplectic structure on the A'Campo space in Section \ref{sec:symplectic}. Further notions, needed to define symplectic monodromy and McLean spectral sequences, will be introduced in Sections \ref{sec:acobs} and \ref{sec:spectral-sequence}.

\subsection{Basic notation}\label{sec:notation}

We write $\log$ for the \emph{natural} logarithm.  For $r\geq 0$ we denote by $\B_{r}^{n}$, or simply $\B_{r}$, the closed Euclidean ball of radius $r$ in $\R^n$, and by $\D_{r} \subseteq \C$ the closed disk of radius $r$ in $\C$. We put $\D_{r}^{*}=\D_{r}\setminus \{0\}$. 

We denote by $\id_{X}$ the identity map $X\to X$; and by $\id_{k\times k}$ the $k\times k$ identity matrix. If $X$ or $k$ is clear from the context, we will sometimes write $\id$ for $\id_X$ or $\id_{k\times k}$. For a map $f\colon X\to X$, we denote its fixed point set by $\Fix f\de \{x\in X: f(x)=x\}$. We write $\pr_{X}$, $\pr_{Y}$ for the projections from a subset of $X\times Y$ to $X$ and $Y$, respectively.

Let $X$ be a topological space. For a subset $U\subseteq X$, we denote by $\bar{U}$ its topological closure, and by $\Int U$ (or $\Int_{X}U$) its interior. Given a sequence $(x_{\nu})_{\nu=1}^{\infty}$ of elements of $X$ (which we abbreviate by $(x_{\nu})\subseteq X$), we write $x_{\nu}\rightarrow x$ for the convergence $\limn x_{\nu}=x$. Similarly, for a subset $Z\subseteq X$ and continuous functions $h$ and $a$ on $X\setminus Z$ and $Z$, respectively, we say that \enquote{$h\rightarrow a$ on $Z$} if $h$ extends to a continuous function on $X$ which restricts to $a$ on $Z$. We denote this continuous extension by $h$, too.

For a smooth manifold $X$, we denote by $\Omega^{k}$ the sheaf of smooth $k$-forms, and put $\Omega^{*}=\bigoplus_{k\geq 0}\Omega^{k}$. Thus for an open subset $U\subseteq X$, the $\R$-algebra $\Omega^{0}(U)$ is the $\R$-algebra $\cC^{\infty}(U)$ of smooth functions on $U$. Whenever no confusion is likely to occur, we will identify a function with its restriction, e.g.\ by writing $\cC^{\infty}(V)\subseteq \cC^{\infty}(U)$ for an open subset $V\subseteq U$.

\subsection{Symplectic manifolds}\label{sec:symplectic-intro}

A $2$-form $\omega\in \Omega^2(M)$ is \emph{symplectic} if it is closed and nondegenerate, i.e.\ $T_{x}M\ni v\mapsto \omega(v,\sdot)\in T^{*}_x M$ is an isomorphism for all $x\in M$. A $1$-form $\lambda$ is \emph{Liouville} if $d\lambda$ is symplectic. Its \emph{Liouville vector field} $X_{\lambda}$ is defined by $d\lambda(X_{\lambda},\sdot)=\lambda$. We say that $(M,\lambda)$ is a \emph{Liouville domain} if $M$ is a compact connected manifold with boundary and $\lambda\in \Omega^{1}(M)$ is a Liouville form such that $X_{\lambda}$ points outwards $\d M$.

Let $J\in \operatorname{End}(TM)$ be an almost complex structure, i.e.\ $J^2=-\id$. We say that $\omega\in \Omega^{2}(M)$ \emph{tames} $J$ if $\omega(v,Jv)>0$ for all nonzero $v\in T_{x}M$, $x\in M$: in this case, $\omega$ is nondegenerate. We say that $\omega$ is $J$-compatible if it tames $J$ and $\omega(J\sdot,J\sdot)=\omega(\sdot,\sdot)$. Note that these definitions vary slightly among the literature; here we follow \cite[\sec 4.1]{McDuff_Salamon}.

Assume that $M$ is a complex manifold, and let $J$ be the multiplication by $\imath=\sqrt{-1}$. A $2$-form on $M$ is \emph{K\"ahler} if it is $J$-compatible and closed, hence symplectic. 

For a smooth function $\varrho\colon M\to \R$ we write $d^{c}\varrho=d\varrho\circ J$. Thus $d^{c}=\imath (\d-\bar{\d})$. Again, we warn the reader that the definition of $d^{c}$ varies among the literature, e.g.\ \cite[Definition 3.13]{Huybrechts} defines $d^{c}$ as $-\imath (\d-\bar{\d})$, and \cite[p.\ 109]{Griffiths-Harris} as $-\frac{\imath}{4\pi}(\d-\bar{\d})$. Our definition of $d^{c}$, adapted from \cite[\sec 2.2]{CE_from-Stein-to-Weinstein}, is used e.g.\ in \cite[\sec 3.1]{CNP_Milnor-open-books} in a context which is very similar to ours, see Example \ref{ex:typical}.% Our operator $d^{c}$ satisfies the identity $d^{c}=\imath (\d-\bar{\d})$, where $\imath=\sqrt{-1}$. However, we warn the reader that other conventions are present in the literature, too, e.g.\ \cite[Definition 3.13]{Huybrechts} defines $d^{c}$ as $-\imath (\d-\bar{\d})$, and \cite[p.\ 109]{Griffiths-Harris} as $-\frac{\imath}{4\pi}(\d-\bar{\d})$.} %For example, if $z=x+\imath y=r\theta$ are standard and polar coordinates on $\C^{*}$, we have $d^{c}x=dy$, $d^{c}y=-dx$ and $d^{c}r=rd\theta$, $d^{c}\theta=\tfrac{dr}{r}$. 

We say that a smooth function $\varrho\colon M\to\R$ is \emph{strictly plurisubharmonic}  if $-dd^{c}\varrho$ tames $J$, see \cite[Definition 7.4.7]{McDuff_Salamon}. In this case, the form $-dd^{c}\varrho$ is K\"ahler. A function $\varrho$ is \emph{exhaustive} if $\varrho^{-1}(-\infty,c]$ is compact for every $c\in \R$. If $\varrho$ is an exhaustive strictly plurisubharmonic function then  $\varrho^{-1}(-\infty,c]$ with $1$-form $-d^{c}\varrho$ is a Liouville domain. Indeed, the Liouville vector field $X_{-d^{c}\varrho}$ is the gradient field of $\varrho$ with respect to the metric $-dd^{c}\varrho(\sdot,J\sdot)$, see \cite[\sec 2.8]{CE_from-Stein-to-Weinstein}. Every Stein manifold admits an exhaustive strictly plurisubharmonic function, see \cite[\sec 5.3]{CE_from-Stein-to-Weinstein}.% for related results.

\begin{example}\label{ex:J}
	Let $M=\C^{*}$, and let $J(v)=\imath \cdot v$, i.e.\ $J$ is a rotation by $\pi$. Then in polar coordinates $z=r\cdot \exp(2\pi \imath\cdot \theta)$ we have $J\frac{\d}{\d r}= \frac{1}{r}\frac{\d}{\d \theta}$ and $J\frac{\d}{\d \theta}=-r\tfrac{\d}{\d r}$. Hence $d^{c}r=-r\, d\theta$ and $d^{c}\theta=\frac{1}{r}\, d r$. 
	
	Now if $M=\C^n$, $\varrho(z)=\frac{\pi}{2}\cdot ||z||^2=\frac{\pi}{2}\cdot\sum_{i}r_{i}^{2}$, then
	%	\begin{equation*}
	$\lambda\std\de -d^{c}\varrho =\pi\cdot\sum_{i}r_{i}^{2}d\theta_i$, 
	$\omega\std \de -dd^{c}\varrho=2\pi\cdot\sum_{i} r_{i}\, dr_{i}\wedge d\theta_{i}$ is the standard area form, and $X_{\lambda\std}=\pi\cdot\sum_{i}r_{i}\frac{\d}{\d r_{i}}=\nabla\varrho$ is the radial vector field.
	
	%	\end{equation*}
	
	Let $f\colon \C^n\to \C$ be a holomorphic function with critical value $0\in \C$, and $\epsilon>0$ be such that $\B_{\epsilon}\trans f^{-1}(0)$. Then for $0<\delta\ll \epsilon $, the \emph{Milnor fiber} $F\de f^{-1}(\delta)\cap \B_{\epsilon}$ is smooth, and $f^{-1}(\delta)\trans \B_{\epsilon}$, so $(F,\lambda\std|_{F})$, is a Liouville domain.
	
	For a more general example, replace $\C^{n}$ by a Stein space $Y$, smooth off $f^{-1}(0)$; $\B_{\epsilon}$ by a sub-level set of an exhaustive strictly plurisubharmonic function $\varrho$; and the standard Liouville form by $-d^{c}\varrho$.% or its small modification.
\end{example}

Let $f\colon M\to B$ be a submersion. A form $\omega\in \Omega^{2}(M)$ is \emph{fiberwise symplectic} if it is closed, and its restriction to each fiber of $f$ is nondegenerate (hence symplectic). A fiberwise symplectic form defines an Ehressmann connection $\{v\in TM: \forall_{w\in \ker f_{*}}\colon \omega(v,w)=0\}$, called a \emph{symplectic connection} \cite[\sec 6.3]{McDuff_Salamon}. Given a vector field on the base $B$, its lifting to the symplectic connection is called the {\em symplectic lift}. If the flow  the symplectic lift of a vector field in $B$ exists at a given time, then it induces symplectomorphisms between fibers, see \cite[Lemma 6.3.5]{McDuff_Salamon}.

\section{The A'Campo space: \enquote{Milnor fibration at radius zero}}\label{sec:ACampo}

Let $X$ be a complex manifold. Let $f\colon X\to \C$ be a holomorphic function with only one critical value $0\in \C$, such that the central fiber $f^{-1}(0)$ is snc. In this section, we recall the A'Campo \cite{A'Campo} model $A$ of the \enquote{radius zero} monodromy of $f$, see Figure \ref{fig:example}, and construct a smooth atlas on it.

After some preparation in Section \ref{sec:AX-def}, we introduce the topological space $A$ in Definition \ref{def:AX} by a hybrid procedure using tropical coordinates and Kato-Nakayama spaces at radius $0$, and classical geometry at positive radius. We list its basic properties in Section \ref{sec:AX-topo}. In Section \ref{sec:AX-C1}, we endow $A$ with a structure of a $\cC^1$-manifold with boundary, which depends only on $f$. In Section \ref{sec:AX-smooth} we improve this structure to a $\cC^{\infty}$ one, designed so that we can introduce a well-behaved fiberwise symplectic form \eqref{eq:omegaAC} in Section \ref{sec:symplectic}. This $\cC^{\infty}$ structure will depend on some additional data chosen along the construction.

In a topological setting, $A$ was introduced by A'Campo in \cite{A'Campo}, who used it to compute the zeta function of the monodromy. We note that loc.\ cit.\  predates the Kato--Nakayama construction \cite{Kato_Nakayama}, so instead of $X_{\log}$ it uses a real blowup. In turn, \enquote{tropical} (or \enquote{hybrid}) constructions have recently been used, for example in \cite{BJ_tropical, AN_hybrid, EM_Lagrangian-fibrations}, to define topological spaces related to our space $\Gamma$ in Definition \ref{def:AX}, together with some additional structure inherited from $f^{-1}(\delta)$ as $\delta\rightarrow 0$: e.g.\ a measure in \cite[3.4]{BJ_tropical} or a Lagrangian torus fibration in \cite{EM_Lagrangian-fibrations}.

\subsection{Construction of the A'Campo space}\label{sec:AX-def}

Let $X$ be a complex manifold of dimension $n$. Let $f\colon X\to \C$ be a holomorphic function, and let $D$ be the (scheme-theoretic) fiber $f^{-1}(0)$. Assume that $D\redd$ is snc. Since we are only interested in $f^{-1}(0)$ and nearby fibers, we assume $\log |f|<-1$ and that $f|_{X\setminus D}$ is a submersion onto $\D_{e^{-1}}^{*}$, so $0\in \C$ is its only critical value. Whenever no confusion is likely to occur, we identify divisors with their supports, e.g.\ by writing $X\setminus D$ for $X\setminus \operatorname{Supp} D$, i.e.\ for $f^{-1}(\D_{e^{-1}}^{*})$; or writing $x\in D$ instead of $x\in \operatorname{Supp} D$.

Let $D_1,\dots, D_N$ be the irreducible components of $D\de f^{-1}(0)$. Write
\begin{equation*}
	D=\sum_{i=1}^{N}m_{i}D_{i},
\end{equation*}
where each $m_{i}$ is a positive integer. Put $X_{\emptyset}=X\setminus D$, and for a nonempty subset $I\subseteq \{1,\dots, N\}$ put
\begin{equation}\label{eq:stratification}
	%X_{I}=\bigcap_{i\in I} D_{i},\quad 
	X_{I}^{\circ}=\bigcap_{i\in I} D_{i}\setminus \bigcup_{j\not \in I} D_{j}.
\end{equation}
Then $X=\bigsqcup_{I}X_{I}^{\circ}$ is a stratification of $X$.

\subsubsection{Charts adapted to $f$} In order to fix the notation, we recall some basic definitions. We say that $(U_X,\psi_X)$ is a \emph{holomorphic chart} for $X$ if $U_X\subseteq X$ is an open set, and $\psi_X\colon U_X\to \C^{n}$ is a biholomorphism onto its image. The map $\psi_X$ will often  be clear from the context, in which case we will abuse the language and use the term \emph{chart} to refer to $U_X$. We will usually keep the subscript \enquote{$X$} when referring to a chart in $X$, and drop it when taking its preimage in the A'Campo space.

A \emph{holomorphic atlas} is a collection $\cU_X=\{(U_X^{p},\psi_X^{p}):p\in R\}$ of holomorphic charts which cover $X$, i.e.\  $X=\bigcup_{p\in R}U^{p}_X$. 
A \emph{partition of unity inscribed in $\cU_X$} is a collection $\{\tau^{p}:p\in R\}$ of smooth functions $\tau^{p}\colon X\to [0,1]$ such that $\tau^{p}|_{X\setminus U_X^p}=0$, the set $\{p\in R:\tau^{p}(x)\neq 0\}$ is finite for every $x\in X$, and $\sum_{p\in R}\tau^{p}(x)=1$. 
We will use charts of the following form.

\begin{definition}\label{def:adapted-chart}
	Let $(U_X,\psi_X)$ be a holomorphic chart for $X$. Its \emph{associated index set} is 
	\begin{equation}\label{eq:index-set}
		S=\{i: D_i\cap \bar{U}_X\neq \emptyset\}\subseteq \{1,\dots, N\}.
	\end{equation}
	Write $\psi_X=(z_{i_1},\dots,z_{i_n})\colon U\to \C^{n}$ for some positive integers $i_1,\dots,i_n$. We say that the chart  $(U_X,\psi_X)$ is \emph{adapted to $f$} if the following conditions are satisfied.
	\begin{enumerate}[(i)]	
		\item\label{item:zero-locus} For every $i\in \{1,\dots, N\}$, we have $D_{i}\cap U_X=\{x\in U_X: z_{i}(x)=0\}$.
		\item\label{item:f-locally} If $S\neq \emptyset$ then the function $f$ restricts to $f|_{U_X}=\prod_{i\in S}z_{i}^{m_i}$.
		\item\label{item:r-small} For every $i\in S$, we have $\log |z_{i}|<-1$.
		\item\label{item:U-compact} The set $U_X$ has compact closure $\bar{U}_X$, and the map $\psi_X$ extends continuously to $\bar{U}_X$.
	\end{enumerate}
	A holomorphic atlas is \emph{adapted to $f$} if all its charts are. We say that a pair $(\cU_X,\btau)$ is \emph{adapted to $f$} if $\cU_X$ is a holomorphic atlas adapted to $f$, and $\btau$ is a partition of unity inscribed in $\cU_X$. 
\end{definition}
Note that, by \ref{item:zero-locus}, the indices of $z_{j}$'s are either in $S$ -- in which case they correspond to the fixed order of components of $D$ -- or are integers bigger than $N$. Although a bit unusual, this convention will make further notation much more convenient. We also note that by \ref{item:f-locally}, each component of $D$ meeting the closure $\bar{U}_X$ (i.e. whose index is in $S$) actually meets $U_X$.

Because $D\redd$ is snc, $X$ admits an atlas adapted to $f$. Indeed, the definition of snc implies that, after suitably indexing the variables, every point of $X$ lies in a chart satisfying  \ref{item:zero-locus}, \ref{item:f-locally}. Then, conditions  \ref{item:r-small}, \ref{item:U-compact} can be achieved by shrinking those charts.%the domains of the charts. %Clearly, any refinement of an atlas adapted to $f$ is adapted to $f$, too. 

\subsubsection{The basic functions}
\label{sec:basicfunctions}
We will use a continuous bijection $\eta\colon [0,1]\to [0,1]$, given by
\begin{equation}\label{eq:smooth_simplex}
	\eta(0)=0,\quad \eta(s)=(1-\log s)^{-1}\mbox{ for } s>0, 
\end{equation}
see Figure \ref{fig:smooth_simplex}. Note that the restriction $\eta|_{(0,1]}$ is smooth.  
Since we have assumed $\log |f|<-1$, we have smooth functions
\begin{equation}\label{eq:def-t-g}
	t\de (-\log|f|)^{-1}\colon X\setminus D\to [0,1),\qquad 
	g\de \eta(t)\colon X\setminus D\to [0,1).
\end{equation}
We extend them to continuous functions on $X$ by putting $t|_{D}=0$, $g|_{D}=0$. We have a smooth map
\begin{equation}\label{eq:def-theta}
	\theta\de \frac{f}{|f|}\colon X\setminus D\to \S^{1}.
\end{equation}

Fix a chart $(z_{i_1},\dots, z_{i_n})\colon U_X\to \C^{n}$ adapted to $f$. Fix $i$ in its index set $S=\{i:\bar{U}_X\cap D_i\neq\emptyset\}$, see \eqref{eq:index-set}. By  Definition  \ref{def:adapted-chart}\ref{item:r-small}, $\log|z_{i}|<-1$, so in particular $|z_{i}|<1$. We have radial coordinates
\begin{equation}\label{eq:def-r_i-t_i}
	r_{i}\de |z_{i}|\colon U_X \to [0,1),
	\quad\mbox{and}\quad
	t_{i}\de (-m_{i}\log r_{i})^{-1}\colon U_X\setminus D_i\to (0,1).
\end{equation}
We extend $t_i$ to a continuous function $U_X\to [0,\infty)$ by setting $t_{i}|_{U_X\cap D_{i}}=0$. We also have an angular coordinate:
\begin{equation}\label{eq:def-theta_i}
	\theta_{i}\de \frac{z_{i}}{|z_{i}|}\colon U_X \setminus D_i\to \S^1.
\end{equation}
We will compare the speeds of convergence towards $D_i$ using \enquote{tropical} functions:
\begin{equation}\label{eq:def-w_i-u_i}
	w_{i}\de \frac{t}{t_{i}}\colon U_X\setminus D_i\to [0,1],\qquad
	u_i\de \eta(w_i)\colon U_X\setminus D_i\to [0,1].
\end{equation}
Note that $\sum_{i\in S} w_{i}=1$, see Lemma \ref{lem:intro}\ref{item:intro-sum}. Eventually, we introduce a \enquote{hybrid} coordinate:
\begin{equation}\label{eq:def-v_i}
	v_{i}\de t_{i}-u_{i}\colon U_X\setminus D_i\to [-1,1).
\end{equation}
To simplify the notation, we introduce a map
\begin{equation}\label{eq:def-rest}
	\rest\de ((\theta_{i})_{i\in S},(z_{j})_{j\in \{i_1,\dots,i_n\}\setminus S}).
\end{equation}
It puts together those coordinates which have no \enquote{tropical} part: as we will see, they are easier to handle when defining a smooth structure on the A'Campo space $A$.

\begin{notation}\label{not:S1}
	Throughout the article, we will identify $\S^{1}$ with the \emph{additive} group $\R / \Z$. This way,  $\theta\in \S^1$ denotes the angle $2\pi \theta$, and we have  $\theta+1=\theta$.
\end{notation}

\subsubsection{The Kato--Nakayama space $X_{\log}$}\label{sec:KN}

As a first step of our construction, we will extend polar coordinates on $X$ to radius zero. This can be conveniently phrased in terms of the Kato--Nakayama construction \cite{Kato_Nakayama}. In general, it  associates to a log structure (i.e.\ a scheme with a sheaf of monoids, subject to some additional conditions), a topological space. We will use it only for the monoid sheaf $\cM\de \cO_{X}\cap i_{*}\cO_{X\setminus D}^{*}$, where $i\colon X\setminus D\into X$ is the inclusion. We follow  \cite[1.2]{Kato_Nakayama}.

For $x\in X$, let $\cM_{x}$ be the stalk of $\cM=\cO_{X}\cap i_{*}\cO_{X\setminus D}^{*}$ at $x$, so $\cM_{x}$ consists of holomorphic germs which are invertible off $D$. Let $\cM_{x}\gp$ be the group associated to the monoid $\cM_{x}$, i.e.\ $\cM_{x}\gp=\{ab^{-1} : a,b\in \cM_{x}\}$ with a natural identification $(ab^{-1})\cdot (ba^{-1})=1$. Let
\begin{equation*}
	X_{\log}=\{(x,h): x\in X,\ h\in \Hom(\cM_{X,x}\gp,\S^1),\ h(f)=\tfrac{f(x)}{|f(x)|} \mbox{ if } f\in \cO_{X,x}^{*}\},
\end{equation*}
and let $\tau\colon X_{\log}\to X$ be the first projection. To define a topology on $X_{\log}$, fix a chart $U_X$ adapted to $f$, see Definition \ref{def:adapted-chart}. Reordering the components of $D$ if needed, we can assume that its associated index set is $\{1,\dots, k\}$. Put $U_{\log}=\tau^{-1}(U_{X})\subseteq X_{\log}$. A \emph{Kato--Nakayama chart over $U_X$} is a  map $U_{\log}\to [0,1)^{k}\times (\S^1)^{k}\times \D^{n-k}$, given by 
\begin{equation}\label{eq:KN-chart}
	(x,h)\mapsto (r_1,\dots,r_k;\theta_1,\dots,\theta_k;z_{i_{k+1}},\dots,z_{i_n}), \quad \mbox{where}\quad r_i=|z_i|,\quad  \theta_i=h(z_i).
\end{equation}
Note that the coordinates $r_i$ and $\theta_{i}$ extend to $U_{\log}$ the pullbacks of the maps $r_i$, $\theta_i$ defined in \eqref{eq:def-r_i-t_i}, \eqref{eq:def-theta_i}. The topology on $(X,D)_{\log}$ is the coarsest one which makes these charts continuous. Clearly, $X_{\log}$ is a topological manifold with boundary, whose interior is homeomorphic with $X\setminus D$. 

Applying this construction to $(X,D)=(\C,0)$, we get
\begin{equation*}
	\C_{\log}\de (\C,0)_{\log}\cong [0,\infty)\times \S^1,
\end{equation*}
so $\C_{\log}$ extends the polar coordinates to a \emph{radius zero circle} $\{0\}\times \S^1$, cf.\  \cite[V.1.2]{Ogus_book}. The holomorphic function $f\colon X\to \C$ extends to a continuous map $f_{\log}\colon X_{\log}\to \C_{\log}$. In local coordinates \eqref{eq:KN-chart} over an adapted chart meting $D$, the map $f_{\log}$ is given by $(r_{1}^{m_1}\cdot\ldots\cdot r_{k}^{m_k},m_1\theta_1+\ldots+m_k\theta_k)$: recall from  Notation \ref{not:S1} that we write the angular coordinates additively. We get a commutative diagram
\begin{equation}\label{eq:KN-diagram}
	\begin{tikzcd}
		X_{\log}
		\ar[r]
		\ar[d, "f_{\log}"] 
		& X 
		\ar[d, "f"]  
		\\
		\C_{\log}
		\ar[r] & \C
	\end{tikzcd}
\end{equation}	

For a subset $I\subseteq \{1,\dots, N\}$ we define $(X_{I}^{\circ})_{\log}\subseteq X_{\log}$ as the preimage of the stratum $X_{I}^{\circ}\subseteq X$, see \eqref{eq:stratification}. In coordinates \eqref{eq:KN-chart}, $(X_{I}^{\circ})_{\log}$ is defined by $r_i=0$ for $i\in I$ and $r_j\neq 0$ for $j\not\in I$.%the zero locus of $r_{i}$ for $i\in I$ and $r_{i}\neq 0$ for $i\notin I$. 

\subsubsection{Definition of the A'Campo space $A$.}
\label{sec:Adef}
Fix a holomorphic atlas $\cU_{X}=\{(U^{p}_{X},\psi^{p}_{X}):p\in R\}$ adapted to $f$, and a partition of unity $\btau=\{\tau^{p}:p\in R\}$ inscribed in $\cU_{X}$. We write $u_i^{p}$, $v_{i}^{p}$ etc.\ for the corresponding maps from Section~\ref{sec:basicfunctions}. For convenience, if $U_{X}^{p}\cap D_{i}=\emptyset$ we set $u_{i}^{p}=0$, $v_{i}^{p}=0$.

Now, we 
%For $i\in \{1,\dots,N\}$, we put 
%$	
%R_{i}=\{p\in R: D_{i}\cap U^{p}_{X}\neq \emptyset\}
%$ 
%and 
define smooth maps $\bar{u}_i,\bar{v}_{i}\colon X\setminus D\to \R$, $\mu\colon X\setminus D\to \R^{N}$ by
\begin{equation}\label{eq:def-vbar-ubar-mu}
	\bar{u}_{i}\de \sum_{p\in R} \tau^{p}u_{i}^{p}, \quad 
	\bar{v}_{i}\de \sum_{p\in R}\tau^{p}v_{i}^{p}, \quad 
	\mu\de (\bar{u}_{1},\dots,\bar{u}_{N}). 
\end{equation}

\begin{definition}\label{def:AX}
	Let $X_{\log}$ be the Kato--Nakayama space associated to $(X,D)$, see Section \ref{sec:KN}. Let $\Gamma\subseteq X\times \R^{N}$ be the closure of the graph of $\mu$. The \emph{A'Campo space} $A$ is the topological space $A=\Gamma \times_{X} X_{\log}$, i.e.\ such that the top-left square of \eqref{eq:AX-diagram} is cartesian.
	\begin{equation}\label{eq:AX-diagram}
		\begin{tikzcd}
			A
			%\de (X,D)\AC 
			\ar[r] 
			\ar[d, %"\pi_{\log}"
			] 
			\ar[dd, bend right=90, looseness=1, "{(g,\theta)\sim f\AC}"'] 
			\ar[dr, phantom, "\ulcorner", very near start]
			\ar[dr, "\pi", start anchor = south east, shorten={5mm}]
			& \Gamma\de \overline{\mbox{\small{graph}}(\mu)} 
			\ar[d, %"\pi_{\Gamma}"
			] 
			\ar[r, hook]
			%			\ar[dr, %"\mu_{\Gamma}", 
			%			start anchor = south east] 
			& X\times  \R^{N} 
			\ar[d] 
			\\
			X_{\log}
			%\de (X,D\redd)_{\log} 
			\ar[r]%, "\tau"] 
			\ar[d, "f_{\log}"] 
			& X 
			\ar[r, dashed, "\mu", "{(\bar{u}_1,\ldots,\bar{u}_{N})}"'] 
			\ar[d, "f"] 
			& \R^{N} 
			\\
			\C_{\log}
			%\de (\C,0)_{\log} 
			\ar[r] & \C &
		\end{tikzcd}
	\end{equation}	
\end{definition}
Put $f\AC=(|f| \circ\pi,\theta)\colon A\to \C_{\log}=[0,\infty)\times \S^1$. It is a composition of the natural map $A\to X_{\log}$ with the map $f_{\log}$ induced by $f$ between the Kato--Nakayama spaces, see \eqref{eq:KN-diagram}. Fibers of $f\AC$ agree with those of $(t,\theta)$ and $(g,\theta)$. Figure \ref{fig:construction} in Example \ref{ex:monodromy} shows the fibers of $f_{A}$ and $f_{\log}$ over $\d \C_{\log}$.

We pull back the stratification \eqref{eq:stratification} of $X$ by putting
\begin{equation}\label{eq:stratification-pullback}
	A_{I}^{\circ}=\pi^{-1}(X_{I}^{\circ}).
\end{equation}
The partition $A=\bigsqcup_{I}A_{I}^{\circ}$ obtained this way is not a stratification anymore. 

We have $A\setminus \pi^{-1}(D)=A_{\emptyset}^{\circ}$ and $\pi^{-1}(D_i)=\bigsqcup_{i\in I}A_{I}^{\circ}$. Thus, for each chart $U_X$ adapted to $f$, the zero locus of the function $t_i$ introduced in \eqref{eq:def-r_i-t_i} is $\bigsqcup_{i\in I} A_{I}^{\circ}\cap U$, where $U=\pi^{-1}(U_X)$. The \enquote{deepest} subset of this partition, namely the one corresponding to the associated index set \eqref{eq:index-set}, equals $A_{S}^{\circ}\cap U=\{x\in U : \forall_{i\in S}\colon\, t_{i}(x)=0\}$ and is closed in $U$.

We put $\d A\de \pi^{-1}(D)=\bigsqcup_{I\neq \emptyset}A_{I}^{\circ}$. In the manifold with boundary structure that we will define in Section \ref{sec:AX-C1}, the set $\d A$ will indeed be the boundary of $A$. Note that $\pi|_{A\setminus \d A}\colon A\setminus \d A\to X\setminus D$ is a homeomorphism. Whenever no confusion is likely to occur, we will identify $A\setminus \d A$ with $X\setminus D$. In particular,  $A\setminus \d A$ has a \emph{natural smooth structure}, pulled back from $X\setminus D$ via $\pi|_{A\setminus \d A}$.

\begin{remark}
	The topological space $A$ introduced in Definition \ref{def:AX} agrees with the one originally constructed by A'Campo, see \cite[p.\ 238]{A'Campo}.
\end{remark}

\subsection{Elementary properties of the A'Campo space}\label{sec:AX-topo}

In this section, we list some simple properties of the A'Campo space $A$ introduced in Definition \ref{def:AX}. They are summarized in Lemma \ref{lem:intro}, which studies the lifts of basic functions introduced in Section \ref{sec:basicfunctions}; and in Proposition \ref{prop:AX-topo}, which describes the topological structure of $A$.

\subsubsection{Lifts of the basic functions}
\label{sec:basicfunctions-extensions}

The following notation, which we will keep throughout Section \ref{sec:ACampo}, will help to make our formulas lighter.
\begin{notation}\label{not:pullbacks}
	For a subset $U_X\subseteq X$, we denote a map from $U_X$ and its pullback to $U\de \pi^{-1}(U_X)$ by the same letter. This way, for every $i\in \{1,\dots, N\}$ we have continuous functions $\bar{u}_{i},\bar{v}_{i},g\colon A\to \R$ given by the formulas \eqref{eq:def-t-g} and \eqref{eq:def-vbar-ubar-mu}. 
	
	Let $U_X$ be a chart adapted to $f$. Pulling back the functions \eqref{eq:def-r_i-t_i}--\eqref{eq:def-v_i}, we get functions $r_i,t_{i}\colon U\to [0,1)$, $\theta_{i}\colon U\setminus \d A\to \S^1$ and  $w_{i},u_{i},v_i\colon U\setminus \d A\to \R$. In Lemma \ref{lem:intro}\ref{item:intro-extends} we will see that they extend continuously to $U$, and we will denote these extensions by the same letters, too.
	
	Whenever a chart $U_X$ is fixed, we will put $U=\pi^{-1}(U_X)$, denote its associated index set \eqref{eq:index-set} by $S$, write $k=\# S$ and use the maps $t_i,w_i$ introduced in Section \ref{sec:basicfunctions} without further emphasis on the dependence on $U_X$. If the chart is called $U_X'$ or $U_X^{p}$ for some index $p$, we will call the corresponding maps $t_{i}',w_{i}',\dots $ or $t_{i}^p,w_i^p,\dots$, respectively.
\end{notation}

In Section \ref{sec:AX-top-proofs} we will prove the following result.

\begin{lema}\label{lem:intro}
	Fix a chart $U_X$ adapted to $f$, not necessarily one from the atlas $\cU_X$ fixed in Definition \ref{def:AX}. Let $S$ be its associated index set \eqref{eq:index-set}, and let $U=\pi^{-1}(U_X)$. Assume $S\neq \emptyset$. The following hold.
	\begin{enumerate}
		\item\label{item:intro-extends} For every $i\in S$, the maps $t_{i},w_{i},u_{i},v_{i},\theta_{i}$ from Section \ref{sec:basicfunctions} extend to continuous maps on $U$.
		\item\label{item:intro-w-zero} Fix $i\in S$ and a nonempty subset $I\subseteq \{1,\dots, N\}$ such that $i\not \in I$. We have $w_{i}|_{A_{I}^{\circ}\cap U}\equiv 0$.
		\item\label{item:intro-sum} For every $i\in S$ we have $w_{i}\in [0,1]$, and $\sum_{i\in S} w_{i}=1$. 
		\item\label{item:intro-independence}% For $i\in S$, the values of $w_{i}$ on $\d A$ do not depend on the choice of the chart. 
		Let $U_{X}'$ be another chart adapted to $f$, let $U'=\pi^{-1}(U_X')$, and let $w_{i}'$ be the corresponding map from \eqref{eq:def-w_i-u_i}. Then on $\d A\cap U\cap U'$ we have $w_{i}'=w_{i}$ for all $i$ such that $U_X\cap U_X'\cap D_i\neq \emptyset$.
		\item\label{item:intro-u=ubar} For every $i\in \{1,\dots, N\}$ we have $\bar{u}_{i}|_{U\cap \d A}=u_{i}|_{U\cap \d A}$ if $i\in S$ and $\bar{u}_{i}|_{U\cap \d A}=0$ if $i\not\in S$.
	\end{enumerate}
\end{lema}

Very importantly for us, part \ref{item:intro-independence} expresses the idea that \enquote{at radius zero}, the values of $w_{i}$ do not depend on the choice of the chart. %We will prove this lemma later in Section \ref{sec:AX-topo}. 
Before entering in the technicalities of the proof, we proceed with the topological description of the A'Campo space.

\subsubsection{Rounded simplices}\label{sec:rounded_simplices}
To state the next result, we introduce the following additional notation. By $\Delta^{N-1}\subseteq \R^N$, we denote the image of the standard ($N-1$)-simplex through $(\eta,\ldots,\eta)$, i.e.
\begin{equation*}
	\Delta^{N-1}=\{(\eta(s_1),\dots,\eta(s_{N})): s_{i}\geq 0,\ \textstyle\sum_{i=1}^{N}s_{i}=1\}\subseteq \R^{N},
\end{equation*}	
see Figure \ref{fig:smooth_simplex}.
For $\emptyset\neq I\subseteq \{1,\dots, N\}$ we denote by $\Delta_{I}^{N-1}$ (or simply $\Delta_{I}$) the face $\{u_{j}= 0, j\not\in I\}\subseteq \Delta^{N-1}$ of $\Delta^{N-1}$ corresponding to $I$, and put $\Delta_{I}^{\circ}=\{u_{i}\neq 0, u_{j}=0: i\in I, j\not\in I\}\subseteq \Delta_{I}$. Hence e.g.\  $\Delta_{1}=\Delta_{1}^{\circ}=\{(1,0,\dots,0)\}$. Note that we identify a subscript $\{i\}$ with $i$. 

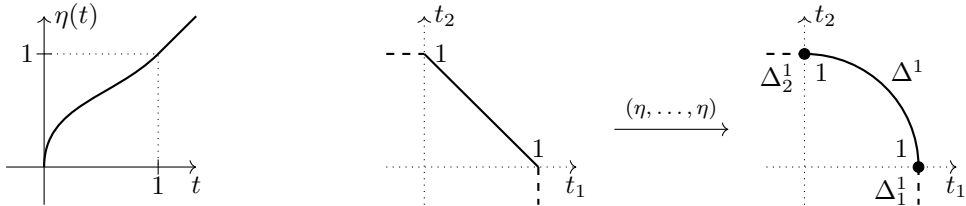
\begin{figure}[htbp]
	\centering
	\captionsetup{width=\linewidth}
	\begin{tikzpicture}
		\begin{scope}[shift={(-5,0)}]
			\path[use as bounding box] (-0.5,-0.5) rectangle (3,2);
			\draw[->] (-0.5,0) -- (2,0);
			\draw[->] (0,-0.5) -- (0,2);
			\draw[thick] (0,0) to[out=90,in=-135] (1.5,1.5) -- (2,2);
			\node [right] at (0,2) {\small{$\eta(t)$}};
			\draw [dotted] (0,1.5) -- (1.5,1.5) -- (1.5,0);
			\draw (-0.1,1.5) -- (0.1,1.5);
			\node [left] at (0,1.5)  {\small{$1$}};
			\draw (1.5,0.1) -- (1.5,-0.1);
			\node [below] at (1.5,0) {\small{$1$}};
			\node [below] at (2,0) {\small{$t$}};
		\end{scope}	
		\begin{scope}%[shift={(-5,0)}]
			\path[use as bounding box] (-0.5,-0.5) rectangle (3,2);	
			\draw[->, dotted] (-0.5,0) -- (2,0);
			\node [above] at (1.5,0) {\small{$1$}};
			\draw[->, dotted] (0,-0.5) -- (0,2);
			\node [right] at (0,1.5) {\small{$1$}};
			\draw [thick, dashed] (-0.5,1.5) -- (0,1.5); 
			\draw [thick] (0,1.5) -- (1.5,0);
			\draw [thick, dashed] (1.5,0) -- (1.5,-0.5);
			\node [right] at (0,2) {\small{$t_2$}};
			\node [below] at (2,0) {\small{$t_1$}};			
		\end{scope}
		\node [above] at (3.25,0.5) {\scriptsize{$(\eta,\ldots,\eta)$}};
		\draw [->] (2.5,0.5) -- (4,0.5);
		%	\ar[r, "{(\eta,\ldots,\eta)}", 
		\begin{scope}[shift={(5,0)}]
			\path[use as bounding box] (-1.5,-0.5) rectangle (2,2);
			\draw[->, dotted] (-0.5,0) -- (2,0);
			\node [above left] at (1.5,0) {\small{$1$}};
			\draw[->, dotted] (0,-0.5) -- (0,2);
			\node [below right] at (0,1.5) {\small{$1$}};
			\draw [thick, dashed] (-0.5,1.5) -- (0,1.5); 
			\draw [thick] (0,1.5) to[out=0,in=90] (1.5,0);
			\draw [thick, dashed] (1.5,0) -- (1.5,-0.5);
			\node [right] at (0,2) {\small{$t_2$}};
			\node [below] at (2,0) {\small{$t_1$}};	
			\node [above right]  at (1,1) {\small{$\Delta^{1}$}};
			\filldraw (1.5,0) circle (2pt);
			\node [below left] at (1.5,0) {\small{$\Delta^{1}_{1}$}};
			\filldraw (0,1.5) circle (2pt);
			\node [below left] at (0,1.5) {\small{$\Delta^{1}_{2}$}};		
		\end{scope}
	\end{tikzpicture}

	\caption{Function $\eta$ from \protect\eqref{eq:smooth_simplex} gives a convenient embedding of an ($N-1$)-simplex to $\R^{N}$}
	\label{fig:smooth_simplex}
\end{figure}

Let $X=\bigsqcup_I X_{I}$ be the stratification \eqref{eq:stratification} of $X$ associated to $D$. The \emph{dual complex} of $D$ is % (or $D\redd$) is
\begin{equation}\label{eq:dual-complex}
	%	\Delta_{D}\de %\Delta_{D\redd}=
	\bigcup_{\{I\neq\emptyset\ :\ X_{I}\neq \emptyset\}}\Delta_{I}\subseteq \Delta^{N-1}.
\end{equation}

\subsubsection{Topological structure.} 
The second main result of the current Section \ref{sec:AX-topo} is the following.

\begin{prop}\label{prop:AX-topo}
	Let $(\cU_X,\btau)$ be a pair adapted to $f$ and let $A$ be the A'Campo space corresponding to $(\cU_X,\btau)$. Then the following hold.
	\begin{enumerate}
		\item\label{item:top-pi} The maps $\pi\colon A\to X$, $\mu\colon A\to \R^N$ and $\bar{v}_i\colon A\to [-1,1]$ for $i\in \{1,\dots, N\}$ are continuous.
		\item\label{item:top-homeo-off-D} The map $\pi\colon A\to X$ is proper. Its restriction $\pi|_{A\setminus \d A}\colon A\setminus \d A\to X\setminus D$ is a homeomorphism.
		\item\label{item:top-S1-bundle} For every nonempty subset $I\subseteq \{1,\dots, N\}$, the restriction $(\pi,\mu)|_{A_{I}^{\circ}}\colon A_{I}^{\circ}\to X_{I}^{\circ}\times \Delta_{I}$ is a topological locally trivial fibration, with fiber $(\S^{1})^{\#I}$. Hence $\mu(\d A)$ is the dual complex of $D$.
		\item\label{item:top-d-independent} The subset $\d A \subseteq X_{\log}\times \R^{N}$
		does not depend on the choice of $(\cU_{X},\btau)$.
		\item\label{item:top-canonical-homeo} 
		Let $A'$ be the A'Campo space corresponding to another pair $(\cU_X',\btau')$ adapted to $f$. Using \ref{item:top-homeo-off-D} and \ref{item:top-d-independent}, we can define a map
		$
		\Phi\colon A\to A'
		$ 	
		by  		
		\begin{equation}\label{eq:canonical_homeo}
			\Phi|_{\d A}=\id_{\d A} \quad\mbox{and}\quad \Phi|_{A\setminus \d A}=\pi'|_{A'\setminus \d A}\circ \pi^{-1}|_{X\setminus D}.
		\end{equation}
		Then $\Phi$ is a homeomorphism.
	\end{enumerate}
\end{prop}

Later, we will show that $A$ is a topological manifold with boundary. We postpone this statement, since at the same time  we will endow $A$ with a $\cC^1$-structure; in such a way that the map \eqref{eq:canonical_homeo} becomes a $\cC^1$-diffeomorphism. In Section \ref{sec:AX-smooth}, we will even make $A$ a $\cC^\infty$-manifold with boundary. As we will see, the $\cC^\infty$ structure is harder to produce, and the map \eqref{eq:canonical_homeo} will not be a $\cC^\infty$-diffeomorphism in general.  

\begin{example}\label{ex:pictures}
	Let $X=\D^{2}_{\epsilon}$ for some $\epsilon\in (0,\frac{1}{e})$, and let $f=z_1z_2$. Consider $\cU_X$ consisting of just one chart $(z_1,z_2)$, and $\btau=\{1\}$. Then $(\cU_X,\btau)$ is adapted to $f$. The corresponding A'Campo space $A$ is the shaded part of Figure \ref{fig:v2}, multiplied by $(\S^{1})^{2}$. Similarly, for $X=\D^{3}_{\epsilon}$, $f=z_1 z_2 z_3$ and trivial $(\cU_X,\btau)$, the corresponding $A$ is Figure \ref{fig:v3} times $(\S^1)^{3}$.
	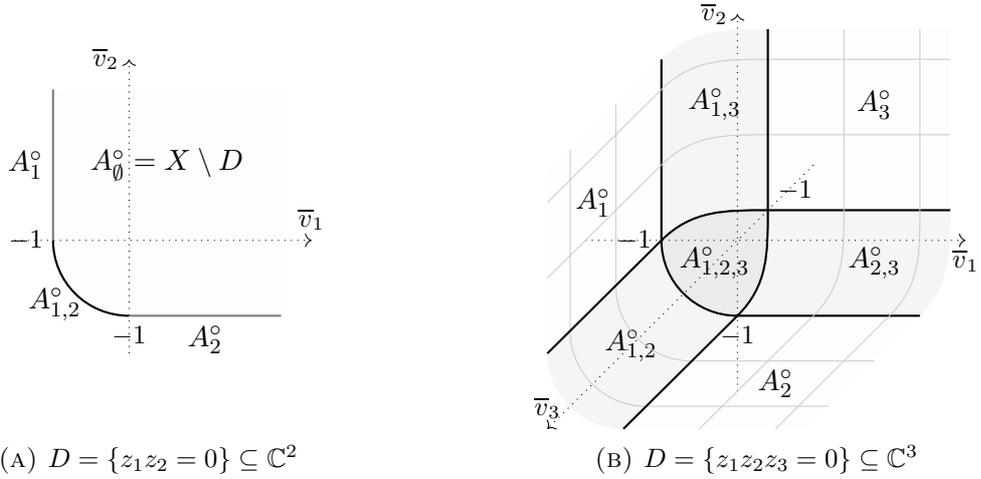
\begin{figure}[htbp]
		%	\centering
		\begin{tabular}{cc}
			\begin{subfigure}{0.45\textwidth}
				\centering			
				\begin{tikzpicture}
					\path[use as bounding box] (-2.5,-2.5) rectangle (3,3);
					\path [fill=black!1] (-1,2) -- (-1,0) to [out=270,in=180] (0,-1) -- (2,-1) -- (2,2) -- (-1,2);
					\draw[->, dotted] (-1.5,0) -- (2.4,0);
					\node [left] at (-1,0) {\small{$-1$}};
					\node [above] at (2.4,0) {\small{$\bar{v}_1$}};
					\draw[->, dotted] (0,-1.5) -- (0,2.4);
					\node [below] at (0,-1) {\small{$-1$}};
					\node [left] at (0,2.4) {\small{$\bar{v}_2$}};
					\draw [thick, gray] (-1,2) -- (-1,0); 
					\draw [thick] (-1,0) to [out=270,in=180] (0,-1);
					\draw [thick, gray] (0,-1) -- (2,-1);
					\node at (.5,1) {$A_{\emptyset}^{\circ}=X\setminus D$};
					\node [left] at (-1,1) {$A_{1}^{\circ}$};
					\node [below] at (1,-1) {$A_{2}^{\circ}$};
					\node [below left] at (-.5,-.5) {$A_{1,2}^{\circ}$};
				\end{tikzpicture}
				\caption{$D=\{z_1 z_2=0\}\subseteq \C^2$}
				\label{fig:v2}
			\end{subfigure}
			&
			\begin{subfigure}{0.45\textwidth}
				\centering			
				\begin{tikzpicture}
					\path[use as bounding box] (-2.5,-2.5) rectangle (3,3);
					\path [fill=black!8] (0.4,0.4) to [out=180,in=45] (-1,0) to [out=270,in=180] (0,-1) to [out=45,in=270] (0.4,0.4);
					\node at (-0.3,-0.3) {$A_{1,2,3}^{\circ}$};
					\path [fill=black!4] (0.4,0.4) to [out=180,in=45] (-1,0) -- (-1,2.4) to [out=45,in=180] (.4,2.8) -- (0.4,0.4);
					\node at (-.3,1.8) {$A_{1,3}^{\circ}$};
					\path [fill=black!4] (-1,0) to [out=270,in=180] (0,-1) -- (-1.5,-2.5) to [out=180,in=270] (-2.5,-1.5) -- (-1,0);
					\node at (-1.4,-1.4) {$A_{1,2}^{\circ}$};
					\path [fill=black!4] (0.4,0.4) to [out=270,in=45] (0,-1) -- (2.4,-1) to [out=45,in=270] (2.8,0.4) -- (-0.4,0.4);
					\node at (1.8,-0.3) {$A_{2,3}^{\circ}$};
					\path [fill=black!1] (0.4,0.4) --(0.4,2.8) -- (2.8,2.8) -- (2.8,0.4) -- (0.4,0.4);
					\node at (1.8,1.8) {$A_{3}^{\circ}$};
					\path [fill=black!1] (-1,0) -- (-1,2.4) -- (-2.5,0.9) -- (-2.5,-1.5) -- (-1,0);
					\node at (-1.9,0.5) {$A_{1}^{\circ}$};
					\path [fill=black!1] (0,-1) -- (2.4,-1) -- (0.9,-2.5) -- (-1.5,-2.5) -- (0,-1);
					\node at (0.5,-1.9) {$A_{2}^{\circ}$};				
					\draw[->, dotted] (-2,0) -- (3,0);
					\node [left] at (-1,0) {\small{$-1$}};
					\node [below] at (3,0) {\small{$\bar{v}_1$}};
					\draw[->, dotted] (0,-2) -- (0,3);
					\node [below] at (0,-1) {\small{$-1$}};
					\node [left] at (0,3) {\small{$\bar{v}_2$}};
					\draw[->, dotted] (1,1) -- (-2.5,-2.5);
					\node [above right] at (.4,.4) {\small{$-1$}};
					\node [above] at (-2.5,-2.5) {\small{$\bar{v}_3$}};
					\draw [thick] (-1,2.4) -- (-1,0) to [out=270,in=180]	(0,-1) -- (2.4,-1);
					\draw [black!20] (-1.6,1.8) -- (-1.6,-0.6) to [out=270,in=180] (-0.6,-1.6) -- (1.8,-1.6);
					\draw [black!20] (-2.2,1.2) -- (-2.2,-1.2) to [out=270,in=180] (-1.2,-2.2) -- (1.2,-2.2);
					\draw [thick] (2.8,0.4) -- (.4,.4) to [out=180,in=45] (-1,0) -- (-2.5,-1.5);
					\draw [black!20] (2.8,1.4) -- (.5,1.4) to [out=180,in=45] (-1,1) -- (-2.5,-0.5);
					\draw [black!20] (2.8,2.4) -- (.5,2.4) to [out=180,in=45] (-1,2) -- (-2.5,0.5);
					\draw [thick] (0.4,2.8) -- (.4,.4) to [out=270,in=45] (0,-1) -- (-1.5,-2.5);
					\draw [black!20] (1.4,2.8) -- (1.4,.5) to [out=270,in=45] (1,-1) -- (-0.5,-2.5);
					\draw [black!20] (2.4,2.8) -- (2.4,.5) to [out=270,in=45] (2,-1) -- (0.5,-2.5);
				\end{tikzpicture}
				\caption{$D=\{z_1 z_2 z_3 =0\}\subseteq \C^3$}				
				\label{fig:v3}
			\end{subfigure}
		\end{tabular}
		
		\caption{Radial coordinates $\bar{v}_{i}$ of the A'Campo space, cf.\ \protect\cite[p.\ 239]{A'Campo}.}
		\label{fig:v}
	\end{figure}
\end{example}

\subsubsection{Proofs of \ref{lem:intro}--\ref{prop:AX-topo}}\label{sec:AX-top-proofs}

The remaining part of Section \ref{sec:AX-topo} is devoted to the proof of Lemma~\ref{lem:intro} and Proposition~\ref{prop:AX-topo}. We remind the reader to use Notation~\ref{not:pullbacks}.

First, we give a precise formula relating $t_{i}$, $w_{i}$ with corresponding $t_{i}'$, $w_{i}'$ defined by another chart. We will use it both to prove the equality  $w_{i}|_{\d A}=w_{i}'|_{\d A}$ asserted in Lemma \ref{lem:intro}\ref{item:intro-independence}, and to make further computations in Lemma \ref{lem:dv}.

\begin{lema}\label{lem:w_extends} 
	Let $U_X,U'_X$ be two charts adapted to $f$. Put $V_X=U_X\cap U_X'$, $V=\pi^{-1}(U_X\cap U_X')$, and use Notation \ref{not:pullbacks}. For every $i\in S\cap S'$, the following hold.
	\begin{enumerate}
		\item\label{item:ti-continuous} The functions $t_{i}$, $t_{i}'$ are continuous on $V$.
		\item\label{item:lambda} There is a nonvanishing holomorphic function $\lambda \in \cO_{X}^{*}(V_X)$ such that $z_{i}'=\lambda \cdot z_{i}$.
		\item\label{item:t-comparison} Let $\lambda$ be as in \ref{item:lambda}, and let  $a =-m_i\log|\lambda|$. Then $a \in \cC^{\infty}(V_X)$, $1+at_{i}\neq 0$, and on $V$ we have:
		\begin{equation*}\label{eq:t-comparison}
			t_{i}'=t_{i}\cdot (1+a \cdot t_{i})^{-1}
			\quad \mbox{and} \quad
			w_{i}'=w_{i}\cdot (1+a \cdot t_{i}).
		\end{equation*}
		\item \label{item:w_independent-on-Ai} The functions $w_i$ and $w_i'$ extend to continuous functions on $V$, such that $w_{i}|_{V\cap \d A}=w_{i}'|_{V\cap \d A}$.
	\end{enumerate}
\end{lema}
\begin{proof}
	\ref{item:ti-continuous} Recall that $t_{i}$, $t_{i}'$ denote the pullbacks by $\pi$ of continuous functions $t_i$, $t_{i}'$ introduced in \eqref{eq:t-t'}. By Definition \ref{def:AX} of $A$, the map $\pi$ is continuous, so $t_{i},t_{i}'$ are continuous, too.
	
	\ref{item:lambda} By Definition \ref{def:adapted-chart}\ref{item:f-locally}, we have $D_{i}\cap V=\{z_i=0\}=\{z_i'=0\}$. 
	Hence $z_i=\lambda z_i'$ for some $\lambda\in \cO_{X}^{*}(V)$, as claimed.
	
	\ref{item:t-comparison} Since $\lambda$ is smooth and nonvanishing, the function $a=-m_i\log|\lambda|$ is smooth. By definition, on $V\cap \pi^{-1}(D_i)$ both  $t_i$ and $t_i'$ are zero, so $1+at_i\neq 0$ and  the first equality of \ref{item:t-comparison} holds trivially there. Away from $\pi^{-1}(D_i)$, we have $t_i, t_{i}'\neq 0$ and $t_{i}'=-(m_{i}\log|z_{i}'|)^{-1}$ by formula \eqref{eq:def-r_i-t_i}. Part \ref{item:lambda} and the definition of $a$ give
	\begin{equation*}
		t_{i}' \cdot (1+at_i)=-\frac{1+at_i}{m_{i}\log|\lambda z_{i}|}=
		-\frac{1+at_i}{m_{i}\log|z_{i}|+m_{i}\log|\lambda|}
		=-\frac{1+at_i}{-t_{i}^{-1}-a}=t_{i},
	\end{equation*}
	so $1+at_i\neq 0$ and $t_{i}'=t_i\cdot (1+at_i)^{-1}$, as needed.
	
	We will now prove that the equality $w_i'=w_i\cdot (1+at_i)$ holds on $V\setminus \pi^{-1}(D_i)$. There, substituting the equality $t_{i}'=t_i\cdot (1+at_i)^{-1}$ to the definition \eqref{eq:def-w_i-u_i} of $w_{i}'$, we get
	\begin{equation*}
		w_{i}'=t\cdot (t_{i}')^{-1}=t\cdot t_{i}^{-1} \cdot (1+at_{i})=w_{i} \cdot (1+at_i).
	\end{equation*}
	%as claimed. 
	Once we prove in \ref{item:w_independent-on-Ai} that $w_{i}$ and $w_{i}'$ are continuous, the above equality will extend to $\pi^{-1}(D_i)$.
	
	\ref{item:w_independent-on-Ai} As in the proof of \ref{item:ti-continuous}, continuity of $\pi$ implies that $w_i,w_i'$ are continuous on $V \setminus \pi^{-1}(D_i)$. Fix $x\in V\cap \pi^{-1}(D_i)$ and a sequence $(x_{\nu})\subset V\setminus \pi^{-1}(D_i)$ such that $\limn x_{\nu}=x$. Since $w_{i}(x_{\nu})\in [0,1]$, after passing to a subsequence we can assume that both $w_{i}(x_{\nu})$ and $w_{i}'(x_{\nu})$ have limits in $[0,1]$. Call these limits $w$ and $w'$, respectively. Since $t_i(x)=0$, applying the second formula in \ref{item:t-comparison} away from $\pi^{-1}(D_i)$ and passing to the limit, we get $w=w'$. Taking for $U_X'$ a chart $U_X^p$ of the atlas $\cU_X$ fixed in Definition \ref{def:AX}, and writing $w_{i}^{p}$ for the corresponding map, we get $\limn w_{i}^{p}(x_{\nu})=w$, for every $p\in R$ such that $x\in U^p$. Continuity of $\eta$ implies that $\limn u_{i}^{p}(x_{\nu})=\limn \eta(w_{i}^{p}(x_{\nu}))=\eta(w)$. Since each $\tau^{p}$ is continuous and 
	supported in $U_{X}^{p}$, we get $\limn \tau^{p}(x_{\nu})u_{i}^{p}(x_{\nu})=\tau^{p}(x)\eta(w)$ for all $p\in R$. 
	
	By Definition \ref{def:AX}, the function $\bar{u}_{i}$ is continuous. Hence %Since $\eta$ and each $\tau^{p}$ are continuous, too, we get 
	$\bar{u}_{i}(x)=\sum_{p} \limn \tau^{p}(x_{\nu})u_{i}^{p}(x_{\nu})=\sum_{p}\tau^{p}(x)\eta(w)=\eta(w)$. Bijectivity of $\eta$ gives $w=\eta^{-1}(\bar{u}_{i}(x))$. 
	
	Thus from any sequence $(x_{\nu})\subset V\setminus \pi^{-1}(D_i)$ such that $\limn x_{\nu}=x$, we can choose a subsequence $(x_{\nu'})$ such that $\lim_{\nu'\rightarrow \infty} w_{i}(x_{\nu'})=\eta^{-1}(\bar{u}_{i}(x))$. In fact, the same is true for all sequences $(x_{\nu})\subset V$, because by Definition \ref{def:AX} of $A$, every point of $V$ is a limit of a sequence of points in $V\setminus \d A$. It follows that the function  $w_i$ is continuous at $x$, and the value of $w_i$ at $x$ equals $\eta^{-1}(\bar{u}_{i}(x))$, so it does not depend on the choice of the chart $U_{X}$ defining $w_i$, as claimed. 
	
	It follows that the second formula in \ref{item:t-comparison} holds on $\pi^{-1}(D_i)$, too, which completes its proof.
\end{proof}

\begin{proof}[Proof of Lemma \ref{lem:intro}]
	\ref{item:intro-extends} Applying Lemma \ref{lem:w_extends}\ref{item:ti-continuous},\ref{item:w_independent-on-Ai} to $U_X$, we infer that $t_{i}$, $w_{i}$ extend to continuous functions on $U$. Since $\eta$ is continuous, the functions $u_{i}=\eta(w_{i})$ and $v_{i}=t_{i}-u_{i}$ are continuous, too. Eventually, $\theta_{i}$ pulls back to a continuous coordinate on the preimage of $U$ in the Kato--Nakayama space $X_{\log}$, see \eqref{eq:KN-chart}, so its further pullback to $A$ is continuous, too.
	
	\ref{item:intro-w-zero} Since $I\neq \emptyset$, we have $A_{I}^{\circ}\subseteq \d A$, so $t|_{A_{I}^{\circ}}=0$. For every $i\not \in I$, definition of $A_{I}^{\circ}$ gives $U\cap A_{I}^{\circ}\subseteq U\setminus \pi^{-1}(D_{i})$, so $t_{i}|_{U\cap A_{I}^{\circ}}>0$. By definition \eqref{eq:def-w_i-u_i} of $w_i$, on $U\cap A_{I}^{\circ}$ we have $w_{i}=t\cdot t_{i}^{-1}=0$.
	
	\ref{item:intro-sum} Because $U_X$ is adapted to $f$, by Definition \ref{def:adapted-chart}\ref{item:f-locally} we have  $|f|_{U}|=\prod_{i\in S} |z_{i}|^{m_i}$. Hence $\log |f|_{U}|=\sum_{i\in S} m_{i}\log |z_{i}|$. On $U\setminus \d A$, substituting the formulas \eqref{eq:def-t-g}, \eqref{eq:def-r_i-t_i} for $t$ and $t_i$, gives $-t^{-1}=-\sum_{i\in S} t_{i}^{-1}$, so $1=\sum_{i\in S} tt_{i}^{-1}=\sum_{i\in S} w_{i}$ by definition \eqref{eq:def-w_i-u_i} of $w_i$. Since by \ref{item:intro-extends} the functions $w_{i}$ are  continuous on $U$, we infer that the equality $1=\sum_{i\in S}w_{i}$ holds on $U\cap \d A$, too, as claimed.
	
	\ref{item:intro-independence} Put $V=U\cap U'$, and fix $i$ such that $D_i\cap V\neq \emptyset$. The set $V\cap \d A$ is a disjoint union of $V\cap \pi^{-1}(D_{i})$ and $V\cap A_{I}^{\circ}$ for all nonempty $I$ not containing $i$. By Lemma \ref{lem:w_extends}\ref{item:w_independent-on-Ai}, $w_{i}=w_{i}'$ on $V\cap \pi^{-1}(D_{i})$. In turn, by \ref{item:intro-w-zero}, both $w_{i}$ and $w_{i}'$ are zero on $V\cap A_{I}^{\circ}$ whenever $i\not \in I$. Thus $w_i=w_i'$ on the whole set $V\cap \d A$.
	
	\ref{item:intro-u=ubar} Fix $i\in S$. By \ref{item:intro-independence}, for every chart $U_{X}^{p}$ of the atlas $\cU_{X}$ fixed in Definition \ref{def:AX}, the corresponding function $w_{i}^{p}$ satisfies $w_{i}^{p}|_{U\cap U^{p}\cap \d A}=w_{i}|_{U\cap U^{p}\cap \d A}$. Hence  $\bar{u}_{i}=\sum_{p}\tau^{p}\eta(w_{i}^{p})=\sum_{p}\tau^{p}\eta(w_{i})=\eta(w_{i})=u_{i}$ on $U\cap \d A$, as claimed. 
	
	Now, fix $i\not\in S$, and $x\in U\cap \d A$. Since $i\not \in S$, we have $x\in A_{I}^{\circ}$ for some nonempty $I$ not containing $i$. Take any $p\in R$ such that $\tau^{p}(x)\neq 0$, so $x\in U^p$. Applying \ref{item:intro-w-zero} to $U^{p}$, we infer that $w_{i}^{p}(x)=0$, so $u_{i}^{p}(x)=0$. Therefore, $\bar{u}_{i}(x)=\sum_{p} \tau^{p}(x)u_{i}^{p}(x)=0$, as claimed.
\end{proof}

\begin{lema}
	\label{lem:Gamma_structure}
	For a nonempty subset $I\subseteq \{1,\dots, N\}$, the following hold.
	\begin{enumerate}
		\item \label{item:fiber-is-a-simplex} For every $x\in X_{I}^{\circ}$, we have  $\mu(\pi^{-1}(x))=\Delta_{I}\subseteq \Delta^{N-1}$. 
		\item\label{item:mu-proper} The projection $\Gamma\to X$ is proper.
		\item \label{item:trivial-fibration} The subset $A_{I}^{\circ}\subseteq X_{\log}\times \R^{N-1}$ equals $(X_{I}^{\circ})_{\log}\times \Delta_{I}$. 
	\end{enumerate}
\end{lema}
\begin{proof}
	\ref{item:fiber-is-a-simplex} Let $U_X$ be a chart adapted to $f$ containing $x$, let $S$ be its associated index set, and let $U=\pi^{-1}(U_X)$. Fix $y\in \pi^{-1}(x)\subseteq U\cap A_{I}^{\circ}$. For any $i\not\in S$, we have $\bar{u}_{i}(y)=0$ by Lemma \ref{lem:intro}\ref{item:intro-u=ubar}. Fix $i\in S$. By Lemma \ref{lem:intro}\ref{item:intro-u=ubar}, we have  $\bar{u}_{i}(y)=u_i(y)=\eta(w_i(y))$. Thus Lemma \ref{lem:intro}\ref{item:intro-sum} implies that $\mu(y)\in \Delta^{N-1}$. Moreover, if $i\not \in I$ then by Lemma \ref{lem:intro}\ref{item:intro-w-zero} we have $w_{i}(y)=0$, so $\bar{u}_{i}(y)=0$. We conclude that $\mu(y)\in \Delta_{I}$, so  $\mu(\pi^{-1}(x))\subseteq \Delta_{I}$. 
	
	For the reverse inclusion,  note that the subset $\mu(\pi^{-1}(x))\subseteq \Delta_{I}$ is closed. Indeed, it consists of all possible limits of $\mu(x_{\nu})$, where $(x_{\nu})\subset X\setminus D$ is a sequence converging to $x$. Therefore, it is sufficient to prove that $\Delta_{I}^{\circ}\subseteq \mu(\pi^{-1}(x))$. Fix $b\in \Delta_{I}^{\circ}$. Since the function $\eta|_{[0,1]}$ is bijective, for each $i\in I$ we get a unique $b_{i}\in (0,1]$ such that $\eta(b_{i})$ is the $i$-th coordinate of $b$; and $\sum_{i\in I} b_{i}=1$. Put $b_{i}=0$ for $i\not\in I$.
	
	Recall that $z_{i}$, for $i\in S$, are some of the coordinates of the chart $U_X$. Hence for sufficiently small $\epsilon>0$, there is a curve $\gamma\colon [0,\epsilon]\to U_X$ such that $z_{i}(\gamma(s))=s^{b_i/m_i}$ for all $s\in [0,\epsilon]$ and all $i\in S$. By Definition \ref{def:adapted-chart}\ref{item:f-locally}, we have $f(\gamma(s))=\prod_{i\in S} z_{i}(\gamma(s))^{m_i}=s^{\sum_{i\in S} b_{i}}=s$, so $t(\gamma(s))=-(\log s)^{-1}$ for $s>0$. By definition \eqref{eq:def-r_i-t_i} of $t_i$, for $s>0$ we have $t_{i}(\gamma(s))=-(m_i\log(s^{b_i/m_i}))^{-1}=-(b_i\log s)^{-1}$, and by \eqref{eq:def-w_i-u_i},  $w_{i}(\gamma(s))=t(\gamma(s))\cdot t_{i}(\gamma(s))^{-1}=b_{i}$. Therefore, $\mu(\gamma(s))=b$ for all $s\in (0,\epsilon]$. Passing with $s$ to $0$ we get $y\in A$ such that $\pi(y)=\gamma(0)=x$ and $\mu(y)=b$, as needed.

	\ref{item:mu-proper}, \ref{item:trivial-fibration} Follow immediately from \ref{item:fiber-is-a-simplex}.
\end{proof}

\begin{proof}[Proof of Proposition \ref{prop:AX-topo}]
	\ref{item:top-pi} Follows directly from Definition \ref{def:AX}.
	
	\ref{item:top-homeo-off-D} Properness follows from the properness of the maps $\Gamma\to X$ (see Lemma \ref{lem:Gamma_structure}\ref{item:mu-proper}) and $X_{\log}\to X$, and the rest is immediate by the definitions. 
	
	\ref{item:top-S1-bundle} Let $U_{X}\subseteq X$ be a chart adapted to $f$. Let $\tau\colon X_{\log}\to X$ be the natural map from the Kato--Nakayama space. In coordinates \eqref{eq:KN-chart} on $\tau^{-1}(U_X)$, the restriction of $\tau$ to $\tau^{-1}(U \cap X_{I}^{\circ})$ maps $\theta_{i}$ to $0$ for $i\in I$, and keeps the remaining coordinates intact. It follows that $\tau|_{\tau^{-1}(X_{I}^{\circ})}\colon \tau^{-1}(X_{I}^{\circ})\to X_{I}^{\circ}$ is a topological $(\S^{1})^{\#I}$-bundle. Thus the first part of \ref{item:top-S1-bundle} follows Lemma \ref{lem:Gamma_structure}\ref{item:trivial-fibration}. The second follows from the definition \eqref{eq:dual-complex} of the dual complex. 
	
	\ref{item:top-d-independent} By Lemma \ref{lem:Gamma_structure}\ref{item:trivial-fibration}, for every nonempty $I\subseteq \{1,\dots, N\}$, the subset $A_{I}^{\circ}\subseteq X_{\log}\times \R^{N}$ does not depend on the choice of $(\cU_{X},\btau)$. Therefore, neither does their disjoint union $\bigsqcup_{I\neq \emptyset}A_{I}^{\circ}=\d A$. 
	
	\ref{item:top-canonical-homeo} Let $\pi_{\log}\colon A\to X_{\log}$ be the natural map induced by $\pi$ in Definition \ref{def:AX}. Let $\pi_{\log}'$, $\mu'$, $\bar{u}_{i}'$ etc.\ denote the maps from Definition \ref{def:AX} associated to $A'$. By \ref{item:top-d-independent}, the subsets $\d A,\d A'\subseteq X_{\log}\times \R^{N}$ are equal, so the restriction $\Phi|_{\d A}$, defined as $\id|_{\d A}$, indeed maps $\d A$ to $\d A'$. Clearly, this restriction is continuous. In turn, $\Phi|_{A\setminus \d A}\colon A\setminus \d A\to A'\setminus \d A'$ is continuous by \ref{item:top-homeo-off-D}. Since $\d A$ is a closed subset of $A$, to prove continuity of $\Phi$ it is sufficient to prove that, for every $y\in \d A$, and every sequence $(y_{\nu})\subset A\setminus \d A$ converging to $y$, the values $y_{\nu}'\de \Phi(y_\nu)$ converge to $y$, too. 
	
	Since a point $y'\in A'$ is by definition a pair $(\pi'_{\log}(y'),\mu'(y'))$, we need to show that $\pi_{\log}'(y_{\nu}')\rightarrow \pi_{\log}(y)$ and $\mu'(y'_{\nu})\rightarrow \mu(y)$. By definition of $\Phi|_{A\setminus \d A}$, we have $\pi(y_{\nu})=\pi'(y_{\nu}')\in X\setminus D$, so $\pi_{\log}(y_{\nu})=\pi'_{\log}(y_{\nu}')$, because the map $X_{\log}\to X$ is a homeomorphism away from $D$. This proves the first convergence.
	
For the second one, choose a chart $U_X$ adapted to $f$ containing $\pi'(y_{\nu}')$ for $\nu\gg 1$, and fix $i$ in the associated index set \eqref{eq:index-set}. By Lemma \ref{lem:intro}\ref{item:intro-extends}, the function $u_i$ from \eqref{eq:def-w_i-u_i} extends to continuous functions $u_{i}\colon U\to [0,1]$ and $u_{i}'\colon U'\to [0,1]$, where $U=\pi^{-1}(U_X)$, $U'=(\pi')^{-1}(U_X)$. Put $x_\nu=\pi(y_{\nu})$, $x_{\nu}'=\pi'(y_{\nu}')$. Clearly, $x_{\nu}=x_{\nu}'$. Therefore, we have the equality 
\begin{equation*}
	\limn u_{i}'(y_{\nu}')=
	\limn u_{i}(x_{\nu}')=
	\limn u_{i}(x_{\nu})=
	\limn u_{i}(y_{\nu})=
	u_{i}(y)=
	\bar{u}_{i}(y),
\end{equation*}
where the last equality follows from Lemma \ref{lem:intro}\ref{item:intro-u=ubar}. Applying it to all charts $U_{X}$ of $\cU_{X}'$, we infer that $\bar{u}_{i}'(y_{\nu}')\rightarrow \bar{u}_{i}(y)$. Thus by definition \eqref{eq:def-vbar-ubar-mu} of $\mu$ we get $\mu'(y_{\nu}')\rightarrow \mu(y)$, as needed.
	
	Therefore, $\Phi$ is continuous. Continuity of the inverse follows by reversing the roles of $A$ and $A'$.
\end{proof}

\subsection{The $\cC^1$-atlas on the A'Campo space}\label{sec:AX-C1}

Fix a pair $(\cU_X,\btau)$ adapted to $f$, see Definition \ref{def:adapted-chart}. In this section we will introduce a $\cC^1$-atlas on the A'Campo space and we will show that if $(\cU'_X,\btau')$ is another pair adapted to $f$, then the homeomorphism \eqref{eq:canonical_homeo} connecting the corresponding A'Campo spaces is a $\cC^1$-diffeomorphism. As the reader may expect, this $\cC^1$-structure agrees outside $\partial A$ with the natural smooth structure in $A\setminus\partial A$, i.e.\ the one pulled back from $X\setminus D$ by the homeomorphism $\pi|_{A\setminus \d A}\colon A\setminus \d A\to X\setminus D$.%under its identification with $X\setminus D$. 

\subsubsection{Definition of the $\cC^1$-atlas}\label{sec:AX-C1-statements}
Fix a chart $U_X$ adapted to $f$, not necessarily one from the fixed atlas $\cU_X$. Put $U=\pi^{-1}(U_X)$, let $S=\{i: U_X\cap D_i\neq \emptyset\}$ be the associated index set, and let $w_i,v_i,\rest,\dots$ be the maps introduced in Section \ref{sec:basicfunctions}. Assume $S\neq \emptyset$. For every $i\in S$, we define 
\begin{equation}\label{eq:Ui}
	U_{i}=\{x\in U: w_{i}(x)>\tfrac{1}{n+1}\}.
\end{equation}
By Lemma \ref{lem:intro}\ref{item:intro-extends}, the subset $U_{i}\subseteq U$ is open. By Lemma  \ref{lem:intro}\ref{item:intro-sum}, the open sets $U_{i}$ for $i\in S$ cover $U$. Put $k=\#S$ and $Q_{k,n}\de [0,\infty)\times \R^{k-1}\times (\S^1)^{k}\times \C^{n-k}$. By Lemma \ref{lem:intro}\ref{item:intro-extends}, for every $i\in S$, we have a continuous map
\begin{equation}\label{eq:AC1-chart}
	\psi_{i}=(g,({v}_{j})_{j\in S\setminus \{i\}};\rest)\colon 
	U_{i}\to Q_{k,n},
\end{equation}
see \eqref{eq:def-t-g}, \eqref{eq:def-v_i} and \eqref{eq:def-rest} for definitions of $g$, $v_i$ and $\rest$, respectively. 

The pairs $(U_{i},\psi_{i})$ will be called the \emph{$\cC^{1}$-charts corresponding to} the chart $U_{X}$ adapted to $f$. The following lemmas, which we will prove later in this section, assert that these charts indeed make $A$ a $\cC^1$-manifold with boundary.

\begin{lema}
	\label{lem:C1chart}
	Each map \eqref{eq:AC1-chart} is a homeomorphism onto an open subset of $Q_{k,n}$. Moreover the restriction $\psi_i$ to $U_i\setminus\partial A$ is $\cC^\infty$ for the natural smooth structure in $A\setminus \partial A$.
\end{lema}

Lemma \ref{lem:C1chart} endows $A$ with a structure of topological manifold with boundary $\d A=\{g=0\}=\pi^{-1}(D)$. The $\cC^1$-structure is provided by Lemma \ref{lem:C1-trans} below, which studies the transition functions between charts \eqref{eq:AC1-chart}, possibly associated with  different charts adapted to $f$.

\begin{lema}\label{lem:C1-trans}
	Let $U_X$, $U'_X$ be charts adapted to $f$, let  $S$, $S'$ be their associated index sets \eqref{eq:index-set}, and let $k=\#S$, $k'=\#S'$. For $i\in S$ and $j\in S'$ let $(U_i,\psi_i)$, $(U_j',\psi_{j}')$ be the $\cC^{1}$-charts defined in \eqref{eq:AC1-chart} above. Then the transition map 
	\begin{equation*}
		\psi'_{j}\circ \psi_{i}^{-1}|_{\psi_{i}(U_i\cap U'_j)}\colon \psi_{i}(U_i\cap U'_j)\to \psi'_j(U_i\cap U'_j)	
	\end{equation*}
	is a $\cC^{1}$-diffeomorphism between open subsets of $Q_{k,n}$ and $Q_{k',n}$.
\end{lema}

	\begin{figure}[htbp]
\begin{tabular}{cc}
\begin{subfigure}{0.45\textwidth}
	\centering			
			\begin{tikzpicture}
%				\path[use as bounding box] (-2.5,-2.5) rectangle (3,3);
		\path [fill=black!5] (-0.6,2) -- (-0.6,-1.85) 
		to [out=170,in=270] (-2,0) -- (-2,2) -- (-0.6,2);
		\path [fill=black!5] (2,-0.6) -- (-1.85,-0.6) 
		to [out=-80,in=180] (0,-2) -- (2,-2) -- (2,-0.6);
		\path [fill=black!10] (-1.85,-0.6) 
		to [out=-80,in=170] (-0.6,-1.85) -- (-0.6,-0.6) -- (-1.85,-0.6);
				\draw[dashed] (-0.6,2) -- (-0.6,-1.85);
				\draw[dashed] (2,-0.6) -- (-1.85,-0.6);
				\draw[->, dotted] (-2.5,0) -- (2.4,0);
				%\node [left] at (-2,0) {\small{$-1$}};
				\draw[->] (-2.2,0.1) -- (-2.2,1);
				\node at (-2.8,0.6) {\small{$r_2>0$}};
				\filldraw (-2,0) circle (2pt);
				\node at (-1.8,0.2) {\tiny{$-1$}}; 
				\draw[->] (-2.2,-0.1) to [out=270,in=130] (-1.9,-1.1);
				\node at (-2.8,-0.6) {\small{$w_2>0$}};
				\draw[->] (0.1,-2.2) -- (1,-2.2);
				\node at (0.7,-2.5) {\small{$r_1>0$}};
				\node at (-0.2,-1.8) {\tiny{$-1$}}; 
				\filldraw (0,-2) circle (2pt);
				\draw[->] (-0.1,-2.2) to [out=180,in=-40] (-1.1,-1.9);
				\node at (-0.7,-2.5) {\small{$w_1>0$}};				
				\node [above] at (2.4,0) {\small{${v}_1$}};
				\draw[->, dotted] (0,-2.5) -- (0,2.4);
				%\node [below] at (0,-2) {\small{$-1$}};
				\node [left] at (0,2.4) {\small{${v}_2$}};
				\draw [thick, gray] (-2,2) -- (-2,0); 
				\draw [thick] (-2,0) to [out=270,in=180] (0,-2);
				\draw [thick, gray] (0,-2) -- (2,-2);
				%
		%		\node at (.5,1) {$A_{\emptyset}^{\circ}=X\setminus D$};
				\node at (-1.3,1) {$U_{1}$};
				\node at (1,-1.3) {$U_{2}$};
				\node at (-2.3,1.7) {$A_{1}^{\circ}$};
				\node at (1.8,-2.3) {$A_{2}^{\circ}$};
			%	\node [below left] at (-.5,-.5) {$A_{1,2}^{\circ}$};
				%
			\end{tikzpicture}
%					\caption{$D=\{z_1 z_2=0\}\subseteq \C^2$}
%	\label{fig:v2}
\end{subfigure}
&
\begin{subfigure}{0.45\textwidth}
\centering			
\begin{tikzpicture}
%	\path[use as bounding box] (-2.5,-2.5) rectangle (3,3);
	%
	\path [fill=black!5] (-1,2.5) -- (-1,0) to [out=270,in=180]	(0,-1) -- (2.5,-1)   
to [out=70,in=-90] (2.8,0.4) -- (2.8,2.8) -- (0.4,2.8) to [out=180, in=20] (-1,2.5);	
	\draw[->, dotted] (-2,-0.15) -- (3,-0.15);
	\node [below] at (3,-0.15) {\small{$v_1$}};
	\draw[->, dotted] (-0.15,-2) -- (-0.15,3);
	\node [left] at (-0.15,3) {\small{${v}_2$}};
	\draw[->, dotted] (1.4,1.4) -- (-2,-2);
	\node [above] at (-2,-2) {\small{$v_3$}};
	\node at (-1.8,0) {\tiny{$-1$}};
	\node at (-0.3,-1.4) {\tiny{$-1$}};
	\node at (0.5,1) {\tiny{$-1$}};
	\draw (-1.55,2.4) -- (-1.55,0) to [out=270,in=180]	(0,-1.55) -- (2.4,-1.55);
	\draw [dashed] (-1,2.5) -- (-1,0) to [out=270,in=180]	(0,-1) -- (2.5,-1) -- (2.5,2.5) -- (-1,2.5);
	\node at (-0.6,2.2) {$U_3$};

	\draw [thick] (2.8,0.8) -- (.4,0.8) to [out=180,in=45] (-1,0.4) -- (-2.4,-1);
	\draw [thick] (0.8,2.8) -- (0.8,.4) to [out=270,in=45] (0.4,-1) -- (-0.9,-2.4);
	\filldraw (-0.15,-1.55)  circle (1.5pt);
	\filldraw (-1.55,-0.15) circle (1.5pt);
\node at (1.8,1.8) {$A_{3}^{\circ}$};
\node at (-2.1,0.7) {$A_{1}^{\circ}$};
\node at (0.7,-2.1) {$A_{2}^{\circ}$};	
\node[draw] at (4.4,-1.2) {\small{$w_1=0$}};
\draw[->] (3.6,-1.2) to[out=180,in=0] (0.9,-0.3);
\node[draw] at (4.4,1.6) {\small{$w_2=0$}};
\draw[->] (3.6,1.6) to[out=180,in=90] (1.8,0.95);	
\node[draw,align=center] at (4.4,0.2) {\small{$w_1=0$} \\ \small{$w_2=0$}};
\filldraw (0.8,0.8) circle (2pt);
\draw[->] (3.6,0.2) to[out=180,in=-45] (0.9,0.7);	
	
\end{tikzpicture}
%\caption{$D=\{z_1 z_2 z_3 =0\}\subseteq \C^3$}				
%\label{fig:v3}
\end{subfigure}
\end{tabular}
	\caption{Coordinate charts \eqref{eq:AC1-chart} of the A'Campo space, cf.\ Figure \protect\ref{fig:v}.}
\label{fig:charts}
\end{figure}

Since for every $k\in \{1,\dots, n\}$, the space $Q_{k,n}$ is a $2n$-manifold with boundary, refining the $\cC^1$-charts corresponding to each adapted chart $U_X$, we obtain an honest $\cC^1$-atlas for a neighborhood of $\partial A$, whose charts have images in half-euclidean spaces. In order to have a $\cC^1$-atlas of the whole $A$, we notice that by Lemma~\ref{lem:C1chart} this $\cC^1$-structure agrees with the natural smooth structure in $A\setminus\partial A=X\setminus D$.  

Proposition \ref{prop:uniqueC1} below allows to speak about \emph{the A'Campo space of $f$}, as a $\cC^1$-manifold with boundary. The key properties of this space are listed in Proposition \ref{prop:AXC1}.

\begin{prop}
	\label{prop:uniqueC1}
	Fix two pairs $(\cU_X,\btau)$ and $(\cU'_X,\btau')$ adapted to $f$. Let $A$ and $A'$ be the corresponding A'Campo spaces, each endowed with the $\cC^1$-structure defined above. Then the homeomorphism $\Phi:A\to A'$ introduced in Proposition \ref{prop:AX-topo}\ref{item:top-canonical-homeo}  is a $\cC^1$ diffeomorphism.
\end{prop}

\begin{prop}\label{prop:AXC1}
	%Let $(\cU_{X},\btau)$ be a pair adapted to $f$, and l
	Let $A$ be the A'Campo space of $f$, with the $\cC^1$-structure defined above.
	\begin{enumerate}
		\item\label{item:AX-piC1} The map $\pi\colon A\to X$ is $\cC^1$. Its restriction $\pi|_{A\setminus \d A}\colon A\setminus \d A\to X\setminus D$ is a $\cC^1$-diffeomorphism.
		\item\label{item:AX-gC1} Let $g,\theta$ be as in \eqref{eq:def-t-g}, \eqref{eq:def-theta}. Then the map $(g,\theta)\colon A\to [0,1)\times \S^{1}$ is a $\cC^1$-submersion. In particular, the map $f\AC\colon A\to \C_{\log}$ extending $f$ in \eqref{eq:AX-diagram} is $\cC^1$.
		\item\label{item:AX-vbarC1} For every $i\in \{1,\dots, N\}$, the function $\bar{v}_i:A\to[-1,1]$ associated in \eqref{eq:def-vbar-ubar-mu} to the pair $(\cU_{X},\btau)$ adapted to $f$, is $\cC^1$.
	\end{enumerate}
\end{prop}

\begin{example}\label{ex:not-C2}
	The $\cC^1$-structure on $A$ constructed above may not be $\cC^2$. Indeed, let $X=\D_{\epsilon}^2$ for some $\epsilon\in (0,\frac{1}{e})$, and let $f=z_1z_2$. Consider two charts $U_X$, $U_X'$ adapted to $f$: one with the standard coordinates $(z_1,z_2)$, and the other with coordinates $(z_1',z_2')$ given by
	\begin{equation*}
		(z_1',z_2')=(e^{-1} \cdot z_1, e \cdot z_2).
	\end{equation*}
	Let $(U_1,\psi_1)$ and $(U_1',\psi_1')$ be the corresponding $\cC^1$-charts. We will show that the transition map $\psi_{1}'\circ\psi_{1}^{-1}\colon (g,v_2,\theta_1,\theta_2)\mapsto (g,v_2',\theta_2,\theta_2')$ is not $\cC^2$ along  $\{t_2=0,u_2=0\}$. 
	
	Using definitions \eqref{eq:def-r_i-t_i},  \eqref{eq:def-w_i-u_i} of $t_2$ and $u_2$, we compute
	\begin{equation*}
		\begin{split}
			t_2'&=-\frac{1}{\log|z_{2}'|}=-\frac{1}{\log|z_2|+1}=-\frac{1}{-t_2^{-1}+1}=\frac{t_{2}}{1-t_2},\\
			u_2'&=\eta\left(\frac{t}{t_2'}\right)=\frac{1}{1-(\log\frac{t}{t_2}+\log(1-t_2))}=\frac{1}{u_2^{-1}-\log(1-t_2)}=\frac{u_2}{1-u_2\log(1-t_2)},
		\end{split}
	\end{equation*}
	so $t_{2}'-t_{2}=\frac{t_{2}^{2}}{1-t_2}$ and $u_2'-u_2=\frac{u_{2}^{2}\log(1-t_2)}{1-u_2\log(1-t_2)}$, cf.\ Lemmas \ref{lem:w_extends}\ref{item:t-comparison} and \ref{lem:dv}\ref{item:vp-v}. Substituting these equalities to the definition \eqref{eq:def-v_i} of $v_2$, we get
	\begin{equation}\label{eq:ex-not-C2}
		v_{2}'=v_{2}+\frac{t_{2}^{2}}{1-t_{2}}-\frac{u_2^2 \log(1-t_2)}{1-u_2\log (1-t_2)}.
	\end{equation}
	Now, we will approach $\{t_2=0,u_2=0\}$ along $\{t_2=0,u_2>0\}=\Int_{\d A}A_{1,2}^{\circ}$ and $\{t_2>0, u_2=0\}=A_{1}^{\circ}$. In {Figure \ref{fig:v2}}, these are, respectively, the circle-arc and the bold vertical line. If $t_{2}=0$, $u_2>0$ then \eqref{eq:ex-not-C2} reads as $v_2'=v_2$, so $(\frac{\d}{\d v_{2}})^2 v_2'=0$. In turn, if $t_2> 0$, $u_2=0$ then $v_2=t_2$, so \eqref{eq:ex-not-C2} reads as $v_2'=v_2+\frac{v_2^2}{1-v_2}=v_2+v_2^2+v_2^3+\dots$, and therefore $(\frac{\d}{\d v_{2}})^2 v_2'=1$. This shows that the transition map $v_2\mapsto v_2'$ is not $\cC^2$, as claimed.
\end{example}

In the remaining part of Section \ref{sec:AX-C1}, we prove Lemmas \ref{lem:C1chart}, \ref{lem:C1-trans} and Propositions \ref{prop:uniqueC1}, \ref{prop:AXC1} stated above. First, in Section \ref{sec:top-atlas} we prove Lemma \ref{lem:C1chart}, which makes $A$ a topological manifold with boundary. To prove that this manifold is $\cC^1$, we will need some preparatory results, which will be useful for the construction of a $\cC^\infty$ atlas, too. In Section \ref{sec:C1-computation}, we gather some handy identities. Next, in Section \ref{sec:bounded-forms} we prove that, as we approach $\d A$, the $\cC^1$-charts corresponding to different charts adapted to $f$ become $\cC^1$-close to each other. To state it in a precise way, we introduce a technical notion of differential forms \emph{bounded from $X$}. With these preparations at hand, we will prove Lemma \ref{lem:C1-trans} and Propositions \ref{prop:uniqueC1}, \ref{prop:AXC1} in Section \ref{sec:C1-proof}.

\subsubsection{The A'Campo space is a topological manifold}\label{sec:top-atlas}
In this section, we prove Lemma \ref{lem:C1chart}, which implies that the formulas \eqref{eq:AC1-chart} define topological charts. We begin with two simple observations.

\begin{lema}\label{lem:Ui-simple}
	Let $U_X$ be a chart adapted to $f$. Let $S=\{i:U_X\cap D_i\neq \emptyset\}$ be the associated index set \eqref{eq:index-set}, and for $i\in S$ let $U_i=\{w_i>\frac{1}{n+1}\}$ be the open subset of $U\de \pi^{-1}(U_X)$ defined in \eqref{eq:Ui}. Then 
	\begin{enumerate}
		\item\label{item:Ui-in-Ai} The intersection $U_{i}\cap \d A$ is an open subset of $\pi^{-1}(D_i)$.
		\item\label{item:Ui-cap-dA} Fix $x\in U_i\cap \d A$. Then $x\in A_{I}^{\circ}$ for some $I\subseteq S$ containing $i$. Moreover, for every $j\in S$ we have $v_{j}(x)=-u_{j}(x)\leq 0$ if $j\in I$ and $v_j(x)=t_j(x)>0$ if $j\not \in I$.
		\item\label{item:vi-smooth-on-strata} Fix $i\in S$ and $I\subseteq S$ not containing $i$. Then the restrictions of $t_i,w_i,u_i$ and $v_i$ to $A_{I}^{\circ}\cap U$ are pullbacks of smooth functions on $X_{I}^{\circ}\cap U_{X}$.
	\end{enumerate}
	
\end{lema}
\begin{proof}
	\ref{item:Ui-in-Ai} By Lemma \ref{lem:intro}\ref{item:intro-extends}, the function $w_i$ is continuous, so $U_i=\{w_i>\frac{1}{n+1}\}$ is an open subset of $A$. To see that $U_i\cap \d A\subseteq \pi^{-1}(D_i)$, fix $x\in U_i\cap\d A$. Since $w_{i}(x)>0$, the formula \eqref{eq:def-w_i-u_i} implies that $t_{i}(x)=t(x)\cdot w_{i}(x)^{-1}$. If $x\in \d A$ then $t(x)=0$, so we get $t_i(x)=0$, i.e.\ $x\in \pi^{-1}(D_i)$, as claimed.
	
	\ref{item:Ui-cap-dA} By \ref{item:Ui-in-Ai}, we have $i\in I$. If $j\in I$ then by definition \eqref{eq:def-r_i-t_i} of $t_j$ we have $t_j(x)=0$, so $v_j(x)=-u_j(x)\leq 0$ by \eqref{eq:def-v_i}. If $j\not\in I$ then $u_j(x)=0$ by Lemma \ref{lem:intro}\ref{item:intro-w-zero}, so $v_j(x)=t_j(x)>0$.
	
	\ref{item:vi-smooth-on-strata} For $I=\emptyset$ we have $A_{I}^{\circ}\cap U=U\setminus \d A=U_X\setminus D$, where the assertion is clear. Assume $I\neq \emptyset$. Since $i\not \in I$, by Lemma \ref{lem:intro}\ref{item:intro-w-zero} the restrictions to $A_{I}^{\circ}\cap U$ of $w_i$, and hence of $u_i$, are zero. Thus $v_i$ restricts to $t_i$, which is a smooth function on $U_X\setminus D_i$, in particular on $X_{I}^{\circ}\cap U_X$.
\end{proof}

The function defining the \enquote{hybrid} coordinate $v_i$ in \eqref{eq:def-v_i} has the following useful property.

\begin{lema}\label{lem:v-check}
	Fix a positive number $\check{t}>0$ and consider a function $\check{v}_{\check{t}}\colon [\check{t},\infty)\ni s\mapsto s-\eta(\check{t}\cdot s^{-1})\in \R$. %Then
	\begin{enumerate}
		\item \label{item:v-increasing} The function $\check{v}_{\check{t}}$ is strictly increasing. Its image equals $[\check{t}-1,\infty)$.
		\item \label{item:v=vi} Let $U_X$ be a chart adapted to $f$, and for $i$ in the associated index set \eqref{eq:index-set} let $v_i$ be the function defined in \eqref{eq:def-v_i}. Then $v_{i}(x)=\check{v}_{t(x)}(t_i(x))$ for every $x\in U_X\setminus D$.
	\end{enumerate}
\end{lema}
\begin{proof}
	For $s\in [\check{t},\infty)$ we have $\check{t}s^{-1}\in (0,1]$, so the function $\check{v}_{\check{t}}$ is well defined. Its derivative is $\check{v}_{\check{t}}'(s)=1+\eta'(\check{t}s^{-1})\cdot \check{t}s^{-2}>1$, since $\eta$ is strictly increasing. Thus $\check{v}_{\check{t}}$ is strictly increasing, too. Since  $\check{v}_{\check{t}}(\check{t})=\check{t}-\eta(1)=\check{t}-1$, this proves \ref{item:v-increasing}. Part \ref{item:v=vi} follows directly from the definition \eqref{eq:def-v_i} of $v_i$.
\end{proof}

Recall that on $A\setminus \d A$ we have the natural smooth structure pulled back from $X\setminus D$. The following lemma asserts that the candidate charts \eqref{eq:AC1-chart} are compatible with this structure.%it in the following sense.

\begin{lema}\label{lem:coordinates-off-D}
	Let $U_{X}$ be a chart adapted to $f$, and let $\psi_{i}\colon U_i\to Q_{k,n}$ be an associated chart \eqref{eq:AC1-chart}.  Then the restriction $\psi_{i}|_{U_i\setminus \d A}$ is a diffeomorphism onto its image.%, with respect to the natural smooth structure on $U_i\setminus \d A$ .%, which is an open subset of $Q_{k,n}$.
\end{lema}
\begin{proof}
	Reordering the components of $D$ if needed, we can assume that the index set associated to $U_X$ is $\{1,\dots, k\}$, and our fixed chart is $\psi_{1}=(g,v_2,\dots,v_k,\rest)$. 
	
	Recall from Definition \ref{def:adapted-chart} that $(z_{i_1},\dots,z_{i_n})$ is a holomorphic coordinate system on $U_X$, and $D\cap U_X$ is the zero locus of $f|_{U_X}=\prod_{i\in S} z_{i}^{m_i}$. Therefore, the map $U_X\setminus D\to Q_{k,n}$ given by polar coordinates $(r_1,\dots,r_k,\rest)$ is a diffeomorphism onto its image.
	Since by definition \eqref{eq:def-r_i-t_i} we have $t_i=-(m_i\log r_i)^{-1}$, the same is true for the map $(t_1,\dots,t_k,\rest)$. Thus $(t_1,\dots,t_k,\rest)$ is a smooth coordinate system on $U_1\setminus \d A$, for the smooth structure inherited from $X$.
	
	We claim that the map $\varphi\de(t,t_2,\dots,t_k,\rest)\colon U_1\setminus \d A\to Q_{k,n}$ is a diffeomorphism onto its image as well. Recall that $|f||_{U_X}=\prod_{i=1}^{k}r_{i}^{m_i}$, $t=-(\log|f|)^{-1}$ and $t_i=-(m_i\log r_i)^{-1}$, so 
	\begin{equation}\label{eq:t-sum}
		t^{-1}=-\log|f|=-\textstyle\sum_{i=1}^{k}m_{i}\log|z_{i}|=\textstyle\sum_{i=1}^{k}t_{i}^{-1},
	\end{equation}
	thus in our smooth coordinates $(t_1,\dots,t_k,\rest)$ we have
	\begin{equation*}
		\tfrac{\d t}{\d t_1}=\tfrac{\d}{\d t_{1}}\textstyle(\sum_{i=1}^{k}t_{i}^{-1})^{-1}=-t^{2}\tfrac{\d}{\d t_1} t_{1}^{-1}=t^{2}t_{1}^{-2}>0.
	\end{equation*}
	Therefore, $\varphi$ is a local diffeomorphism. Suppose that for some $x,y\in U \setminus \d A$ we have $\varphi(x)=\varphi(y)$. Then $t(x)=t(y)>0$ and $t_{i}(x)=t_{i}(y)>0$ for all $i\in \{2,\dots, k\}$, so $t_1(x)=t_1(y)$ by \eqref{eq:t-sum}. Since  $\rest(x)=\rest(y)$, too, we get $x=y$. Thus $\varphi$ is injective, hence a diffeomorphism onto its image, as claimed.
	
	The above claim shows that $(t,t_2,\dots, t_k,\rest)$ is another smooth coordinate system on $U\setminus \d A$. With respect to these coordinates, we have $\frac{\d g}{\d t}=\eta'(g)>0$ by \eqref{eq:def-t-g} and $\frac{\d v_{i}}{\d t_{j}}=0$ for $i\neq j$ by \eqref{eq:def-v_i}. For a fixed $t>0$, the function $t_{i}\mapsto t_{i}-\eta(t^{-1}\cdot t_i)=v_i$ is strictly increasing by Lemma \ref{lem:v-check}, so $\frac{\d v_{i}}{\d t_{i}}>0$. Thus the Jacobian determinant of $\psi_1$ in coordinates $(t,t_2,\dots, t_k,\rest)$ is $\frac{\d g}{\d t}\cdot \prod_{i=2}^{k} \frac{\d v_i}{\d t_i}>0$, so $\psi_1$ is a local diffeomorphism.
	
	Assume  $\psi_{1}(x)=\psi_{1}(y)$ for some $x,y\in U\setminus \d A$. Then $g(x)=g(y)>0$, so both $t(x)$ and $t(y)$ are equal to some number $\check{t}>0$. By Lemma \ref{lem:v-check}\ref{item:v=vi}, $v_{i}(x)=\check{v}_{\check{t}}(t_{i}(x))$ for every $i\in \{2,\dots, k\}$, and similarly $v_{i}(y)=\check{v}_{\check{t}}(t_{i}(y))$. Since $v_i(x)=v_i(y)$, Lemma \ref{lem:v-check}\ref{item:v-increasing} implies that $t_i(x)=t_{i}(y)$ for all $i\in \{2,\dots, k\}$. The equality $\psi_1(x)=\psi_1(y)$ implies that $\rest(x)=\rest (y)$, too, so the values of all coordinates $(t,t_2,\dots,t_k,\rest)$ at $x$ and $y$ are equal. Thus $x=y$, so $\psi_1$ is injective, as claimed.
\end{proof}

\begin{lema}\label{lem:psi-open}
	The image of each chart \eqref{eq:AC1-chart} is an open subset of $Q_{k,n}$.
\end{lema}
\begin{proof}
	Fix a chart $U_X$ adapted to $f$. As before, we order the components of $D$ so that the index set \eqref{eq:index-set} associated to $U_X$ is $\{1,\dots, k\}$, and study the map $\psi_{1}\colon U_1\to Q_{k,n}$ introduced in \eqref{eq:AC1-chart}. 
	
	Fix a point $x\in U_1$. We need to prove that some neighborhood of $\psi_1(x)$ in $Q_{n,k}$ lies in $\psi_1(U_1)$. This is clear if $x\not\in \d A$, since by Lemma \ref{lem:coordinates-off-D} the restriction $\psi_{1}|_{U\setminus \d A}$ is a diffeomorphism onto its image. Hence we can assume $x\in \d A$. Take a sequence $(y^{\nu})\subseteq Q_{k,n}$ such that $y^{\nu}\rightarrow \psi_{1}(x)$. We need to show that $y^{\nu}\in \psi_{1}(U_1)$ for all $\nu\gg 1$. To do this, we will construct a sequence $(x^{\nu})$ such that $\psi_{1}(x^{\nu})=y^{\nu}$ and $x^{\nu}\in U_1$ for $\nu\gg 1$. 
	
	Write $y^{\nu}=(g^{\nu},v^{\nu}_1,\dots,v_{k}^{\nu},\rest^{\nu})$ for some $g^{\nu}\geq 0$, $v_i^{\nu}\in \R$ and $\rest^{\nu}\in (\S^{1})^{k}\times \C^{n-k}$. Then $g^{\nu}\rightarrow g(x)=0$, $v_{i}^{\nu}\rightarrow v_i(x)$ and $\rest^{\nu}\rightarrow \rest(x)$. Passing to a subsequence, we can assume that either $g^{\nu}>0$ for all $\nu$, or $g^{\nu}=0$ for all $\nu$: in other words, $y^{\nu}$ approaches $\psi_1(x)$ either from $\Int Q_{k,n}$, or from $\d Q_{k,n}$.
	\smallskip

	Consider the case $g^{\nu}>0$. Put $t^{\nu}:=\eta^{-1}(g^{\nu})$. Since $\eta^{-1}$ is continuous, we have $t^{\nu}\rightarrow \eta^{-1}(0)=0$. By definition \eqref{eq:Ui} of $U_i$, we have $w_1(x)>\tfrac{1}{n+1}$, so by Lemma \ref{lem:intro}\ref{item:intro-sum}, for $i\in \{2,\dots, k\}$ we have $w_{i}(x)=1-w_1(x)-\sum_{j\neq 1,i} w_j(x)<1-\tfrac{1}{n+1}<1$. The formula \eqref{eq:def-v_i} for $v_i$ implies that $v_{i}(x)>-1$. Hence there is a positive number $\epsilon>0$ such that, after taking $\nu$ sufficiently large, we have $t^{\nu}<\epsilon$ and $v_{i}^{\nu}>\epsilon-1$ for $i\in \{2,\dots, k\}$. Therefore, each number $v_{i}^{\nu}$ lies in the image of the function $\check{v}_{t^{\nu}}$ from Lemma \ref{lem:v-check}\ref{item:v-increasing}. In other words, there is a number $t_{i}^{\nu}\geq t^{\nu}$ such that $\check{v}_{t^{\nu}}(t_{i}^{\nu})=v_{i}^{\nu}$. 
	
	We claim that for each $i\in \{2,\dots, k\}$, the sequence $(t_{i}^{\nu})$ is bounded. Suppose the contrary. Passing to a subsequence, we get $t_{i}^{\nu}\rightarrow \infty$. Since $t^{\nu}\rightarrow 0$, we get $t^{\nu}\cdot (t_{i}^{\nu})^{-1}\rightarrow 0$, so $\eta(t^{\nu}\cdot (t_{i}^{\nu})^{-1})\rightarrow \eta(0)=0$. Now $v_{i}^{\nu}=\check{v}_{t^{\nu}}(t_{i}^{\nu})=t_{i}^{\nu}-\eta(t^{\nu}\cdot (t_{i}^{\nu})^{-1})\rightarrow \infty-0=\infty$, a contradiction since $v_{i}^{\nu}\rightarrow v_i(x)\in \R$.
	
	Thus passing to a subsequence, we can assume that $t_{i}^{\nu}\rightarrow \check{t}_{i}$ for some $\check{t}_{i}\geq 0$. We claim that $\check{t}_{i}=t_i(x)$. If $\check{t}_{i}>0$, then the sequence $(t_{i}^{\nu})$ is bounded from below by a positive number, so $t^{\nu}\cdot (t_{i}^{\nu})^{-1}\rightarrow 0$, and therefore $\check{v}_{t^{\nu}}(t_{i}^{\nu})=t_i^{\nu}-\eta(t^{\nu}\cdot (t_{i}^{\nu})^{-1})\rightarrow \check{t}_{i}$. Hence $\check{t}_{i}=v_{i}(x)$. In particular, $v_i(x)> 0$, so by Lemma \ref{lem:Ui-simple}\ref{item:Ui-cap-dA} we have $v_i(x)=t_i(x)$, and therefore $\check{t}_i=t_i(x)$, as claimed. On the other hand, if $\check{t}_{i}=0$ then $v_{i}(x)=\limn \check{v}_{t^{\nu}}(t_{i}^{\nu})=\limn -\eta(t^{\nu}\cdot (t_{i}^{\nu})^{-1})\leq 0$, so Lemma \ref{lem:Ui-simple}\ref{item:Ui-cap-dA} gives $t_{i}(x)=0$, as needed.
	
	Thus for all $i\in \{2,\dots, k\}$ we have $t_{i}^{\nu}\rightarrow t_{i}(x)$. Put $w_{i}^{\nu}:=t^{\nu}\cdot (t^{\nu}_{i})^{-1}$. By definition of $t_{i}^{\nu}$ we have  $v_{i}^{\nu}=\check{v}_{t^{\nu}}(t_{i}^{\nu})$, so $v_{i}^{\nu}=t_{i}^{\nu}-\eta(w_{i}^{\nu})$ by definition of the function $\check{v}_{t^{\nu}}$. Hence $w_{i}^{\nu}=\eta^{-1}(t_{i}^{\nu}-v_{i}^{\nu})\rightarrow \eta^{-1}(t_{i}(x)-v_{i}(x))=\eta^{-1}(u_i(x))=w_{i}(x)$.
	
	By definition \eqref{eq:Ui} of $U_1$, we have $w_{1}(x)>\frac{1}{n+1}$. Lemma \ref{lem:intro}\ref{item:intro-sum} implies that $\sum_{i=2}^{k}w_{i}(x)=1-w_{1}(x)< \tfrac{n}{n+1}$. Hence for $\nu\gg 1$ we have $\sum_{i=2}^{k}w_{i}^{\nu}<\tfrac{n}{n+1}$, too. By definition of $w_{i}^{\nu}$ we get $\sum_{i=2}^{k}(t_{i}^{\nu})^{-1}<\tfrac{n}{n+1} (t^{\nu})^{-1}<(t^{\nu})^{-1}$, so the positive number $t^{\nu}_{1}\de ((t^{\nu})^{-1} - \sum_{i=2}^{k}(t_{i}^{\nu})^{-1})^{-1}$ satisfies $(t^{\nu}_1)^{-1}=(t^{\nu})^{-1} - \sum_{i=2}^{k}(t_{i}^{\nu})^{-1}>\tfrac{1}{n+1}(t^{\nu})^{-1}$, and therefore $t_{1}^{\nu}<(n+1)t^{\nu}$. Since $t^{\nu}\rightarrow 0$, we infer that $t_{1}^{\nu}\rightarrow 0$, too. By Lemma \ref{lem:Ui-simple}\ref{item:Ui-in-Ai} we have $t_{1}(x)=0$, so the convergence $t_{i}^{\nu}\rightarrow t_i(x)$ holds for all $i\in \{1,\dots, k\}$.
	
	For $i\in \{1,\dots, k\}$ put $r_{i}^{\nu}:=e^{-(m_{i}t_{i}^{\nu})^{-1}}$. By \eqref{eq:def-r_i-t_i}, we have $r_{i}(x)=0$ if $t_i(x)=0$ and $r_{i}(x)=e^{-(m_i t_i(x))^{-1}}$ otherwise, so $r_i^{\nu}\rightarrow r_i(x)$ for all $i\in \{1,\dots, k\}$. %Recall that $\rest^{\nu}\rightarrow \rest$ by definition of $\rest^{\nu}$.
	
	Let $x_{\log}$ be the image of $x$ in the Kato--Nakayama space $X_{\log}$. In the chart \eqref{eq:KN-chart}, the coordinates of $x_{\log}$ are $(r_1(x),\dots,r_k(x),\rest(x))$. Since $r_{i}^{\nu}\rightarrow r_{i}(x)$ and $\rest^{\nu}\rightarrow \rest(x)$ by definition of $\rest^{\nu}$, for $\nu\gg 1$ the point $(r_1^{\nu},\dots, r_k^{\nu},\rest^{\nu})$ lies in the image of that chart, too. Let $x_{\log}^{\nu}$ be the corresponding point of $X_{\log}$, and let $x^{\nu}$ be its image in $X$. Since ${r}^{\nu}_{i}>0$ for all $i\in \{1,\dots,k\}$, we have $x^{\nu}\not\in D$, so we can identify $x^{\nu}$ with its preimage in $A\setminus \d A$. This way, $\psi_{1}(x^{\nu})=y^{\nu}$ and $x^{\nu}\rightarrow x$, so $x^{\nu}\in U_1$ for $\nu\gg 1$ because $U_1$ is open. This ends the proof in case $g^{\nu}>0$.
	\smallskip
	
	Consider now the case when $g^{\nu}=0$ for all $\nu$. Say that $x\in A_{I}^{\circ}$ for some $I\subseteq \{1,\dots, k\}$. By Lemma \ref{lem:Ui-simple}\ref{item:Ui-cap-dA} we have $1\in I$, $v_{i}(x)=-u_{i}(x)$ if $i\in I$ and $v_i(x)=t_i(x)>0$ if $i\not \in I$. Since $v_{i}^{\nu}\rightarrow v_{i}(x)$ for all $i\in \{2,\dots, k\}$, after passing to a subsequence we get a subset $J\subseteq I\setminus \{1\}$ such that $v_{i}^{\nu}\leq 0$ if $i\in J$ and $v_{i}^{\nu}> 0$ if $i\not\in J$. %Reordering the components of $D$ if needed, we can assume $J=\{2,\dots, l\}$ for some $l\in \{1,\dots,k\}$.
	
	For $i\in J$ put $w_{i}^{\nu}=\eta^{-1}(-v_{i}^{\nu})$. Since $\eta^{-1}$ is continuous, we get $w_{i}^{\nu}\rightarrow \eta^{-1}(-v_{i}(x))=w_{i}(x)$. For $i\in I\setminus J$ we have $-u_{i}(x)=v_{i}(x)=0$, so $w_{i}(x)=0$. For $i\not\in I$ we have $w_{i}(x)=0$ by Lemma \ref{lem:intro}\ref{item:intro-w-zero}. Thus by Lemma \ref{lem:intro}\ref{item:intro-sum} we have $\sum_{i\in J}w_i(x)=1-w_{1}(x)<\tfrac{n}{n+1}$, where the inequality follows from definition \eqref{eq:Ui} of $U_i$. Since $w_{i}^{\nu}\rightarrow w_{i}(x)$ for $i\in J$, we infer that for $\nu\gg 1$, the numbers $w_{1}^{\nu}\de 1-\sum_{i\in J}w_{i}^{\nu}$ are positive. Put $v_{1}^{\nu}=-\eta(w_{1}^{\nu})$. Then $v_{1}^{\nu}\rightarrow -\eta(w_1(x))=v_{1}(x)$ by Lemma \ref{lem:Ui-simple}\ref{item:Ui-cap-dA}. Let $u^{\nu}$ be the point of $\R^{N}$ whose $i$-th coordinate, for $i\in J\cup\{1\}$, is $-v_{i}^{\nu}$, and the remaining coordinates are zero. Then $u^{\nu}\in \Delta_{J\cup \{1\}}\subseteq \Delta^{N-1}$.
	
	For $i \not\in J$ put $r_{i}^{\nu}=e^{-(m_{i}v_{i}^{\nu})^{-1}}$. Like before, we have $r_{i}^{\nu}\rightarrow e^{-(m_{i}v_{i}(x))^{-1}}=e^{-(m_{i}t_{i}(x))^{-1}}=r_i(x)$. For $i\in J\cup \{1\}$ put $r_{i}^{\nu}=0$, so $r_{i}^{\nu}\rightarrow r_i(x)$, too. The image of $x$ in the chart \eqref{eq:KN-chart} of $X_{\log}$ has coordinates $(r_1(x),\dots,r_k(x),\rest(x))$, so for $\nu\gg 1$ the point $(r_{1}^{\nu},\dots,r_{k}^{\nu},\rest^{\nu})$ lies in the image of this chart, too. The corresponding point $x_{\log}^{\nu}\in X_{\log}$ lies in $(X_{J\cup\{1\}}^{\circ})_{\log}$.  Since $u^{\nu}\in \Delta_{J\cup \{1\}}$, the pair $x^{\nu}\de (x_{\log}^{\nu},u^{\nu})\in X_{\log}\times \R^{N}$ lies in $A_{J\cup\{1\}}^{\circ}$ by Lemma \ref{lem:Gamma_structure}\ref{item:trivial-fibration}. By construction, $x^{\nu}\rightarrow x$ and $\psi_{1}(x^{\nu})=y^{\nu}$. Since $U_{1}$ is open, we conclude that $x^{\nu}\in U_1$ for $\nu\gg 1$, as required.
\end{proof}

\begin{lema}\label{lem:psi-homeo}
	Each chart \eqref{eq:AC1-chart} is a homeomorphism onto its image.
\end{lema}
\begin{proof}
	Fix a chart \eqref{eq:AC1-chart} corresponding to a chart $U_X$ adapted to $f$, say $\psi_{1}\colon U_1\to Q_{k.n}$. Let $S$ be the index set \eqref{eq:index-set} associated to $U_X$. By Definition \ref{def:adapted-chart}\ref{item:U-compact}, the coordinates of $U_X$ extend continuously to the compact closure $\bar{U}_X$. Together with Lemma \ref{lem:intro}\ref{item:intro-extends}, this implies that $\psi_1$ extends to a continuous map $\bar{U}_{1}\to Q_{k,n}$, which we will denote by the same letter. Since $\bar{U}_1$ is compact, it is sufficient to prove that this extension is injective.
	
	For a subset $I\subseteq  S$ put $V_{I}=\bar{U}_{1}\cap A_{I}^{\circ}$. We have $\psi_{1}(V_{\emptyset})\subseteq \Int Q_{k,n}$ and $\psi(V_{I})\subseteq \d Q_{k,n}$ for all $I\neq \emptyset$, so $\psi_{1}(V_{\emptyset})\cap \psi_{1}(V_{I})=\emptyset$. Lemma \ref{lem:Ui-simple}\ref{item:Ui-cap-dA} implies that $\psi_{1}(V_{I})\cap \psi_{1}(V_{J})=\emptyset$ for all $I\neq J$. Therefore, it is sufficient to prove the injectivity of each restriction $\psi_{1}|_{V_{I}}$. 
	
	For $I=\emptyset$, this follows Lemma \ref{lem:coordinates-off-D}. Assume $I\neq \emptyset$ and fix $x\in V_I$. By  Lemma \ref{lem:Gamma_structure}\ref{item:trivial-fibration}, the point $x$ is a pair $(x_{\log},x_{\Delta})\in (X_{I}^{\circ})_{\log}\times \Delta_{I}$. In coordinates \eqref{eq:KN-chart} on $X_{\log}$, the point $x_{\log}$ is given by $((r_{i}(x))_{i\in S},\rest(x))$, where $r_{i}(x)=0$ if $i\in I$ and, by \eqref{eq:def-r_i-t_i}, $r_{i}(x)=e^{-(m_i t_i(x))^{-1}}$ if $i\not\in I$. By Lemma \ref{lem:Ui-simple}\ref{item:Ui-cap-dA}, for $i\not\in I$ we have $i\neq 1$ and  $t_i(x)=v_i(x)$, so the point $x_{\log}$ is determined by the value of $\psi_1$ at $x$. In turn, the $i$-th coordinate of $x_{\Delta}\in \Delta_{I}\subseteq \R^N$ is $\bar{u}_{i}(x)$ if $i\in I$ and zero otherwise. For $i\in I$, we have $\bar{u}_{i}(x)=u_{i}(x)$ by Lemma \ref{lem:intro}\ref{item:intro-u=ubar} and $u_i(x)=-v_{i}(x)$ by Lemma  \ref{lem:Ui-simple}\ref{item:Ui-cap-dA}. Hence for $i\neq 1$, the $i$-th coordinate of $x_{\Delta}\in \R^{N}$ is determined by the value of $\psi_1$ at $x$. The remaining coordinate is determined by the fact that $x_{\Delta}$ lies in the rounded simplex $\Delta_{I}$.
\end{proof}

\begin{proof}[Proof of Lemma \ref{lem:C1chart}]
	By Lemma \ref{lem:psi-homeo}, each chart \eqref{eq:AC1-chart} is a homeomorphism onto its image, which is an open subset of $Q_{k,n}$ by Lemma \ref{lem:psi-open}. By Lemma \ref{lem:coordinates-off-D}, the chart \eqref{eq:AC1-chart} restricts to a diffeomorphism away from $\d A$, as needed.
\end{proof}

\subsubsection{Useful computations}\label{sec:C1-computation} 
By Lemma \ref{lem:C1chart}, our atlas makes $A$ a topological manifold. Now, our goal is to show that this manifold is $\cC^1$. Lemmas \ref{lem:computations}--\ref{lem:computations-d} gather some technical computations which will be useful throughout the remaining part of Section \ref{sec:ACampo}. To state them in a convenient way, we introduce the following notation.

Recall that in \eqref{eq:def-t-g} we defined $t=-(\log|f|)^{-1}$ and $g=\eta(t)$. We put $t'=\tfrac{d}{dg} t$, so $t'=\zeta(g)$, where $\zeta\colon s\mapsto s^{-2}e^{1-s^{-1}}$ is the derivative of the inverse of $\eta$, see \eqref{eq:smooth_simplex}. More generally, for every $l\geq 0$ we put
\begin{equation}\label{eq:def-tl}
	t^{(l)}=(\tfrac{d}{d g})^{l}\, t=\zeta^{(l-1)}(g).
\end{equation}

Let $U_X$ be a chart adapted to $f$. For every $i$ in its associated index set \eqref{eq:index-set}, we define functions 
\begin{equation}\label{eq:rho-sigma}
	\rho_{i}=(t_{i}+u_i^2)^{-1}
	\quad \mbox{and}\quad
	\sigma_{i}=t_{i}^{2}+t_{i}u_{i}^{2},
\end{equation}
where $t_{i}, u_{i}$ are the functions introduced in \eqref{eq:def-r_i-t_i} and \eqref{eq:def-w_i-u_i}. Note that $\sigma_{i}=t_{i}\rho_{i}^{-1}$.

\begin{lema}\label{lem:computations} 
	Let $U_X$ be a chart adapted to $f$. Fix $i$ in its associated index set \eqref{eq:index-set}. The functions introduced in Section \ref{sec:basicfunctions} and in \eqref{eq:def-tl}, \eqref{eq:rho-sigma} above satisfy the following identities on $U_X\setminus D$:
	\begin{enumerate}
		\item\label{item:t-r} $r_{i}=e^{-(m_{i}t_{i})^{-1}}$,
		\item\label{item:t=tiwi} $t=t_{i}\cdot w_{i}$,
		\item\label{item:t'} $t'=t\cdot g^{-2}=t\cdot (1-\log t)^{2}$,
		\item \label{item:tl} for every integer $l\geq 0$, there is a polynomial $p_{l}\in \R[s]$ such that $t^{(l)}=t\cdot p_{l}(\log t)$,
		\item\label{item:u-gt} $u_{i}=g\cdot(1+g\log t_{i})^{-1}$, 
		\item \label{item:rho-gt} $\rho_{i}=(1+g\log t_{i})^{2}\cdot (t_{i}(1+g\log t_{i})^{2}+g^2)^{-1}$,
		\item\label{item:eta-u} $\eta'(w_{i})=u_{i}^{2}w_{i}^{-1}$, where $\eta\colon s\mapsto (1-\log s)^{-1}$ is the function defined in \eqref{eq:smooth_simplex},
		\item\label{item:eta-t} $\eta'(w_i)t'=t_{i}(1+g\log t_{i})^{-2}$,
		\item\label{item:eta-t-bounded} for every $\epsilon\in (0,1)$, there is a bounded function $l\in \cC^{\infty}(U_X\setminus D)$ such that $\eta'(w_{i})t'=t_{i}^{\epsilon}\cdot l$.
	\end{enumerate}
\end{lema}
\begin{proof}
	\ref{item:t-r} In \eqref{eq:def-r_i-t_i}, the function $t_i$ was defined as $t_{i}=-(m_{i}\log r_{i})^{-1}$, so $\log r_{i}=-(m_it_i)^{-1}$, and therefore $r_{i}=e^{-(m_{i}t_{i})^{-1}}$, as claimed.
	
	\ref{item:t=tiwi} In \eqref{eq:def-w_i-u_i}, the function $w_i$ was defined as $w_{i}=tt_{i}^{-1}$, so $t=t_iw_i$.
	
	\ref{item:t'} Recall from \eqref{eq:def-t-g} and \eqref{eq:smooth_simplex} that $g=\eta(t)=(1-\log t)^{-1}$, so $t=e^{1-g^{-1}}$. Now $t'=\frac{d t}{d g}=e^{1-g^{-1}} g^{-2}=tg^{-2}=t(1-\log t)^{2}$, as needed. 
	
	\ref{item:tl} We argue by induction on $l\geq 0$. For $l=0$ we put $p_0=1$. Assume that for some $l\geq 0$, there is $p_{l}\in \R[s]$ such that $t^{(l)}=p_{l}(\log t)t$. Since $t=e^{1-g^{-1}}$ and $\log t=1-g^{-1}$, we have
	\begin{equation*}
		\begin{split}
			t^{(l+1)}&=\tfrac{d}{d g}t^{(l)}=\tfrac{d}{dg} [p_{l}(1-g^{-1})e^{1-g^{-1}}]=p_{l}'(1-g^{-1})g^{-2}e^{1-g^{-1}}+p_{l}(1-g^{-1})g^{-2}e^{1-g^{-1}}=\\
			&= (p_{l}(1-g^{-1})+p_{l}'(1-g^{-1}))g^{-2} e^{1-g^{-1}}
			=(p_{l}(\log t)+p_{l}'(\log t)) (1-\log t)^{2} t.
		\end{split}
	\end{equation*}
	The inductive claim follows by putting $p_{l+1}(s)=(p_{l}(s)+p_{l}'(s))(1-s)^{2}\in \R[s]$.
	
	\ref{item:u-gt} The function $u_i$ is defined in \eqref{eq:def-w_i-u_i} as $u_{i}=\eta(w_{i})=(1-\log w_{i})^{-1}$. By \eqref{eq:def-w_i-u_i}, we have $(1-\log w_{i})^{-1}=(1-\log t+\log t_{i})^{-1}$. Since by definition $g=(1-\log t)^{-1}$, we have $(1-\log t+\log t_{i})^{-1}=(g^{-1}+\log t_{i})^{-1}=g\cdot (1+g\log t_{i})^{-1}$. Thus eventually $u_{i}=g\cdot (1+g\log t_{i})^{-1}$, as claimed.
	
	\ref{item:rho-gt} The function $\rho_{i}$ was defined in \eqref{eq:rho-sigma} by $\rho_i=(t_{i}+u_{i}^{2})^{-1}$. Substituting the formula for $u_i$ from \ref{item:u-gt}, we get $\rho_{i}=(t_{i}+g^2\cdot (1+g \log t_i)^{-2})^{-1}=(1+g\log t_{i})^{2}\cdot (t_{i}(1+g\log t_{i})^2+g^2)^{-1}$.
	
	\ref{item:eta-u} By \eqref{eq:def-w_i-u_i}, we have $u_{i}=\eta(w_i)=(1-\log w_{i})^{-1}$, so $\eta'(w_{i})=(1-\log w_{i})^{-2}w_{i}^{-1}=u_{i}^2w_{i}^{-1}$.
	
	\ref{item:eta-t} 	Substituting  the formula for $u_i$ from \ref{item:u-gt} to \ref{item:eta-u}, we get $\eta'(w_{i})=g^{2}\cdot (1+g\log t_{i})^{-2}\cdot w_{i}^{-1}$. By \eqref{eq:def-w_i-u_i} we have $w_{i}=t\cdot t_{i}^{-1}$, so $\eta'(w_{i})=g^{2}\cdot (1+g\log t_{i})^{-2} \cdot t^{-1}t_{i}$. In turn, by \ref{item:t'} we have $t'=tg^{-2}$, so $t'\eta'(w_i)=tg^{-2}\cdot g^{2}(1+g\log t_{i})^{-2}\cdot t^{-1}t_{i}=t_{i}(1+g\log t_{i})^{-2}$, as claimed.
	
	\ref{item:eta-t-bounded} Recall from \eqref{eq:def-t-g} that $g=\eta(t)=(1-\log t)^{-1}$. By \ref{item:t=tiwi}, we have  $t=t_{i}w_{i}$, so $g=(1-\log t_{i}-\log w_{i})^{-1}$. Substituting this formula to \ref{item:eta-t}, we get
	\begin{equation*}
		\eta'(w_{i})t'=t_{i}\cdot\left(1+\frac{\log t_{i}}{1-\log t_{i}-\log w_{i}}\right)^{-2}=
		t_{i}\cdot \left( \frac{1-\log w_{i}}{1-\log t_{i}-\log w_{i}}\right)^{-2}=
		t_{i} \cdot \left( 1-\frac{\log t_{i}}{1-\log w_{i}}\right)^{2}.
	\end{equation*}
	Since $w_{i}\in (0,1]$, we have $\log w_{i}\leq 0$, so the function $s\de (1-\log w_{i})^{-1}$ is bounded. Fix $\epsilon\in (0,1)$ and put $\delta=\frac{1}{2}(1-\epsilon)>0$. Then 
	\begin{equation*}
		\eta'(w_{i})t'=t_{i}\cdot (1-s\cdot \log t_{i})^{2}=t_{i}^{\epsilon}\cdot (t_{i}^{\delta}-s\cdot t_{i}^{\delta}\log t_{i})^{2}.
	\end{equation*}
	Since $\delta>0$, we have $t_{i}^{\delta}\log t_{i}\rightarrow 0$ as $t_{i}\rightarrow 0$, so $l\de (t_{i}^{\delta}-s\cdot t_{i}^{\delta}\log t_{i})^{2}$ is bounded, as claimed.
\end{proof}

We write $d$ for the exterior derivative. We will apply it to smooth functions on (subsets of) $X\setminus D$, and study the behavior of the resulting forms as we approach $D$. 

\begin{lema}\label{lem:computations-d}
	In the setting of Lemma \ref{lem:computations}, the following identities hold on $U_X\setminus D$: 
	\begin{enumerate}
		\item \label{item:dr} $dr_{i}=m_{i}^{-1}t_{i}^{-2}r_{i}\, dt_{i}$,	
		\item \label{item:dv} $	dv_{i}=(1+t_{i}^{-1}u_{i}^{2})\, dt_{i}-t_{i}^{-1}\eta'(w_i)t'\, dg$,
		\item \label{item:dt-rho} $d t_i = \rho_{i}(t_{i}\, dv_{i}+u_{i}^{2}w_{i}^{-1}t'\, dg)$,
		\item \label{item:dt-g} $dt_i=t_{i}\cdot (t_i(1+g\log t_i)^{2}+g^2)^{-1}\cdot ((1+g\log t_i)^{2}\, dv_{i}+ dg)$,	
		\item \label{item:dw} $d w_{i}=-\rho_{i}(w_{i}\, dv_{i}-t'\, dg)$,
		\item \label{item:du} $d u_{i}=-\rho_{i}(u_{i}^{2}\, dv_{i}-\eta'(w_i)t'\, dg)$,
		\item \label{item:du-dt} $du_{i}=-t_{i}^{-1}(u_{i}^{2}\, dt_{i}-\eta'(w_{i})t'\, dg)$.
	\end{enumerate}
\end{lema}
\begin{proof}
	\ref{item:dr} By Lemma \ref{lem:computations}\ref{item:t-r}, we have 
	$r_{i}=e^{-(m_{i}t_{i})^{-1}}$, so $dr_{i}=m_{i}\cdot (m_{i}t_{i})^{-2} \cdot e^{-(m_{i}t_{i})^{-1}}\, dt_{i}=m_{i}^{-1}t_{i}^{-2}r_{i}\, dt_{i}$, as claimed.
	
	\ref{item:dv} The functions $v_i$, $w_i$ were defined in \eqref{eq:def-v_i} and \eqref{eq:def-w_i-u_i} as  $v_{i}=t_{i}-\eta(w_i)$ and $w_{i}=tt_{i}^{-1}$. Hence
	\begin{equation*}
		dv_{i}=dt_{i}-\eta'(w_{i})\cdot d(tt_{i}^{-1})=dt_{i}-\eta'(w_{i})\cdot (-tt_{i}^{-2}\, dt_{i}+t_{i}^{-1}\, dt)=(1+\eta'(w_{i})tt_{i}^{-2})\, dt_{i}-t_{i}^{-1}\eta'(w_{i})\, dt.
	\end{equation*}
	By Lemma \ref{lem:computations}\ref{item:eta-u}, we have $\eta'(w_{i})=u_{i}^{2}w_{i}^{-1}$. Since $t=t_{i}w_{i}$ by Lemma \ref{lem:computations}\ref{item:t=tiwi}, we get  $\eta'(w_{i})tt_{i}^{-2}=u_{i}^{2}w_{i}^{-1}\cdot t_{i}w_{i}\cdot t_{i}^{-2}=u_{i}^{2}t_{i}^{-1}$. 
	In turn, $dt=t' dg$ by definition \eqref{eq:rho-sigma} of $t'$. Thus $dv_{i}=(1+u_{i}^{2}t_{i}^{-1})dt_{i}-t_{i}^{-1}\eta'(w_{i})t'\, dg$, as claimed.
	
	\ref{item:dt-rho} Solving \ref{item:dv} for $dt_{i}$, we get 
	\begin{equation*}
		dt_{i}=(1+t_{i}^{-1}u_{i}^{2})^{-1}\cdot (dv_{i}+t_{i}^{-1}\eta'(w_{i})t'\, dg)=(t_{i}+u_{i}^{2})^{-1}\cdot (t_{i}\, dv_{i}+\eta'(w_{i})t'\, dg).
	\end{equation*}
	We have $(t_{i}+u_{i}^{2})^{-1}=\rho_{i}$ by \eqref{eq:rho-sigma} and $\eta'(w_{i})=u_{i}^{2}w_{i}^{-1}$ by Lemma  \ref{lem:computations}\ref{item:eta-u}, so  $dt_{i}=\rho_{i}(t_{i}\, dv_{i}+u_{i}^{2}w_{i}^{-1}t'\, dg)$, as claimed.
	
	\ref{item:dt-g}  Substituting the formula for $\eta'(w_i)$ from Lemma \ref{lem:computations}\ref{item:eta-u} to Lemma \ref{lem:computations}\ref{item:eta-t}, we get the identity $u_{i}^{2}w_{i}^{-1}t'=t_{i}(1+g\log t_{i})^{-2}$. Substituting it to \ref{item:dt-rho} gives
	\begin{equation*}
		dt_{i}=\rho_{i}(t_{i}\,dv_{i}+t_{i}(1+g\log t_{i})^{-2}\, dg)=\rho_{i}t_{i} (1+g\log t_{i})^{-2} \cdot((1+g\log t_{i})^{2} dv_{i}+dg).
	\end{equation*}
	The claim follows by substituting the formula for $\rho_{i}$ from Lemma \ref{lem:computations}\ref{item:rho-gt}.
	
	\ref{item:dw} The function $w_i$ was defined in \eqref{eq:def-w_i-u_i} as $w_{i}=tt_{i}^{-1}$, so $dw_{i}=-tt_{i}^{-2}\, dt_{i}+t_{i}^{-1}\, dt=-w_{i}t_{i}^{-1}\, dt_{i}+t_{i}^{-1}\, dt$. We have $dt_{i}=\rho_{i}(t_{i}\, dv_{i}+u_{i}^{2}w_{i}^{-1}t'\, dg)$ by \ref{item:dt-rho} and $dt=t'\, dg$ by definition \eqref{eq:rho-sigma} of $t'$, so
	\begin{equation*}
		dw_{i}=-w_{i}t_{i}^{-1}\rho_{i}(t_{i}\, dv_{i}+u_{i}^{2}w_{i}^{-1}t'\, dg)+t_{i}^{-1}t'\, dg=
		-\rho_{i}(w_{i}\, dv_{i}+(u_{i}^{2}-\rho_{i}^{-1})t_{i}^{-1}t'\, dg).
	\end{equation*}
	By definition \eqref{eq:rho-sigma} of $\rho_{i}$, we have $(u_{i}^{2}-\rho_{i}^{-1})t_{i}^{-1}=(u_{i}^{2}-t_{i}-u_{i}^{2})t_{i}^{-1}=-1$, which proves the claim.
	
	\ref{item:du} By \eqref{eq:def-w_i-u_i}, we have $u_i=\eta(w_i)$, so using \ref{item:dw} we get   $du_{i}=\eta'(w_i)dw_{i}=-\rho_{i}(\eta'(w_i)w_{i}\, dv_{i}-\eta'(w_i)t'\, dg)$. By Lemma \ref{lem:computations}\ref{item:eta-u}, we have $\eta'(w_{i})w_{i}=u_{i}^{2}$, so $du_{i}=-\rho_{i}(u_{i}^{2}\, dv_{i}-\eta'(w_i)t'\, dg)$, as claimed.
	
	\ref{item:du-dt} Substituting \ref{item:dv} to \ref{item:du}, we get
	\begin{equation*}
		\begin{split}
			d u_{i}&=-\rho_{i}(u_{i}^{2}\cdot ((1+t_{i}^{-1}u_{i}^{2})\, dt_{i}-t_{i}^{-1}\eta'(w_i)t'\, dg)-\eta'(w_i)t'\, dg)=\\
			&= 
			-\rho_{i}(u_{i}^{2}(1+t_{i}^{-1}u_{i}^{2})\, dt_{i}-(u_{i}^{2}t_{i}^{-1}+1)\eta'(w_{i})t'\, dg)=\\
			&=
			-\rho_{i}(1+t_{i}^{-1}u_{i}^{2}) (u_{i}^{2}\, dt_{i}-\eta'(w_{i})t'\, dg)=
			-t_{i}^{-1}(u_{i}^{2}\, dt_{i}-\eta'(w_{i})t'\, dg),
		\end{split}
	\end{equation*}
	where the last equality follows from definition \eqref{eq:rho-sigma} of $\rho_{i}$.
\end{proof}

\subsubsection{Forms bounded from $X$}\label{sec:bounded-forms}
Once we endow $A$ with a smooth structure, it will be important to distinguish those forms on $A$ which come from bounded forms on $X$. To this end, we introduce the following definition.

\begin{definition}\label{def:bounded-from-X}
	Let $U_{X}$ be a chart adapted to $f$, and let $V$ be an open subset of $\pi^{-1}(U_X)$. On $V\setminus \d A$, we consider the natural smooth structure pulled back from $\pi(V)\setminus D$. We say that a form $\beta\in \Omega^{*}(V\setminus \d A)$ is \emph{bounded from $X$} if $\beta=\pi^{*}\beta'$ for some $\beta'\in \Omega^{*}(\pi(V)\setminus D)$ which is bounded in the usual sense, with respect to coordinates $(z_{i_1},\dots,z_{i_n})$ of the adapted chart $U_X$.
\end{definition}
\begin{remark}\label{rem:bounded-from-X}
	Let $U_i$ be a domain of a chart \eqref{eq:AC1-chart}. A function on $U_i\setminus \d A$ is bounded from $X$ if and only if it is bounded. Since the closure $\bar{U}_i$ is compact by Definition \ref{def:adapted-chart}\ref{item:U-compact}, all continuous functions on $\bar{U}_i$, in particular those from Lemma \ref{lem:intro}\ref{item:intro-extends}, are bounded from $X$. 
	
	If two forms $\beta$ and $\beta'$ are bounded from $X$, then so are $\beta+\beta'$ and $\beta\wedge \beta'$. Nonetheless, $d\beta$ might not be bounded from $X$. For example, functions $\theta_i$, $t_{i}$ are bounded from $X$, but the $1$-forms $d\theta_i$, $d t_i$ are not. Indeed, $d\theta_{i}=-\frac{1}{r_i}\, d^{c}r_i$ has a \enquote{logarithmic pole} along $D_i$. For $dt_{i}$, note that by Lemma \ref{lem:computations-d}\ref{item:dr} we have $d t_{i}=m_{i}t_{i}^{2}r_{i}^{-1}\, dr_{i}$, so Lemma \ref{lem:computations}\ref{item:t-r} gives
	\begin{equation*}
		dt_{i}=m_{i}t_{i}^{2}e^{(m_{i}t_{i})^{-1}}\, dr_{i}\rightarrow \infty\quad \mbox{as } t_i\rightarrow 0. 
	\end{equation*}	
\end{remark}

In principle, one should think of $1$-forms bounded from $X$ as the ones \enquote{decaying exponentially fast as we approach $\d A$}. One way of making this precise will be given in Lemma \ref{lem:Qi}\ref{item:Qi-1-form}.
\smallskip

In the next computation, the following elementary observation will be useful.

\begin{lema}\label{lem:inequality}
	Let $a,b,x,y\in \R$. Assume $a,b \geq 0$ and $a+b>0$. Then the following inequality holds:
	\begin{equation*}
		\frac{|ax+by|}{a+b}\leq |x|+|y|.
	\end{equation*}
\end{lema}
\begin{proof}
	By symmetry, we can assume $a\geq b$. Then $a>0$ and $1\geq \frac{b}{a} \geq  0$, so
	\begin{equation*}
		\frac{|ax+by|}{a+b}=\frac{|x+\frac{b}{a}y|}{1+\frac{b}{a}}\leq \frac{|x|+\frac{b}{a}|y|}{1+\frac{b}{a}} \leq |x|+|y|. \qedhere
	\end{equation*} 
\end{proof}

We will now use Lemma \ref{lem:w_extends}\ref{item:t-comparison} to compare the $1$-forms $d v_{i}$ defined by two charts adapted to $f$. We warn the reader that functions $b,c,q,s$ and a $1$-form $\gamma$ introduced below depend on the fixed charts $U_X$, $U_X'$ and index $i$; even though the notation does not reflect this explicitly. The same warning applies to functions $\lambda,a$ from Lemma \ref{lem:w_extends}\ref{item:lambda},\ref{item:t-comparison}.  

\begin{lema}\label{lem:dv}
	Let $U_X,U'_X$ be two charts adapted to $f$. As in Lemma \ref{lem:w_extends}, we put $V_X=U_X\cap U_X'$, $V=\pi^{-1}(U_X\cap U_X')$, and use Notation \ref{not:pullbacks}. Fix $i\in S\cap S'$ and a subset $I\subseteq S\cap S'$ not containing $i$. Put $V_{X,I}^{\circ}=V_{X}\cap X_{I}^{\circ}$. Then the following hold.
	\begin{enumerate}
		\item\label{item:theta-theta'} The map $\theta'_i-\theta_i\colon V\to \S^1$ is a pullback of a smooth map $V_{X}\to \S^1$.
		\item\label{item:a-bounded} Let $a\in \cC^{\infty}(V_X)$ be as in Lemma \ref{lem:w_extends}\ref{item:t-comparison}. Then the functions $|a|$, $|1+at_{i}|^{-1}$, $|\log(1+at_{i})|$ and $|1-u_i\log(1+at_i)|^{-1}$ are continuous and bounded on $V$.
		\item\label{item:vp-v} Let $a\in \cC^{\infty}(V_X)$ be as in \ref{item:a-bounded}. The following identity holds on $V$
		\begin{equation*}\label{eq:vp-v}
			v_{i}-v_{i}'=\frac{a t_{i}^{2}}{1+a t_{i}}+\frac{u_{i}^{2}\log(1+a t_{i})}{1-u_{i}\log(1+a t_i)}.
		\end{equation*}
		\item \label{item:tame_0} There is a bounded function $b\in \cC^{\infty}(V_{X,I}^{\circ})$ such that $v_i-v_i'=\sigma_i b$ on $V_{X,I}^{\circ}$. %In particular, $v_{i}|_{\pi^{-1}(D_i)\cap V}=v_{i}'|_{\pi^{-1}(D_i)\cap V}$.
		\item \label{item:tame_1} For every $\epsilon\in (0,1)$, there are bounded functions $c,q\in \cC^{\infty}(V_{X,I}^{\circ})$, and a bounded 
		$1$-form $\gamma\in \Omega^{1}(V_{X,I}^{\circ})$, such that the following equality holds on $V_{X,I}^{\circ}$: 
		\begin{equation*}
			d(v_i-v'_{i})
			%=\sigma_{i}(ct_{i}^{-1}\, dt_{i}+\gamma)+ t_i^{\epsilon} u_{i} q\, dg%, \\
			= t_{i}\cdot c\, dv_{i}+t_i^{\epsilon}\cdot q\, dg+\sigma_{i}\cdot \gamma.
		\end{equation*}
	\end{enumerate}
	In particular, taking $I=\emptyset$, we get that the equalities from \ref{item:tame_0}, \ref{item:tame_1} hold on $V_{X}\setminus D=V\setminus \d A$, where functions $b,c,q$ are bounded, and the $1$-form $\gamma$ is bounded from $X$.
\end{lema}
\begin{proof}
	\ref{item:theta-theta'} By Lemma \ref{lem:w_extends}\ref{item:lambda}, there is a nonvanishing holomorphic function $\lambda \in \cO_{X}^{*}(V_X)$ such that $z_{i}'=\lambda z_{i}$. In polar coordinates, we have $\lambda=|\lambda|e^{2\pi\imath\cdot \beta}$ for some smooth function $\beta$. Using the additive notation \ref{not:S1} on $\S^1$, we get $\theta'_i-\theta_i=\beta$, as claimed.
	
	\ref{item:a-bounded} Recall that by Definition \ref{def:adapted-chart}\ref{item:U-compact}, the coordinates of an adapted chart $U_X$ extend continuously to the closure $\bar{U}_X$, which is compact. Therefore, the function $a\in \cC^{\infty}(V_X)$ introduced in Lemma \ref{lem:w_extends}\ref{item:t-comparison} extends to a continuous function on a compact set $\bar{V}_X$. In particular, the function $|a|$ is bounded. 
	
	Similarly, the function $1+at_{i}$ extends to a continuous function on $\bar{V}_X$, which does not vanish by Lemma \ref{lem:w_extends}\ref{item:t-comparison}. Thus the functions $|1+at_{i}|^{-1}$ and $|\log(1+at_{i})|$ are continuous on $\bar{V}_X$, hence bounded because $\bar{V}_{X}$ is compact.
	
	On $U_X\setminus D_i$, we have $u_{i}=\eta(w_i)=(1-\log w_{i})^{-1}$ by the formulas \eqref{eq:def-w_i-u_i} and \eqref{eq:smooth_simplex}. Using the second equation from Lemma \ref{lem:w_extends}\ref{item:t-comparison}, we get that on $V \setminus \pi^{-1}(D_i)$, the following identity holds:
	\begin{equation*}
		u_{i}'-u_{i}=
		\frac{1}{1-\log w_{i}-\log(1+at_{i})}-u_i=
		\frac{1}{u_{i}^{-1}-\log(1+at_{i})}-u_i=
		\frac{u_i}{1-u_{i}\log(1+at_{i})}-u_i.
	\end{equation*}
	Therefore, 
	\begin{equation}\label{eq:u-u'}
		u_{i}'-u_{i}=\frac{u_{i}^{2}\log(1+at_{i})}{1-u_{i}\log(1+at_{i})}.
	\end{equation}
	On $V\cap \pi^{-1}(D_i)$ we have $t_i=0$ by definition \eqref{eq:def-r_i-t_i} of $t_i$, so the right-hand side of \eqref{eq:u-u'} is well defined there. By Lemma \ref{lem:intro}\ref{item:intro-independence} we have $u'_i=u_i$ on $V\cap \pi^{-1}(D_i)$, so both sides of \eqref{eq:u-u'} are zero. Thus \eqref{eq:u-u'} holds everywhere on $V$. As before, it extends to continuously to $\bar{V}$. In particular, the function $|1-u_{i}\log(1+at_{i})|^{-1}$ is continuous on $\bar{V}$: to see this, divide both sides of \eqref{eq:u-u'} by $u_i^2 \log(1+at_i)$ and use the right or left side of the equality when $u_i$ is close or away from $0$, respectively. Since the map $\pi$ is proper by Proposition \ref{prop:AX-topo}\ref{item:top-homeo-off-D}, the set $\bar{V}=\pi^{-1}(\bar{V}_{X})$ is compact. Thus $|1-u_{i}\log(1+at_{i})|^{-1}$ is bounded, as claimed.
	
	\ref{item:vp-v} 
	By Lemma \ref{lem:w_extends}\ref{item:t-comparison}, we have 
	\begin{equation}\label{eq:t-t'}
		t_{i}-t_{i}'=t_{i}-\frac{t_{i}}{1+at_{i}}=\frac{at_{i}^2}{1+at_{i}}.
	\end{equation}
	By \eqref{eq:def-v_i}, we have $v_i-v_i'=(t_i-u_i)-(t_i'-u_i')=(t_i-t_i')+(u_i'-u_i)$, so \ref{item:vp-v} follows from  \eqref{eq:t-t'} and \eqref{eq:u-u'}.

	\ref{item:tame_0}
	Since $i\not \in I$, by Lemma \ref{lem:Ui-simple}\ref{item:vi-smooth-on-strata} the functions $t_{i}$, $u_{i}$, $v_i$, $v_i'$ are smooth on $V_{X,I}^{\circ}$. Moreover, since $t_i>0$ on $V_{X,I}^{\circ}$, it  follows that $\sigma_{i}|_{V_{X,I}^{\circ}}$ is a smooth, positive function, hence the restriction  $\sigma_{i}^{-1}(v_i-v_{i}')|_{V_{X,I}^{\circ}}$ is smooth. We need to prove that it is bounded. 
	
	By \ref{item:a-bounded}, the functions $\check{b}\de \frac{a}{1+at_{i}}$ and  $(1-u_{i}\log(1+at_{i}))^{-1}$ are bounded. Since $\frac{1}{t_i}\log(1+a t_{i})=a-\tfrac{1}{2}a^{2}t_{i}+\dots$ is bounded, too, we infer that $\check{h} \de \frac{1}{t_{i}} \cdot \frac{\log(1+a t_{i})}{1-u_{i}\log(1+a t_{i})}$ is bounded as well. 
	By \ref{item:vp-v}, we have $v_{i}-v_{i}'=\check{b}t_{i}^2+\check{h} t_{i}u_{i}^{2}$. By definition \eqref{eq:rho-sigma} of $\sigma_i$, we have $\sigma_{i}=t_{i}^2+t_{i}u_{i}^2$.  Lemma \ref{lem:inequality} implies that
	\begin{equation*}
		\frac{|v_{i}-v_{i}'|}{\sigma_{i}}=\frac{|\check{b} t_{i}^{2}+\check{h} t_{i}u_{i}^{2}|}{t_{i}^{2}+t_{i}u_{i}^{2}} \leq |\check{b}|+|\check{h}|.
	\end{equation*}
	Therefore, $v_{i}-v_{i}'=\sigma_{i}b $ for some bounded function $b$, as claimed. 
	
	\ref{item:tame_1} As before, we note that since $i\not \in I$, the functions $t_{i}$ and $u_{i}$ are smooth on $V_{X,I}^{\circ}$. Thus we can apply the exterior derivative $d$ to the formulas \eqref{eq:t-t'} and \eqref{eq:u-u'}. We get bounded functions $\tilde{c},\tilde{q},\tilde{s} \in \cC^{\infty}(V_{X,I}^{\circ})$ and bounded $1$-forms $\tilde{\beta}$, $\tilde{\gamma} \in \Omega^{1}(V_{X,I}^{\circ})$, such that the following identities hold on $V_{X,I}^{\circ}$: 
	\begin{equation*}
		\begin{split}
			& d(t_{i}-t_{i}') =\frac{t_{i}^{2}\, da +2a t_{i}\, dt_{i}}{1+a t_{i}}-
			\frac{a t_{i}^{2}(a \, dt_{i}+t_{i}\, da )}{(1+a t_{i})^{2}}=
			t_{i} \tilde{c} \, dt_{i}+t_{i}^{2}\, \tilde{\beta}, \\%, \quad\mbox{and}\quad  \\
			& d(u_{i}' -u_{i})=\frac{2u_{i}\log(1+a t_{i})\, du_i}{1-u_{i}\log(1+a t_{i})}
			+\frac{u_{i}^{2}(a \, dt_{i}+t_{i}\ da )}{(1+a t_{i})(1-u_{i}\log(1+a t_{i}))}+\\
			+&
			\frac{u_{i}^{2}\log(1+a t_{i})}{(1-u_{i}\log(1+a t_{i}))^{2}}\cdot \left(
			\log(1+a t_{i})\, du_{i}+
			\frac{u_{i}(a \, dt_{i}+t_{i}\, da )}{1+a t_{i}}
			\right)= u_{i}t_{i}\tilde{s} \, du_{i}+
			u_{i}^{2}\tilde{q} \, dt_{i}+t_{i}u_{i}^{2}\, \tilde{\gamma} .
		\end{split}
	\end{equation*}
	By Lemma \ref{lem:computations-d}\ref{item:du-dt}, we have $du_{i}=-t_{i}^{-1}(u_{i}^{2}\, dt_{i}-\eta'(w_i)t'\, dg)$. Fix $\epsilon\in (0,1)$. By Lemma \ref{lem:computations}\ref{item:eta-t-bounded}, $\eta'(w_i)t'=t_{i}^{\epsilon}l$ for some bounded function $l$. Therefore, $du_{i}=-t_{i}^{-1}(u_{i}^{2}\, dt_{i}-t_{i}^{\epsilon}l\, dg)$. Substituting this to the above formula for $d(u_{i}'-u_{i})$, we get that for some bounded functions $\check{q}$, $s$, the following holds:
	\begin{equation*}
		d(u_{i}' -u_{i})=-u_{i}\tilde{s} (u_{i}^{2}\, dt_{i}-t_{i}^{\epsilon}l\, dg)+
		u_{i}^{2} \tilde{q} \, dt_{i}+t_{i}u_{i}^{2}\, \tilde{\gamma}=
		u_{i}^{2}\check{q}\, dt_{i}+t_{i}^{\epsilon}u_{i}s\, dg+t_{i}u_{i}^{2}\, \tilde{\gamma}.
	\end{equation*}
	Adding this equality to the above formula for $d(t_{i}-t_{i}')$ gives
	\begin{equation*}
		d(v_{i}'-v_{i})=(t_{i}\tilde{c}  + u_{i}^2\check{q} )\, dt_{i}+
		t_i^{\epsilon}u_{i} s \, dg+
		(t_{i}^{2}\, \tilde{\beta} +t_{i}u_{i}^{2}\, \tilde{\gamma} ).	
	\end{equation*}
	By Lemma \ref{lem:inequality}, $\frac{|t_{i}\tilde{c}+u_{i}^{2}\check{q}|}{t_{i}+u_{i}^{2}}\leq |\tilde{c}|+|\check{q}|$, so $t_{i}\tilde{c}+u_{i}^{2}\check{q}=(t_{i}+u_{i}^{2})c$ for some bounded function $c$. By definition \eqref{eq:rho-sigma} of $\sigma_i$, we get $t_{i}\tilde{c}+u_{i}^{2}\check{s}=\sigma_{i}t_{i}^{-1}c$. 
	
	Recall that, by Definition \ref{def:adapted-chart}, the holomorphic coordinates on $U_X$ are $(z_{i_1},\dots,z_{i_n})$. Write $x_{j}$, $y_{j}$ for the real and imaginary part of $z_{i_j}$, respectively, and put $\gamma_{2j-1}=dx_{j}$, $\gamma_{2j}=dy_{j}$. Then we can write $\tilde{\beta}=\sum_{j=1}^{2n} \tilde{b}_{j}\gamma_{j}$ and $\tilde{\gamma}=\sum_{j=1}^{2n} \tilde{q}_{j}\gamma_{j}$ for some bounded functions $\tilde{b}_{j}$, $\tilde{q}_{j}$. As before, Lemma \ref{lem:inequality} shows that there is a bounded function $\tilde{c}_{j}$ such that   $t_{i}^{2}\tilde{b}_{j}+t_{i}u_{i}^{2}\tilde{q}_{j}=(t_{i}^{2}+t_{i}u_{i}^{2})\tilde{c}_{j}=\sigma_{i}\tilde{c}_{j}$, where the last equality follows from definition \eqref{eq:rho-sigma} of $\sigma_{i}$. Therefore,  $t_{i}^{2}\tilde{\beta}+t_{i}u_{i}^{2}\tilde{\gamma}=\sigma_{i}\sum_{j=1}^{2n}\tilde{c}_{j}\gamma_{j}=\sigma_{i}\gamma$ for some $1$-form $\gamma$ bounded from $X$. Substituting the above equalities to the formula for $d(v_i'-v_i)$, we get
	\begin{equation*}
		d(v_{i}'-v_i)=\sigma_{i}t_{i}^{-1}c\, dt_{i}+
		t_i^{\epsilon}u_{i} s \, dg+
		\sigma_{i}\gamma=
		\sigma_{i}(ct_{i}^{-1}\, dt_{i}+\gamma)+
		t_i^{\epsilon}u_{i} s \, dg.
	\end{equation*}
	In case $I\neq \emptyset$, we have $u_{i}=0$, so $v_{i}=t_{i}$ and $\sigma_{i}=t_{i}^2$, and therefore, the above equality reads as $d(v_{i}-v_{i}')=t_{i}c\, dv_{i}+\sigma_{i}\gamma$, which is what we wanted to prove. In case $I=\emptyset$, we substitute the formula for $dt_{i}$ from Lemma \ref{lem:computations-d}\ref{item:dt-rho}. Since by \eqref{eq:rho-sigma} we have $\sigma_{i}t_{i}^{-1}\rho_{i}=1$, we get
	\begin{equation*}
		\begin{split}
			& d(v_{i}'-v_i)=
			\sigma_{i}(ct_{i}^{-1}\, dt_{i}+\gamma)+
			t_i^{\epsilon}u_{i} s \, dg
			=\\
			=\ &
			\sigma_{i}ct_{i}^{-1}\, (\rho_{i}(t_{i}\, dv_{i}+u_{i}^{2}w_{i}^{-1}t'\, dg)) +\sigma_{i}\gamma+
			t_i^{\epsilon}u_{i} s \, dg
			= \\
			=\ & 
			ct_{i}\, dv_{i}+(cu_{i}^{2}w_{i}^{-1}t'+t_{i}^{\epsilon}u_{i}s)\, dg+\sigma_{i}\gamma.
		\end{split}
	\end{equation*}
	By Lemma \ref{lem:computations}\ref{item:eta-u}, $u_{i}^{2}w_{i}^{-1}=\eta'(w_{i})$, and by Lemma \ref{lem:computations}\ref{item:eta-t-bounded}, $\eta'(w_{i})t'=t_{i}^{\epsilon}l$ for some bounded function $l$. Hence the coefficient near $dg$ above equals $t_{i}^{\epsilon}q$ for some bounded function $q$, which proves %the second equality of 
	\ref{item:tame_1}. 
\end{proof}

\subsubsection{The functions $r_j$, $w_j$ are $\cC^1$.} We aim at proving Lemma \ref{lem:C1-trans}, which asserts that the  transition maps for our atlas are $\cC^1$. The key step is to prove that in a chart $(U_i,\psi_i)$ introduced in \eqref{eq:AC1-chart}, all functions $r_j$, $w_j$ for $j\neq i$ are $\cC^1$. As a consequence, in Lemma \ref{lem:v1C1} we will infer that the function $v_i$, which is not a coordinate of $\psi_i$ (see formula \eqref{eq:AC1-chart}); as well as the pullbacks of all smooth functions from $U_X$, are $\cC^1$. In Lemmas \ref{lem:w_smooth}\ref{item:w-smooth} and \ref{lem:pullbacks}\ref{item:pullbacks} we will prove a stronger result, replacing  $\cC^1$ by $\cC^{\infty}$. Nonetheless, we provide an independent, direct argument here, too. This way, we give a complete construction of a $\cC^1$-atlas without introducing further technical tools, and illustrate computations behind the proof of Lemmas \ref{lem:w_smooth}, \ref{lem:pullbacks}, which is more complicated.

\begin{notation}\label{not:S}
	Fix a chart $U_X$ adapted to $f$ such that $U_X\cap D\neq \emptyset$. We reorder the components of $D$ so that the associated index set \eqref{eq:index-set} is $S=\{1,\dots, k\}$. By Lemma \ref{lem:C1chart} the map $\psi_{1}\colon U_1\to Q_{k,n}$ is a homeomorphism onto an open subset of a smooth manifold with boundary, so  it endows $U_1$ with a smooth structure. 
	In other words, we take $(g,v_2,\dots,v_k,\rest)$ as smooth coordinates on $U_1$. 
	
	For a subset $V\subseteq U_1$, we denote by $\cS(V)$ (respectively, $\cS^{l}(V)$ for $l\in \N$) the $\R$-algebra of functions on $V$ which are smooth (respectively, $\cC^{l}$), with respect to this smooth structure. We do not use the symbol \enquote{$\cC^{\infty}$} to avoid confusion with the eventual $\cC^{\infty}$-structure on $A$, which we will introduce in Section \ref{sec:AX-smooth}, and whose restriction to $U_1$ will differ from the $\cC^{\infty}$-structure that $U_1$ inherits via $\psi_1$. Since by Lemma \ref{lem:C1chart} the map $\pi|_{U_1\setminus \d A}$ is a diffeomorphism onto its image,  we can -- identifying functions on $X$ with their pullbacks to $A$  -- write an inclusion $\cC^{\infty}(U_X\setminus D)\subseteq \cS(U_1\setminus \d A)$.
	
	For $j\in \{1,\dots, k\}$, we introduce differential operators $\d_{j}\colon \cS^{l}(V)\to \cS^{l-1}(V)$ by 
	\begin{equation*}
		\d_{1}=\tfrac{\d}{\d g},\quad \d_{i}=\tfrac{\d}{\d v_{i}}\quad \mbox{for } i\in \{2,\dots, k\},
	\end{equation*}
	and for $\alpha=(\alpha_1,\dots,\alpha_k)\in \N^k=\Z_{\geq 0}^k$ we write $\d^{\alpha}=\d_{1}^{\alpha_{1}}\dots\d^{\alpha_k}_{k}$, $|\alpha|=\alpha_1+\dots+\alpha_k$. 
	
	For $i\in \{2,\dots, k\}$, we put
	\begin{equation*}
		W_{i}=\{x\in U_1: w_{i}(x)>0\}\quad\mbox{and}\quad T_{i}=\{x\in U_1:t_{i}(x)>0\}.
	\end{equation*}
\end{notation}

Using the above notation, we have the following immediate corollary of Lemma \ref{lem:computations-d}.

\begin{lema}\label{lem:mixed-vanish}
	For every $i,j\in \{2,\dots,k\}$, such that $i\neq j$, we have
	\begin{equation*}
		\d_j v_i=0,\quad 
		\d_j r_i=0,\quad 
		\d_j t_i=0,\quad
		\d_j w_i=0,\quad
		\d_j u_i=0\quad\mbox{and}\quad
		\d_{j}t^{(l)}=0\mbox{ for any } l\geq 0.
	\end{equation*}
\end{lema}
\begin{proof}
	Since $(g,v_2,\dots,v_k)$ are coordinates on $U_1$, we have $\d_j v_i=0$ for $j\neq i$ and $\d_{j}g=0$ for $j\neq 1$. This proves the first equality. The last one follows by definition \eqref{eq:def-tl} of $t^{(l)}$. 
	
	For the remaining ones, note that the basis of $T^{*}_{y}U_1$ at any point $y\in U_1$ is given by $dg,dv_2,\dots,d v_k$ and $1$-forms $dz$, where $z$ ranges through the coordinates of $\rest$. For a function $h$, the coefficient of $dh$ near $dv_j$ is $\d_{j} h$. By Lemma \ref{lem:computations-d}\ref{item:dr}, \ref{item:dt-rho}, \ref{item:dw} and \ref{item:du}, this coefficient is zero for $h=r_i,t_i,w_i,u_i$ whenever $i\neq 1,j$.% It is also trivially zero for $h=v_i$, $i\neq 1,j$. The last assertion follows from the fact that $t=e^{1-g^{-1}}$ depends only on $g$.
\end{proof}

\begin{lema}\label{lem:v1C1}
	In the setting of Notation \ref{not:S}, we have 
	\begin{enumerate}
		\item\label{item:t-smooth}	The function $t$ is smooth on $U_1$, i.e.\ $t\in \cS(U_1)$,
		\item\label{item:viwi-C1} For all $i\in \{1,\dots, k\}$, the functions $v_i,w_{i}$ are in $\cS^{1}(U_1)$.
		\item\label{item:ri-C1} For all $i\in \{1,\dots, k\}$, the function $r_{i}$ is in $\cS^{1}(U_1)$.
		\item\label{item:pullbacks-C1} The pullback of any $\cC^{1}$-function $h$ on $U_X$ is in $\cS^{1}(U_1)$, i.e.\ $\cC^{1}(U_X)\subseteq \cS^{1}(U_1)$.
	\end{enumerate}
\end{lema}
\begin{proof}
	\ref{item:t-smooth} Since $t=\eta^{-1}(g)=e^{1-g^{-1}}$ by \eqref{eq:def-t-g}, and $g$ is one of the coordinates on $U_1$, we have $t\in \cS(U_1)$.
	
	\ref{item:viwi-C1} Fix $i\in \{2,\dots, k\}$. By definition \eqref{eq:AC1-chart} of the coordinate chart $\psi_{1}$, we have  $v_{i}\in \cS^{1}(U_1)$. 
	
	By Lemma \ref{lem:computations-d}\ref{item:dw}, we have $\d_i w_i=-\rho_{i}w_{i}$, $\d_1 w_i=\rho_{i}t'$ and $\d_j w_i=0$ for $j\neq i,1$. Recall that $\rho_i$ was defined in \eqref{eq:rho-sigma} by  $\rho_{i}=(t_{i}+u_{i}^{2})^{-1}$, so since $t_i,u_i$ are continuous by Lemma \ref{lem:intro}\ref{item:intro-extends}, the function $\rho_{i}$ extends to a continuous function away from the common zero locus of $t_i$ and $u_i$, i.e.\ on $T_i\cup W_i$. It follows that $w_{i}\in \cS^{1}(T_i\cup W_i)$. We claim that $w_{i}\in \cS^{1}(U_1)$. 
	
	By definitions \eqref{eq:rho-sigma} and \eqref{eq:def-w_i-u_i}, we have $\rho_{i}=(t_{i}+u_i^2)^{-1}$ and $u_i=\eta(w_i)=(1-\log w_{i})^{-1}$,  so
	\begin{equation*}
		-\d_{i}w_{i}=\rho_{i}w_{i}=\frac{w_{i}}{t_{i}+u_{i}^{2}}=\frac{w_{i}(1-\log w_{i})^{2}}{t_{i}(1-\log w_{i})^{2}+1}\rightarrow 0\quad \mbox{as } w_{i}\rightarrow 0.
	\end{equation*}
	By Lemma \ref{lem:computations}\ref{item:t'} and \ref{item:t=tiwi}, we have $t'=t(1-\log t)^{2}$ and $t=t_{i}w_i$, so, using Lemma~\ref{lem:computations-d}\ref{item:dw}, we get
	\begin{equation*}
		\d_{1}w_{i}=\rho_{i} t'=\frac{t'(1-\log w_{i})^{2}}{t_{i}(1-\log w_{i})^{2}+1}=\frac{t_{i}w_{i}(1-\log t_{i}-\log w_{i})^{2}(1-\log w_{i})^{2}}{t_{i}(1-\log w_{i})^2+1} \rightarrow 0 \quad\mbox{as } t_{i},w_i\rightarrow 0.
	\end{equation*}
	Therefore, $\d_{i}w_{i}$ and $\d_1 w_{i}$ extend to continuous functions on $U_1$, so $w_i\in \cS^{1}(U_1)$, as claimed.
	
	By Lemma \ref{lem:intro}\ref{item:intro-sum}, we have $w_{1}=-\sum_{i=2}^{k}w_{i}$, so $w_{1}\in \cS^{1}(U_1)$, too. The remaining function $v_1$ is defined in \eqref{eq:def-v_i} as $v_{1}=t_{1}-u_{1}=tw_{1}^{-1}-\eta(w_1)$. Recall that $t\in \cS(U_1)$ by \ref{item:t-smooth}, and the restriction $\eta|_{(0,1]}$ is smooth. Therefore, since $w_{1}|_{U_{1}}>0$ by definition \eqref{eq:Ui} of $U_{1}$, we infer that $v_{1}\in \cS^{1}(U_1)$, as claimed.
	
	\ref{item:ri-C1} Consider first the case $i\in \{2,\dots, k\}$. Lemma \ref{lem:computations-d}\ref{item:dt-g} implies that $\d_{1}t_{i}$ and $\d_{i}t_{i}$ are continuous on $T_i$, and $\d_{j}t_{i}=0$ for $j\neq i,1$, so like before we conclude that $t_{i}\in \cS^{1}(T_i)$. By Lemma \ref{lem:computations}\ref{item:t-r}, we have $r_i=e^{(-m_it_i)^{-1}}$, so $r_i\in \cS^{1}(T_i)$, too. We need to show that $r_i$ is $\cC^1$ on its zero locus $U_1\setminus T_i$.
	
	First, we claim that for every $j\in \{1,\dots, k\}$, the function $t_{i}\cdot \d_{j} t_{i}$ is bounded. By Lemma \ref{lem:computations-d}\ref{item:dt-rho}, we have $\d_{i}t_{i}=\rho_{i}t_{i}$, $\d_{1}t_{i}=\rho_{i}u_{i}^{2}w_{i}^{-1}t'$ and $\d_jt_i=0$ for $j\neq 1,i$. By definition \eqref{eq:rho-sigma} of $\rho_{i}$, we have $\rho_{i}t_{i}=(1+u_{i}^{2}t_{i}^{-1})^{-1}\leq 1$, so $\d_{i}t_{i}$ is bounded, hence $t_i \cdot \d_i t_i$ is bounded, too, as needed. To deal with $\d_{1}t_{i}=\rho_{i}u_{i}^{2}w_{i}^{-1}t'$, recall that $u_{i}^{2}w_{i}^{-1}=\eta'(w_{i})$ by Lemma \ref{lem:computations}\ref{item:eta-u}, and by Lemma \ref{lem:computations}\ref{item:eta-t-bounded}, we have $\eta'(w_i)t'=t_{i}^{\epsilon}l$ for some bounded function $l$ and $\epsilon>0$. It follows that  $\d_1t_i=\rho_it_{i}^{\epsilon} l$. We have seen that $\rho_{i}t_{i}=(1+t_{i}^{-1}u_{i}^{2})^{-1}$ is bounded, so $t_{i}\d_{1}t_{i}=\rho_{i}t_{i}^{\epsilon+1}l$ is bounded, too, which proves the claim. 
	
	Now by Lemma \ref{lem:computations-d}\ref{item:dr}, we have $\d_{j}r_{i}=m_{i}^{-1}\cdot t_{i}^{-3}r_{i}\cdot t_{i}\d_j t_i$. By Lemma \ref{lem:computations}\ref{item:t-r}, we have $r_{i}=e^{(-m_{i}t_{i})^{-1}}$, so $t_{i}^{-3}r_{i}\rightarrow 0$ as $t_{i}\rightarrow 0$. Since we have shown that $t_{i}\d_{j}t_{i}$ is bounded, we infer that $\d_j r_{i}\rightarrow 0$ as $t_{i}\rightarrow 0$, so $r_{i}\in \cS^{1}(U_1)$, as needed.
	
	It remains to show that $r_{1}\in \cS^{1}(U_1)$. By Lemma \ref{lem:computations}\ref{item:t=tiwi} we have  $t_{1}=tw_{1}^{-1}$, and $w_{1}>0$ on $U_1$ by definition  \eqref{eq:Ui}. By \ref{item:t-smooth} and \ref{item:viwi-C1}, we have $t,w_{1}\in \cS^{1}(U_1)$, so $t_1\in \cS^{1}(U_1)$. Now Lemma \ref{lem:computations}\ref{item:t-r} gives $r_{1}=e^{-(m_{1}t_{1})^{-1}}\in \cS^{1}(U_1)$, as claimed.
	
	\ref{item:pullbacks-C1} For $i\in \{1,\dots,k\}$ we have $r_{i}\in \cS^{1}(U_1)$ by \ref{item:pullbacks} and $\theta_{i}\in \cS^{1}(U_1)$ by definition \eqref{eq:AC1-chart} of the chart. Thus the smooth coordinates $r_i\cos(2\pi\theta_i)$ and $r_i\sin(2\pi\theta_i)$ on $U_X$ are in $\cS^1(U_1)$. By \eqref{eq:AC1-chart}, the remaining coordinates $z_{i_{k+1}},\dots,z_{i_n}$ are in $\cS^{1}(U_1)$, too. Now, part \ref{item:pullbacks-C1} follows from the chain rule.
\end{proof}

\subsubsection{The A'Campo space is a $\cC^1$-manifold}\label{sec:C1-proof} We are now in the position to prove the remaining results stated in Section \ref{sec:AX-C1-statements}, which endow $A$ with a $\cC^1$-structure.

\begin{proof}[Proof of Lemma \ref{lem:C1-trans}]
	Fix two $\cC^{1}$-charts $(U_i,\psi_i)$ and $(U_j',\psi_{j}')$ associated to charts $(U_X,\psi_X)$ and $(U'_{X},\psi_{X}')$ adapted to $f$. By Lemma \ref{lem:C1chart}, the transition map $\psi_{j}'\circ \psi_{i}^{-1}|_{\psi_{i}(U_i\cap U_j')}$ is a homeomorphism between open subsets of $Q_{k,n}$ and $Q_{k',n}$. To prove that it is a $\cC^1$-diffeomorphism, we need to prove that it is $\cC^1$ at $\psi_{i}(x)$ for any $x\in U_i\cap U_{j}'$, and that its Jacobian determinant does not vanish there.
	
	If $x\not\in \d A$ then near $x$ we have the natural smooth structure pulled back from $X\setminus D$, and by Lemma \ref{lem:C1chart} both $\psi_{i}$ and $\psi_{j}$ are diffeomorphisms with respect to this structure; so the assertion follows.
	
	Assume $x\in \d A$, so $x\in A_{I}^{\circ}$ for some nonempty $I\subseteq \{1,\dots, N\}$, say $I=\{1,\dots, k\}$. Without loss of generality, we can assume that $U_{X}\subseteq U_{X}'$,  and that the associated index set \eqref{eq:index-set} for $U_X$ is $\{1,\dots, k\}$. 
	
	Assume first that $\psi_{X}=\psi_{X}'|_{U_X}$. Then the index set associated to $U_{X}'$ is $\{1,\dots, k\}$, too. Say that $i=1$, $j=2$, and put $V=U_1\cap U_2$. The map $\psi_1$ is the smooth chart used in Notation \ref{not:S}, and $\psi_2$ differs from it by replacing $v_2$ with $v_1$. Thus it is sufficient to show that $v_2\in \cS^{1}(V)$, and that the Jacobian determinant of $\psi_{2}\circ\psi_{1}^{-1}$, which equals $\d_{2}v_1$, does not vanish on $V$.
	
	The first assertion follows from Lemma \ref{lem:v1C1}\ref{item:viwi-C1}. For the second one, recall from \eqref{eq:def-v_i} that $v_1=t_{1}-u_{1}=tw_1^{-1}-\eta(w_1)$, and this formula is valid on $U_1$ since $w_1|_{U_1}>0$ by definition \eqref{eq:Ui} of $U_1$. By Lemma \ref{lem:mixed-vanish} we have $\d_{2} t=0$, so by Lemma \ref{lem:computations}\ref{item:t=tiwi} and \ref{item:eta-u}:
	\begin{equation*}
		\d_2 v_1=-(tw_1^{-2}+\eta'(w_1))\cdot \d_2 w_1=-w_{1}^{-1}\cdot (t_{1}+u_{1}^{2})\cdot \d_2 w_1=
		-w_{1}^{-1}\rho_{1}^{-1}\cdot \d_2 w_1,
	\end{equation*}
	where the last equality follows from definition \eqref{eq:rho-sigma} of $\rho_1$. 
	
	By Lemma \ref{lem:intro}\ref{item:intro-sum}, we have $w_{1}=-\sum_{i=2}^{k} w_i$, so Lemma \ref{lem:mixed-vanish} implies that  $\d_{2} w_{1}=-\d_2 w_2$. Now by Lemma  \ref{lem:computations-d}\ref{item:dw}, we have $\d_2 w_2=-\rho_2 w_2$, so $\d_2 w_1=\rho_2 w_2$, and therefore 
	\begin{equation*}
		\d_2 v_1=-w_{1}^{-1}\rho_{1}^{-1}\cdot w_{2}\rho_2.
	\end{equation*}
	By definition \eqref{eq:Ui} of $U_i$, we have $w_1,w_2>0$ on $V=U_1\cap U_2$, hence $\rho_1,\rho_2>0$ on $V$ by  definition \eqref{eq:rho-sigma} of $\rho_i$. We conclude that $\d_2 v_1\neq 0$, as needed.
	\smallskip
	
	Consider now the general case, where $\psi_{X}$ may be different than $\psi_{X}'$. Since $U_X\subseteq U_X'$, the index set $S'$ associated to $U_{X}'$ contains $\{1,\dots, k\}$, so, say, $S'=\{1,\dots, l\}$ for some $l\geq k$. The special case considered above allows to assume $i=j=1$. We can write the transition map $\psi_{1}'\circ \psi_{1}^{-1}$ as a composition $\varphi_{2}\circ\varphi_1$, where
	\begin{equation*}
		\begin{tikzcd}
			(g,v_2,\dots,v_k,\rest)
			\ar[r, mapsto, "\varphi_1"]
			& 
			(g,v_2',\dots,v_k',\theta'_1,\dots,\theta'_k,z'_{i_{k+1}},\dots z'_{i_{n}})
			\ar[r, mapsto, "\varphi_2"] 
			&
			(g,v_2',\dots, v_{l}',\rest').
		\end{tikzcd}
	\end{equation*}
	
	We will first show that $\varphi_1$ is a $\cC^{1}$-diffeomorphism at $x$. By Lemma \ref{lem:v1C1}\ref{item:pullbacks-C1}, the pullbacks of smooth functions on $U_X$ are in $\cS^{1}(U_1)$, so $z_{i_{k+1}}',\dots,z_{i_n}'\in \cS^{1}(U_1)$, and, by Lemma \ref{lem:dv}\ref{item:theta-theta'}, $\theta_{i}'-\theta_{i}\in \cS^{1}(U_1)$ for all $i\in \{1,\dots, k\}$. Since $g,\theta_{1},\dots,\theta_{k}$ are coordinates on $U_1$, we infer that $g,\theta_{1}',\dots,\theta_{k}'\in \cS^{1}(U_1)$.
	
	Fix $i\in \{2,\dots,k\}$. We claim that $v_{i}'\in \cS^{1}(T_i)$. By Lemma  \ref{lem:dv}\ref{item:vp-v}, it is enough to show that the functions $a$, $t_{i}$ and $u_{i}$ are in $\cS^{1}(T_{i})$. Since the function $a$ is smooth on $U_{X}$, it is in $\cS^{1}(U_1)$ by Lemma \ref{lem:v1C1}\ref{item:pullbacks-C1}. By Lemma  \ref{lem:v1C1}\ref{item:ri-C1}, $r_{i}\in \cS^{1}(U_1)$, so since $t_{i}=-(m_i\log r_i)^{-1}$, we have $t_{i}\in \cS^{1}(T_i)$. Now by Lemma \ref{lem:computations-d}\ref{item:du} we have $\d_{j}u_{i}=0$ for $j\neq 1,i$ and $\d_{i}u_{i}=-\rho_{i}u_{i}^{2}=-(t_{i}u_{i}^{-2}+1)^{-1}$, so $\d_{j}u_{i}$ is continuous on $U_1$ for $j\neq 1$. Eventually, $\d_{1}u_{i}=\rho_{i}\eta'(w_{i})t'=\rho_{i}t_{i}(1+g\log t_{i})^{-1}$ by Lemma \ref{lem:computations}\ref{item:eta-t}, so $\d_{1}u_i$ is continuous on $T_{i}$. Hence $u_{i}\in \cS^{1}(T_i)$, and therefore $v_{i}'\in \cS^{1}(T_i)$, as claimed. 
	%Thus Lemma \ref{lem:dv}\ref{item:vp-v} implies that $v_{i}'\in \cS^{1}(T_i)$. 
	In turn, by Lemma \ref{lem:dv}\ref{item:tame_1} we have $d(v_i'-v_{i})\rightarrow 0$ as $t_{i}\rightarrow 0$, so $v_{i}'-v_{i}\in \cS^1(U_1)$, and therefore $v_{i}'\in \cS^{1}(U_1)$. Thus $\varphi_{1}$ is $\cC^{1}$.
	
	To prove that $\varphi_1$ is a $\cC^1$-diffeomorphism at $x$, it remains to show that its Jacobian matrix is invertible at $x$. To this end, we will first write the Jacobian matrix of the transition map $\psi_{X}'\circ \psi_{X}^{-1}\colon (z_{i_1},\dots, z_{i_n})\mapsto (z_{i_1}',\dots,z_{i_n}')$. 
	By Lemma \ref{lem:w_extends}\ref{item:lambda}, for any $i\in \{1,\dots,k\}$ there is a nonvanishing holomorphic function  $\lambda_{i}\in \cO_{X}^{*}(V)$ such that $z_{i}'=\lambda_i z_{i}$. Hence for $j\neq i$ we have $\frac{\d z_{i}'}{\d z_{j}}=z_{i}\frac{\d \lambda_{i}}{\d z_{j}}$, so $\frac{\d z_{i}'}{\d z_{j}}|_{D_{i}}=0$. It follows that the Jacobian matrix of  $\psi_{X}^{-1}\circ \psi_{X}'$ at $\pi(x)$  has a block form $\begin{psmallmatrix} * & 0 \\ * & Z \end{psmallmatrix}$, for some $2(n-k)\times 2(n-k)$ matrix $Z$. Since this block matrix is invertible, so is $Z$.
	
	%	Now, the invertible matrix $Z$ occurs as the bottom-right block of the Jacobian matrix of $\varphi_1$ at $x$. We claim that 
	%	the latter matrix has the form 
	Now, we claim that the Jacobian matrix of $\varphi_1$ at $x$ has the form 
	\begin{equation*}%\label{eq:trans-jacobian}
		\left[
		\begin{matrix}
			\id_{k\times k} & 0 & 0 \\
			* & \id_{k\times k} & * \\
			* & 0 & Z
		\end{matrix}
		\right].
	\end{equation*} 
	%and therefore it is invertible.	
	The first row consists of partial derivatives of the first coordinate $g$, so it is the first unit vector. For $i\in \{2,\dots, k\}$ we have $t_i(x)=0$, so $\sigma_{i}(x)=0$ by definition \eqref{eq:rho-sigma} of $\sigma_i$. Thus by Lemma \ref{lem:dv}\ref{item:tame_1} we have $dv_{i}'=dv_{i}$ at $x$, so the $i$-th row is the $i$-th unit vector, as claimed.
	
	The middle $k\times k$ block is $[\frac{\d \theta_{j}'}{\d \theta_{i}}]_{1\leq i,j\leq k}$. By Lemma \ref{lem:dv}\ref{item:theta-theta'}, we have $\theta_{j}'=\theta_{j}+\beta$ for some smooth $\beta\colon U_1\to \S^1$ (recall that we use additive notation \ref{not:S1} for $\S^1=\R/\Z$). Write $z_{i}=x_{i}+\imath y_{i}$. Then $\frac{\d \beta}{\d \theta_{i}}=-r_{i}\sin(2\pi\theta_{i}) \frac{\d \beta}{\d x_{i}}+r_{i}\cos(2\pi\theta_{i})\tfrac{\d \beta}{\d y_{i}}$, so $\frac{\d \beta}{\d \theta_{i}}=0$ at $x$. Thus the middle block is indeed $\id_{k\times k}$.
	
	The last $2(n-k)$ rows correspond to the (real and imaginary parts of) the functions $z_{i_{k+1}}',\dots, z_{i_{n}}'$. The above computation shows that their partial derivatives with respect to $\theta_{1},\dots, \theta_{k}$ vanish. The ones with respect to (real and imaginary parts of) $z_{i_j}$ give the matrix $Z$. This proves that the Jacobian of $\varphi_1$ has the required form, so it is invertible because $Z$ is.
	\smallskip
	
	Now, consider the map $\varphi_2$. Its inverse $\varphi_{2}^{-1}$ is defined away from the zero locus of $r_{j}$, for $j\in \{k+1,\dots, l\}$, and it replaces the coordinates $(v_{j}',\theta_{j}')$ of $U_{1}'$ by real and imaginary parts of $r_{j}'\cdot \exp(2\pi\imath \theta_{j}')$. We will show  $\varphi_{2}^{-1}$ is a $\cC^{1}$-diffeomorphism. It is sufficient to show it for the map $v_{j}'\mapsto r_{j}'$: the remaining part is passing from polar to standard coordinates. Since $r_{j}'$ is a pullback of a smooth function on $U_{X}'$, by Lemma \ref{lem:v1C1}\ref{item:ri-C1} it is $\cC^1$ in the coordinates of $U_{1}'$. Hence our map is $\cC^1$. By Lemma \ref{lem:mixed-vanish}, we have $\frac{\d r_{j}'}{\d v_{i}'}=0$ if $i\neq j$, so its Jacobian determinant is $\prod_{j=k+1}^{l}\frac{\d r_{j}'}{\d v_{j}'}$.  By Lemma \ref{lem:computations-d}\ref{item:dr},\ref{item:dt-rho}, we have $\frac{\d r_{j}'}{\d v_{j}'}=m_{j}^{-1}(t_{j}')^{-2}r_{j}'\cdot \frac{\d t_{j}'}{\d v_{j}'}=m_{j}^{-1}(t_{j}')^{-2}r_{j}'\cdot \rho_{j}'t_{j}'$, which is nonzero whenever $r_{j}'$ is, as needed.
\end{proof}

\begin{proof}[Proof of Proposition \ref{prop:uniqueC1}]
	Since the $\cC^{1}$-structure on $A\setminus \d A$ agrees with the one on $X\setminus D$, the restriction $\Phi|_{A\setminus \d A}$ is a $\cC^{1}$-diffeomorphism. We will show that $\Phi$ is a $\cC^{1}$-diffeomorphism at every point $x\in \d A$. Fix a chart $U_{X}$ adapted to $f$ containing $\pi(x)$, and let $U$ and $U'$ be its preimages in $A$ and $A'$, respectively. Fix $i$ in the associated index set such that  $w_{i}(x)>\tfrac{1}{n+1}$. Then $x$ lies in the subset $U_{i}\subseteq U$ defined in \eqref{eq:Ui}, as well as in its counterpart $U_{i}'\subseteq U'$. To prove that $\Phi$ is a $\cC^{1}$ diffeomorphism at $x$, we need to prove that the composition $\psi_{i}'|_{\Phi(U_i)\cap U_{i}'}\circ \Phi|_{U_{i}} \circ \psi_{i}^{-1}$ is a $\cC^{1}$-diffeomorphism between open subsets of $Q_{k,n}$. Since both $\psi_{i}$ and $\psi_{i}'$ correspond to the same adapted chart $U_{X}$, this map is the identity.
\end{proof}

\begin{proof}[Proof of Proposition \ref{prop:AXC1}]
	\ref{item:AX-piC1} By Lemma \ref{lem:v1C1}\ref{item:pullbacks}, the map $\pi$ pulls back $\cC^{1}$-functions from the charts $U_X\subseteq X$ to $\cC^{1}$-functions on their preimages in $A$, so it is $\cC^{1}$. Its restriction $\pi|_{A\setminus \d A}$ is a $\cC^{1}$-diffeomorphism by Lemma \ref{lem:C1chart}.
	
	\ref{item:AX-gC1} The function $g\colon A\to [0,1)$ is a coordinate in each chart \eqref{eq:AC1-chart} meeting $\d A$, hence it is a $\cC^{1}$-submersion (onto its image) near $\d A$. Away from $\d A$, it is a composition of a $\cC^{1}$-diffeomorphism $\pi|_{A\setminus \d A}$ with a submersion $|f||_{X\setminus D}$, so it is a $\cC^{1}$-submersion there, too. Similarly, over each chart adapted to $f$ we have $\theta=\sum_{i\in S}m_{i}\theta_{i}$ by Definition \ref{def:adapted-chart}\ref{item:f-locally}, so $\theta$ is a nonzero linear combination of coordinates of each chart \eqref{eq:AC-chart}, hence a $\cC^1$-submersion near $\d A$. In turn, $\theta|_{A\setminus \d A}$ is a pullback of a submersion $\frac{f}{|f|}$. Thus $(g,\theta)$ is a $\cC^{1}$-submersion. It follows that the map $f\AC=(\exp(-\exp(g^{-1}-1)),\theta)$ is $\cC^{1}$.
	
	\ref{item:AX-vbarC1} Let $U_{X}^{p}$ be the chart adapted to $f$ from the atlas $\cU_{X}$ used in Definition \ref{def:bounded-from-X}, and for $i$ in the associated index set let $v_{i}^{p}$ be the corresponding function \eqref{eq:def-v_i}. By Lemma \ref{lem:v1C1}\ref{item:viwi-C1}, $v_{i}^{p}$ is $\cC^{1}$ on $\pi^{-1}(U_{X}^{p})$. Let $\tau^{p}$ be the element of the smooth partition of unity $\btau$ inscribed in $\cU_{X}$, supported in $U_{X}^{p}$. Since the map $\pi$ is $\cC^{1}$ by \ref{item:AX-piC1}, the pullback of $\tau^{p}$ is $\cC^{1}$; hence $\tau^{p}v_{i}^{p}$ is a $\cC^{1}$-function on $A$. Therefore, the function $\bar{v}_{i}$ defined in \eqref{eq:def-vbar-ubar-mu} as $\sum_{p}\tau^{p}v_{i}^{p}$ is $\cC^{1}$, too, as needed.
\end{proof}

\subsection{\texorpdfstring{$\cC^\infty$-}{Smooth }atlases on the A'Campo space}\label{sec:AX-smooth}

Let $(\cU_X,\btau)$ be a pair adapted to $f$, and let $A$ be its associated A'Campo space. We will now use the function $g\colon A\to [0,1)$, see \eqref{eq:def-t-g}, and the distinguished global functions $\bar{v}_i\colon A\to [-1,1]$ corresponding to $(\cU_X,\btau)$, see \eqref{eq:def-vbar-ubar-mu}, to produce a $\cC^\infty$-atlas on $A$. This atlas will be compatible with the $\cC^1$-atlas defined in Section \ref{sec:AX-C1}. However, if $(\cU_X,\btau)$ and $(\cU'_X,\btau')$ are pairs adapted to $f$, and $A$, $A'$ are corresponding A'Campo spaces, then the $\cC^1$-diffeomorphism $\Phi\colon A\to A'$ from Proposition~\ref{prop:uniqueC1} will not be $\cC^\infty$ in general.  

\subsubsection{Definition of the \texorpdfstring{$\cC^\infty$-}{smooth }atlas associated with $(\cU_X,\btau)$}\label{sec:AX-smooth-def}
Fix a pair $(\cU_X,\btau)$ adapted to $f$. Let $U_X$ be a chart adapted to $f$, which perhaps does not belong to the fixed atlas $\cU_X$. Let $S$ be the index set \eqref{eq:index-set} associated with $U_X$. Assume $S\neq \emptyset$, let $U\de \pi^{-1}(U_X)$ and let $\bigcup_{i\in S}U_i$ be the covering of $U$ introduced in \eqref{eq:Ui}. For every $i\in S$, we define 
\begin{equation}\label{eq:AC-chart}
	\bar{\psi}_{i}=(g,(\bar{v}_{j})_{j\in S\setminus \{i\}};\rest)\colon 
	U_{i}\to Q_{k,n}.
\end{equation}
It differs from the $\cC^{1}$-chart $\psi_{i}$ introduced in \eqref{eq:AC1-chart} by replacing local functions $v_{j}$ by global $\bar{v}_{j}$. In particular, by Proposition \ref{prop:AXC1}\ref{item:AX-vbarC1}, $\bar{\psi}_i$ is $\cC^{1}$ with respect to the $\cC^{1}$-structure on $A$ introduced in Section \ref{sec:AX-C1}. Note that the formula \eqref{eq:def-vbar-ubar-mu} for the function $\bar{v}_j$ used in \eqref{eq:AC-chart} depends on the fixed pair $(\cU_X,\btau)$.

We will now use charts \eqref{eq:AC-chart} to define a $\cC^{\infty}$ structure on $A$. As in Section \ref{sec:AX-C1}, we state the main results first, and postpone their proofs to subsequent Sections \ref{sec:smoothchart}--\ref{sec:AX-smooth-proofs}.

\begin{lema}
	\label{lem:smoothchart}
	Consider in $A$ the $\cC^1$-structure constructed in the previous section. For every $x\in D$ there exists a chart $U_X$ adapted to $f$ containing $x$, such that each map \eqref{eq:AC-chart} is a $\cC^{1}$-diffeomorphism onto an open subset of $Q_{k,n}$. Moreover, the restriction $\psi_i|_{U_i\setminus\partial A}$ is $\cC^\infty$ for the natural smooth structure in $A\setminus \partial A$.
\end{lema}

By Lemma \ref{lem:smoothchart}, pairs $(U_{i},\bar{\psi}_{i})$ defined in \eqref{eq:AC-chart} are candidates for $\cC^\infty$ charts on $A$. We will call them the {\em smooth charts} corresponding to the adapted chart $U_X$ \emph{and to the pair $(\cU_X,\btau)$}. As before, we abuse the definition of a $\cC^\infty$ atlas by allowing open subsets of $Q_{k,n}$ as targets for the charts. Nonetheless, the obvious refinement produces charts whose targets are open subsets of half-euclidean spaces. 

In order to complete the definition of the $\cC^\infty$ atlas we need to prove that the corresponding transition functions are $\cC^\infty$. This is the content of Lemma \ref{lem:transsmooth} below.

\begin{lema}\label{lem:transsmooth}
	Let $U_X$, $U'_X$ be charts adapted to $f$,  satisfying the statement of Lemma~\ref{lem:smoothchart}. Let  $S$, $S'$ be their associated index sets \eqref{eq:index-set} and let $k=\#S'$, $k'=\#S'$. For $i\in S$ and $j\in S'$ let $(U_i,\psi_i)$, $(U_j',\psi_{j}')$ be the charts defined in \eqref{eq:AC-chart} above. Then the transition map 
	\begin{equation*}
		\bar{\psi}'_{j}\circ \bar{\psi}_{i}^{-1}|_{\bar{\psi}_{i}(U_i\cap U'_j)}\colon \bar{\psi}_{i}(U_i\cap U'_j)\to \bar{\psi}'_j(U_i\cap U'_j)	
	\end{equation*}
	is a $\cC^{\infty}$-diffeomorphism between open subsets of $Q_{k,n}$ and $Q_{k',n}$.
\end{lema}

Lemmas \ref{lem:smoothchart}, \ref{lem:transsmooth} endow $A$ with a $\cC^{\infty}$ structure, compatible with the $\cC^{1}$-structure introduced in Section \ref{sec:AX-C1}, and with the natural smooth structure on $A\setminus \d A=X\setminus D$.% Remark \ref{rem:not-C2} below shows that Proposition \ref{prop:uniqueC1} no longer holds in the $\cC^{\infty}$ setting.

\begin{remark}\label{rem:not-C2}
	Example \ref{ex:not-C2} shows that the $\cC^{1}$-diffeomorphism $\Phi$ comparing the A'Campo spaces defined using different adapted pairs is not $\cC^2$ in general. Indeed, let $X=\D^{2}_{\epsilon}$, $f\colon \D^2_{\epsilon}\ni (z_1,z_2)\mapsto z_1z_2\in \C$; and let $\cU_{X}$, $\cU_{X}'$ be atlases adapted to $f$, each consisting of one chart: for $\cU_{X}$, we take the standard one $(z_1,z_2)$, and for $\cU_{X}'$ we take $(z_1',z_2')=(e^{-1}z_1,ez_2)$  as in Example \ref{ex:not-C2}. Let $A$ and $A'$ be the A'Campo spaces corresponding to $(\cU_{X},\{1\})$ and $(\cU_{X}',\{1\})$. Since the partition of unity is trivial in each case, the smooth charts \eqref{eq:AC-chart} are the same as the $\cC^{1}$-charts \eqref{eq:AC1-chart}. The map $\Phi\colon A\to A'$ is defined in \eqref{eq:canonical_homeo} so that $\Phi|_{\d A}=\id_{\d A}$. Example \ref{ex:not-C2} shows that the identity is not a $\cC^{2}$-diffeomorphism between these two smooth structures on $\d A$: thus $\Phi$ is not $\cC^2$.
\end{remark}

The $\cC^{\infty}$ structure defined above has the following important properties, cf.\ Proposition \ref{prop:AXC1}.
\begin{prop}\label{prop:AXsmooth}
	Let $(\cU_{X},\btau)$ be a pair adapted to $f$, and let $A$ be the corresponding A'Campo space, with the $\cC^\infty$-structure defined above.
	\begin{enumerate}
		\item\label{item:AX-pismooth} The map $\pi\colon A\to X$ is smooth. Its restriction $\pi|_{A\setminus \d A}\colon A\setminus \d A\to X\setminus D$ is a diffeomorphism.
		\item\label{item:AX-gsmooth} The map $(g,\theta)\colon A\to [0,1)\times \S^{1}$ is a smooth submersion. In particular, $f\AC\colon A\to \C_{\log}$ is smooth.%, where $g$ is introduced in QUOTE.  
		\item \label{item:AX-vbar-smooth}For every $i\in \{1,\dots, N\}$ the function $\bar{v}_i:A\to [-1,1]$ defined in \eqref{eq:def-vbar-ubar-mu} for $(\cU_{X},\btau)$ is smooth.%$\C^\infty$.
		\item\label{item:AX-stratification} For every nonempty subset $I \subseteq \{1,\dots, N\}$, the restriction $(\pi,\mu)\colon \Int_{\d A} A_{I}^{\circ}\to X_{I}^{\circ}\times \Delta_{I}^{\circ}$ is a smooth  $(\S^{1})^{\#I}$-bundle.
	\end{enumerate}
\end{prop}

\begin{remark}\label{rem:Ehressmann}
	Proposition \ref{prop:AXsmooth}\ref{item:AX-gsmooth} and the Ehressmann lemma give a diffeomorphism $A\cong g^{-1}(\delta)\times [0,1)$ for any $\delta\in [0,1)$. By definition \eqref{eq:def-t-g} of $g$, for every $\delta>0$ we have $g^{-1}(\delta)=f^{-1}(\S_{\epsilon}^{1})$, where $\epsilon=\exp(-\exp(\delta^{-1}-1))>0$. Therefore, the A'Campo space $A$ is diffeomorphic to $f^{-1}(\S^{1}_{\epsilon})\times [0,1)$ for any $0<\epsilon\ll 1$. In particular, its diffeomorphism type does not depend on the choice of $(\cU_X,\btau)$. Nonetheless, as we have seen in Remark \ref{rem:not-C2}, given two A'Campo spaces $A$ and $A'$ the composition of any diffeomorphism from $A$ to $f^{-1}(\S^{1}_{\epsilon})\times [0,1)$ with the inverse of any diffeomorphism from $A'$ to $f^{-1}(\S^{1}_{\epsilon})\times [0,1)$ cannot be equal to the canonical map $\Phi\colon A\to A'$ defined in \eqref{eq:canonical_homeo}.
\end{remark}

Lemma \ref{lem:smoothchart} will be proved in Section \ref{sec:smoothchart} below. The proofs of Lemma~\ref{lem:transsmooth} and Proposition~\ref{prop:AXsmooth} are somewhat intricate and need preparation which will be developed in the next sections.

\subsubsection{Compatibility with the $\cC^1$-structure}\label{sec:smoothchart} 
In this section, we prove Lemma \ref{lem:smoothchart}. It asserts that charts \eqref{eq:AC-chart} yield a particular $\cC^{1}$-atlas on for the $\cC^1$-structure on $A$ defined in Section \ref{sec:AX-C1}. To this end, we will use the following consequence of Lemma \ref{lem:dv}\ref{item:tame_0},\ref{item:tame_1}. We use the notion of forms bounded from $X$, introduced in Definition \ref{def:bounded-from-X}, and the function $\sigma_{i}=t_{i}^{2}+t_iu_i^2$ defined in \eqref{eq:rho-sigma}.

\begin{lema}\label{lem:dvbar}
	Let $U_X$ be a chart adapted to $f$, with associated index set $S$. Fix $i\in S$, and a subset $I\subseteq S$ not containing $i$. Put $U_{X,I}^{\circ}=U_{X}\cap X_{I}^{\circ}$. Then the following hold.
	\begin{enumerate}
		\item \label{item:tame_0-bar} There is a bounded function $b\in \cC^{\infty}(U_{X,I}^{\circ})$ such that on $U_{X,I}^{\circ}$ we have $v_i-\bar{v}_i=\sigma_{i}b$.
		\item \label{item:tame_1-bar} For every $\epsilon\in (0,1)$, there are bounded functions $c,q\in \cC^{\infty}(U_{X,I}^{\circ})$, and a bounded $1$-form $\gamma\in \Omega^{1}(U_{X,I}^{\circ})$ such that on $U_{X,I}^{\circ}$ we have 
		\begin{equation*}
			d(v_i-\bar{v}_{i})
	%		=\sigma_{i}(ct_{i}^{-1}\, dt_{i}+\gamma)+ t_i^{\epsilon} u_{i}^2 q\, dg%. \\
			= t_{i}\cdot c\, dv_{i}+t_i^{\epsilon}\cdot q\, dg+\sigma_{i}\cdot \gamma.
		\end{equation*}
	\end{enumerate}
In particular, taking $I=\emptyset$, we get that the above equalities hold on $U_{X,\emptyset}^{\circ}=U_{X}\setminus D=U\setminus \d A$, with $b,c,q$ and $\gamma$ bounded from $X$.
\end{lema}
\begin{proof}
	For $p\in R$, let $U_{X}^{p}$ be the corresponding chart from $\cU_{X}$, and let $v_{i}^{p}$ be its associated function \eqref{eq:def-v_i}. By \eqref{eq:def-vbar-ubar-mu}, we have $\bar{v}_{i}=\sum_{p\in R}\tau^{p}v_{i}^{p}$. Put $V_{X,I}^{p}=U_{X}^{p}\cap U_{X,I}^{\circ}$.%, $V^{p}=\pi^{-1}(V_{X}^{p})$, $U=\pi^{-1}(U_X)$. 
	
	\ref{item:tame_0-bar} Since $\sum_{p\in R}\tau^{p}=1$, we have $v_{i}-\bar{v}_{i}=\sum_{p\in R}\tau^{p}(v_{i}-v_{i}^{p})$. For every $p\in R$ there is a bounded function $b^{p}\in \cC^{\infty}(V_{X,I}^{p})$ such that $v_{i}-v_{i}^{p}=\sigma_{i}b^{p}$: indeed, if $D_{i}$ meets $\bar{U}_{X}^{p}$ then it follows from Lemma \ref{lem:dv}\ref{item:tame_0}; otherwise $v_{i}^{p}=0$ and $b^{p}\de v_{i}\sigma_{i}^{-1}|_{V_{X,I}^{p}}$ is bounded. Since $\tau^{p}|_{U_{X,I}^{\circ}}$ is supported on $V_{X,I}^{p}$, the equality $\tau^{p}(v_{i}-v_{i}^{p})=\sigma_{i}\tau^{p}b^{p}$ holds on the whole $U_{X,I}^{\circ}$. Thus $v_{i}-\bar{v}_{i}=\sigma_{i}\sum_{p\in R} \tau^{p}b^{p}=\sigma_{i}b$, where $b\de \sum_{p\in R}\tau^{p}b^{p}\in \cC^{\infty}(U_{X,I}^{\circ})$ is bounded, as needed.
	
	\ref{item:tame_1-bar} Since $\sum_{p\in R}\tau^{p}=1$, we have  $d(v_{i}-\bar{v}_{i})=\sum_{p\in R}d[\tau^{p}(v_{i}-v_{i}^{p})]=\sum_{p\in R}(v_{i}-v_{i}^{p})\, d\tau^{p}+\sum_{p\in R}\tau^{p}\, d(v_{i}-v_{i}^{p})$. As before, Lemma \ref{lem:dv}\ref{item:tame_0} implies that on $V^{p}_{X,I}$ we have $v_{i}-v_{i}^{p}=\sigma_{i}b^{p}$ for some bounded $b^{p}$, so on $U_{X,I}^{\circ}$ we have  $\sum_{p\in R}(v_{i}-v_{i}^{p})d\tau^{p}=\sigma_{i}\sum_{p\in R}b^{p}d\tau^{p}=\sigma_{i}\gamma_{0}$ for some $1$-form $\gamma_{0}$ bounded from $X$. 
	
	Similarly, Lemma \ref{lem:dv}\ref{item:tame_1} implies that on $V^{p}_{X,I}$ we have $d(v_i-v_{i}^{p})=t_{i}c^{p}\, dv_{i}+t_{i}^{\epsilon}q^p\, dg+\sigma_{i}\gamma^{p}$ for some bounded $c^p,q^p\in \cC^{\infty}(V_{X,I}^{p})$ and $\gamma^{p}\in \Omega^{1}(V_{X,I}^{p})$. Indeed, if $D_i$ meets $\bar{U}_{X}^{p}$ then this is the statement of \ref{lem:dv}\ref{item:tame_1} for $U_X'=U_X^p$. Otherwise, the restriction of $t_{i}^{-1}$ to $V_{X,I}^p$ is bounded, and by convention $v_{i}^{p}=0$, so we can take $c^{p}=t_{i}^{-1}$, $q^{p}=0$ and $\gamma^{p}=0$. Multiplying by $\tau^{p}$ and taking the sum we get a formula $\sum_{p\in R}\tau^{p}d(v_{i}-v_{i}^{p})=t_{i}c\, dv_{i}+t_{i}^{\epsilon}q\, dg+\sigma_{i}\gamma_{1}$ valid on the whole $U_{X,I}^{\circ}$, where the functions $c\de \sum_{p\in R}\tau^{p}c^{p}$, $q\de \sum_{p\in R}\tau^{p}q^p$ and $1$-form $\gamma_{1}\de \sum_{p\in R}\tau^{p}\gamma^{p}$ are smooth and bounded on $U_{X,I}^{\circ}$.
	
	Adding these formulas together and taking $\gamma=\gamma_0+\gamma_1$ gives \ref{item:tame_1-bar}. %The second one follows from the second formula for $d(v_{i}-v_{i}')$ in Lemma \ref{lem:dv}\ref{item:tame_1} by exactly the same computation.
\end{proof}

\begin{proof}[Proof of Lemma \ref{lem:smoothchart}]
	Fix $x\in D$, so $x\in X_{S}^{\circ}$ for some nonempty $S\subseteq \{1,\dots, N\}$. Let $U_{X}$ be a chart adapted to $f$ containing $x$, whose associated index set is $S$. For $i\in S$, let $U_{i}$ be the domain of the corresponding chart $\bar{\psi}_i$  introduced in \eqref{eq:AC-chart}. We need to show that there is a neighborhood $V_{X}$ of $x$ in $U_X$, such that for every $i\in S$, the map $\bar{\psi}_{i}$ is a $\cC^{1}$-diffeomorphism on $U_{i}\cap \pi^{-1}(V_X)$.
	
	Fix $y\in \pi^{-1}(x)\cap U_{i}$. Then $t_{j}(y)=0$ for all $j\in S$, so $\sigma_{j}(y)=0$  by definition \eqref{eq:rho-sigma} of $\sigma_{j}$. Thus by Lemma \ref{lem:dvbar}\ref{item:tame_1-bar}, the forms $dv_{j}$ and $d\bar{v}_{j}$ are equal at $y$. Since the map $\bar{\psi}_{i}$ introduced in \eqref{eq:AC-chart} differs from the $\cC^{1}$-chart $\psi_{i}$ introduced in \eqref{eq:AC1-chart} only by replacing $v_{j}$ by $\bar{v}_{j}$, we infer that the Jacobian matrix of $\bar{\psi}_{i}$ with respect to the $\cC^{1}$-coordinates $\psi_{i}$ is the identity at $y$. By the inverse function theorem, there is a neighborhood $V_{y}$ of $y$ in $U_i$ such that $\bar{\psi}_{i}|_{V_{y}}$ is a $\cC^{1}$-diffeomorphism. 
	
	We claim that there is a neighborhood $V_{X}$ of $x$ in $U_X$, such that $\bar{\psi}_{i}|_{U_i\cap \pi^{-1}(V_X)}$ is injective. Suppose the contrary. Then there are sequences $(y^{\nu}_1),(y^{\nu}_2)\subseteq U_i$ such that both  $\pi(y^{\nu}_1)$ and $\pi(y^{\nu}_2)$ converge to $x$, and we have $y_{1}^{\nu}\neq y_{2}^{\nu}$,  $\bar{\psi}_{1}(y^{\nu}_1)=\bar{\psi}_1(y^{\nu}_2)$ for all $\nu$. By Proposition \ref{prop:AX-topo}\ref{item:top-homeo-off-D}, the map $\pi$ is proper, so passing to subsequences we can assume that for $l\in \{1,2\}$ we have $y^{\nu}_l\rightarrow y_l$ for some $y_l\in \pi^{-1}(x)$. Since $\bar{\psi}_i$ is continuous, we have $\bar{\psi}_{i}(y_1)=\bar{\psi}_i(y_2)$. 
	
	Since $x\in X_{S}^{\circ}$, for every $j\in S$ and $l\in \{1,2\}$ we have, as before, $t_{j}(y_l)=0$ and $\sigma_{j}(y_l)=0$. Lemma \ref{lem:dvbar}\ref{item:tame_0-bar} implies that $\bar v_{j}(y_l)={v}_{j}(y_l)$, so $\bar{\psi}_{i}(y_l)=\psi_{i}(y_{l})$. Therefore, $\psi_{i}(y_1)=\bar\psi_i(y_1)=\bar{\psi}_i(y_2)=\psi_i(y_2)$. Since $\psi_{i}$ is injective by Lemma \ref{lem:C1chart}, we get $y_1= y_2$, so the sequences $(y_{1}^{\nu})$ and $(y_{2}^{\nu})$ converge to the same limit $y_1$. Thus for $\nu \gg 1$ the points $y_{l}^{\nu}$ lie in the neighborhood $V_{y_1}$ of $y_1$ where $\psi_{i}$ is a $\cC^1$-diffeomorphism. Since by assumption $\bar{\psi}_i(y_1^{\nu})=\bar{\psi}_{i}(y_2^{\nu})$, we get $y_{1}^{\nu}=y_{2}^{\nu}$, a contradiction.
	
	By Proposition \ref{prop:AX-topo}\ref{item:top-homeo-off-D} the map $\pi$ is proper, so we can shrink the neighborhood $V_{X}$ of $x$ so that $\pi^{-1}(V_{X})\cap U_{i}\subseteq \bigcup_{y\in \pi^{-1}(x)}V_{y}$. This way, $\bar{\psi}_{i}|_{U_i\cap \pi^{-1}(V_{X})}$ is a $\cC^{1}$-diffeomorphism, as needed.
\end{proof}

The remaining part of Section \ref{sec:AX-smooth} is devoted to the proof of Lemma \ref{lem:transsmooth} and Proposition \ref{prop:AXsmooth}. Throughout this proof, we will use the following notation. 

\begin{notation}\label{not:S-smooth}
	We fix a chart $U_{X}$ adapted to $f$ satisfying the statement of Lemma \ref{lem:smoothchart}, and put $U=\pi^{-1}(U_X)$. We reorder the components of $D$ so that the associated index set \eqref{eq:index-set} of $U_X$ is $S=\{1,\dots, k\}$. As in \eqref{eq:Ui}, we put $U_1=\{w_1>\tfrac{1}{n+1}\}\subseteq U$. Now on $U_1$, we have a smooth structure given by the $\cC^{1}$-chart $\psi_{1}=(g,v_2,\dots,v_k,\rest)$, see \eqref{eq:AC1-chart}. We use Notation \ref{not:S} for this chart. That is, for an open subset $V\subseteq U_1$ we denote by $\cS(V)$ and $\cS^{l}(V)$ the algebras of smooth (resp.\ $\cC^{l}$) functions on $V$; put $\d_1=\frac{\d}{\d g}$, $\d_{i}=\frac{\d}{\d v_i}$ and $T_{i}=\{t_{i}>0\}\subseteq U_1$, $W_{i}=\{w_i>0\}\subseteq U_1$ for $i\in \{2,\dots, k\}$.
\end{notation}

Recall that by Lemma \ref{lem:C1chart}, the restriction $\pi|_{U_1\setminus \d A}$ is a $\cC^{\infty}$-diffeomorphism. Thus, with our usual abuse of notation, we can write $\cC^{\infty}(U_X\setminus D)\subseteq \cS(U_1\setminus \d A)$. 

We end this section with a simple consequence of Lemma \ref{lem:computations-d}. 

\begin{lema}\label{lem:simple-smoothness} 
	For all $i\in \{2,\dots, k\}$, we have $w_{i}\in \cS(W_{i})$ and $t_i\in \cS(T_i\cup W_i)$.
\end{lema}
\begin{proof}
	Recall that $w_i,t_i$ are continuous on $U_1$ by Lemma \ref{lem:intro}\ref{item:intro-extends}. Assume that for some $l\geq 0$ we have $w_i,t_i\in \cS^{l}(W_i)$, that is, the restrictions $w_i|_{W_{i}}$, $t_i|_{W_{i}}$ are of class $\cC^{l}$ with respect to the smooth coordinates $(g,v_2,\dots,v_k,\rest)$. Then by Lemma \ref{lem:computations-d}\ref{item:dw},\ref{item:dt-rho}, all partial derivatives of $w_i|_{W_i}$, $t_i|_{W_i}$ are of class $\cC^{l}$, too. Hence $w_{i}|_{W_{i}},t_{i}|_{W_{i}}\in \cS^{l+1}(W_i)$. By induction, $t_{i},w_i\in \cS(W_i)$. Similarly, Lemma \ref{lem:computations-d}\ref{item:dt-g} implies that on $T_{i}$, all  partial derivatives of $t_i$ are of class $\cC^{l}$ whenever $t_{i}$ is, so $t_i\in \cS (T_i)$.
\end{proof}

\subsubsection{Flattening algebras}
We will now introduce a practical tool which allows to extend the smoothness domain of a function along its zero locus. First, we settle the following notation.

Let $V$ be a subset of $U_1$. Given a subset $\cF\subseteq \cS(V)$, we denote by $\langle \cF \rangle$ the $\R$-algebra (possibly without $1$) generated by $\cF$. For two $\R$-subalgebras $\cA,\cB\subseteq \cS(V)$ we put $\cA+\cB=\langle \cA\cup \cB\rangle$ and $\cA\cdot \cB=\langle a\cdot b: a\in \cA, b\in \cB\rangle=\{\sum_{i} \lambda_{i} a_{i}b_{i}: \lambda_{i}\in \R,a_i\in \cA,b_i\in \cB\}$. In particular, if $\cF$ is a subset of $\cA$ then $\langle \cF \rangle \cdot \cA$ is the ideal of $\cA$ generated by $\cF$. For a function $h\in \cS(V)$, we write $h\cdot \cA$ for the algebra $\langle h\rangle \cdot \cA$. Thus if $h\in \cA$, then $h\cdot \cA$ is, as usual, the principal ideal of $\cA$ generated by $h$.

Recall that we have introduced differential operators $\d_{1}=\frac{\d}{\d g}$, $\d_{i}=\frac{\d}{\d v_i}$ for $i\in \{2,\dots, k\}$. For an $\R$-subalgebra $\cA\subseteq \cS(V)$ we write $\d_{j}\cA=\{\d_j a: a\in \cA\}\subseteq \cS(V)$. We say that $\cA$ is \emph{closed under derivation by $\d$} (or simply \emph{closed under $\d$}) if $\d_{j}\cA\subseteq \cA$ for all $j\in \{1,\dots, k\}$. An easy application of Leibniz rule shows that an algebra $\cA=\langle \cF \rangle$ generated by $\cF$ is closed under $\d$ if and only if $\d_{j}h\in \cA$ for every generator $h\in \cF$ and every $j\in \{1,\dots, k\}$.

\begin{definition}\label{def:flattening}
	Fix a closed subset $Z\subseteq U_1\cap \d A$ and a function $h\in \cS(U_1\setminus Z)$. Let $\cA$ be an $\R$-subalgebra of $\cS(U_1\setminus Z)$ which is closed under derivation by $\d$. 
	
	We say that \emph{$h$ flattens $\cA$ on $Z$} if for every $a\in \cA$, and every $\epsilon>0$ we have 
	\begin{equation}\label{eq:flattens}
		|h|^{\epsilon}\cdot a \rightarrow 0\quad \mbox{ on } Z.
	\end{equation}
\end{definition}

The following Lemma \ref{lem:flattening} lists some elementary consequences of Definition \ref{def:flattening}. Part \ref{item:flat-on-gens} asserts that condition \eqref{eq:flattens} -- just like closedness under $\d$ -- can be checked on generators. Part \ref{item:flat-by-ders} gives a practical condition under which not only $h$, but all its derivatives flatten $\cA$. Together with part \ref{item:flat-is-smooth}, it will be used to infer smoothness of $h$.

\begin{lema}\label{lem:flattening}
	Fix a closed subset $Z\subseteq U_1$ and a continuous function $h$ on $U_1$ such that $h\in \cS(U_1\setminus Z)$ and $h|_{Z}=0$. Let $\cF\subseteq \cS(U_1\setminus Z)$, and let $\cA \de \langle \cF \rangle$ be the $\R$-algebra generated by $\cF$. Assume that $\cA$ is closed under $\d$. Then the following hold.
	\begin{enumerate}
		\item\label{item:flat-on-gens} If for every $a\in \mathcal{F}$ and every $\epsilon>0$ we have $|h|^{\epsilon}\cdot a\rightarrow 0$ on $Z$, then $h$ flattens $\cA$ on $Z$.
		\item\label{item:flat-by-ders} Assume that $h$ flattens $\cA$ on $Z$, and  for every $j\in \{1,\dots, k\}$ we have  $h^{-1}\cdot \d_{j}h\in \cA$. Then for all $l\geq 0$ and all $j_1,\dots,j_l\in \{1,\dots, k\}$, the function $\d_{j_1}\dots\d_{j_l} h$ flattens $\cA$ on $Z$.
		\item\label{item:flat-is-smooth} Assume that all derivatives $\d_{j_1}\dots\d_{j_l} h$ flatten $\cA$ on $Z$. Then for every $a\in \cA$, we have $ha\in \cS(U_1)$, and $\d^{\alpha}(ha)|_{Z}=0$ for all $\alpha\in \N^k$. In particular, if $1\in \cA$ then $h\in \cS(U_1)$ and all derivatives of $h$ vanish on $Z$.
	\end{enumerate}
\end{lema}
\begin{proof}
	\ref{item:flat-on-gens} Since by assumption $h|_{Z}=0$ and $h$ is continuous, all constant functions $a$ satisfy \eqref{eq:flattens}. Now if $a,b\in \cA$ satisfy \eqref{eq:flattens} then $|h|^{\epsilon}\cdot (a+b)=|h|^{\epsilon}\cdot a+|h|^{\epsilon}\cdot b\rightarrow 0$, and $|h|^{\epsilon}\cdot (ab)=(|h|^{\epsilon/2}\cdot a)\cdot (|h|^{\epsilon/2}\cdot b)\rightarrow 0$ on $Z$, so $a+b$ and $ab$ satisfy \eqref{eq:flattens}, too, as needed. 
	
	\ref{item:flat-by-ders}  %Put $\boldsymbol{j}=(j_1,\dots,j_l)$ and $\d_{\boldsymbol{j}}=\d_{j_1}\dots\d_{j_l}$. 
	We need to prove that for every $\epsilon>0$ and every $a\in \cA$, we have  $|\d_{j_1}\dots\d_{j_l} h|^{\epsilon}\cdot a\rightarrow 0$ on $Z$. It is sufficient to prove this claim for all $a\in \cA$ and $\epsilon=1$. Indeed, by continuity it is sufficient to consider the case $\epsilon\in \Q$, which follows from case $\epsilon=1$ applied to some power of $a$.
	
	For $\boldsymbol{j}=(j_1,\dots,j_l)\in \{1,\dots, k\}^{l}$ put $\d_{\boldsymbol{j}}=\d_{j_1}\dots\d_{j_l}$ and $|\boldsymbol{j}|=l$. 
	
	We argue by induction on $l\geq 0$. Case $l=0$ holds by assumption. Fix $l\geq 1$ and assume that the claim holds for all $l'<l$. Fix $\boldsymbol{j}=(j_1,\dots,j_l)$ and put $\boldsymbol{j}'=(j_1,\dots,j_{l-1})$. By assumption, the function $b\de h^{-1}\cdot \d_{j_{l}}h$ belongs to $\cA$. Now $a\cdot \d_{\boldsymbol{j}}h=a\cdot \d_{\boldsymbol{j}'}(hb)=a\cdot \sum_{|\boldsymbol{i}|,|\boldsymbol{i}'|<l} c_{\boldsymbol{i},\boldsymbol{i}'} \d_{\boldsymbol{i}} h \cdot \d_{\boldsymbol{i}'}b$ for some $c_{\boldsymbol{i},\boldsymbol{i}'}\in \R$. Since $\cA$ is closed under $\d$, we have $a\cdot \d_{\boldsymbol{i}'}b\in \cA$, so by inductive assumption, $\d_{\boldsymbol{i}} h \cdot (a\cdot \d_{\boldsymbol{i}'}b)\rightarrow 0$ on $Z$. Thus $a\cdot \d_{\boldsymbol{j}} h\rightarrow 0$ on $Z$, as needed.
	
	\ref{item:flat-is-smooth} By assumption, $ha\in \cS(U_1\setminus Z)$. As before, for any $\boldsymbol{j}\in \{1,\dots,k\}^{l}$ we can write $\d_{\boldsymbol{j}}(ah)=\sum_{|\boldsymbol{i}|,|\boldsymbol{i'}|}c_{\boldsymbol{i},\boldsymbol{i}'} \cdot \d_{\boldsymbol{i}} a\cdot \d_{\boldsymbol{i}'}h$ for some $c_{\boldsymbol{i},\boldsymbol{i}'}\in \R$. Since $\cA$ is closed under $\d$, we have $\d_{\boldsymbol{i}} a\in \cA$. Since $\d_{\boldsymbol{i}'}h$ flattens $\cA$, we have $\d_{\boldsymbol{i}} a\cdot \d_{\boldsymbol{i}'}h\rightarrow 0$ on $Z$. Hence $\d_{\boldsymbol{j}}(ah)$ extends to a continuous function on $U_1$, which vanishes on $Z$, as claimed. If furthermore $1\in \cA$ then substituting $a=1$ to the first claim we get $h\in \cS(U_1)$ and $\d^{\alpha}h|_{Z}=0$, as claimed.% (note that we can ritenow $\d_{j}$'s do commute)
\end{proof}

\subsubsection{The functions $r_j$, $w_j$ are smooth.} 
In this section, we establish a \enquote{smooth version} of Lemma \ref{lem:v1C1}. Its most important consequence will be the fact that, for all $i\in \{2,\dots, k\}$, the functions $w_{i},r_{i}$ are in $\cS(U_1)$, and for all $\alpha\in \N^k$, their derivatives $\d^{\alpha}w_{i}$, $\d^{\alpha}r_{i}$, vanish on the zero locus of $w_i$ and $r_i$, respectively. To prove this result, for each $i\in \{2,\dots, k\}$ we introduce the following $\R$-algebras:
\begin{equation}\label{eq:RW}
	\begin{split}
		\cW_i& =\langle 1, t_i, w_i, u_i, \rho_{i}, t^{(l)}w_{i}^{-1}\ :\ l\geq 0 \rangle,\\
		\cR_{i}& = \langle 1, t_{i}, t_{i}^{-1}, \log t_{i}, g, (1+g\log t_i)^{-1},(t_{i}(1+g\log t_i)^{2}+g^{2})^{-1} \rangle.
	\end{split}
\end{equation}
We will study their properties in Lemmas \ref{lem:w_smooth} and \ref{lem:pullbacks}, respectively.

\begin{lema}\label{lem:w_smooth} 
	Fix $i\in \{2,\dots,k\}$, and let $\cW_{i}$ be as in \eqref{eq:RW}. Then the following hold.
	\begin{enumerate} 
		\item\label{item:w-inverse} For every $j\in \{1,\dots, k\}$ we have $w_{i}^{-1}\cdot \d_{j}w_{i}\in \cW_i$.		
		\item \label{item:w-closed} The algebra $\cW_{i}$ is contained in $\cS(W_i)$ and is closed under derivation by $\d$.
		\item\label{item:w-algebra} For every $\alpha\in \N^k$, the function $\d^{\alpha}w_{i}$ flattens $\cW_{i}$ on $U_{1}\setminus W_{i}$.
		\item\label{item:w-smooth} We have $w_{i}\in \cS(U_1)$, for every $i\in \{2,\dots, k\}$.
		\item\label{item:t1v1-smooth} We have $w_1,t_1,v_1\in \cS(U_1)$.
	\end{enumerate}
\end{lema}
\begin{proof}
	\ref{item:w-inverse} Recall that by Lemma \ref{lem:mixed-vanish} we have $\d_{j}w_{i}=0$ for all $j\neq i,1$. By Lemma \ref{lem:computations-d}\ref{item:dw}, we have
	\begin{equation*}
		w_{i}^{-1}\cdot \d_{i} w_{i}=-\rho_{i}\in \cW_{i},
		\quad\mbox{and}\quad 
		w_{i}^{-1}\cdot \d_1 w_i=\rho_{i}\cdot (t'w_{i}^{-1}) \in \cW_i.
	\end{equation*}
	
	\ref{item:w-closed} By Lemma \ref{lem:v1C1}\ref{item:t-smooth}, we have $t^{(l)}\in \cS(U_1)$ for all $l\geq 0$. By Lemma \ref{lem:simple-smoothness}, we have $w_{i}\in \cS(W_{i})$: recall that, by definition, $W_{i}=\{w_{i}>0\}$ is the complement of the zero locus of $w_i$. Hence $t^{(l)}w_{i}^{-1}\in \cS(W_i)$. The function $u_i$ was defined in \eqref{eq:def-w_i-u_i} as $u_{i}=\eta(w_i)$, where $\eta|_{(0,1]}\colon (0,1]\to (0,1]$ is smooth. Therefore, $u_{i}\in \cS(W_i)$ and $u_{i}|_{W_{i}}>0$. It follows that $\rho_{i}=(t_{i}+u_{i}^{2})^{-1}\in \cS(W_i)$, so $\cW_{i}\subseteq \cS(W_i)$, as claimed.
	
	Now, we check that $\cW_i$ is closed under $\d$. It is sufficient to check $\d_{j}a\in \cW_{i}$ for all generators $a$ of $\cW_{i}$, and all $j\in \{1,\dots, k\}$. %Lemma \ref{lem:mixed-vanish} implies that $\d_{j}h=0$ for all $j\neq 1,i$. 
	By Lemma \ref{lem:computations-d}\ref{item:dt-rho}, we have 
	\begin{equation*}
		\d_{i}t_{i}=\rho_{i}t_{i}\in \cW_i,\quad
		\d_{1}t_{i}=\rho_{i}u_{i}^{2}\cdot w_{i}^{-1}t'\in \cW_i,\quad \mbox{and}\quad 
		\d_{j} t_{i}=0\mbox{ for }j\neq i,1.
	\end{equation*}
	Thus $\d_{j}t_{i}\in \cW_{i}$. Since $w_i\in \cW_i$, part \ref{item:w-inverse} implies that $\d_{j} w_{i}\in \cW_{i}$, too. Recall that $u_{i}=\eta(w_{i})$. By Lemma \ref{lem:computations}\ref{item:eta-u}, we have $\eta'(w_i)=u_{i}^{2}w_{i}^{-1}$, so $\d_{j} u_{i}=\eta'(w_{i})\d_{j}w_{i}=u_{i}^{2}\cdot w_{i}^{-1}\d_{j}w_{i}\in \cW_{i}$ by \ref{item:w-inverse}. In \eqref{eq:rho-sigma}, we have defined $\rho_i$ as $(t_{i}+u_{i}^{2})^{-1}$, so $\d_{j}\rho_{i}=-\rho_{i}^{2}(\d_{j}t_{i}+2u_{i}\d_j u_i)\in \cW_{i}$. Eventually, 
	\begin{equation*}
		\d_{j}(w_{i}^{-1} t^{(l)})=-w_{i}^{-1}t^{(l)}\cdot w_{i}^{-1}\d_{j}w_{i}+\delta_{j}^{1}\cdot w_{i}^{-1}t^{(l+1)}\in \cW_{i}.
	\end{equation*}
	
	\ref{item:w-algebra}, \ref{item:w-smooth} It is sufficient to show that $w_{i}$ flattens $\cW_{i}$ on $U_1\setminus W_i$. Indeed, since by \ref{item:w-closed} $\cW_{i}$ is a subalgebra of $\cS(W_i)$ closed under $\d$, and by \ref{item:w-inverse} it contains $w_i^{-1}\cdot \d_{j}w_{i}$, Lemma \ref{lem:flattening}\ref{item:flat-by-ders} shows that if $w_{i}$ flattens $\cW_{i}$ on $U_1\setminus W_i$, then so do all its partial derivatives (taken in any order). Once this is shown, Lemma \ref{lem:flattening}\ref{item:flat-is-smooth} will imply that $w_i\in \cS(U_1)$.
	
	By Lemma \ref{lem:flattening}\ref{item:flat-on-gens}, it is sufficient to check that, for every generator $a$ of $\cW_{i}$ and every $\epsilon>0$, we have $w_{i}^{\epsilon}a\rightarrow 0$ as $w_{i}\rightarrow 0$. This is clear for 
	all bounded generators, i.e.\ for $1$, $t_i$, $w_i$ and $u_{i}$. Recall from \eqref{eq:rho-sigma} and  \eqref{eq:def-w_i-u_i} that $\rho_{i}=(t_i+u_i^2)^{-1}$ and $u_{i}=\eta(w_i)=(1-\log w_i)^{-1}$, so
	\begin{equation*}
		\rho_{i}w_{i}^{\epsilon}=\frac{w_{i}^{\epsilon}}{t_{i}+(1-\log w_{i})^{-2}}=
		\frac{w_{i}^{\epsilon}\cdot(1-\log w_{i})^{2}}{t_{i}(1-\log(w_{i}))^{2}+1}\rightarrow 0
		\quad\mbox{as }w_i\rightarrow 0,	
	\end{equation*}
	because $t_{i}(1-\log w_{i})^{2}\geq 0$. Eventually, by Lemma \ref{lem:computations}\ref{item:tl} there is a polynomial $p\in \R[s]$ such that $t^{(l)}=tp(\log t)$. Since by Lemma \ref{lem:computations}\ref{item:t=tiwi} we have $t=t_iw_i$, we get
	\begin{equation*}
		w_{i}^{\epsilon}\cdot t^{(l)}w_{i}^{-1}=w_{i}^{\epsilon}t_{i}\cdot p(\log t_i+\log w_i)\rightarrow 0, 
		\quad\mbox{as }w_i\rightarrow 0. 		
	\end{equation*}
	%	\smallskip
	%	
	
	\ref{item:t1v1-smooth} We argue as in the proof of Lemma \ref{lem:v1C1}\ref{item:viwi-C1}. Part \ref{item:w-smooth} applied to $i\in \{2,\dots, k\}$ gives $\sum_{i=2}^{k}w_i\in \cS(U_1)$. By Lemma \ref{lem:intro}\ref{item:intro-sum}, we have  $w_{1}=1-\sum_{i=2}^{k}w_{i}\in \cS(U_1)$. Because  $t\in \cS(U_1)$ by Lemma \ref{lem:v1C1}\ref{item:t-smooth}, and  $w_{1}|_{U_1}>0$ by \eqref{eq:Ui}, we have $t_{1}=tw_{1}^{-1}\in \cS(U_1)$, and $v_1=t_{1}-\eta(w_1)\in \cS(U_1)$, as claimed.
\end{proof}

\begin{lema}\label{lem:pullbacks} 
	Fix $i\in \{2,\dots,k\}$, and let $\cR_{i}$ be as in \eqref{eq:RW}. Then the following hold.
	\begin{enumerate} 
		\item\label{item:r-inverse} We have $u_i,\rho_i\in \cR_i$, and $r_{i}^{-1}\cdot \d_{j} r_{i}\in \cR_{i}$ for every $j\in \{1,\dots, k\}$.
		\item \label{item:r-closed} The algebra $\cR_{i}$ is contained in $\cS(T_i)$ and is closed under derivation by $\d$.
		\item\label{item:r-algebra} For every $\alpha\in \N^k$, the functions $\d^{\alpha}r_{i}$ and $\d^{\alpha}r_1$ flatten $\cR_{i}$ on $U_{1}\setminus T_{i}$ and on $U_1\cap \d A$, respectively.
		\item\label{item:r-smooth} We have $r_{i},r_{1}\in \cS(U_1)$.
		\item\label{item:pullbacks} For every $h\in \cC^{\infty}(U_X)$, the pullback of $h$ to $U_1$ is in $\cS(U_1)$ (in short: $\cC^{\infty}(U_X)\subseteq \cS(U_1)$). 
	\end{enumerate}
\end{lema}
\begin{proof}
	Denote the last two generators of $\cR_i$ by $b=(1+g\log t_{i})^{-1}$ and $c=(t_i(1+g\log t_{i})^{2}+g^2)^{-1}$. 
	
	\ref{item:r-inverse}  
	By Lemma \ref{lem:computations}\ref{item:u-gt} we have $u_{i}=gb\in \cR_{i}$, and by Lemma \ref{lem:computations}\ref{item:rho-gt} we have $\rho_{i}=(1+g\log t_{i})^{2}\cdot c\in \cR_{i}$. By Lemma \ref{lem:computations-d}\ref{item:dt-g}, we have 
	\begin{equation*}
		\d_{i}t_{i}=t_{i}c(1+g\log t_{i})^{2}\in \cR_{i},\quad
		\d_{1}t_{i}=t_{i}c\in \cR_{i}\quad
		\mbox{and} \quad
		\d_{j}t_{i}=0\mbox{ for } j\neq 1,i.
	\end{equation*}
	Hence by Lemma \ref{lem:computations-d}\ref{item:dr}, $r_{i}^{-1}\cdot \d_{j}r_{i}=m_{i}^{-1}t_{i}^{-2} \d_{j}t_{i}\in \cR_{i}$.
	
	\ref{item:r-closed} Recall that $T_{i}=\{t_{i}>0\}$. Clearly, the coordinate $g$ is in $\cS(T_i)$. By Lemma \ref{lem:simple-smoothness} we have $t_{i}\in \cS(T_i)$, so $t_{i}^{-1},\log t_{i}\in \cS(T_i)$, and in consequence $b,c\in \cS(T_i)$, so $\cR_{i}\subseteq \cS(T_i)$, as needed.
	
	Now, we check that $\cR_{i}$ is closed under $\d$. We have already seen that $\d_{j}t_{i}\in \cR_{i}$. Hence $\d_{j}t_{i}^{-1}=-t_{i}^{-2}\d_j t_i\in \cR_{i}$ and $\d_{j} \log t_{i}=t_{i}^{-1} \d_{j}t_{i}\in \cR_{i}$. By definition of $\d_{j}$, we have $\d_{j}g=\delta_{j}^{1}\in \cR_{i}$.  Eventually, 
	\begin{equation*}
		\begin{split}
			\d_{j} b &=-b^{2} \cdot (\d_{j} g\cdot \log t_{i}+gt_{i}^{-1}\cdot \d_j t_i)\in \cR_i,\\
			\d_j c & = -c^{2}\cdot (\d_{j}t_{i} \cdot (1+g\log t_i)^2+2t_{i}(1+g\log t_i)(\log t_{i}\cdot \d_{j} g+ gt_{i}^{-1}\cdot \d_{j} t_{i})+2g\d_j g)\in \cR_i.
		\end{split}
	\end{equation*}
	
	\ref{item:r-algebra}, \ref{item:r-smooth} We claim that $r_{i}$ flattens $\cR_{i}$ on $U_{1}\setminus T_i$. Since $\cR_{i}\subseteq \cS(U_1\setminus T_i)$, by Lemma \ref{lem:flattening}\ref{item:flat-on-gens} it is enough to show that all generators of $\cR_{i}$ satisfy \eqref{eq:flattens}. Since $r_{i}|_{U_1\setminus T_i}=0$, this condition is clear for all bounded functions, hence for $1$, $t_{i}$, $g$ and $b$. By Lemma \ref{lem:computations}\ref{item:t-r}, we have $r_{i}^{\epsilon}=e^{-\epsilon(m_it_i)^{-1}}$, so $r_{i}^{\epsilon}t_{i}^{-1}\rightarrow 0$ and $r_{i}^{\epsilon}\log t_{i}\rightarrow 0$. It remains to check that $r_{i}^{\epsilon}c\rightarrow 0$. By Lemma \ref{lem:computations}\ref{item:eta-t}, we have $(1+g\log t_{i})^{2}=t_{i}\cdot (\eta'(w_i)t')^{-1}$, and by Lemma \ref{lem:computations}\ref{item:eta-t-bounded}, $\eta'(w_i)t'=t_{i}^{\delta}l$ for some $\delta\in (0,1)$ and a bounded function $l$. Hence $(1+g\log t_{i})^{2}=t_{i}^{1-\delta} l^{-1}$, and therefore
	\begin{equation*}
		c=(t_i(1+g\log t_{i})^{2}+g^2)^{-1}=(t_{i}^{2-\delta}l^{-1}+g^{2})^{-1}=t_{i}^{\delta-2}l\cdot q,
	\end{equation*}
	where $q\de (1+g^{2}t_{i}^{\delta-2}l)^{-1}$ is bounded. Hence $r_{i}^{\epsilon}c=e^{-\epsilon(m_i t_i)^{-1}}t_{i}^{\delta-2} q\rightarrow 0$ as $t_i\rightarrow 0$, as claimed.

	Thus $r_i$ flattens $\cR_i$ on $U_1\setminus T_i$. Since $r_{i}^{-1}\d_jr_i\in \cR_{i}$ by \ref{item:r-inverse}, Lemma \ref{lem:flattening}\ref{item:flat-by-ders} implies that all derivatives $\d_{j_1}\dots\d_{j_l}r_{i}$ flatten $\cR_{i}$ on $U_i\setminus T_i$, too, as claimed. Part \ref{item:r-inverse} and Lemma \ref{lem:flattening}\ref{item:flat-is-smooth} imply that $r_{i}\in \cS(U_1)$.
	\smallskip
	
	Recall that $t_1$, hence $r_1$, are in $\cS(U_1)$ by Lemma \ref{lem:w_smooth}\ref{item:w-smooth}. We claim that $r_1$ flattens $\cR_{i}$ on $U_1\cap \d A$.
	
	By definition \eqref{eq:Ui} of $U_1$, we have $w_1>0$, so by Lemma \ref{lem:computations}\ref{item:t=tiwi}, on $U_1\cap \d A$ we have $t_1=tw_{1}^{-1}=0$, and therefore $r_1=0$. In \ref{item:r-closed}, we have shown that $\cR_{i}\subseteq \cS(T_i)$, so all elements of $\cR_{i}$ are bounded on $T_i$. Hence $r_1$ flattens $\cR_{i}$ on $T_{i}\cap \d A$. We claim that $r_1$ flattens $\cR_{i}$ on $U_1\setminus T_i$, too.
	
	Define $\cR_{i}'=\cR_{i}+\langle t_{1}^{-1}, \d^{\alpha}t_{1} :\alpha\in \N^k \rangle$. Because $\cR_{i}$ is closed under $\d$ by \ref{item:r-closed}, so is $\cR_{i}'$. We claim that $r_1$ flattens $\cR_{i}'$ on $U_1\setminus T_i$. By Lemma \ref{lem:flattening}\ref{item:flat-on-gens}, we need to prove that  all elements of $\cR_{i}$, and all functions $t_{1}^{-1}$, $\d^{\alpha}t_{1}$ satisfy \eqref{eq:flattens}. By Lemma \ref{lem:computations}\ref{item:t-r},\ref{item:t=tiwi}, we have on $U_{1}\setminus \d A$
	\begin{equation}\label{eq:r1<ri}
		r_{1}=\exp(-\tfrac{1}{m_1t_{1}})=
		\exp(- \tfrac{1}{m_1}\tfrac{w_1}{t})=
		\exp(-\tfrac{m_i}{m_1}\tfrac{1}{m_i}\tfrac{1}{t_i}\tfrac{w_1}{w_i})=r_{i}^{\frac{m_i}{m_1}\frac{w_1}{w_i}}<
		r_{i}^{\gamma},\quad \mbox{where } \gamma=\tfrac{m_i}{m_1(n+1)}.
	\end{equation}
	For the last inequality, we have used the fact that $r_1<1$ by Definition \ref{def:adapted-chart} of a chart adapted to $f$; $w_{i}\leq 1$ by Lemma \ref{lem:intro}\ref{item:intro-sum}; and $w_1>\tfrac{1}{n+1}$ by definition \eqref{eq:Ui} of $U_1$. 
	
	We have seen that $r_i$ flattens $\cR_{i}$ on $U_1\setminus T_i$, so for every $\epsilon>0$ and every $a\in \cR_{i}$ we have $|r_{1}^{\epsilon}a|<|r_{i}^{\gamma \epsilon}a|\rightarrow 0$ as $t_i\rightarrow 0$. Hence all elements of $\cR_{i}$ satisfy \eqref{eq:flattens}, as needed.
	
	To prove \eqref{eq:flattens} for $t_{1}^{-1}$, recall from Lemma \ref{lem:computations}\ref{item:t-r} that $r_1=e^{-\epsilon(m_1t_1)^{-1}}$, so $t_{1}^{-1}r_{1}^{\epsilon}=t_{1}^{-1}e^{-\epsilon(m_1t_1)^{-1}}$. By \eqref{eq:r1<ri}, we have $t_{1}|_{U_1\setminus T_i}=0$, so $t_{1}^{-1}r_{1}^{\epsilon}\rightarrow 0$ on $U_1\setminus T_i$, as claimed. Eventually, by Lemma \ref{lem:w_smooth}\ref{item:w-smooth}, we have $t_{1}\in \cS(U_1)$, so for all $\alpha\in \N^k$, the function $\d^{\alpha}t_1$ is bounded on $U_1$, in particular $r_{1}^{\epsilon}\d^{\alpha}t_{1}\rightarrow 0$ on $U_1\setminus T_i$. Thus we have shown that $r_{1}$ flattens $\cR_{i}'$.%, so in particular its subalgebra $\cR_{i}$.
	
	By Lemma \ref{lem:computations-d}\ref{item:dr}, we have $r_{1}^{-1}\d_j r_1=m_{1}^{-1}t_{1}^{-2}\d_j t_1\in \cR_{i}'$ for all $j\in \{1,\dots, k\}$, so by Lemma \ref{lem:flattening}\ref{item:flat-by-ders} all derivatives of $r_1$ flatten $\cR_{i}'$, too. Since $\cR_i\subseteq \cR_i'$, it follows that they flatten $\cR_{i}$, as claimed. 
	
	\ref{item:pullbacks} Follows from \ref{item:r-smooth} by the chain rule applied as in the proof of Lemma \ref{lem:v1C1}\ref{item:pullbacks-C1}. 
	%By \ref{item:r-smooth} and the definition of the chart \eqref{eq:AC-chart}, the standard coordinates on $U_X$, namely $r_{j}\cos(2\pi\theta_j)$, $r_j\sin(2\pi \theta_j)$ for $j\in \{1,\dots, k\}$ and the (real and imaginary parts of) $z_{i_{k+1}},\dots,z_{i_n}$ pull back to functions in  $\cS(U_1)$. Thus \ref{item:pullbacks} follows from the chain rule.
\end{proof}

\subsubsection{The Jacobian matrix of \texorpdfstring{$\bar{\psi}_1$}{psi-bar}}\label{sec:M}
In Section \ref{sec:smoothchart}, we have used the fact that the Jacobian matrix of $\bar{\psi}_{1}=(g,\bar{v}_2,\dots,\bar{v}_{k},\rest)$ with respect to our auxiliary coordinates $\psi_{1}=(g,v_2,\dots,v_{k},\rest)$ is a small deformation of the identity. The difference can be computed using Lemma \ref{lem:dvbar}\ref{item:tame_1-bar}. We will now study this difference in more detail. In Lemma \ref{lem:jacobian}, proved at the end of this section, we will see that this difference belongs to a particular matrix ring $\cM$ introduced in Definition \ref{def:M}. Properties of $\cM$ will allow to conveniently express the coordinate vector fields of the smooth chart $\bar{\psi}_{1}$.%in terms of the $\cC^1$-chart $\psi_1$.
\smallskip

Recall that we work with a fixed $(\cU_{X},\btau)$, where $\cU_{X}=\{U_{X}^{p}\}_{p\in R}$ is an atlas adapted to $f$, and $\btau=\{\tau^{p}\}_{p\in R}$ is a partition of unity inscribed in $\cU_{X}$. 
To simplify the notation, we will now restrict our attention to charts which are small in the following sense.

\begin{lema}\label{lem:small-chart}
	Every point $x\in X$ lies in a chart $U_X$ adapted to $f$ with the following property. For every $p\in R$ such that $\tau^{p}$ is not identically zero on $U_X$, we have $U_{X}\subseteq U_{X}^{p}$.
\end{lema}
\begin{proof}
	For $p\in R$ let $G^{p}$ be the closure of the support of $\tau^{p}$, i.e.\ $G^{p}=\bar{\{y\in X: \tau^{p}(y)>0\}}$. Put $R_{x}=\{p\in R: x\in G^{p}\}$.  
	Since $\btau$ is locally finite, there is an open neighborhood $V'$ of $x$ such that the set $R'=\{p\in R: \tau^{p}|_{V'}\not\equiv 0\}$ is finite. Clearly, $R_x\subseteq R'$. Put $V=V'\cap \bigcap_{p\in R_x} U^{p}_{X} \setminus \bigcup_{p\in R'\setminus R_{x}} G^{p}$. Clearly $x\in V$, and since $R'$ and $R_x$ are finite, the set $V$ is open. 
	
	Now, fix $p\in R$ such that $\tau^{p}$ is not identically zero on $V$. Then $\tau^{p}|_{V'}\not\equiv 0$, so $p\in R'$ by definition of $R'$. Definition of $V$ implies that $p\in R_x$, and therefore, $V\subseteq U^p_X$. We conclude that any chart adapted to $f$ contained in $V$ has the required property.
\end{proof}

From now on, we fix a chart $U_X$ adapted to $f$, satisfying the statements of Lemmas \ref{lem:smoothchart} and  \ref{lem:small-chart}, and use Notation \ref{not:S-smooth} for this chart.

Put $R_0=\{p\in R : \tau^{p}|_{U_X}\not\equiv 0\}$. Then the statement of Lemma \ref{lem:small-chart} asserts that  $U_X\subseteq U_X^p$ for every $p\in R_0$. Thus for every $p\in R_0$ and $i\in \{2,\dots, k\}$ we have a \enquote{transition} function $a\in \cC^{\infty}(U_X)$ defined in Lemma \ref{lem:w_extends}\ref{item:t-comparison}. We define $\cT_{i}$ as the set of those functions $a$, as $p$ ranges among the set $R_0$. By definition, $\cT_i\subseteq \cC^{\infty}(U_X)$, so by Lemma \ref{lem:pullbacks}\ref{item:pullbacks} we have $\cT_i\subseteq \cS(U_1)$. 
\smallskip

Fix $i\in \{2,\dots, k\}$, and let $\cR_{i}$, $\cW_{i}$ be the $\R$-algebras defined in \eqref{eq:RW}. We define

\begin{equation}\label{eq:AI}
	\begin{split}
		\cA_{i}&=\cW_{i}\cap \cR_{i}
		% +\cC^{\infty}(U_X) 
		+\langle \d^{\alpha}h,\ \d^{\alpha}r_{j},\ \cos(2\pi\theta_j),\ \sin(2\pi\theta_j),\ 
		(1+at_{i})^{-1},\ \log(1+at_{i}),\\
		&\phantom{=\cW_{i}\cap \cR_{i}+\langle\ }
		(1-u_{i}\log(1+at_{i}))^{-1} 
		: \alpha\in \N^{k},\ h\in \cC^{\infty}(U_X),\ j\in \{1,i\},\ a\in \cT_{i} \rangle, \\
		\cI_{i}&=\langle \d^{\alpha} r_{1},\d^{\alpha}r_{i}:\alpha\in \N^k \rangle \cdot \cA_i.
	\end{split}
\end{equation}
This way, $\cI_{i}$ is an ideal of $\cA_{i}$ generated by all the derivatives of $r_{1}$, $r_{i}$. Intuitively speaking, the ideal $\cI_{i}$ consists of elements which decay exponentially fast with respect to $t_i$. In fact, we have $r_{i}=e^{-(m_{i}t_{i})^{-1}}$ by Lemma \ref{lem:computations}\ref{item:t-r}; and we have seen in \eqref{eq:r1<ri} that $r_1<r_i^{\gamma}$ for some constant $\gamma>0$. 

We now list basic properties of the algebra $\cA_i$ and its ideal $\cI_{i}$.

\begin{lema}\label{lem:AI}
	For each $i\in \{2,\dots, k\}$, the following hold.
	\begin{enumerate}
		\item \label{item:AI-smooth-off-corner} We have $\cA_{i}\subseteq \cS(W_i\cup T_i)$.
		\item\label{item:AI-closed} Both the algebra $\cA_{i}$ and its ideal $\cI_{i}$ are closed under $\d$.
		\item\label{item:AI-flattened} For all $\alpha\in \N^{k}$, the functions $\d^{\alpha}w_i$, $\d^{\alpha}r_i$ and $\d^{\alpha}r_1$ flatten $\cA_{i}$ on $U_1\setminus W_i$, $U_1\setminus T_i$ and on $U_1\cap \d A$, respectively.
		\item\label{item:AI-I-smooth} We have $\cI_{i}\subseteq \cS(U_1)$.		
		\item\label{item:dl-smooth} For every $h\in \cC^{\infty}(U_X)$, we have $\d_{i}h\in \cI_i$.		
		\item\label{item:AI-dl} For every $j\in \{2,\dots, k\}$ such that $j\neq i$ we have $\d_{j}\cA_{i}\subseteq \cI_{j}\cdot \cA_{i}$.
	\end{enumerate}
\end{lema}
\begin{proof}
	Recall that $t_{i},u_{i}\in \cW_{i}\cap \cR_{i}$. Indeed, by definition \eqref{eq:RW} of $\cW_i$, $\cR_i$ we have $t_{i}, u_i\in \cW_{i}$ and $t_i\in \cR_i$; and  by Lemma \ref{lem:pullbacks}\ref{item:r-inverse} we have $u_i\in \cR_i$.
	\smallskip
	
	\ref{item:AI-smooth-off-corner}	We have $\cW_{i}\subseteq \cS(W_i)$ by Lemma \ref{lem:w_smooth}\ref{item:w-closed} and $\cR_{i}\subseteq \cS(T_i)$ by Lemma \ref{lem:pullbacks}\ref{item:r-closed}, so $\cW_i\cap \cR_i\subseteq \cS(W_i\cup T_i)$. In particular, the above remark shows that $t_i,u_i\in \cS(W_i\cup T_i)$. 
	
	By definition of the coordinate chart \eqref{eq:AC-chart}, we have $\cos(2\pi \theta_j), \sin(2\pi\theta_j)\in \cS(U_1)\subseteq \cS(W_i\cup T_i)$.  By Lemma \ref{lem:pullbacks}\ref{item:r-smooth},\ref{item:pullbacks} we have $r_{j}\in \cS(U_1)$ and $\cC^{\infty}(U_X)\subseteq \cS(U_1)$, so $\d^{\alpha}r_{j},  \d^{\alpha}h\in \cS(U_1)\subseteq \cS(W_i\cup T_i)$ for all $h\in \cC^{\infty}(U_X)$; in particular $a\in \cS(W_i\cup T_i)$ for every $a\in \cT_i$. By  Lemma \ref{lem:dv}\ref{item:a-bounded} we have $1+at_i\neq 0$ and $1-u_{i}\log(1+at_i)\neq 0$ on $U_1$, so the fact that  $t_i,u_i,a\in \cS(W_i\cup T_i)$ implies that the remaining generators $(1+at_i)^{-1}$, $\log(1+at_i)$ and $(1-u_i\log(1+at_i))^{-1}$ are in $\cS(W_i\cup T_i)$, as claimed.
	
	\ref{item:AI-closed} The algebras $\cW_i$ and $\cR_i$ are closed under $\d$ by Lemmas \ref{lem:w_smooth}\ref{item:w-closed} and \ref{lem:pullbacks}\ref{item:r-closed}, respectively. Hence their intersection $\cW_i\cap \cR_i$ is closed under $\d$, too. %Fix $j\in \{1,\dots, k\}$. For any $h\in \cC^{\infty}(U_X)$ is in $\cS(U_1)$ and hence closed under $\partial$.

	To show that $\cA_{i}$ is closed under $\d$, it remains to check that for all $j\in \{1,\dots, k\}$, $\d_{j}$ maps the remaining generators of $\cA_i$ to $\cA_i$. For all derivatives of $r_{j}$ and functions $h\in \cC^{\infty}(U_X)$, this is clear by definition of $\cA_{i}$. Moreover, $\d_{j}\cos(2\pi\theta_{l})=0$ and $\d_{j}\sin(2\pi \theta_l)=0$ for all $l$ by definition of the chart \eqref{eq:AC-chart}. %Thus by the chain rule, we have $\d_j h\in \cA_i$ for all $h\in \cC^{\infty}(U_X)$. 
	Fix $a\in \cT_i$, and put $b=1+a t_i$. Then $b^{-1}\in \cA_i$. Since $a\in \cC^{\infty}(U_X)$, we have  $\d_{j}a\in \cA_{i}$. Since $t_i$ belongs to the algebra $\cW_i\cap \cR_i$ which is closed under $\d$, we have $\d_jt_i\in \cA_i$. Hence $\d_jb=a \d_j t_i+t_i\d_ja \in \cA_i$. As a consequence, we have $\d_j b^{-1}=-b^{-2}\d_j b\in \cA_i$ and $\d_j\log b=b^{-1}\d_jb\in \cA_i$. It remains to prove that $\d_j (1-u_i\log b)^{-1}\in \cA_i$. To see this, recall that $u_i$ belongs to $\cW_i\cap \cR_i$, which is closed under $\d$, so $\d_j u_i\in \cW_i\cap\cR_i\subseteq \cA_i$. Now $\d_j (1-u_i\log b)^{-1}=(1-u_i\log b)^{-2}\cdot (u_i\d_j\log b+\log b\, \d_j u_i)\in \cA_i$, as needed.
	
	Thus $\cA_{i}$ is closed under $\d$. Closedness of its ideal $\cI_{i}$ follows by the Leibniz rule.
	
	\ref{item:AI-flattened} By Lemma \ref{lem:w_smooth}\ref{item:w-algebra}, all derivatives $\d^{\alpha}w_i$ flatten $\cW_i$ on $U_1\setminus W_i$ (and vanish there), so in particular they flatten there a subalgebra $\cW_i\cap \cR_i$. Similarly, Lemma \ref{lem:pullbacks}\ref{item:r-algebra} implies that $\d^{\alpha}r_i$, $\d^{\alpha}r_1$ flatten $\cW_i\cap \cR_{i}$ on $U_1\setminus T_i$ and $U_1\cap \d A$, respectively. By Lemma \ref{lem:flattening}\ref{item:flat-on-gens}, it is sufficient to prove that all the remaining generators of $\cA_{i}$ satisfy \eqref{eq:flattens}. This is clear since they are all bounded. Indeed, by Lemma \ref{lem:pullbacks}\ref{item:r-smooth},\ref{item:pullbacks}, $r_{j}$ and all $h\in \cC^{\infty}(U_1)$ are in $\cS(U_1)$, so their derivatives $\d^{\alpha}r_{j}, \d^{\alpha}h\in \cS(U_1)$ are bounded. Of course, $|\cos(2\pi\theta_j)|,|\sin(2\pi\theta_j)|\leq 1$. %for every $h\in \cC^{\infty}(U_X)$ we have $h\in \cS(U_1)$ by Lemma \ref{lem:pullbacks}\ref{item:pullbacks}, so $\d^{\alpha}h\in \cS(U_1)$ is bounded. 
	The remaining generators of $\cA_i$ are bounded by Lemma \ref{lem:dv}\ref{item:a-bounded}. 
	
	\ref{item:AI-I-smooth} 	We need to show that, for every $a\in \cA_{i}$ and every $\alpha\in \N^{k}$, the functions $(\d^{\alpha}r_1)\cdot a$ and $(\d^{\alpha}r_i)\cdot a$ are in $\cS(U_1)$. By \ref{item:AI-smooth-off-corner}, we have $\cA_{i}\subseteq \cS(T_i)$. By \ref{item:AI-flattened}, both $\d^{\alpha} r_1$ and $\d^{\alpha}r_i$ flatten $\cA_i$ on $U_1\setminus T_i$. Thus the claim follows from Lemma  \ref{lem:flattening}\ref{item:flat-is-smooth}.
	
	\ref{item:dl-smooth} Since $\cC^{\infty}(U_X)\subseteq \cA_{i}$, by the chain rule it is sufficient to prove that the coordinate functions $x_{j}=r_{j}\cos(2\pi\theta_j)$, $y_{j}=r_{j}\sin(2\pi \theta_{j})$ satisfy $\d_{i}x_{j}\in \cI_{i}$, $\d_{i}y_{j}\in \cI_{i}$. By definition of the chart \eqref{eq:AC-chart}, we have $\d_{i}x_{j}$, $\d_{i}y_{j}=0$ for $j\neq 1,i$. If $j\in \{1,i\}$ then $\d_{i}\theta_{j}=0$, so $\d_{i}x_{j}=\cos(2\pi\theta_{j}) \cdot \d_{i}r_{j}\in \cI_{i}$; and similarly $\d_{i}y_{j}=\sin(2\pi\theta_{j}) \cdot \d_{i}r_{j}\in \cI_{i}$.

	\ref{item:AI-dl}  Assume $j\neq 1,i$. 
	By Lemma \ref{lem:mixed-vanish}, we have $\d_{j}t_{i}=0$, $\d_{j}u_{i}=0$, and more generally, $\d_{j}c=0$ for all $c\in \cW_{i}\cap \cR_{i}$. Thus $\d_j(\cW_i\cap \cR_i)=\{0\}\subseteq \cA_{i}\cdot \cI_{j}$. By \ref{item:dl-smooth}, for any $h\in \cC^{\infty}(U_X)$ we have $\d_{j} h\in \cI_{j}$. Since by \ref{item:AI-closed} the ideal $\cI_{j}$ is closed under $\d$, we get $\d_{j}\d^{\alpha}h=\d^{\alpha}\d_jh\in \cI_{j}$ for all $h\in \cC^{\infty}(U_X)$, $\alpha\in \N^{k}$. 
	Moreover, $\d_{j}\cos(2\pi\theta_l)$, $\d_j\sin(2\pi \theta_l)=0$ for $l\in \{1,i\}$; $\d_{j}r_{i}=0$ and $\d_{j}r_{1}\in \cI_{j}$ by definition of $\cI_{j}$.
	
	Let $b$ be one of the remaining generators of $\cA_i$. Since $\d_j t_i=0$ and $\d_j u_i=0$, when computing $\d_jb$ by the chain rule, we get an element of $\cA_{i}$ multiplied by $\d_{j} a$, where $a\in \cT_{i}\subseteq \cC^{\infty}(U_X)$. By \ref{item:dl-smooth} we have $\d_j a\in \cI_j$, so $\d_j b\in \cA_{i}\cdot \cI_{j}$, as claimed.
\end{proof}

We will now introduce additional notation. For $i\in \{2,\dots, k\}$, we put
\begin{equation*}
	\cW_{i}'=\langle w_{i}^{(s)} :s\geq 0\rangle,\quad\mbox{where}\quad w_{i}^{(s)}=\d_i^{s} w_i.
\end{equation*}
Moreover, we put 
\begin{equation*}
	\cA = \sum_{i=2}^{k}\cA_{i},\quad 
	\cJ = \langle \d^{\alpha} r_1: \alpha\in \N^{k}\rangle \cdot \cA,\quad 
	\cP=\R+\cJ+\sum_{i=2}^{k}\cI_{i}.
\end{equation*}
Then $\cJ$ is an ideal of $\cA$, and $\cP$ is a subalgebra of $\cA$, containing $1$. One should think of $\cJ$ as an ideal whose elements decay exponentially with respect to all $t_i$, see \eqref{eq:r1<ri}, so they can be used to flatten elements of $\cA$. In turn, viewing the elements of $\cI_i$ as those which decay exponentially with respect to $t_i$, we can view the elements of $\cP$ as those which do not require further flattening. In particular, we will see in Lemma \ref{lem:products}\ref{item:I-smooth} below that $\cP\subseteq \cS(U_1)$.% This intuition is made more precise by the following lemma.% \ref{lem:AI} yields the following properties of $\cI_{J}$, $\cP$.

\begin{lema}\label{lem:products}
	The algebras defined above have the following properties.
	\begin{enumerate}
		\item\label{item:I-closed} The algebras $\cA$, $\cJ$ and $\cP$ are closed under $\d$. 
		\item\label{item:dl-iterate} For every $i\in \{2,\dots, k\}$ we have $\cA_{i}\cdot \d_{i}\cP\subseteq \cP$.		
		\item\label{item:I-smooth} We have $\cP \subseteq \cS(U_1)$.  
		\item \label{item:AW-smooth} For every $i\in \{2,\dots, k\}$ we have $\cW_{i}'\cdot \cA_i\subseteq \cS(U_1)$.
	\end{enumerate}
\end{lema}
\begin{proof}
	\ref{item:I-closed} Follows from Lemma \ref{lem:AI}\ref{item:AI-closed}
	
	\ref{item:dl-iterate} By \ref{item:I-closed}, $\cJ$ is an ideal of $\cA$ which is closed under $\d$, so $\cA_{i}\cdot \d_{i}\cJ\subseteq \cA_{i}\cdot \cJ\subseteq \cJ \subseteq \cP$. It remains to show that for every $j\in \{2,\dots, k\}$ we have $\cA_{i}\cdot \d_{i}\cI_{j}\subseteq \cP$. 
	
	Assume $j=i$.  By Lemma \ref{lem:AI}\ref{item:I-closed} the ideal $\cI_{i}$ of $\cA_{i}$ is closed under $\d$, so we have $\cA_{i}\cdot \d_{i}\cI_{i}\subseteq \cI_{i}\subseteq \cP$. 
	
	Assume $j\neq i$. We claim that $\d_{i}\cI_{j}\subseteq \cJ+ \cI_{i}\cdot \cI_{j}$. Let $a\in \cA_{j}$ and $\alpha\in \N^{k}$. We have $a\d^{\alpha}r_{1}\in \cJ$, so $\d_{i}(a\d^{\alpha}r_{1})\in \cJ$ because $\cJ$ is closed under $\d$ by \ref{item:I-closed}. For the second generator, we compute $\d_{i}(a\d^{\alpha}r_{j})=a\d^{\alpha}\d_{i}r_{j}+\d_ia \cdot \d^{\alpha}r_j=\d_ia \cdot \d^{\alpha}r_j$, because $\d_{i}r_{j}=0$ by Lemma \ref{lem:mixed-vanish}. By Lemma \ref{lem:AI}\ref{item:AI-dl}, we have $\d_ia \in \cI_{i}\cA_{j}$, so $\d_i a\cdot \d^{\alpha} r_{j}\in \cI_{i}\cA_{j}\cdot \cI_{j}=\cI_{i}\cI_{j}$. Hence $\d_{i}(a\d^{\alpha}r_{j})\in \cJ+ \cI_{i}\cdot \cI_{j}$, as claimed.
	
	Now $\cA_{i}\cdot \d_{i}\cI_{j}\subseteq \cA_{i}\cJ+\cA_{i}\cI_{i}\cI_{j}\subseteq \cJ+\cI_{i}\cI_{j}\subseteq \cP$, as needed.
	
	\ref{item:I-smooth} We need to show that $\cJ\subseteq \cS(U_1)$ and $\cI_{i}\subseteq \cS(U_1)$ for all $i\in \{2,\dots, k\}$. The second assertion follows from Lemma \ref{lem:AI}\ref{item:AI-I-smooth}. For the first one, recall from Lemma \ref{lem:AI}\ref{item:AI-flattened} that each generator $\d^{\alpha}r_1$ of $\cJ$ flattens $\cA_{i}$ on $U_1\cap \d A$, for all $i\in \{2,\dots, k\}$, so it flattens the whole $\cA=\sum_{i=2}^{k}\cA_{i}$ there. Thus by Lemma \ref{lem:flattening}\ref{item:flat-is-smooth} we have $\d^{\alpha}r_1\cdot \cA\subseteq \cS(U_1)$, and therefore $\cJ\subseteq \cS(U_1)$, as claimed.
	
	\ref{item:AW-smooth} By Lemma \ref{lem:AI}\ref{item:AI-flattened}, for any $s$ the function $w_{i}^{(s)}$ flattens $\cA_{i}$ on $U_1\setminus W_i$. Since $\cA_{i}\subseteq \cS(W_i)$ by Lemma \ref{lem:AI}\ref{item:AI-smooth-off-corner}, Lemma \ref{lem:flattening} implies that $\cW_{i}'\cdot \cA_{i}\subseteq \cS(U_1)$, as needed.
\end{proof}

\begin{definition}\label{def:M}
	Let $\cM$ be the set of $k\times k$ matrices $[a_{ij}]_{1\leq i,j\leq k}$ satisfying the following properties.
	\begin{enumerate}[(i)]
		\item\label{item:M-first-row} For every $j\in \{1,\dots, k\}$ we have $a_{1j}=0$ (i.e.\ the first row is zero),
		\item\label{item:M-Jij} For every $i\in \{2,\dots, k\}$, $j\in \{1,\dots, k\}$ we have $a_{ij}\in \cA_{i}\cdot \cP$,
		\item\label{item:M-off-diagonals} For every $i,j\in \{2,\dots, k\}$ such that $i\neq j$ we have $a_{ij}\in \cA_{i}\cdot \cI_{j}\cdot \cP$,
	\end{enumerate}
\end{definition}

\begin{lema}\label{lem:M-ring}
	The set $\cM$ is a ring.
\end{lema}
\begin{proof}
	Let $M,M'\in \cM$. Clearly, $M+M'\in \cM$. To show that $M\cdot M'\in \cM$, write $M=[a_{il}]$, $M'=[a'_{lj}]$, $M\cdot M'=[b_{ij}]$. Since $M'$ satisfies Definition \ref{def:M}\ref{item:M-first-row}, we have $a'_{1j}=0$ for every $j$, so
	\begin{equation*}
		b_{ij}=\sum_{l=2}^{k} a_{il}a_{lj}'.
	\end{equation*}
	Condition \ref{def:M}\ref{item:M-first-row} for $M$ gives $a_{1l}=0$, so $b_{1j}=0$, too, i.e.\ $MM'$ satisfies condition  \ref{def:M}\ref{item:M-first-row}.
	
	To check condition \ref{def:M}\ref{item:M-Jij}, it is enough to show that $a_{il}a_{lj}'\in \cA_{i}\cP$ for every $i,l\in \{2,\dots, k\}$ and $j\in \{1,\dots, k\}$. Consider the case $i=l$. We have $a_{ii}\in \cA_{i}\cP$ and $a_{ij}'\in \cA_{i}\cP$, so $a_{ii}a_{ij}'\in \cA_{i}\cP$, too, as needed. Consider now the case $i\neq l$. Then $a_{il}\in \cA_{i}\cI_{l}\cP$ and $a_{lj}'\in \cA_{l}\cP$. Since $\cI_{l}$ is an ideal of $\cA_{l}$, we have $\cI_{l}\cA_{l}=\cI_{l}$. Thus $a_{il}a_{lj}'\in \cA_{i}\cI_{l}\cP\cdot \cA_{l}\cP= \cA_{i}\cI_{l}\cdot \cP\subseteq \cA_{i}\cP$ because $\cI_{l}\subseteq \cP$, as needed.
	
	Condition \ref{def:M}\ref{item:M-off-diagonals} is checked similarly. We claim that for every $i,l,j\in \{2,\dots, k\}$ such that $i\neq j$, we have $a_{ij}a_{lj}'\in \cA_{i}\cI_{j}\cP$. Since $i\neq j$, we have $a_{ii}a_{ij}'\in \cA_{i}\cP\cdot \cA_{i}\cI_{j}\cP=\cA_{i}\cI_{j}\cP$, which settles the case $i=l$. Assume $i\neq l$, so $a_{il}\in \cA_{i}\cI_{l}\cP$. If $l=j$ then the equality $\cI_j\cA_j=\cI_j$ implies that $a_{ij}a_{jj}'\in \cA_{i}\cI_{j}\cP\cdot \cA_{j}\cP=\cA_{i}\cI_{j}\cP$, as needed. Eventually, if $l\neq j$ then using the equality $\cI_l\cA_l=\cI_l$ we get  $a_{il}a_{lj}'\in \cA_{i}\cI_{l}\cP\cdot \cA_{l}\cI_{j}\cP=\cA_{i}\cI_{l}\cI_{j}\cP\subseteq \cA_{i}\cI_{j}\cP$, because $\cI_{l}\subseteq \cP$, which ends the proof.
\end{proof}

Recall that our goal is to study the map $\bar{\psi}_{1}=(g,\bar{v}_2,\dots,\bar{v}_{k},\rest)\colon U_1\to Q_{k,n}$ introduced in \eqref{eq:AC-chart}, in the coordinate system $(g,v_{2},\dots, v_{k},\rest)$ fixed in Notation \ref{not:S-smooth}. The Jacobian matrix of $\bar{\psi}_1$ with respect to these coordinates has a block form
\begin{equation}\label{eq:M-def}
	\left[
	\begin{matrix}
		\id -\Psi & * \\
		0 & \id
	\end{matrix}
	\right]
\end{equation}
for some $k\times k$ matrix $\Psi$.  By Lemma \ref{lem:dv}\ref{item:tame_1}, the matrix $\Psi$ is zero on $A_{S}^{\circ}\cap U_1$. Therefore, we can, and will, shrink $U_X$ so that on $U_1$ we have
\begin{equation}\label{eq:M-small}
	||\Psi||<1.
\end{equation}
Recall that we still require $U_X$ to satisfy the statements of Lemmas \ref{lem:smoothchart} and \ref{lem:small-chart}. 

The following result is what motivates Definition \ref{def:M} of the matrix ring $\cM$.

\begin{lema}\label{lem:jacobian}
	The matrix $\Psi$ defined above belongs to the ring $\cM$.
\end{lema}
\begin{proof}
	Recall that for $i\in \{1,\dots, k\}$, the function $\bar{v}_{i}$ is defined as $\sum_{p}\tau^{p}v_{i}^{p}$, where $\{\tau^{p}\}$ is the fixed partition of unity, and each function $v_{i}^{p}$ corresponds to the chart $U_{X}^{p}$ of the fixed atlas $\cU_{X}$. Since by assumption our chart $U_X$ satisfies the statement of Lemma \ref{lem:small-chart}, the above sum runs over the finite set $R_{0}=\{p\in R: \tau^{p}|_{U_X}\not\equiv 0\}$, and for each $p\in R_0$, the function $v^{p}_{i}$ is defined on the entire $U_1$. We have
	\begin{equation*}
		d(v_{i}-\bar{v}_i)=\sum_{p}\tau^{p}\,d(v_i-v_i^p)+\sum_{p}(v_i-v_i^p)d\tau^{p}
	\end{equation*}
	so we can write 
	\begin{equation*}
		\Psi=\sum_{p} [a_{ij}^{p}]+[b_{ij}^p],
	\end{equation*}
	where the $k\times k$ matrices $[a_{ij}^p]$ and $[b_{ij}^p]$ are defined by 
	\begin{equation*}
		a_{1j}^p=0,\quad a_{ij}^p=\tau^{p}\cdot \d_{j}(v_{i}-v_{i}^{p}),
		\quad\mbox{and}\quad
		b_{1j}^{p}=0,\quad b_{ij}^p=(v_{i}-v_{i}^{p})\cdot \d_{j}\tau^{p},
	\end{equation*}
	for all $i\in \{2,\dots, k\}$ and $j\in \{1,\dots, k\}$. We will show that $[a_{ij}^p]\in \cM$ and $[b_{ij}^p]\in \cM$. Clearly, the first rows of these matrices are zero, so they satisfy Definition \ref{def:M}\ref{item:M-first-row}.
	
	We have $\tau^{p}\in \cC^{\infty}(U_X)$, so for every $i\in \{2,\dots, k\}$ we have $\tau^{p} \in \cA_{i}$ by definition \eqref{eq:AI} of $\cA_{i}$. Definition of $\cA_{i}$ and Lemma \ref{lem:dv}\ref{item:vp-v} imply that $v_{i}-v_{i}^{p}\in \cA_{i}$, too. By Lemma \ref{lem:AI}\ref{item:AI-closed}, the algebra $\cA_i$ is closed under $\d$, so $\d_j \tau^{p}\in \cA_{i}$ and $\d_j(v_i-v_i^p)\in \cA_i$. Since $1\in \cP$, we have $\cA_i\subseteq \cA_i\cP$, and therefore  the matrices $[a_{ij}^p]$ and $[b_{ij}^p]$ satisfy  Definition \ref{def:M}\ref{item:M-Jij}. 
	
	Assume $j\neq 1,i$. Then $\d_j\cA_i\subseteq \cI_j\cA_i$ by Lemma \ref{lem:AI}\ref{item:AI-dl}. Hence $\tau^{p}\cdot \d_{j}(v_{i}-v_{i}^{p})\in \cA_i\cdot \cI_j\cA_i=\cA_i\cI_j$, and similarly $(v_{i}-v_{i}^{p})\cdot \d_{j}\tau^{p}\in \cA_i\cI_j$. Again, using the fact that $1\in \cP$ we infer that $[a_{ij}^p]$ and $[b_{ij}^p]$ satisfy Definition \ref{def:M}\ref{item:M-off-diagonals}. too.
\end{proof}

\subsubsection{Differential operators induced by $\cM$} \label{sec:dM}
Our goal is to compare differential operators  $\d_j$ with the ones given by the coordinate vector fields of our eventual chart $\bar{\psi}_i=(g,\bar{v}_2,\dots, \bar{v}_k,\rest)$. To make this comparison in an efficient way, we introduce the following notation.
\smallskip

For a matrix $M\in \cM$, define differential operators $\d_{1}^{M},\dots,\d_{k}^{M}$ by
\begin{equation*}
	[\d_{1}^{M},\dots,\d_{k}^{M}]^{\top}=M^{\top}\cdot [\d_{1},\dots,\d_{k}]^{\top}.
\end{equation*}
%As before, w
We say that an algebra $\cB$ is \emph{closed under $\d^{M}$} if $\d^{M}_{j}b\in \cB$ for every $j\in \{1,\dots, k\}$ and every $b\in \cB$. For a multi-index $(j_1,\dots,j_s)\in \{1,\dots, k\}^{s}$, with $s\geq 0$, we write  $\d_{\boldsymbol{j}}^{M}=\d_{j_1}^{M}\dots, \d_{j_s}^{M}$ and $|\boldsymbol{j}|=s$.

\begin{lema}\label{lem:operators-M} For every matrix $M\in \cM$, the following hold.
	\begin{enumerate} 
		\item \label{item:M-closed} The algebras $\cP$ and $\cA_{i}\cdot \cW_{i}'\cdot \cP$, for every $i\in \{2,\dots, k\}$, are closed under $\d^{M}$.
		\item \label{item:M-smooth} For every multi-index $\boldsymbol{j}$ we have $\d^{M}_{\boldsymbol{j}}v_1\in \cS(U_1)$ and $\d^{M}_{\boldsymbol{j}} h\in \cS(U_1)$ for every $h\in \cC^{\infty}(U_X)$.
	\end{enumerate}
\end{lema}
\begin{proof}
	\ref{item:M-closed} Write $M=[a_{lj}]$ and fix $j\in \{1,\dots, k\}$. Since $a_{1j}=0$ by Definition \ref{def:M}\ref{item:M-first-row}, we have
	\begin{equation}\label{eq:dM}
		\d_{j}^{M}=\sum_{l=2}^{k} a_{lj}\d_{l}.
	\end{equation}
	For every $l\in \{2,\dots, k\}$ we have $a_{lj}\in \cA_l\cdot \cP$ by Definition \ref{def:M}\ref{item:M-Jij}, and $\cA_l \cdot \d_l\cP\subseteq \cP$ by Lemma \ref{lem:products}\ref{item:dl-iterate}, so $a_{lj}\cdot \d_l \cP\subseteq \cA_l \cP\cdot \d_l\cP \subseteq \cP$. By \eqref{eq:dM}, it follows that $\d_j^M\cP\subseteq \cP$. Hence $\cP$ is closed under $\d^M$.
	
	Now, fix $b\in \cA_i$ and an integer $s\geq 0$. By Lemma \ref{lem:mixed-vanish}, we have $\d_{j}w_{i}^{(s)}=0$ for $j\neq i,1$, so by \eqref{eq:dM} 
	\begin{equation*}
		\d_{j}^{M}(b\cdot w_{i}^{(s)})=\sum_{l=2}^{k} a_{lj}\cdot(\d_{l}b)\cdot w_{i}^{(s)}+
		a_{ij}\cdot b\cdot w_{i}^{(s+1)}.
	\end{equation*}
	To prove that $\cA_{i}\cdot \cW_{i}'\cdot \cP$ is closed under $\d^{M}$, we need to show that all the above summands are in $\cA_{i}\cdot \cW_{i}'\cdot \cP$. By Definition \ref{def:M}\ref{item:M-Jij}, we have $a_{ij}\in \cA_i\cP$, so $a_{ij}bw_i^{(s+1)}\in \cA_i\cP\cdot \cA_i\cdot \cW_i'=\cA_{i}\cdot \cW_{i}'\cdot \cP$, as needed. It remains to show that $a_{lj}\cdot \d_lb\in \cA_{i}\cP$ for all $l\in \{2,\dots, k\}$. 
	
	Fix $l\in \{2,\dots, k\}$. By Definition \ref{def:M}\ref{item:M-Jij}, we have $a_{lj}\in \cA_{l}\cP$. If $l\neq i$ then $\d_{l}b\in \cI_{l}\cA_{i}$ by Lemma \ref{lem:AI}\ref{item:AI-dl}, so $a_{lj}\d_lb\in \cA_{l}\cP\cdot  \cI_l\cA_i=\cA_{i}\cI_{l}\cP\subseteq \cA_{i}\cP$, as needed. For the remaining case $l=i$, recall that $\cA_{i}$ is closed under $\d$ by Lemma \ref{lem:AI}\ref{item:AI-closed}, so $\d_{i}b\in \cA_{i}$, and therefore $a_{ij}\d_{i}b\in \cA_{i}\cP\cA_{i}=\cA_{i}\cP$, as needed.
	\smallskip
	
	\ref{item:M-smooth} 
	Fix $i\in \{2,\dots, k\}$. Since $w_i\in \cW_{i}'$ and $1\in \cA_{i}\cdot \cP$, we have $w_{i}\in \cA_{i}\cdot \cW_{i}'\cdot \cP$. By \ref{item:M-closed}, we have $\d^{M}_{\boldsymbol{j}}w_i\in \cA_{i}\cdot \cW_{i}'\cdot \cP$. By Lemma \ref{lem:products}\ref{item:I-smooth},\ref{item:AW-smooth} the latter algebra is contained in $\cS(U_1)$, so $\d^{M}_{\boldsymbol{j}}w_i\in \cS(U_1)$.% , for all $\boldsymbol{j}$.
	
	As in the proof of Lemma \ref{lem:v1C1}\ref{item:viwi-C1}, we conclude that $\d^{M}_{\boldsymbol{j}}v_1\in \cS(U_1)$. Let us recall the argument here. Lemma \ref{lem:intro}\ref{item:intro-sum} implies that  $\d^{M}_{\boldsymbol{j}}w_{1}=-\sum_{i=2}^{k}\d^{M}_{\boldsymbol{j}}w_{i}\in \cS(U_1)$. By Lemma \ref{lem:mixed-vanish} we have $\d_{l}t=0$ for $l\neq 1$, so by \eqref{eq:dM} we have $\d^{M}_{\boldsymbol{j}}t=0$, for all $\boldsymbol{j}$. Since $w_{1}|_{U_1}>0$ by \eqref{eq:Ui}, the formula $v_1=t_1-u_1=tw_1^{-1}-\eta(w_1)$ shows that  $\d^{M}_{\boldsymbol{j}}v_{1}\in \cS(U_1)$, as claimed.
	\smallskip
	
	For the last statement, fix $h\in \cC^{\infty}(U_1)$. By Lemma \ref{lem:pullbacks}\ref{item:pullbacks} we have $h\in \cS(U_1)$, so the claim holds for $|\boldsymbol{j}|=0$. Assume $|\boldsymbol{j}|>0$. By Lemma \ref{lem:AI}\ref{item:dl-smooth}, for any $l\in \{2,\dots,k\}$ we have $\d_{l}h\in \cI_{l}$. By Definition \ref{def:M}\ref{item:M-Jij}, it follows that for any $j\in \{1,\dots, k\}$ we have  $a_{lj}\d_l h\in \cA_l \cP\cdot \cI_l=\cI_l\cP\subseteq \cP$, so $\d^{M}_{j}h\in \cP$ by \eqref{eq:dM}. Since by \ref{item:M-closed} the algebra $\cP$ is closed under $\d^M$, we infer that $\d^{M}_{\boldsymbol{j}}h\in \cP$. But $\cP\subseteq \cS(U_1)$ by Lemma \ref{lem:products}\ref{item:I-smooth}, so eventually $\d^{M}_{\boldsymbol{j}}h\in \cS(U_1)$, as claimed.
\end{proof}

\subsubsection{The A'Campo space admits a \texorpdfstring{$\cC^\infty$}{smooth}  structure.}\label{sec:AX-smooth-proofs}
We are now ready to prove Lemma \ref{lem:transsmooth}, which makes $A$ a $\cC^{\infty}$ manifold with boundary, and establish its properties listed in Proposition \ref{prop:AXsmooth}. 

Recall that in Notation \ref{not:S-smooth}, we have fixed a chart $U_X$, considered the subset $U_1\subseteq\pi^{-1}(U_X)$ with smooth coordinates $(g,v_2,\dots,v_{k},\rest)$, and denoted the sheaf of functions which are smooth with respect to these coordinates by $\cS$. By Lemma \ref{lem:smoothchart}, the map $\bar{\psi}_{1}=(g,\bar{v}_2,\dots,\bar{v}_k,\rest)$ is a $\cC^{1}$-diffeomorphism from $U_1$ onto an open subset of $Q_{k,n}$, so it gives \emph{another} smooth structure on $U_1$, which is (only) $\cC^{1}$-diffeomorphic to the previous one. We denote the sheaf of smooth functions with respect to that smooth structure by $\bar{\cS}$. As in Notation \ref{not:S}, this structure gives differential operators
\begin{equation*}%\label{eq:dbar}
	\bar{\d}_1=\tfrac{\d}{\d g} \quad\mbox{and}\quad \bar{\d}_{i}=\tfrac{\d}{\d\bar{v}_i}\quad \mbox{for}\quad  i\in \{2,\dots, k\}.
\end{equation*}
Note that $\tfrac{\d}{\d g}$ has different meaning here than in Notation \ref{not:S}, i.e.\ $\bar{\d}_{1}$ may not be equal to $\d_1$: this is because the vector field $\bar{\d}_1$ is tangent to the fibers of $(\bar{v}_2,\dots,\bar{v}_k)$, which differ from the fibers of $(v_2,\dots,v_k)$, to which the vector field $\d_1$ is tangent.

Combining Lemmas \ref{lem:operators-M} and \ref{lem:jacobian}, we get the following.

\begin{lema}\label{lem:v-smooth}
	We have $\bar{v}_{i}\in \bar{\cS}(U_1)$ for all $i\in \{1,\dots, N\}$ and $h\in \bar{\cS}(U_1)$ for all $h\in \cC^{\infty}(U_X)$.
\end{lema}
\begin{proof}
	Let $\Psi$ be the matrix introduced in \eqref{eq:M-def}. Then the definition of $\bar{\d}_1,\dots\bar{\d}_k$ reads as
	\begin{equation*}
		[\bar{\d}_{1},\dots,\bar{\d}_{k}]^{\top}= ((\id-\Psi)^{-1})^{\top}\cdot [\d_{1},\dots,\d_{k}]^{\top}   =\sum_{s\geq 0} (\Psi^{s})^{\top}\cdot [\d_{1},\dots,\d_{k}]^{\top}.
	\end{equation*}
	By assumption \eqref{eq:M-small}, the above sequence is uniformly convergent.
	
	By Lemma \ref{lem:jacobian} we have $\Psi\in \cM$, so by Lemma \ref{lem:M-ring}, for any $s\geq 1$ we have $\Psi^{s}\in \cM$, too. Therefore, by Lemma \ref{lem:operators-M} for all $s\geq 1$ and all multi-indices $\boldsymbol{j}$ we have $\d^{\Psi^{s}}_{\boldsymbol{j}} v_1, \d^{\Psi^{s}}_{\boldsymbol{j}} h\in \cS(U_1)$, for any $h\in \cC^{\infty}(U_X)$. Recall that $v_1\in \cS(U_1)$ by Lemma \ref{lem:w_smooth}\ref{item:t1v1-smooth} and $h\in \cS(U_1)$ by Lemma \ref{lem:pullbacks}\ref{item:pullbacks}, so the above result holds for $s=0$, too. It follows that $\bar{\d}_{\boldsymbol{j}}v_1$ and $\bar{\d}_{\boldsymbol{j}}h$, for any $\boldsymbol{j}$, are continuous on $U_1$. In other words, $v_1$ and $h$ are smooth with respect to the coordinates $(g,\bar{v}_2,\dots, \bar{v}_{k},\rest)$, i.e.\ $v_1,h\in \bar{\cS}(U_1)$.
	
	Taking for $U_{X}$ the restriction of the $p$-th chart of the fixed atlas $\cU_{X}$, we infer that $v_1^{p}\in \bar{\cS}(U_1)$, for all $p\in R_0$. Since $\tau^{p}\in \cC^{\infty}(U_X)$, we have shown that $\tau^{p}\in \bar{\cS}(U_1)$, too. Hence $\bar{v}_1=\sum_{p\in R_0} \tau^{p}v_{i}^{p}\in \bar{\cS}(U_1)$. The functions $\bar{v}_2,\dots, \bar{v}_{k}$ are coordinates of the smooth chart $\bar{\psi}_1$ defining $\bar{\cS}$, so they are in $\bar{\cS}(U_1)$, too. 
	
	It remains to show that $\bar{v}_{j}\in \bar{\cS}(U_1)$ for all $j\in \{k+1,\dots, N\}$. We have $\bar{v}_{j}=\sum_{p\in R_0}\tau^{p}v_{j}^{p}$, where $\tau^{p}\in \cC^{\infty}(U_X)\subseteq \bar{\cS}(U_1)$. Thus it is enough to show that $v_{j}^{p}\in \bar{\cS}(U_{1})$ for all $p\in R_0$. Since $j>k$, our chart $U_{X}$ is disjoint from $D_{j}$. Hence $t_{j}^{p}|_{U_X}>0$ and $t_{j}^{p}|_{U_X},\log t_{j}^{p}|_{U_X}\in \cC^{\infty}(U_{X})\subseteq \bar{\cS}(U_1)$. By definition \eqref{eq:def-v_i} of $v_{j}^{p}$, we have $v_{j}^{p}=t_{j}^{p}-u_{j}^{p}$, so it remains to show that $u_{j}^{p}\in \bar{\cS}(U_1)$. Formulas \eqref{eq:def-w_i-u_i} and \eqref{eq:smooth_simplex} give $u_{j}^{p}=\eta(\frac{t}{t_{j}^{p}})=(1-\log t+\log t_{j}^{p})^{-1}=(g^{-1}+\log t_{j}^{p})^{-1}=g\cdot (1+g\log t_{j}^{p})^{-1}$. Since both the coordinate $g$, and the restriction $\log t_{j}^{p}|_{U_X}$ are  in $\bar{\cS}(U_{1})$, we infer that $u_{j}^{p}\in \bar{\cS}(U_1)$, as needed.
\end{proof}

\begin{proof}[Proof of Lemma \ref{lem:transsmooth}]
	Since by Lemma \ref{lem:smoothchart} each $\bar{\psi}_{i}$ is a $\cC^{1}$-diffeomorphism, the transition map $\bar{\psi}_j'\circ \bar{\psi}_{i}^{-1}$ is a $\cC^{1}$-diffeomorphism, too. It remains to prove that it is smooth. 
	
	Without loss of generality, we can assume that $i=1$, and choose $U_X$ so small that it satisfies the statement of Lemma \ref{lem:small-chart} and the assumption \eqref{eq:M-small} above. Moreover, we can assume that $U_{X}\subseteq U_{X}'$, so the index for $U_X$ is contained with the one for $U_{X}'$. Call the latter $S'$. The chart $\bar{\psi}_j'$ is defined in \eqref{eq:AC-chart} as $(g,(\bar{v}_{l})_{l\in S'\setminus \{j\}},\rest')$.  Using above notation, we need to prove that all these functions are in $\bar{S}(U_1)$, i.e.\ that they are smooth with respect to the coordinates $(g,\bar{v}_{2},\dots,\bar{v}_{k},\rest)$.
	
	Clearly, the coordinate $g$ is in $\bar{S}(U_1)$. By Lemma \ref{lem:v-smooth}, we have $\bar{v}_{l}\in \bar{\cS}(U_1)$ for all $l$. The first $k$ functions in $\rest'$ are $\theta_{1}',\dots,\theta_{k}'$. By Lemma \ref{lem:dv}\ref{item:theta-theta'}, they differ from coordinates $\theta_{1},\dots, \theta_{k}$ of $\rest$ by a smooth function on $U_X$, so by Lemma \ref{lem:v-smooth} they are in $\bar{\cS}(U_1)$, too. The remaining functions in $\rest'$ are smooth on $U_X$, so again by Lemma  \ref{lem:v-smooth} they are in $\bar{\cS}(U_1)$, as claimed.
\end{proof}

\begin{proof}[Proof of Proposition \ref{prop:AXsmooth}]
	
	\ref{item:AX-pismooth} By Lemma \ref{lem:v-smooth}, the map $\pi$ pulls back smooth functions to smooth functions, so it is smooth. Its restriction to $A\setminus \d A$ is a diffeomorphism by Lemma \ref{lem:C1chart}.
	
	\ref{item:AX-gsmooth} We argue exactly as in the proof of Proposition \ref{prop:AXC1}\ref{item:AX-gC1}. The function $g$ is a coordinate in each smooth chart \eqref{eq:AC-chart} meeting $\d A$; and $\theta$ is a nonzero linear combination of such coordinates. Thus $(g,\theta)$ is a submersion near $\d A$. It follows that $f\AC=(\exp(-\exp(g^{-1}-1)),\theta)$ is smooth near $\d A$. Away from $\d A$, the result follows from the fact that $f$ is a submersion.
	
	\ref{item:AX-vbar-smooth} This part was proved in Lemma \ref{lem:v-smooth}.
	
	\ref{item:AX-stratification}
	Recall that for every $\emptyset\neq I \subseteq \{1,\dots, N\}$, the restriction $(\pi,\mu)\colon \Int_{\d A} A_{I}^{\circ}\to X_{I}^{\circ}\times \Delta_{I}^{\circ}$ is a topological $(\S^{1})^{\# I}$-bundle by Proposition \ref{prop:AX-topo}\ref{item:top-S1-bundle}. It remains to prove that it is smooth. The map $\pi$ is smooth by \ref{item:AX-pismooth}. The $i$-th coordinate of $\mu$ is $\bar{u}_{i}|_{A_{I}^{\circ}}$, so by Lemma \ref{lem:intro}\ref{item:intro-u=ubar} it is zero if $i\not\in I$ and equals $-\bar{v}_{i}|_{A_{I}^{\circ}}$ if $i\in I$. Thus smoothness of $\mu|_{\Int_{\d A} A_{I}^{\circ}}$ follows from \ref{item:AX-vbar-smooth}.
\end{proof}

\section{Fiberwise symplectic form on the A'Campo space}\label{sec:symplectic}

Let us recall some notation and assumptions from Section \ref{sec:ACampo}. Let $X$ be a complex manifold of dimension $n$, and let $f\colon X\to \C$ be a holomorphic function with only one critical value $0\in \C$, whose central fiber $D\de f^{-1}(0)$ is snc. We denote the irreducible components of $D$ by $D_1,\dots, D_N$, and write $D=\sum_{i=1}^{N}m_{i}D_{i}$ for some positive integers $m_i$.

Assume that $\log |f|<-1$, and that $f|_{X\setminus D}$ is a submersion onto $\D_{e^{-1}}^{*}$. Using certain additional data $(\cU_X,\btau)$, 
we have constructed in Definition \ref{def:AX} the A'Campo space $A$. In Section \ref{sec:AX-smooth-def}, we have endowed $A$ with a structure of a smooth manifold with boundary. By Proposition \ref{prop:AXsmooth}, this manifold comes with a smooth map $\pi\colon A\to X$ which restricts to a diffeomorphism $\pi|_{A\setminus \d A}\to X\setminus D$; and with a submersion $(g,\theta)\colon A\to [0,1)\times \S^1$ whose fibers over $(0,1)\times \S^1$ agree with those of $f$.
\smallskip

Fix a K\"ahler form $\omega_{X}\in \Omega^{2}(X)$. In this section, we will introduce, for any $\delta>0$ and any open subset $W\subseteq A$ so that $f_A|_{\bar{W}}$ is proper,  a $2$-form  on $A$ which agrees with $\omega_{X}$ on $g^{-1}([\delta,1))\subseteq A\setminus \d A=X\setminus D$, and whose restriction to $W$ is fiberwise symplectic with respect to the submersion $(g,\theta)\colon A\to [0,1)\times \S^1$. 

The restriction to $W$ is important for the compactness argument in the proof, and comes naturally in our applications. For example, if $f$ is a log resolution of an isolated hypersurface singularity, we can take for $W$ the preimage of a bounded domain in $\C^{n}$, e.g.\ a ball, cf.\ Example \ref{ex:typical}. In turn, if $f$ is a projective family then one can take $W=X$, i.e.\ no restriction is needed at all.% for other possibilities in a more general setting.}

\subsection{Construction of the fiberwise symplectic form} 

\begin{definition}
	\label{def:finechartz}
	We say that a holomorphic chart $U_{X}$ on $X$ is \emph{fine} if it satisfies the statement of Lemma \ref{lem:smoothchart}, or is disjoint from $D$.
\end{definition}

Lemma \ref{lem:smoothchart} shows that we can choose a fine chart at any point of $X$; and any fine chart meeting $D$ gives rise to a collection of smooth charts \eqref{eq:AC-chart} on its preimage in $A$. Fine charts are, by definition, adapted to $f$, so they come equipped with functions $t_i,w_i,\theta_{i},v_{i}$ introduced in Section \ref{sec:basicfunctions}. We will now use the angular coordinates $\theta_{i}$ to define \emph{global angular $1$-forms} $\alpha_{i}\in \Omega^{1}(A)$, as follows.

Fix an atlas $\cV_{X}=\{(V^{q}_{X},\psi^{q}_{X}):q\in \tilde{R}\}$ whose all charts are fine; and a locally finite partition of unity $\tilde{\btau}=\{\tilde{\tau}^{q}\}_{q\in \tilde{R}}$ inscribed in $\cV_{X}$ (note that $(\cV_X,\tilde{\btau})$ may differ from the adapted pair $(\cU_X,\btau)$ used in Section \ref{sec:AX-smooth} to define the smooth structure on $A$). Since the map $\pi\colon A\to X$ is smooth, the pullback through $\pi$ of each $\tilde{\tau}^{q}$ to $A$ is smooth, too. As usual, we denote these pullbacks by the same letters. 

Fix $i\in \{1,\dots, N\}$ and $q\in \tilde{R}$ such that $i$ lies in the index set \eqref{eq:index-set} of the chart $V^{q}_X$. Let $\theta_{i}^{q}$ be the corresponding angular coordinate on $V^{q}\de \pi^{-1}(V_X^q)$, see \eqref{eq:def-theta_i}: it is smooth by definition \eqref{eq:AC-chart} of the smooth charts on $V^q$. Then $d\theta_{i}^{q}$ is a smooth $1$-form on $V^{q}$. Since $\tilde{\tau}^{q}$ vanishes identically away from $V^q$, the form $\tilde{\tau}^{q}d\theta_{i}^{q}$ extends to a global smooth form on $A$, which is zero off $V^{q}$. If $V_{X}^{q}\cap D_{i}=\emptyset$, we put $\theta_{i}^{q}=0$, so $\tilde{\tau}^{q}d\theta_{i}^{q}=0$.  
This way, for each $i\in \{1,\dots, N\}$ we define a \emph{global angular $1$-form}
\begin{equation*}
	\alpha_{i}\de \sum_{q\in \tilde{R}}\tilde{\tau}^{q}\, d\theta_{i}^{q}.%\in \Omega^{1}(A).%,\quad i\in \{1,\dots, N\}.
\end{equation*}

By Proposition \ref{prop:AXsmooth}\ref{item:AX-vbar-smooth}, for each $i\in \{1,\dots, N\}$ we have a smooth function $\bar{v}_{i}\colon A\to [-1,1]$. Put
\begin{equation}\label{eq:omegaE}
	\lambda_{E}=\sum_{i=1}^{N}\bar{v}_{i}\alpha_{i}, \quad \omega_{E}\de d\lambda_{E}.%, \quad \omega_{\epsilon}=\pi^{*}\omega_{X}+\epsilon\omega_{E}.
\end{equation}
Given any $\epsilon\geq 0$ we define 
\begin{equation}\label{eq:omegaAC}
	\omega\AC^\epsilon \de \pi^{*}\omega_{X}+\epsilon \omega_{E}.
\end{equation}
For every $\delta>0$ we fix a smooth function $\rho_{\delta}\colon [0,\infty)\to [0,1]$ satisfying  $\rho_{\delta}(0)=1$ and $\rho_{\delta}|_{[\delta,\infty)}\equiv 0$. We define
%\begin{equation*}%\label{eq:omegaEdelta}
%	\lambda^\delta_{E}\de \rho_{\delta}(g)\lambda_E, \quad \omega^\delta_{E}\de d\lambda^\delta_{E}%, \quad \omega_{\epsilon}=\pi^{*}\omega_{X}+\epsilon\omega_{E}.
%\end{equation*}
%and 
	\begin{equation}\label{eq:omegaACdelta}
		\lambda^\delta_{E}\de \rho_{\delta}(g)\lambda_E, \quad \omega^\delta_{E}\de d\lambda^\delta_{E}
		\quad\mbox{and}\quad 
		\omega\AC^{\delta,\epsilon}=\pi^{*}\omega_{X}+\epsilon\omega^\delta_{E}.
	\end{equation}
Note that the restrictions of $\omega\AC^{\delta,\epsilon}$ and $\pi^{*}\omega_{X}$ to $g^{-1}([\delta,+\infty))$ are equal for any $\epsilon\geq 0$.
\smallskip

If $\omega_{X}|_{X\setminus D}=d\lambda_{X}$ for some $\lambda_{X}\in \Omega^{1}(X\setminus D)$, we put
	\begin{equation}\label{eq:lambdaACdelta}
		\lambda\AC^{\delta,\epsilon} \de \pi^{*}\lambda_{X}+\epsilon\lambda^\delta_{E},\quad\mbox{so}\quad \omega_{A}^{\delta,\epsilon}=d\lambda_{A}^{\delta,\epsilon}.
	\end{equation}
In our applications, the pullback $\pi^{*}\lambda_{X}$ will extend to a $1$-form on $A$, hence the forms $\omega_{A}^{\delta,\epsilon}$ constructed above will be exact.

Proposition \ref{prop:omega} below is the main result of this section. Part \ref{item:omega-lift} shows that the symplectic monodromy induced by $\omega_{A}^{\delta,\epsilon}$ at \enquote{radius zero} is particularly easy to handle. In Section \ref{sec:monodromy} we will see that -- assuming $D$ is $m$-separating -- its dynamics are exactly as for the topological monodromy constructed in  \cite{A'Campo}, and \emph{nearly} as for the symplectic monodromy associated with a model resolution of a germ of isolated singularity constructed in \cite[Example 5.14]{McLean}, see Remark \ref{rem:McLean-5.41}. The main difference is the fact that all iterates of our monodromy will have one more component of fixed points, whose analysis is crucial for the proof of Theorem \ref{theo:Zariski}. %main application of this paper).

\begin{prop}\label{prop:omega}
	Let $W\subseteq A$ be an open set such that $f_{A}|_{\bar{W}}$ is proper. Then the following hold.
	\begin{enumerate}
		\item\label{item:omega-symplectic} 	There is an $\epsilon_0>0$ and a neighborhood $W'$ of $\d A\cap W$ in $W$, such that for every $\delta>0$ and every $\epsilon\in (0,\epsilon_0]$, the restriction of each of the forms $\omega^\epsilon\AC$ and $\omega\AC^{\delta,\epsilon}$ defined in \eqref{eq:omegaAC} and \eqref{eq:omegaACdelta} is symplectic on each fiber of  $(g,\theta)|_{W'}\colon W'\to [0,1)\times \S^1$
		\item\label{item:omega-lift} Let $U_X$ be a fine chart such that $U\de \pi^{-1}(U_X)\subseteq W$. Let $S$ be its index set \eqref{eq:index-set}. The symplectic lift, with respect to $\omega^\epsilon\AC$ or $\omega\AC^{\delta,\epsilon}$, of the unit angular vector field on $\{0\}\times \S^1$ to $U\cap \d A$ in coordinates \eqref{eq:AC-chart} on $U$ is equal to 
		\begin{equation}\label{eq:monodromy-vector-field}
			\left(\sum_{i\in S} m_{i} \zeta(u_{i})\right)^{-1}\cdot \sum_{i\in S} \zeta(u_{i})\cdot \frac{\d}{\d \theta_{i}},
		\end{equation}
	\end{enumerate}
	where for each $i\in S$, the function $u_{i}\colon U\to [0,1]$ is defined in \eqref{eq:def-w_i-u_i}, see Lemma \ref{lem:intro}\ref{item:intro-extends}, and $\frac{\d}{\d \theta_i}$ is the coordinate vector field of the chart \eqref{eq:AC-chart}. The function $\zeta\colon s\mapsto s^{-2}e^{1-s^{-1}}\colon [0,\infty)\to [0,\infty)$ is the derivative of the inverse to \eqref{eq:smooth_simplex}. In particular, $\zeta$ is smooth and all its derivatives vanish at $0$.
\end{prop}

The proof of Proposition \ref{prop:omega} will be carried out in two steps, as follows. Since $\bar{W}$ is compact, we can work locally, in a preimage $U$ of a fine chart $U_X$. In Proposition \ref{prop:positivity}, we prove that the forms $\omega\AC^\epsilon (\sdot,J\sdot)$ and $\omega\AC^{\delta,\epsilon} (\sdot,J\sdot)$ are fiberwise positive definite on $U \setminus \d A$, where $J$ is the standard almost complex structure on $X$, pulled back by a diffeomorphism $\pi|_{A\setminus \d A}$. Unfortunately, $J$ does not extend to $\d A$, so this argument does not show non-degeneracy at radius zero. We prove the latter in Proposition \ref{prop:lifts} by directly computing that the symplectic orthogonal in $T\d A$ to the fibers of $(g,\theta)$ is spanned by the vector field  \eqref{eq:monodromy-vector-field}.

\subsection{Positivity away from \texorpdfstring{$\d A$}{d A}}\label{sec:positivity}

Let $J$ be the standard almost complex structure on $X\setminus D$. We use the same letter $J$ to denote its pullback to $A\setminus \d A$ via the diffeomorphism $\pi|_{A\setminus \d A}$. We say that a $2$-form $\omega\in \Omega^2(A\setminus \d A)$ \emph{fiberwise tames $J$} if $\omega(v,Jv)>0$ for every nonzero vector $v$ tangent to the fiber of $(g,\theta)$. In particular, if $\omega$ (fiberwise) tames $J$, then it is (fiberwise) nondegenerate. Note that in this definition we do not assume that $\omega(Jv,Jw)=\omega(v,w)$, that is, $\omega$ may not be (fiberwise) $J$-compatible.

In this section, we prove the following result.

\begin{prop}\label{prop:positivity}
	For every point $x_0\in D$ there is an $\epsilon_0>0$ and a neighborhood $U_X$ containing $x_0$, such that for any $\epsilon\in [0,\epsilon_0]$, $\delta>0$, each of the forms $\omega^\epsilon\AC$ and $\omega\AC^{\delta,\epsilon}$ 
	fiberwise tames $J$ on $\pi^{-1}(U_X)\setminus \d A$.	
\end{prop}

\begin{remark}
	\label{rem:positivityreductions}
	For the proof of Proposition~\ref{prop:positivity} it is enough to consider the case $\omega^\epsilon\AC$, for  $\epsilon>0$.
\end{remark}
\begin{proof}
	For $\epsilon=0$ we have $\omega_{A}^{\delta,0}=\omega^{0}\AC=\pi^*\omega_X$, which tames $J$. Assume that $\omega_{A}^{\epsilon}$ fiberwise tames $J$ for all $\epsilon\in [0,\epsilon_0]$. Fix $\epsilon \in [0,\epsilon_0]$. The restriction of $\omega\AC^{\epsilon,\delta}$ to the level set $g^{-1}(c)$ equals  $\omega\AC^{\rho_{\delta}(c)\epsilon}$. Since $\rho_\delta(c)\leq 1$ we have $\rho_{\delta}(c)\epsilon\leq\epsilon_0$, so by assumption this restriction fiberwise tames $J$, as needed. 
\end{proof}

We will use the following notation. Fix $x_0\in D$, so $x_0\in X_{S}^{\circ}$ for some nonempty $S\subseteq \{1,\dots, N\}$. We order the components of $D$ so that $S=\{1,\dots,k\}$. By Lemma \ref{lem:smoothchart}, there is a fine chart around $x_0$ whose associated index set \eqref{eq:index-set} is $S$. We fix such a chart $U_X$, put $U=\pi^{-1}(U_X)$, and use functions $w_i,\theta_i$ etc.\ introduced for $U_X$ in Section \ref{sec:basicfunctions}, cf.\ Lemma \ref{lem:intro}\ref{item:intro-extends}.

\begin{lema}\label{lem:d-theta} The following hold.
	\begin{enumerate}
		\item\label{item:alpha-dtheta} For every $i\in S$ there is a $1$-form  $\beta_i\in\Omega^1(U_X)$ such that we have $d\theta_{i}-\alpha_{i}=\pi^{*}\beta_i$.
		\item\label{item:dvbar-rest-fiberwise} Fix $j\in \{1,\dots, N\}\setminus S$ and a subset $I\subseteq S$. Then there is a bounded $1$-form $\mu_{j}\in \Omega^{1}(U_{X}\cap X_{I}^{\circ})$ which satisfies  $\mu_{j}|_{g^{-1}(\tau)}=d\bar{v}_{j}|_{g^{-1}(\tau)}$ for every $\tau\geq 0$.
		\item\label{item:alpha-rest} For every $j\in \{1,\dots, N\}\setminus S$, the $1$-form $\alpha_{j}$ is a pullback of a smooth $1$-form on $U_X$.
		\item \label{item:dalpha} For every $i\in \{1,\dots, N\}$, the $2$-form $d\alpha_{i}$ is a pullback of a smooth $2$-form on $X$.
	\end{enumerate}
	%	In particular, all the above forms are bounded from $X$.
\end{lema}
\begin{proof}	
	\ref{item:alpha-dtheta} 
	Fix $q\in \tilde{R}$. Put $V_{X}=V_{X}^{q}\cap U_{X}$, $V=\pi^{-1}(V_X)$. On $V$, we have an equality $\theta_{i}=\theta_{i}^q+b_{i}^{q}$ for some smooth map $b_{i}^{q}\colon V_{X}\to \S^1$. Indeed, if $V_{X}^{q}\cap D_{i}\neq \emptyset$, then this follows from Lemma \ref{lem:dv}\ref{item:theta-theta'}, applied to $U_{X}'=V_{X}^{q}$. In turn, if $V_{X}^{q}\cap D_{i}=\emptyset$, then by convention $\theta_{i}^{q}=0$, and the restriction $\theta_{i}|_{V}$ is a pullback of a smooth function from $V_{X}$, as needed. 
	
	Since the restriction $\tilde{\tau}^{q}|_{U_X}$ is supported on $V_{X}$, the $1$-form $\tilde{\tau}^{q}\, db_{i}^{q}$ extends to a smooth $1$-form on $U_{X}$. Thus on $U$, we have an equality $\tilde{\tau}^{q}\, d\theta_{i}=\tilde{\tau}^{q}d\theta_{i}^{q}+\pi^{*}(\tilde{\tau}^{q}\, db_{i}^{q})$. Since $\sum_{q\in \tilde{R}}\tilde{\tau}^{q}=1$, taking a sum over all $q\in \tilde{R}$ we get $d\theta_{i}=\alpha_{i}+\pi^{*}\beta_{i}$, where $\beta_{i}\de \sum_{q\in \tilde{R}}\tilde{\tau}^{q}\, db_{i}^{q}$ is a smooth form on $U_X$, as needed.
	
	\ref{item:dvbar-rest-fiberwise} By definition \eqref{eq:def-vbar-ubar-mu}, we have $\bar{v}_{j}=\sum_{p\in R}\tau^{p}v_{j}^{p}$, so $d\bar{v}_{j}=\sum_{p}\tau^{p}dv_{j}^{p}+\sum_{p}v_{j}^{p}d\tau^{p}$. Since $j\not\in I$, by Lemma \ref{lem:Ui-simple}\ref{item:vi-smooth-on-strata}  $v_{j}^{p}$ restricts to a smooth function on $V_{X,I}^{p}\de U_{X}\cap U_{X}^{p}\cap X_{I}^{\circ}$. Since each $\tau^{p}$ is a smooth function on $X$, the form $\sum_{p} v_{j}^{p}d\tau^{p}\in \Omega^{1}(U_{X}\cap X_{I}^{\circ})$ is bounded. It remains to show that each $dv_{j}^{p}$ is fiberwise bounded from $X$. More precisely, for each $p$ we need to find a bounded $1$-form $\mu_{j}^{p}\in \Omega^{1}(V_{X,I}^{\circ})$, such that for every $\tau\geq 0$ we have an equality of restrictions $\mu_{j}^{p}|_{g^{-1}(\tau)\cap V^{p}_{X,I}}=dv_{j}^{p}|_{g^{-1}(\tau)\cap V_{X,I}^{p}}$. 
	
	If $U_{X}^{p}\cap D_{j}=\emptyset$ 	then by convention $v_{j}^{p}=0$, so we can take $\mu_{j}^{p}=0$. Assume $U_{X}^{p}\cap D_{j}\neq \emptyset$. Since $dg|_{g^{-1}(\tau)}=0$, Lemma \ref{lem:computations-d}\ref{item:dv} implies that $dv_{j}^{p}|_{g^{-1}(\tau)}=(1+(t_{j}^{p})^{-1}\cdot (u_{j}^{p})^{2})\cdot dt_{j}^{p}|_{g^{-1}(\tau)}$. By assumption $j\not\in S$, so the closure $\bar{U}_{I}^{\circ}$ does not meet $D_{j}$. Hence the restrictions $(1+(t_{j}^{p})^{-1}\cdot (u_{j}^{p})^{2})|_{V_{X,I}^{p}}$ and $dt_{j}^{p}|_{V_{X,I}^{p}}$ are smooth and bounded. Thus $\mu_{j}^{p}\de (1+(t_{j}^{p})^{-1}\cdot (u_{j}^{p})^{2})\cdot dt_{j}^{p}|_{V_{X,I}^{p}}\in \Omega^{1}(V_{X,I}^{p})$ is a bounded form, whose restriction to each fiber $g^{-1}(\tau)$ is equal to $dv_{j}^{p}|_{V_{X,I}^{p}}$, as needed.
	%
	%For each $p\in R$, the function $v_{j}^{p}$ is smooth on $U_{X}^{p}\setminus D_{j}$: indeed, if $U_{X}^{p}\cap D_{j}\neq \emptyset$ this follows from the definition \eqref{eq:def-v_i} of $v_{j}^{p}$; and if $U_{X}^{p}\cap D_{j}=\emptyset$ then by convention $v_{j}^{p}=0$. Since $\tau^{p}$ is smooth and supported on $U_X^p$, each $\tau^{p}v_{j}^{p}$ extends to a smooth function on $X\setminus D_j$. Hence $\bar{v}_{j}|_{X\setminus D_j}$ is a smooth function on $X\setminus D_{j}$. Since $U_{X}\cap D_j=\emptyset$, it follows that the restriction $\bar{v}_{j}|_{U_X}$ is a smooth function on $U_X$, so $d\bar{v}_{j}$ restricts to a smooth form on $U_X$, as needed.
	
	\ref{item:alpha-rest} For every $q\in \tilde{R}$ we have a smooth map $\theta_{j}^{q}\colon V_{X}^{q}\setminus D_{j}\to \S^1$, defined by the formula \eqref{eq:def-theta_i} if $V_{X}^{q}$ meets $D_j$, and by $\theta_{j}^{q}=0$ otherwise. Hence for each $q\in \tilde{R}$, the form $\tilde{\tau}^{q}d\theta_{j}^{q}$ is smooth on $X\setminus D_j$, in particular, on $U_X$. Taking a sum over all $q\in \tilde{R}$ we infer that $\alpha_j$ is a smooth form on $U_X$, as needed. 
	%Since $D_{j}\cap U_{X}=\emptyset$, the functions $v_{j}^{p}$ and $\theta_{j}^{q}$ are pullbacks of smooth ones from  $U_X\cap U_{X}^p$ and $U_X\cap V_{X}^{q}$ for $p\in R$, $q\in \tilde{R}$. It follows that $d\bar{v}_{j}$ and $\alpha_{j}$ are pullbacks of smooth forms on $X$, as needed. 
	
	\ref{item:dalpha} Follows from \ref{item:alpha-dtheta} if $i\in S$ and from  \ref{item:alpha-rest} if $i\not\in S$.
\end{proof}

The following easy consequence of Lemma \ref{lem:d-theta} %implies that, away from the singularities of $D$, the restriction $\lambda_{A}^{\delta,\epsilon}|_{\d A}$ is pulled back from $D$. This observation 
will be important in Section \ref{sec:basic-setting} to control various monodromies near the boundary of the Milnor fiber.

\begin{remark}
	\label{rem:vanishing_restriction}
	For $\theta\in \S^1$ put $B_{i,\theta}^{\circ}=f_{A}^{-1}(0,\theta)\cap A_{i}^{\circ}$. %, i.e.\ $B^{\circ}_{i,\theta}$ is the intersection of a radius-zero fiber with the preimage of $X_{i}^{\circ}=D_{i}\setminus (D-D_{i})$, see \eqref{eq:stratification-pullback}. 
	Then $\pi_{i,\theta}\de \pi|_{B^{\circ}_{i,\theta}}\colon B_{i,\theta}^{\circ}\to X_{i}^{\circ}$ is an $m_{i}$-fold covering. Moreover, there is a form $\beta_{i}\in \Omega^{1}(X_{i}^{\circ})$ such that for every $\theta\in \S^{1}$ we have $\lambda_{E}|_{B_{i,\theta}^{\circ}}=\pi_{i,\theta}^{*}\beta_{i}$. As a consequence, if $\omega_{X}|_{X_{i}^{\circ}}=d\lambda_{X}$ for some $\lambda_{X}\in \Omega^{1}(X_{i}^{\circ})$, then for every $\delta,\epsilon>0$ the form $\lambda_{A}^{\delta,\epsilon}$ defined in \eqref{eq:lambdaACdelta} satisfies the equality  $\lambda_{A}^{\delta,\epsilon}|_{B_{i,\theta}^{\circ}}=\pi_{i,\theta}^{*}(\lambda_{X}+\epsilon \beta_{i}^{\delta})$ for some $\beta_{i}^{\delta}\in \Omega^{1}(X_{i}^{\circ})$.
\end{remark}
\begin{proof}
	%Clearly, $\pi_{i,\theta}$ is an isomorphism. 
	The first assertion is clear. Indeed, in coordinates \eqref{eq:AC-chart} the subset $B_{i}^{\circ}$ is given by $\{m_{i}\theta_i=0\}$, and the restriction $\pi|_{A_{i}^{\circ}}\colon A_{i}^{\circ}\to X_{i}^{\circ}$ is  $(0,\theta_{i},z)\mapsto (0,z)$, i.e.\ it collapses the angular coordinate $\theta_{i}$.
	
	To find $\beta_{i}$, recall from  Lemma \ref{lem:intro}\ref{item:intro-u=ubar} that on $A_{i}^{\circ}$ we have $\bar{v}_{i}=-\bar{u}_{i}=-1$ and $\bar{u}_{j}=0$ for all $j\neq i$. Hence by Lemma \ref{lem:d-theta}\ref{item:alpha-rest}, we have  $\lambda_{E}|_{A_{i}^{\circ}}=-\alpha_{i}+\pi^{*}\gamma$ for some $\gamma\in \Omega^{1}(X_{i}^{\circ})$. Choose a fine chart $V_X$ with associated index set $\{i\}$, and let $V=\pi^{-1}(V_X)$. By Lemma \ref{lem:d-theta}\ref{item:alpha-dtheta}, we have $\alpha_{i}|_{V}=d\theta_{i}+\pi^{*}\beta$ for some $\beta\in \Omega^{1}(V_X)$. Since $f_{A}|_{V}=(r_{i}^{m_i},m_i\theta_{i})$, the restriction of $d\theta_{i}$ to any fiber of $f_{A}|_{V}$ is zero. Hence $\alpha_i|_{V\cap B_{i,\theta}^{\circ}}=\pi^{*}_{i,\theta}\beta|_{V\cap B_{i,\theta}^{\circ}}$. Let $V_{X}'$ be another such chart, let $\beta'\in \Omega^{1}(V_{X}')$ be the corresponding $1$-form, and let $V''_{X}=V_X\cap V_X'$, $V''=\pi^{-1}(V_{X}'')$. Then we have an equality of restrictions $\pi_{i,\theta}^{*}\beta|_{V''_X\cap X_{i}^{\circ}}=\alpha_i|_{V''\cap B_{i,\theta}^{\circ}}=\pi_{i,\theta}^{*}\beta'|_{V''_X\cap X_{i}^{\circ}}$. Since $\pi_{i,\theta}\colon B_{i,\theta}^{\circ}\to X_{i}^{\circ}$ is a local diffeomorphism, it follows that $\beta|_{V''_X\cap X_{i}^{\circ}}=\beta'|_{V''_X\cap X_{i}^{\circ}}$. We conclude that the above forms glue to a form $\check{\beta}\in \Omega^{1}(X_{i}^{\circ})$. Eventually, we put $\beta_{i}=-\check{\beta}+\gamma$.
\end{proof}

We will now introduce smooth coordinates on $U_X\setminus D=U\setminus \d A$ which are well suited to the standard complex structure $J$. Recall that our holomorphic coordinates on $U_X$ are $(z_{1},\dots,z_{k},z_{i_{k+1}},\dots, z_{i_{n}})$, for some $\{i_{k+1},\dots, i_{n}\}\subseteq \Z_{>N}$. For $i\in \{1,\dots, k\}$, we have $r_{i}=|z_i|$, $\theta_{i}=\frac{z_i}{|z_i|}$, see \eqref{eq:def-r_i-t_i}, \eqref{eq:def-theta_i}. Put
\begin{equation*}
	s_{i}=\log r_{i}\colon U_{X}\setminus D_i \to \R,\qquad s_{i+k}=\theta_{i}\colon U_{X}\setminus D_i  \to \S^1.
\end{equation*}
For $j\in \{1,\dots, n-k\}$ let $x_{2j-1},x_{2j}\colon U_X\to \R$ be the real and imaginary part of $z_{i_j}$. Now
\begin{equation}\label{eq:chart-sx}
	(s_{1},\dots,s_{2k},x_{1},\dots,x_{2(n-k)})\colon U\setminus \d A\to \R^{k}\times (\S^1)^{k}\times \R^{2(n-k)}
\end{equation}
is a smooth chart on $U\setminus \d A=U_X\setminus D$. With respect to this chart, we can define coordinate vector fields $\nu_{1},\dots,\nu_{2k},\xi_{1},\dots,\xi_{2(n-k)}$ by the conditions
\begin{equation}\label{eq:coordinates-sx}
	d s_{i}(\nu_{j})=\delta_{j}^{i},\quad d s_{i}(\xi_{a})=0,\quad dx_{a}(\nu_{j})=0,\quad dx_{a}(\xi_{b})=\delta_{b}^{a}
\end{equation}
for all $i,j\in \{1,\dots, 2k\}$ and all $a,b\in \{1,\dots, 2(n-k)\}$. For an integer $i\in \{1,\dots,2k\}$ we put $\bar{i}=i$ if $i\leq k$ and $\bar{i}=i-k$ if $i>k$. We will also use the functions $\sigma_i$, $i\in \{1,\dots, k\}$, defined in \eqref{eq:rho-sigma} by
\begin{equation*}%\label{eq:sigma_i}
	\sigma_i=t_i^2+t_i u_i^2.
\end{equation*}
Note that $\sigma_{i}\rightarrow 0$ as $t_i\rightarrow 0$. In particular, $\sigma_{i}$ is bounded. 

\begin{lema}\label{lem:nu-j}
	The following hold.
	\begin{enumerate}
		\item\label{item:dt-ds} For every $i\in \{1,\dots,k\}$, we have $ds_{i}=m_{i}^{-1}t_{i}^{-2}\, dt_{i}$,
		\item\label{item:ds-J} For every $i\in \{1,\dots,k\}$, we have  $ds_{i} \circ J= -d s_{i+k}$ and $ds_{i+k}\circ J=ds_{i}$.
	\end{enumerate}
\end{lema}
\begin{proof}
	\ref{item:dt-ds} The function $t_i$ was defined in \eqref{eq:def-r_i-t_i} as $t_{i}=-(m_{i} \log r_{i})^{-1}=-(m_{i}s_{i})^{-1}$, so $s_{i}=-(m_{i}t_{i})^{-1}$, and therefore $ds_{i}=m_{i}^{-1}t_{i}^{-2}\, dt_{i}$, as claimed.
	
	\ref{item:ds-J} We have $ds_{i}=d\log r_i=r_{i}^{-1}\, dr_{i}$. Recall from Example \ref{ex:J} that $dr_{i}\circ J=-r_{i}\, d\theta_{i}$, so $ds_{i}\circ J=r_{i}^{-1}\cdot (-r_{i}\, d\theta_{i})=- d\theta_{i}=-d s_{i+k}$ by definition. Now $ds_{i+k}\circ J=-ds_{i}\circ J\circ J=ds_{i}$, as claimed.
\end{proof}

%\red{We now focus on...} 

%\smallskip
%

For a $2$-form $\omega\in \Omega^{2}(U\setminus \d A)$, denote by $\omega^{J}$ the symmetric part of $(v,w)\mapsto \omega(v,Jw)$. In the next lemma, we compute $\omega^{J}_{E}$. 

Recall from Definition \ref{def:bounded-from-X} that a form $\gamma\in \Omega^{*}(U\setminus \d A)$ is \emph{bounded from $X$} if $\gamma=\pi^{*}\gamma'$ for some bounded $\gamma'\in \Omega^{*}(U_X\setminus D)$. For example, if $i\in \{1,\dots, k\}$ then $dr_{i}$ is bounded from $X$, but $dt_i$ is not.

\begin{lema}\label{lem:omegaAC-sym}
	For $i\in \{1,\dots, 2k\}$ there are bounded $b_i\in \cC^{\infty}(U \setminus \d A)$ and forms $\beta_i\in \Omega^{1}(U\setminus \d A)$, $\omega_1\in \Omega^2(U\setminus \d A)$ bounded from $X$, such that 
	for every $\tau>0$, we have the equality of restrictions
	\begin{equation}\label{eq:omegaAC-sym}
		\omega_{E}^{J}|_{g^{-1}(\tau)}=\tilde{\omega}_E^{J}|_{g^{-1}(\tau)},
		\quad\mbox{where}\quad 
		\tilde{\omega}_E^{J}\de
		\sum_{i=1}^{2k} \sigma_{\bar{i}} (b_{i}\, (ds_{i})^{2}+ds_{i}\cdot \beta_{i})
		+
		\omega_1^{J}.	
	\end{equation}
	Moreover, we can shrink the neighborhood $U_X$ of $x_0$ so that $b_i>\frac{1}{2}$ for all $i\in \{1,\dots,2k\}$.
\end{lema}
\begin{proof}
	Fix $i\in \{1,\dots, k\}$. Since we work with forms restricted to fibers of $g$, we will tacitly use the equality $dg=0$. This way, Lemma \ref{lem:computations-d}\ref{item:dv} reads as $dv_{i}=(1+t_{i}^{-1}u_{i}^{2})dt_i$. Since $(1+t_{i}^{-1}u_{i}^{2})=t_{i}^{-2}(t_{i}^{2}+t_iu_{i}^2)=t_i^{-2}\sigma_i$ by definition \eqref{eq:rho-sigma} of $\sigma_{i}$, we get $dv_{i}=\sigma_{i}t_{i}^{-2}\, dt_{i}$. Hence by Lemma \ref{lem:nu-j}\ref{item:dt-ds}
	\begin{equation*}\label{eq:mu}
		d v_{i}=\sigma_{i}m_{i}\, ds_{i}.
	\end{equation*}
	Substituting $dg=0$ to Lemma  \ref{lem:dvbar}\ref{item:tame_1-bar}, we infer that there is a bounded function $c_i$, and a $1$-form $\gamma_{i}$ bounded from $X$, such that $d\bar{v}_{i}=(1+ct_{i})\, dv_{i}+\sigma_{i}\gamma_{i}$. 
	Substituting the above formula for $dv_i$, we get 
	\begin{equation}\label{eq:dv-mu}
		d\bar{v}_{i}=\sigma_{i}(m_{i}(1+c_{i}t_{i})\, ds_i +\gamma_{i}).
	\end{equation}
	
	By definition \eqref{eq:omegaE} of the form $\omega_E$, we have 
	$\omega_{E}=\sum_{i=1}^{k}d\bar{v}_{i}\wedge \alpha_{i}+\sum_{j=k+1}^{N}d\bar{v}_{j}\wedge \alpha_{j}+\sum_{i=1}^{N}\bar{v}_i\, d\alpha_{i}.
	$ 
		By Lemma \ref{lem:d-theta}\ref{item:alpha-dtheta}, for every $i\in \{1,\dots, k\}$ there is a $1$-form $\check{\beta}_{i}\in \Omega^{1}(U)$, bounded from $X$, such that $\alpha_{i}=d\theta_{i}+\check{\beta}_{i}$. By Lemma \ref{lem:d-theta}\ref{item:dvbar-rest-fiberwise},\ref{item:alpha-rest}, for every $j\in \{k+1,\dots, N\}$, the $2$-form $d\bar{v}_{j}\wedge \alpha_{j}$ is fiberwise equal to a $2$-form bounded from $X$. Eventually, by Lemma \ref{lem:d-theta}\ref{item:dalpha}, the $2$-form $d\alpha_{i}$ is bounded from $X$ for all $i\in \{1,\dots, N\}$. We thus obtain a fiberwise equality
		\begin{equation*}
			\omega_{E}=%\sum_{i=1}^{k}d\bar{v}_{i}\wedge \alpha_{i}+\sum_{j=k+1}^{N}d\bar{v}_{j}\wedge \alpha_{j}+\sum_{i=1}^{N}\bar{v}_i\, d\alpha_{i}
			\sum_{i=1}^{k}d\bar{v}_{i}\wedge (d\theta_{i}+ \check{\beta}_{i})+\check{\omega}_1,
		\end{equation*}
	for some $1$-forms $\check{\beta}_{i}$ and a $2$-form $\check{\omega}_1$, all bounded from $X$. Using the formula \eqref{eq:dv-mu} for $d\bar{v}_i$ and the notation $d\theta_i=d s_{k+i}$, we conclude that
	\begin{equation}\label{eq:omegaAC-mu}
		\omega_{E}=\sum_{i=1}^{k} \sigma_{i} \left(
		m_{i}(1+c_it_i)\, ds_{i}\wedge ds_{i+k}
		+
		\gamma_{i}\wedge d s_{i+k}
		+
		m_i(1+c_{i}t_{i})\, ds_{i}\wedge \check{\beta}_{i}
		\right)
		+
		\omega_1,
	\end{equation}
	where the $2$-form $\omega_1\de \sum_{i=1}^{k} \sigma_{i}\gamma_{i}\wedge \check{\beta}_{i}+\check{\omega}_1$ is bounded from $X$.  %We will now study each of the $2$-forms in parentheses.
	By Lemma \ref{lem:nu-j}\ref{item:ds-J}, we have 
	\begin{gather*}
		(ds_{i}\wedge ds_{i+k})^{J}=(ds_{i})^{2}+(d s_{i+k})^{2},\quad \mbox{and} \\
		(\gamma_{i}\wedge d s_{i+k})^{J}=d s_{i}\cdot \gamma_{i}-ds_{i+k}\cdot (\gamma_{i}\circ J),\quad
		(ds_{i}\wedge \check{\beta}_{i})^{J}=ds_{i}\cdot (\check{\beta}_{i}\circ J)+d s_{i+k}\cdot \check{\beta}_{i}.
	\end{gather*}
	Since the function $c_i$ is bounded, and the forms $\check{\beta}_i$, $\gamma_{i}$ are bounded from $X$, the forms $\beta_{i}\de \gamma_{i}+m_i(1+c_{i}t_{i})(\check{\beta}_{i}\circ J)$ and $\beta_{i+k}\de -(\gamma_{i}\circ J)+m_i(1+c_{i}t_{i})\check{\beta}_{i}$ are bounded from $X$, too. 
	
	Substituting these definitions to \eqref{eq:omegaAC-mu}, we get
	\begin{equation*}
		\tilde{\omega}_{E}^{J}=\sum_{i=1}^{2k}\sigma_{\bar{i}}\left( m_{\bar{i}}(1+c_{\bar{i}}t_{\bar{i}})\, (ds_{i})^{2}+ds_{i}\cdot \beta_{i}\right)+\omega_1^{J}.
	\end{equation*}
	Put $b_{i}=m_{\bar{i}}(1+c_{\bar{i}}t_{\bar{i}})$. As we approach the point $x_0$, we have $t_{\bar{i}}\rightarrow 0$. Since $m_{\bar{i}}\geq 1$ and $c_{\bar{i}}$ is bounded, for $t_{\bar{i}}$ sufficiently small we have $b_{i}>\tfrac{1}{2}$, as needed. 
\end{proof}

From now on, we assume that our fine chart $U_X$ is so small that the functions $b_i$ from Lemma \ref{lem:omegaAC-sym} satisfy $b_i>\tfrac{1}{2}$. For the next result, we introduce the following notation. 
\smallskip

Let $\cB\subseteq \cC^{\infty}(U_X\setminus D)$ be the algebra of bounded functions. For $i\in \{1,\dots, k\}$, we put
\begin{equation*}
	\cQ_{i}= \langle r_{i}t_{i}^{l}:l\in \Z \rangle\cdot \cB.
\end{equation*}
Since by Lemma \ref{lem:computations}\ref{item:t-r} we have $r_{i}t_{i}^{l}=e^{-(m_it_i)^{-1}}t_{i}^{l}\in \cB$, the algebra $\cQ_{i}$ is an ideal of $\cB$. One should think of $\cQ_{i}$ as the ideal of functions which decay exponentially with respect to $t_i$. In Section \ref{sec:AX-smooth}, a similar role was played by the ideal $\cI_{i}\subseteq \cA_i$.% We will frequently use the following properties of $\cQ_{i}$.

In the following lemma, we use the coordinate vector fields $\nu_{i}$, $\xi_{j}$ introduced in \eqref{eq:coordinates-sx}.

\begin{lema}\label{lem:Qi} The ideals $\cQ_i\subseteq \cB$ satisfy the following properties.
	\begin{enumerate}
		\item\label{item:Qi-t} For every $i\in \{1,\dots, k\}$ we have $t_{i}\cdot \cQ_{i}=\cQ_i$ and $\sigma_{i}\cdot \cQ_{i}=\cQ_i$.
		\item\label{item:Qi-1-form} Let $\beta\in \Omega^{1}(U \setminus \d A)$ be a $1$-form bounded from $X$. Then for every $i\in \{1,\dots, 2k\}$ and every $j\in \{1,\dots, 2(n-k)\}$ we have $\beta(\nu_{i})\in \cQ_{\bar{i}}$ and $\beta(\xi_{j})\in \cB$.
		\item\label{item:Qi-2-form} Let $\gamma\in \Omega^{2}(U\setminus \d A)$ be a $2$-form bounded from $X$. Then for every $i,j\in \{1,\dots, 2k\}$ and every $l\in \{1,\dots, 2(n-k)\}$ we have $\gamma^{J}(\nu_{i},\nu_{j})\in \cQ_{\bar{i}}\cdot \cQ_{\bar{j}}$ and $\gamma^{J}(\nu_i,\xi_l)\in \cQ_{\bar{i}}$.
	\end{enumerate}
\end{lema}
\begin{proof}
	\ref{item:Qi-t} The first equality follows directly from the definition of $\cQ_i$. For the second one, recall that $\sigma_{i}$ was defined in \eqref{eq:rho-sigma} by $\sigma_i=t_i^2+t_iu_i^2$, so $\sigma_i\in \cB$. Since $\cQ_i$ is an ideal of $\cB$, we have $\sigma_i \cdot \cQ_i\subseteq \cQ_i$. 

	To see the other inclusion, note that $\sigma_{i}^{-1}t_{i}^{2}=(t_{i}^{2}+t_iu_i^2)^{-1}t_i^2=(1+t_{i}^{-1}u_i^2)^{-1}\in \cB$. Since $\cQ_{i}$ is an ideal of $\cB$, we have $\sigma_{i}^{-1}t_{i}^{2}\cdot \cQ_{i}\subseteq \cQ_i$. The first equality of \ref{item:Qi-t} shows that $t_i^2\cdot \cQ_i=\cQ_i$, so $\sigma_{i}^{-1}\cdot \cQ_i= \sigma_{i}^{-1}t_{i}^2\cdot \cQ_i\subseteq \cQ_i$, as needed.
	
	\ref{item:Qi-1-form} By Definition \ref{def:bounded-from-X} of a form bounded from $X$, we have $\beta=\pi^{*}\beta'$ for some bounded $\beta'\in \Omega^{1}(U_X\setminus D)$. Recall that our holomorphic chart on $U_X$ is $(z_1,\dots,z_k,z_{i_{k+1}},\dots, z_{i_n})$. Write $y_i$, $y_{i}'$ for the real and imaginary parts of $z_i$, $i\in \{1,\dots, k\}$. Then the $1$-forms $dy_i$, $dy'_i$ for $i\in \{1,\dots, k\}$  and $dx_{l}$ for $l\in \{1,\dots, 2(n-k)\}$ give a basis of $T^{*}U_X$. Away from the zero locus of $z_i$, we can change this basis by replacing $dy_i,d y_i'$ with polar coordinates $dr_i$, $r_id\theta_i$. Hence there are bounded functions $p_i,q_l\in \cC^{\infty}(U_X\setminus D)$ such that
	\begin{equation*}
		\beta'=\sum_{i=1}^{k} p_{i}\, dr_{i}+\sum_{i=1}^{k}p_{k+i} r_i\, d\theta_i+\sum_{l=1}^{2(n-k)}q_{l}\, dx_{l}.
	\end{equation*}
	We have $dr_{i}=de^{s_i}=e^{s_i}\, ds_i=r_{i}\, ds_i$, and $r_{i}\, d\theta_{i}=r_{i}\, d s_{i+k}=r_{\bar{i+k}}ds_{i+k}$ by definition. Since we denote the functions on $U_X\setminus D$ and their pullbacks to $U\setminus \d A$ by the same letters, we get
	\begin{equation*}
		\beta=\sum_{i=1}^{2k} p_{i}\, r_{\bar{i}}\, ds_{i}+\sum_{l=1}^{2(n-k)}q_{l}\, dx_{l}.
	\end{equation*}
	Now by definition \eqref{eq:coordinates-sx} of the coordinate vector fields $\nu_{i},\xi_{l}$ we get $\beta(\nu_i)=p_{i}r_{\bar{i}}\in r_{\bar{i}}\cB\subseteq \cQ_{\bar{i}}$ and $\beta(\xi_{l})=q_{l}\in \cB$, as claimed.
	
	\ref{item:Qi-2-form} By definition of a $2$-form bounded from $X$ we have $\gamma=\pi^{*}\gamma'$ for some bounded $\gamma'\in \Omega^{2}(U_X\setminus D)$. Using some basis of $T^{*}(U_X\setminus D)$, we can write $\gamma'=\sum_{l} \gamma'_{1l}\wedge \gamma'_{2l}$ for some bounded  $\gamma'_{1l},\gamma'_{2l}\in \Omega^{1}(U_X\setminus D)$. Thus $\gamma=\sum_{l} \gamma_{1l}\wedge \gamma_{2l}$ for some  $\gamma_{1l},\gamma_{2l}$ bounded from $X$. Now  \ref{item:Qi-2-form} follows from   \ref{item:Qi-1-form}.
\end{proof}
In a fixed basis of $T(U\setminus \d A)$, we write the matrix of the symmetric $2$-form $\tilde{\omega}_E^{J}$ introduced in \eqref{eq:omegaAC-sym} as
\begin{equation}\label{eq:omega-matrix}
	\left[
	\begin{matrix}
		A & Q^{\top} \\
		Q & P  \\
	\end{matrix}
	\right]
\end{equation}
for some $2k\times 2k$ matrix $A=[a_{ij}]$; some $2(n-k)\times 2k$ matrix $Q=[q_{ij}]$, and some $2(n-k)\times 2(n-k)$ matrix $P=[p_{ij}]$. The entries of $A,Q,P$ are smooth functions on $U\setminus \d A$. 

The next lemma summarizes some properties of these entries for our basis $((\nu_{i})_{i=1}^{2k},(\xi_{i})_{i=1}^{2(n-k)})$ given by coordinate vector fields \eqref{eq:coordinates-sx}. In Lemma \ref{lem:Gauss} we will show that these properties are preserved after change of basis given by the Gaussian diagonalization of \eqref{eq:omega-matrix}.

\begin{lema}\label{lem:matrix} The matrix \eqref{eq:omega-matrix} of $\tilde{\omega}_{E}^{J}$ in basis $((\nu_{i})_{i=1}^{2k}, (\xi_{i})_{i=1}^{2(n-k)})$ has the following properties:
	\begin{enumerate}
		\item \label{item:a_ii} For every $i\in \{1,\dots, 2k\}$, there is a neighborhood $V_X$ of $x_0$ and a function $\check{b}_{i} \in \cB$ such that $a_{ii}=\sigma_{\bar{i}} \check{b}_i$ and $\check{b}_i>\frac{1}{4}$ on $\pi^{-1}(V_X)\setminus \d A$.
		\item \label{item:a_ij} For every $i,j\in \{1,\dots, 2k\}$ such that $i\neq j$ we have  $a_{ij}\in  \sigma_{\bar{i}}\sigma_{\bar{j}} \cdot \cB$.
		\item \label{item:q_ij} For every $i\in \{1,\dots, 2(n-k)\}$ and every $j\in \{1,\dots, 2k\}$ we have $q_{ij}\in \sigma_{\bar{j}}\cdot \cB$.
		\item \label{item:P} All entries of the matrix $P$ are bounded functions.
		\item \label{item:vectorsmall} For every $i\in \{1,\dots, 2k\}$ the vector fields  $\sigma_{\bar{i}}^{-1/2}\cdot \pi_*\nu_{i}$ and $\pi_*\nu_{i}$ are bounded and converge to $0$ as $r_{\bar{i}}$ converges to $0$.
		\item \label{item:xi-iso} Let $\Xi$ be the subbundle of $T(U\setminus \d A)$ spanned by $\xi_{1},\dots,\xi_{2(n-k)}$, let $||\sdot||$ be the maximum norm on $\Xi$ with respect to the coefficients of this basis; and let $||\sdot ||_{X}$ be any norm on $T\bar{U}_{X}$. Then there is a constant $K_0>0$ such that for every $\xi\in \Xi$ we have $||\pi_{*}\xi||_{X}\geq K_{0}||\xi||$.
	\end{enumerate}
\end{lema}
\begin{proof}
	%	Lemma \ref{lem:omegaAC-sym} gives
	%	\begin{equation}\label{eq:omegaAC-sym}
	%		\omega_E^{J}=
	%		\sum_{i=1}^{2k} \sigma_{\bar{i}} (b_{i}\, (ds_{i})^{2}+ds_{i}\cdot \beta_{i})
	%		+
	%		\omega_1^{J}.
	%	\end{equation}
	%	where the functions $b_i$ are bounded, $b_i>\tfrac{1}{2}$, and the forms $\beta_{i},\omega_1$ are bounded from $X$.
	%	
	\ref{item:a_ii} By definition of the matrix \eqref{eq:omega-matrix}, we have $a_{ii}=\tilde{\omega}_E^{J}(\nu_i,\nu_i)$. 
	By definition \eqref{eq:coordinates-sx} of $\nu_i$ we have $d s_{j}(\nu_i)=\delta_{j}^{i}$, for all $j\in \{1,\dots, 2k\}$. Substituting this to the formula \eqref{eq:omegaAC-sym} for $\tilde{\omega}_E^{J}$, we get 
	\begin{equation*}
		a_{ii}=\sigma_{\bar{i}}(b_{i}+\beta_{i}(\nu_i))+\omega_1^{J}(\nu_i,\nu_i).
	\end{equation*}	
	Since the $1$-form $\beta_{i}$ is bounded from $X$, by Lemma \ref{lem:Qi}\ref{item:Qi-1-form} we have $\beta_{i}(\nu_i)\in \cQ_{\bar{i}}$, so $\sigma_{\bar{i}}\beta_{i}(\nu_i)\in \sigma_{\bar{i}}\cQ_{\bar{i}}=\cQ_{\bar{i}}$, where the last equality holds by Lemma \ref{lem:Qi}\ref{item:Qi-t}. Similarly, since $\omega_1$ is a $2$-form bounded from $X$, by Lemma \ref{lem:Qi}\ref{item:Qi-2-form}, we have $\omega_1^{J}(\nu_i,\nu_i)\in \cQ_{\bar{i}}\cQ_{\bar{i}}\subseteq \cQ_{\bar{i}}$. Hence
	\begin{equation*}
		a_{ii}=\sigma_{\bar{i}}b_{i}+q
	\end{equation*}
	for some $q\in \cQ_{\bar{i}}$. By Lemma \ref{lem:Qi}\ref{item:Qi-t} we have $\cQ_{\bar{i}}=\sigma_{\bar{i}}t_{\bar{i}} \cQ_{\bar{i}}$, so we can write $q=\sigma_{\bar{i}}t_{\bar{i}} q'$ for some $q'\in \cQ_{\bar{i}}\subseteq \cB$. This way, $a_{ii}=\sigma_{\bar{i}}(b_{i}+t_{\bar{i}}q')$, with bounded $q'$. Since we have chosen $U_X$ so that the function $b_i$ from Lemma \ref{lem:omegaAC-sym} satisfies $b_{i}>\tfrac{1}{2}$, we have $\check{b}_i\de b_{i}+t_{\bar{i}}q'>\frac{1}{4}$ whenever $t_{\bar{i}}$ is small enough. In particular, we have $\check{b}_{i}>\frac{1}{4}$ sufficiently close to the point $x_0$, as needed.
	
	\ref{item:a_ij} By definition of the matrix \eqref{eq:omega-matrix}, we have $a_{ij}=\tilde{\omega}_E^{J}(\nu_i,\nu_j)$.
	As before, using definition \eqref{eq:coordinates-sx} of $\nu_i,\nu_j$ and the formula \eqref{eq:omegaAC-sym} for $\tilde{\omega}_E^J$, we get that for $i\neq j$:
	\begin{equation*}
		a_{ij}=\sigma_{\bar{i}}\beta_{i}(\nu_j)+\sigma_{\bar{j}}\beta_{j}(\nu_i)+\omega_1^J(\nu_i,\nu_j).
	\end{equation*}
	We need to show that each of the above summands is in $\sigma_{\bar{i}}\sigma_{\bar{j}}\cdot \cB$. Since  $\beta_{i}$ is bounded from $X$, by Lemma \ref{lem:Qi}\ref{item:Qi-1-form} we have $\beta_{i}(\nu_j)\in \cQ_{\bar{j}}$. By Lemma \ref{lem:Qi}\ref{item:Qi-t}, we have $\cQ_{\bar{j}}=\sigma_{\bar{j}}\cQ_{\bar{j}}\subseteq \sigma_{\bar{j}}\cB$, so $\sigma_{\bar{i}}\beta_{i}(\nu_j)\in \sigma_{\bar{i}}\sigma_{\bar{j}}\cB$, as needed. For the second summand, we apply the same argument interchanging the roles of $i$ and $j$. Eventually, since the $2$-form $\omega_1$ is bounded from $X$, by Lemma \ref{lem:Qi}\ref{item:Qi-2-form} we have $\omega_1^{J}(\nu_i,\nu_j)\in \cQ_{\bar{i}}\cQ_{\bar{j}}=\sigma_{\bar{i}}\sigma_{\bar{j}}\cQ_{\bar{i}}\cQ_{\bar{j}}\subseteq \sigma_{\bar{i}}\sigma_{\bar{j}}\cB$, where the equality follows from Lemma \ref{lem:Qi}\ref{item:Qi-t}.

	\ref{item:q_ij} The entry $q_{ij}$ is defined as $\tilde{\omega}_E^J(\xi_i,\nu_j)$. Now, substituting \eqref{eq:coordinates-sx} to \eqref{eq:omegaAC-sym} gives
	\begin{equation*}
		q_{ij}=\sigma_{\bar{j}}\beta_{j}(\xi_i)+\omega_1^{J}(\xi_i,\nu_j).
	\end{equation*}
	By Lemma \ref{lem:Qi}\ref{item:Qi-1-form}, we have $\beta_{j}(\xi_i)\in \cB$, so $\sigma_{\bar{j}} \beta_{j}(\xi_i)\in \sigma_{\bar{j}}\cB$, as needed. Lemma \ref{lem:Qi}\ref{item:Qi-2-form} gives $\omega_1^{J}(\xi_i,\nu_j)\in \cQ_{\bar{j}}$. By Lemma \ref{lem:Qi}\ref{item:Qi-t}, we have $\cQ_{\bar{j}}=\sigma_{\bar{j}}\cQ_{\bar{j}}\subseteq \sigma_{\bar{j}}\cB$, so  $\omega_1^{J}(\xi_i,\nu_j)\in \sigma_{\bar{j}}\cB$, as claimed.
	
	\ref{item:P} Formula \eqref{eq:omegaAC-sym} shows that the entries of $P$ are $p_{ij}=\omega_{1}^{J}(\xi_{i},\xi_{j})$ for $i,j\in \{1,\dots, 2(n-k)\}$. By Lemma \ref{lem:omegaAC-sym}, the form $\omega_{1}$ is bounded from $X$, so these entries are bounded functions.
	
	\ref{item:vectorsmall} Fix a basis $\beta_1,\dots,\beta_{2n}$ of $T^{*}U_X$. We need to show that for all $j\in \{1,\dots,2n\}$ and $i\in \{1,\dots,2k\}$, the functions $\sigma_{i}^{-1/2}\beta_{j}(\pi_{*}\nu_i)$ and $\beta_{j}(\pi_{*}\nu_i)$ are bounded and converge to zero as $r_{\bar{i}}\rightarrow 0$. Since each $1$-form $\pi^{*}\beta_{j}$ is bounded from $X$, by Lemma \ref{lem:Qi}\ref{item:Qi-1-form} we have $\beta_{j}(\pi_{*}\nu_i)=\pi^{*}\beta_{j}(\nu_i)\in \cQ_{\bar{i}}$. By Lemma \ref{lem:Qi}\ref{item:Qi-t}, we have $\cQ_{\bar{i}}\subseteq \sigma_{\bar{i}}\cB$, so $\beta_{j}(\pi_{*}\nu_i)=\sigma_{\bar{i}}b_{ji}$ for some bounded function $b_{ji}$. In particular, the functions $\sigma_{\bar{i}}^{-1/2}\cdot \beta_{j}(\pi_{*}\nu_{i})$ and $\beta_{j}(\pi_{*}\nu_{i})$ are bounded. Moreover, since $\sigma_{\bar{i}}\rightarrow 0$ as $r_{\bar{i}}\rightarrow 0$, it follows that $\sigma_{\bar{i}}^{-1/2} \beta_{j}(\pi_{*}\nu_i)=\sigma_{\bar{i}}^{1/2}b_{ji}\rightarrow 0$ and $\beta_{j}(\pi_{*}\nu_i)=\sigma_{\bar{i}}b_{ji}\rightarrow 0$ as $r_{\bar{i}}\rightarrow 0$, as claimed.
	
%	For $i>k$, we have $J(\pi_{*}\nu_i)=-\pi_{*}\nu_{i-k}$ by Lemma \ref{lem:nu-j}\ref{item:ds-J}, so it is enough to consider $i\leq k$. By definition \eqref{eq:rho-sigma}, the function $\sigma_{i}$ is bounded, so it is enough to prove that $\sigma_{i}^{-1/2}\cdot \pi_{*}\nu_i\rightarrow 0$ as $r_i\rightarrow 0$.
%	
%	By definition \eqref{eq:coordinates-sx} of $\nu_i$, we have $\pi_{*}\nu_{i}=r_{i}\frac{\d}{\d r_{i}}$. By definitions \eqref{eq:rho-sigma} and \eqref{eq:def-r_i-t_i} of $\sigma_{i}$ and $t_{i}$, we have $\sigma_{i}^{1/2}=(1+t_{i}^{-1}u_{i}^{2})^{1/2}\cdot t_{i}=b\cdot (-m_i\log r_{i})^{-1}$, where $b\de (1+t_{i}^{-1}u_{i}^{2})^{1/2}\geq 1$. In particular, the function $b^{-1}$ is bounded. Now   $\sigma_{i}^{-1/2}\cdot  \pi_{*}\nu_{i}=b^{-1}\cdot (-m_i\log r_{i}) \cdot r_{i} \cdot \tfrac{\d}{\d r_{i}}\rightarrow 0$ as $r_{i}\rightarrow 0$; as needed.

	\ref{item:xi-iso} By definition \eqref{eq:coordinates-sx} of $\xi_{i}$, for $\xi\in \Xi$ the norm $||\xi||$ equals the maximum norm of $\pi_{*}\xi$ with respect to the natural coordinates on $\bar{U}_{X}$. The result follows since all norms on $T \bar{U}_{X}$ are equivalent.
\end{proof}

We now apply the Gaussian diagonalization to the first $2k$ rows and columns of \eqref{eq:omega-matrix}. For $j\in \{1,\dots, 2k\}$, denote by  $((\nu_{i}^{j})_{i=1}^{2k}, (\xi_{i}^{j})_{i=1}^{2(n-k)})$ the basis of $T(U\setminus \d A)$ obtained from $((\nu_{i})_{i=1}^{2k}, (\xi_{i})_{i=1}^{2(n-k)})$ after the $j$-th diagonalization step.

\begin{lema}\label{lem:Gauss}
	For every $j\in \{1,\dots, 2k\}$ we can shrink the fine chart $U_X$ around $x_0$ so that the matrix \eqref{eq:omega-matrix} of $\tilde{\omega}_E^{J}$ in basis $((\nu_{i}^j)_{i=1}^{2k}, (\xi_{i}^j)_{i=1}^{2(n-k)})$ keeps the properties listed in Lemma \ref{lem:matrix}.
\end{lema}
\begin{proof}
	We argue by induction on $j\in \{1,\dots, 2k\}$. Assume that we have eliminated the off-diagonal entries from rows and columns up to $j-1$; and that the resulting basis still satisfies Lemma \ref{lem:matrix}. While eliminating the off-diagonal entries from the $j$-th column, we perform the following operations on the columns $l\in \{j+1,\dots, 2k\}$ and $\check{l}\in \{2k+1,\dots, n\}$:
	\begin{equation*}
		\begin{split}	
			& a_{il}\rightsquigarrow a_{il}- a_{jj}^{-1}a_{jl}\cdot a_{ij},\mbox{ for all } i\in \{j+1,\dots,2k\} \\
			& q_{il}\rightsquigarrow q_{il}- a_{jj}^{-1}a_{jl}\cdot q_{ij}, \mbox{ for all } i\in \{1,\dots, 2(n-k)\}. \\
			\mbox{and}\quad 
			&	p_{il'} \rightsquigarrow p_{il'} - a_{jj}^{-1}q_{l'j}\cdot q_{ij}, \mbox{ for all } i\in \{1,\dots, 2(n-k)\}, \mbox{ where } l'=\check{l}-2k.%\in \{1,\dots,2(n-k)\}.}
		\end{split}
	\end{equation*}
	Put $c=a_{jj}^{-1} a_{jl}$, $c'=a_{jj}^{-1}q_{l'j}$. We claim that, after possibly shrinking $U_{X}$, we have  
	\begin{equation}\label{eq:cl-c}
		c\in \sigma_{\bar{l}}\cdot \cB\quad\mbox{and} \quad c'\in \cB. 
	\end{equation}
	Indeed, by property \ref{lem:matrix}\ref{item:a_ii} we can shrink $U_X$ so that $a_{jj}=\sigma_{\bar{j}}\check{b}_{j}$ for a function $\check{b}_{j}$ bounded from below by a positive constant. This way, $\check{b}_j^{-1}\in \cB$, so $a_{jj}^{-1}\in \sigma_{\bar{j}}^{-1}\cB$. Since we assume $l>j$, by property \ref{lem:matrix}\ref{item:a_ij} we have $a_{jl}\in \sigma_{\bar{j}}\sigma_{\bar{l}}\cB$, so $c=a_{jj}^{-1} \cdot a_{jl}\in \sigma_{\bar{j}}^{-1}\cdot \sigma_{\bar{j}}\sigma_{\bar{l}}\cdot \cB= \sigma_{\bar{l}}\cdot \cB$.  Similarly, since $q_{l'j}\in \sigma_{\bar{j}}\cdot \cB$ by \ref{lem:matrix}\ref{item:q_ij}, we have $c'=a_{jj}^{-1}q_{l'j}\in \sigma_{\bar{j}}^{-1}\sigma_{\bar{j}}\cB=\cB$. This proves \eqref{eq:cl-c}.
	
	Recall that $\sigma_j$ was defined in \eqref{eq:rho-sigma} as $\sigma_{j}=t_{j}^2+t_ju_j^2$, so  $\sigma_{j}\in \cB$. Now for $i>j$, property \ref{lem:matrix}\ref{item:a_ij} gives $a_{ij}\in \sigma_{\bar{i}}\sigma_{\bar{j}}\cB\subseteq \sigma_{\bar{i}}\cB$, so the formula \eqref{eq:cl-c} implies that
	\begin{equation}\label{eq:claim-bcd}
		ca_{ij}\in \sigma_{\bar{i}}\sigma_{\bar{l}}\cdot \cB,
	\end{equation}
	so the property \ref{lem:matrix}\ref{item:a_ij} is preserved. 
	
	Similarly, property \ref{lem:matrix}\ref{item:q_ij} shows that $q_{ij}\in \sigma_{\bar{j}}\cB\subseteq \cB$, so formula \eqref{eq:cl-c} gives $cq_{ij}\in \sigma_{\bar{l}}\cB$ and $c'q_{ij}\in \cB$. It follows that the properties of $q_{il}$ and $p_{il'}$ listed in Lemma \ref{lem:matrix}\ref{item:q_ij},\ref{item:P} are preserved.
	\smallskip

	Thus we have shown that the off-diagonal entries in the first $2k$ columns keep their properties. For the diagonal ones, we need to show that \ref{lem:matrix}\ref{item:a_ii} is preserved. To do this, assume $i=l$. By \eqref{eq:claim-bcd}, we have $ca_{ij}\in \sigma_{\bar{i}}\sigma_{\bar{l}}\cdot \cB=\sigma_{\bar{i}}^2\cdot \cB$. Recall that by definition \eqref{eq:rho-sigma} we have  $\sigma_{\bar{i}}=t_{\bar{i}}^2+t_{\bar{i}}u_{\bar{i}}^2\in t_{\bar{i}}\cdot \cB$, so $ca_{ij}\in \sigma_{\bar{i}}t_{\bar{i}}\cdot\cB$, i.e.\ $ca_{ij}=\sigma_{\bar{i}}t_{\bar{i}} b$ for some $b\in \cB$. Thus the new $(i,i)$-th entry of $A$ is
	\begin{equation*}
		a_{ii}-ca_{ij}=\sigma_{\bar{i}}\check{b}_{\bar{i}}-\sigma_{\bar{i}}t_{\bar{i}}b=\sigma_{\bar{i}}(\check{b}_i+t_{\bar{i}}b).
	\end{equation*}
	As we approach $x_0$, we have $t_{\bar{i}}\rightarrow 0$, so $t_{\bar{i}}b\rightarrow 0$ because $b\in \cB$. Thus if $\check{b}_{i}>\frac{1}{4}$ then $\check{b}+t_{\bar{i}}b>\frac{1}{4}$, too, for $t_{\bar{i}}$ small enough. Therefore, shrinking $U_X$  we can ensure that the diagonal entries of $A$ keep their property \ref{lem:matrix}\ref{item:a_ii}.
	\smallskip
	
	To see property~\ref{lem:matrix}\ref{item:vectorsmall}, we recall the Gaussian elimination formula
	\begin{equation*}
		\nu_{i}^{j}=\nu_{i}^{j-1}-a_{ij}a_{jj}^{-1}\cdot \nu_{j}^{j-1}\quad \mbox{for } i>j.
	\end{equation*}
	We need to prove that $\sigma_{\bar{i}}^{-1/2}\pi_{*}\nu_{i}^{j}\rightarrow 0$ as $r_{\bar{i}}\rightarrow 0$. The convergence $\pi_{*}\nu_{i}^{j}\rightarrow 0$ will follow since $\sigma_{\bar{i}}\in\cB$.
	
	By the inductive assumption, $\sigma_{\bar{i}}^{-1/2}\pi_{*}\nu_{i}^{j-1}\rightarrow 0$ as $r_{\bar{i}}\rightarrow 0$. We now study the second summand.  We have $a_{ij}\in \sigma_{\bar{i}}\sigma_{\bar{j}}\cdot \cB$ by property \ref{lem:matrix}\ref{item:a_ij} and  $a_{jj}^{-1}\in \sigma_{\bar{j}}^{-1}\cdot \cB$ by property \ref{lem:matrix}\ref{item:a_ii}, so $\sigma_{\bar{i}}^{-1/2}\cdot a_{ij}a_{jj}^{-1}\in \sigma_{\bar{i}}^{-1/2}\cdot \sigma_{\bar{i}}\sigma_{\bar{j}}\cdot \sigma_{\bar{j}}^{-1}\cdot \cB=\sigma_{\bar{i}}^{1/2}\cdot \cB$. Thus $\sigma_{\bar{i}}^{-1/2}a_{ij}a_{jj}^{-1} =\sigma^{1/2}_{\bar{i}}\cdot b$ for some $b\in \cB$. By the inductive assumption, the vector field $\pi_{*}\nu_{j}^{j-1}$ is bounded, hence so is $b\cdot \pi_{*}\nu_{j}^{j-1}$. By definition \eqref{eq:rho-sigma} of $\sigma_{\bar{i}}$, we have $\sigma_{\bar{i}}\rightarrow 0$ as $r_{\bar{i}}\rightarrow 0$. Thus $\sigma_{\bar{i}}^{-1/2}\cdot a_{ij}a_{jj}^{-1}\cdot \pi_{*}\nu_{j}^{j-1}=\sigma_{\bar{i}}^{1/2}\cdot b \pi_{*}\nu_{j}^{j-1}\rightarrow 0$ as $r_{\bar{i}}\rightarrow 0$, as needed. 
	\smallskip
	
	It remains to show that the property \ref{lem:matrix}\ref{item:xi-iso} is preserved. Let $\Xi^{j}$ be the subbundle of $T(U\setminus \d A)$ spanned by $\{\xi^{j}_{1},\dots,\xi_{2(n-k)}^{j}\}$, and let $||\sdot ||^{(j)}$ be the corresponding maximum norm on $\Xi^j$.  We need to find a constant $K_{0}'>0$ such that for every $\xi\in \Xi^{j}$ we have $||\pi_{*}\xi||_{X}\geq K_{0}' ||\xi||^{(j)}$. Clearly, it is enough to ensure this inequality for $\xi\in \Xi^{j}$ such that $||\xi||^{(j)}=1$.
	
	Take such $\xi\in \Xi^{j}$ and write it as  $\xi=\sum_{i=1}^{2(n-k)} \lambda_{i}\xi_{i}^{j}$ for some smooth functions $\lambda_{i}$. Then $\max_{i}||\lambda_{i}||=||\xi||^{(j)}=1$. Put $\xi'=\sum_{i}\lambda_{i}\xi_{i}^{j-1}\in \Xi^{j-1}$. We have $||\xi'||^{(j-1)}=1$, so by the inductive assumption, $||\pi_{*}\xi'||_{X}\geq K_0$. Moreover, $||\pi_{*}\xi'-\pi_{*}\xi||_{X}\leq \sum_{i} ||\pi_{*}(\xi^{j-1}_{i}-\xi^{j}_{i})||_{X}$.
	
	The Gaussian elimination formula reads as $\xi^{j}_{i}=\xi^{j-1}_{i}-q_{ij}a_{jj}^{-1}\cdot \nu_{j}^{j-1}$, so $\pi_{*}(\xi^{j-1}_{i}-\xi^{j}_{i})=q_{ij}a_{jj}^{-1}\pi_{*}\nu_{j}^{j-1}$. Property \ref{lem:matrix}\ref{item:vectorsmall} implies that $\pi_{*}\nu_{j}^{j-1}\rightarrow 0$ as $r_{\bar{j}}\rightarrow 0$. Like before, combining properties \ref{lem:matrix}\ref{item:q_ij} and \ref{item:a_ii} we infer that $q_{ij}\cdot a_{jj}^{-1}\in \sigma_{\bar{j}}\cdot \sigma_{\bar{j}}^{-1}\cdot \cB=\cB$. Therefore, $\pi_{*}(\xi^{j-1}_{i}-\xi^{j}_{i})\rightarrow 0$ as $r_{\bar{j}}\rightarrow 0$. Since $r_{\bar{j}}(x_0)=0$ for all $j\in \{1,\dots, 2k\}$, we can shrink the neighborhood $U_X$ of $x_0$ so that $||\pi_{*}(\xi^{j-1}_{i}-\xi^{j}_{i})||_{X}\leq \tfrac{1}{2(n-k)}K_{0}$. Then $||\pi_{*}\xi'-\pi_{*}\xi||_{X}\leq \tfrac{1}{2}K_0$, and as a consequence, $||\pi_{*}\xi||_{X}\geq ||\pi_{*}\xi'||_X-||\pi_{*}\xi'-\pi_{*}\xi||_{X}\geq \tfrac{1}{2}K_0$. Thus property \ref{lem:matrix}\ref{item:xi-iso} holds with the constant $K_0>0$ replaced by $\tfrac{1}{2}K_0>0$.
	%
	%It follows that the vector space $\Xi^{j+1}_{y}$ converges to $\Xi^{j}_{y}$ as $y\rightarrow x_0$. Thus the property \ref{lem:matrix}\ref{item:xi-iso} is preserved after shrinking the neighborhood $U_{X}$ and the constant $K_0>0$.
	%	
	%	 on the size of $\nu^j_i$ follows from the Gaussian elimination formula $\nu^{j+1}_i=\nu^{j}_i-a_{ij}a_{jj}^{-1}\nu^{j}_i$ for $i>j$, and the facts that $a_{jj}\geq \frac{1}{4}\sigma_{\bar j}$ by \ref{lem:matrix}\ref{item:a_ii} and $a_{ij}\in \sigma_{\bar{i}}\sigma_{\bar j}\cdot \cB\subseteq \sigma_{\bar{j}}\cdot \cB$ by \ref{lem:matrix}\ref{item:a_ij}.
\end{proof}

\begin{proof}[Proof of Proposition \ref{prop:positivity}]
	Consider the matrix \eqref{eq:omega-matrix} of $\tilde{\omega}_E^{J}$ after the first $2k$ steps of the Gaussian diagonalization, that is, in the basis $((\nu_{i}^{2k})_{i=1}^{2k},(\xi_{i}^{2k})_{i=1}^{2(n-k)})$. Then its off-diagonal block $Q$ is zero, and the top-left block $A$ is diagonal.  By Lemma \ref{lem:Gauss}, we can shrink $U_X$ so that the properties \ref{lem:matrix}\ref{item:a_ii}--\ref{item:xi-iso} are preserved.  Furthermore, we change the basis by replacing each $\nu_{i}^{2k}$ with $\tilde{\nu}_{i}^{2k}\de \sigma_{\bar{i}}^{-1/2} \cdot \nu_{i}^{2k}$. 
	
	Now, property \ref{lem:matrix}\ref{item:a_ii} asserts that the diagonal terms of $A$ are bounded from below by $\tfrac{1}{4}$, and by \ref{lem:matrix}\ref{item:vectorsmall} the corresponding basis vector fields $(\tilde{\nu}_{i}^{2k})_{i=1}^{2k}$ approach zero as we approach $x_0$.
	
	%Fix $y\in U\setminus \d A$. 
	Any vector field $v\in T(U\setminus\partial A)$ splits as $v=\nu +\xi $, where $\nu $ and $\xi $ are linear combinations of $\{\tilde{\nu}_i^{2k}\}_{i=1}^{2k}$ and $\{\xi_i^{2k}\}_{i=1}^{2(n-k)}$, respectively. 	 Since the block form \eqref{eq:omega-matrix} of $\tilde{\omega}_E^J$ has the off-diagonal block $Q$ equal to zero, we have  the equality
	\begin{equation}\label{eq:omega-splitting}
		\tilde{\omega}_E^J(v,v)=\tilde{\omega}_E^J(\nu ,\nu )+\tilde{\omega}_E^J(\xi ,\xi ). 
	\end{equation}
	In the tangent bundle $T(U\setminus\partial A)$, we consider the maximum norm $||\sdot||$ with respect to the coefficients of the basis $((\tilde{\nu}_i^{2k})_{i=1}^{2k},(\xi_i^{2k})_{i=1}^{2(n-k)})$. The diagonal entries of $A$ are equal to $\tilde{\omega}_{E}^{J}(\tilde{\nu}_{i}^{2k},\tilde{\nu}_{i}^{2k})$, so by property \ref{lem:matrix}\ref{item:a_ii}, they are bounded from below by $\tfrac{1}{4}$. It follows that %	Property \ref{lem:matrix}\ref{item:a_ii}, which is kept by the diagonalization, implies that 
	\begin{equation}\label{eq:constant-nu}
		\tilde{\omega}_E^J(\nu ,\nu )\geq\tfrac{1}{4} ||\nu ||^2.
	\end{equation}
	Property \ref{lem:matrix}\ref{item:P} asserts that the functions $\tilde{\omega}_{E}^{J}(\xi_{i}^{2k},\xi_{j}^{2k})$, which are entries of the matrix $P$, are bounded. Thus there is a constant $K>0$ such that, uniformly on $U\setminus \d A$, we have the inequality
	\begin{equation}\label{eq:constant-xi}
		|\tilde{\omega}_E^J(\xi ,\xi )|\leq K ||\xi ||^2. 
	\end{equation}
	
	%	Now, we fix $\epsilon>0$. 
	Recall that the form $\omega\AC^{\epsilon}$ was defined in  \eqref{eq:omegaAC} as 	$\pi^{*}\omega_{X}+\epsilon\omega_{E}$. Hence by Lemma \ref{lem:omegaAC-sym}, for every $\tau>0$ we have the equality of restrictions
	\begin{equation*}
		(\omega\AC^{\epsilon})^{J}|_{g^{-1}(\tau)}=(\tilde{\omega}^{\epsilon}\AC)^{J}|_{g^{-1}(\tau)},
		\quad\mbox{where}\quad
		(\tilde{\omega}\AC^{\epsilon})^{J}= \pi^{*}\omega_{X}^{J}+\epsilon\tilde{\omega}_E^{J}.
	\end{equation*}
	We need to show that there is $\epsilon_0>0$ and a neighborhood $V_X$ of $x_0$ such that whenever $\epsilon\in [0,\epsilon_0]$, %and $y\in \pi^{-1}(V_X)$, 
	we have the inequality $(\tilde{\omega}\AC^{\epsilon})^{J}(v,v)>0$ for every nonzero $v\in T(\pi^{-1}(V_{X})\setminus \d A)$.
	
	We split $v=\nu +\xi $ as above. Consider the case $||\nu ||> 2K^{\frac{1}{2}}||\xi ||$. Since the form $\omega_{X}^{J}$ is positive definite, we have $(\pi^{*}\omega_{X})^{J}(v,v)\geq 0$, so  $(\tilde{\omega}\AC^{\epsilon})^{J}(v,v)\geq \epsilon\tilde{\omega}_{E}^{J}(v,v)$. Substituting the inequalities \eqref{eq:constant-nu} and \eqref{eq:constant-xi} to \eqref{eq:omega-splitting}, we get $\tilde{\omega}_E^J(v,v)\geq \frac{1}{4}||\nu ||^2-K||\xi ||^2>0$ by assumption. Thus  $(\tilde{\omega}\AC^{\epsilon})^{J}(v,v)>0$, as needed.
	
	Hence we can assume 
	\begin{equation}
		\label{eq:assumption}
		||\nu ||\leq 2K^{\frac{1}{2}}  ||\xi ||.
	\end{equation}
	%	Note that since we assume $v\neq 0$, the assumption \eqref{eq:assumption} implies that $\xi \neq 0$. 
	%	
	Using \eqref{eq:omega-splitting}, we expand the equality $(\tilde{\omega}\AC^{\epsilon})^{J}(v,v)=(\pi^*\omega_{X})^J(v,v)+\epsilon\tilde{\omega}_{E}^{J}(v,v)$ as 
	\begin{equation}
		\label{eq:quadraticexpansion}
		(\tilde{\omega}\AC^{\epsilon})^J(v,v)=(\pi^*\omega_X)^J(\nu ,\nu )+2(\pi^*\omega_X)^J(\nu ,\xi )+(\pi^*\omega_X)^J(\xi ,\xi )+\epsilon\tilde{\omega}_E^J(\nu ,\nu )+\epsilon\tilde{\omega}_E^J(\xi ,\xi ).
	\end{equation}
	The inequality \eqref{eq:constant-nu} implies that $\tilde{\omega}_{E}^{J}(\nu ,\nu )\geq 0$, so since $(\pi^*\omega_X)^J$ is positive definite, for $v\neq 0$ we have
	\begin{equation*}
		(\pi^*\omega_X)^J(\nu ,\nu )+\tfrac{1}{3}(\pi^*\omega_X)^J(\xi ,\xi )+\epsilon\tilde{\omega}_E^J(\nu ,\nu )>0.
	\end{equation*} 
	We group the remaining summands of the expansion~\eqref{eq:quadraticexpansion} as follows:
	\begin{equation*}
		\begin{split}
			B&\de \tfrac{1}{3}(\pi^*\omega_X)^J(\xi ,\xi )+2(\pi^*\omega_X)^J(\nu ,\xi ),\\
			C &\de \tfrac{1}{3}(\pi^*\omega_X)^J(\xi ,\xi )+\epsilon\tilde{\omega}_E^J(\xi ,\xi ).
		\end{split}
	\end{equation*}
	
	Let $||\sdot ||_X$ be a norm in the tangent bundle $T \bar{U}_X$. By property \ref{lem:matrix}\ref{item:xi-iso}, there is a constant $K_0>0$ such that $||\pi_*\xi ||_X\geq K_0||\xi ||$. By positive definiteness of $\omega_X^J$ there exists a constant $K_1'>0$ such that $\omega_X^J(\eta ,\eta )\geq K_{1}'||\eta ||_{X}^2$ for every $\eta$ tangent to $\bar{U}_{X}$. Putting $K_{1}=K_{1}'K_{0}^2$ we get
	\begin{equation}\label{eq:C-bound}
		(\pi^*\omega_X)^J(\xi ,\xi )\geq K_1||\xi ||^2.
	\end{equation}
	There are also constants $K_{2}',K_{2}''>0$ such that $|(\pi^{*}\omega_{X})^{J}(\nu ,\xi )|\leq  K_{2}'\cdot ||\pi_{*}\nu ||_{X}\cdot ||\pi_{*}\xi ||_{X}$ and $||\pi_{*}\xi ||_{X}\leq K_{2}'' ||\xi ||$, so putting $K_{2}=K_{2}'K_{2}''$ we get $|(\pi^*\omega_X)^J(\nu ,\xi )|\leq K_2||\xi ||\cdot||\pi_*\nu ||_X$.  This inequality together with \eqref{eq:C-bound} implies that
	\begin{equation}\label{eq:B-bound}
		B\geq \tfrac{1}{3}K_{1}||\xi ||^{2}-2K_{2} ||\xi ||\cdot||\pi_*\nu ||_X.
	\end{equation}
	
	Recall that $\tilde{\nu}_{i}^{2k}=\sigma_{i}^{-1/2}\nu_{i}^{2k}$, so by  property~\ref{lem:matrix}\ref{item:vectorsmall}, we have $\tilde{\nu}_{i}^{2k}\rightarrow 0$ as $r_{\bar{i}}\rightarrow 0$. Hence there is a continuous function $\kappa\colon [0,\infty)^k\to[0,\infty)$ such that $\kappa(0)=0$ and $||\pi_*\nu ||_X\leq \kappa(r_1,\dots,r_k)||\nu ||$. Together with the assumption \eqref{eq:assumption}, this implies that $||\pi_{*}\nu ||_{X}\leq 2K^{\frac{1}{2}} \kappa(r_1,\dots,r_k) ||\xi ||$. Substituting this inequality to \eqref{eq:B-bound}, we get
	\begin{equation*}
		%\begin{split}
		B%&
		%\overset{\mbox{\tiny{\eqref{eq:B-bound}}}}{\geq}
		%\tfrac{1}{3}K_1||\xi ||^2-2K_2||\xi ||\cdot ||\pi_*\nu ||_X=
		%||\xi ||\cdot (\tfrac{1}{3}K_1||\xi ||-2K_2||\pi_*\nu ||_X)\geq \\
		%&
		%\geq||\xi ||\cdot (\tfrac{1}{3}K_1||\xi ||-2K_2\kappa(r_1,\dots,r_k)\cdot ||\nu ||)
		%\overset{\mbox{\tiny{\eqref{eq:assumption}}}}{\geq} 
		%%%%
		\geq 
		%\tfrac{1}{3}K_{1}||\xi ||^{2}-2K_{2} ||\xi ||\cdot 2K^{\frac{1}{2}} \kappa(r_1,\dots,r_k) ||\xi ||
		%= 
		||\xi ||^2(\tfrac{1}{3}K_1-4K^{\frac{1}{2}}K_2\cdot \kappa(r_1,\dots,r_k)).
		%\end{split}
	\end{equation*} 
	Since $r_{1}, \dots,r_k$ vanish at $x_0$, we can shrink the neighborhood $U_X$ of $x_0$ so that $\kappa(r_1,\dots,r_k)<\frac{\frac{1}{3}K_1}{4K^{\frac{1}{2}}K_2}$. After this shrinking of $U_X$, we get $B\geq 0$. Notice that this shrinking is independent of $\epsilon$. 
	\smallskip
	
	The inequalities \eqref{eq:C-bound} and \eqref{eq:constant-xi} imply that
	\begin{equation*}
		C\geq %\tfrac{1}{3}K_{1}||\xi ||^{2}-\epsilon K ||\xi ||^{2}=
		(\tfrac{1}{3}K_{1}-\epsilon K)\cdot ||\xi ||^{2}.
	\end{equation*}
	Choose $\epsilon_0\de \frac{1}{3}\frac{K_1}{K}$. Then for every $\epsilon\in [0,\epsilon_0]$ we have $C\geq 0$, so $(\tilde{\omega}\AC^{\epsilon})^J(v,v)>0$, as needed.
\end{proof}

\subsection{Non-degeneracy on \texorpdfstring{$\d A$}{dA}}\label{sec:lifts}

In Proposition \ref{prop:positivity} we have shown that the forms $\omega\AC^{\epsilon}$ and $\omega\AC^{\delta,\epsilon}$ are fiberwise nondegenerate near $\d A$. Proposition \ref{prop:lifts} below, which we will prove in this section, shows that the same holds on $\d A$. Recall from Proposition \ref{prop:AXsmooth}\ref{item:AX-gsmooth} that $\theta\colon A\to \S^1$  is a submersion extending $f/|f|\colon X\setminus D\to \S^1$ to $\d A$.

\begin{prop}%[Monodromy vector field]
	\label{prop:lifts}
	%Fix a K\"ahler form $\omega_{X}\in \Omega^{2}(X)$.  
	Fix $x_0\in D$. There is an $\epsilon_0>0$, and a fine chart $U_X$ containing $x_0$, such that for every $\epsilon\in (0,\epsilon_0]$ and any $\delta>0$, the restriction of each of the forms $\omega\AC^{\epsilon}$ and $\omega\AC^{\delta,\epsilon}$  to $U\cap \d A$ is nondegenerate on every fiber of $\theta|_{U\cap \d A}$, where $U=\pi^{-1}(U_X)$. Moreover, the symplectic lift to $U\cap \d A$ of the unit angular vector field on $\S^1$ %=\{0\}\times \S^1$
	is independent of the choice of the form $\omega\AC^{\epsilon}$ and $\omega\AC^{\delta,\epsilon}$ defining the symplectic connection, and is given by the formula \eqref{eq:monodromy-vector-field}.
\end{prop}

We notice that at $\partial A$ the forms $\omega\AC^{\epsilon}$ and $\omega\AC^{\delta,\epsilon}$ coincide, so it is enough to deal with the former one.

As in Section \ref{sec:positivity}, we fix $x_{0}\in X_{S}^{\circ}$ for some $S\neq \emptyset$ and order the components of $D$ so that $S=\{1,\dots, k\}$. We fix a fine chart $U_{X}$ around $x_0$ with index set $S$; and put $U=\pi^{-1}(U_X)\subseteq A$.

Let us recall the definition of the smooth structure on $U$. In \eqref{eq:Ui} we have defined an open covering $U=\bigcup_{i=1}^{k}U_i$ by $U_i=\{w_{i}>\frac{1}{n+1}\}$. The smooth coordinate chart \eqref{eq:AC-chart} on $U_1$ is
\begin{equation*}
	(g,\bar{v}_2,\dots,\bar{v}_k,\theta_1,\dots,\theta_k,z_{i_{k+1}},\dots, z_{i_n})\colon U_1\to [0,1)\times \R^{k-1}\times (\S^1)^{k}\times \C^{n-k},
\end{equation*}
where $\{i_{k+1},\dots, i_{n}\}$ is some subset of $\Z_{>N}$. The charts on $U_2,\dots, U_k$ are defined analogously. The vector fields 
	$
	\tfrac{\d}{\d\theta_{1}},\dots,\tfrac{\d}{\d \theta_{k}} 
	$
	on $U_1$ used in formula \eqref{eq:monodromy-vector-field} are defined as the ones dual to $d\theta_{1},\dots,d\theta_{k}$ with respect to the basis of $T^{*}U_1$ given by the above chart \eqref{eq:AC-chart}.
\smallskip

	In the proof of Proposition \ref{prop:lifts}, we will treat separately the subsets $U_{1}\cap A_{I}^{\circ}$ lying over the strata $U_{X}\cap X_{I}^{\circ}$, see \eqref{eq:stratification-pullback}. Therefore, it will be convenient to use slightly different charts over each stratum, which we introduce in formula \eqref{eq:coordinates-on-stratum} below. For the deepest stratum $X_{S}^{\circ}$, formula \eqref{eq:coordinates-on-stratum} agrees with \eqref{eq:AC-chart}.
	
	\begin{lema}\label{lem:coordinates-on-stratum}
		 Fix a subset $I\subseteq S$ containing $1$. Put $U_{I}^{\circ}\de U_1\cap A_{I}^{\circ}$, and reorder the components of $D$ so that $I=\{1,\dots, l\}$. Then, after possibly shrinking $U_X$, the following hold.
		\begin{enumerate}
			\item\label{item:coordinates-on-stratum} %Let $x_{2j-1}$, $x_{2j}$ be the real and imaginary part of $z_{i_j}$ for $j\in \{1,\dots,n-k\}$. Then t
			For $i\in \{l+1,\dots, k\}$, put $s_{i}=\log r_{i}$, see \eqref{eq:coordinates-sx}. Then the map 
		\begin{equation}\label{eq:coordinates-on-stratum}
			(\bar{v}_2,\dots,\bar{v}_{l},s_{l+1},\dots,s_{k},\theta_{1},\dots,\theta_k,z_{i_{k+1}},\dots,z_{i_{n}})\colon U_{I}^{\circ}\to \R^{k}\times (\S^1)^{k}\times \C^{n-k}.
		\end{equation}
		is a smooth coordinate chart on a neighborhood of $U_{I}$ in $\d A$.
			\item\label{item:coordinate-vector-fields-on-stratum} The vector fields $
			\tfrac{\d}{\d\theta_{1}},\dots,\tfrac{\d}{\d \theta_{k}} 
			$ are dual to $d\theta_1,\dots,d\theta_k$ with respect to the basis of $T^{*}U_{I}$ given by coordinates \eqref{eq:coordinates-on-stratum}.
			\item\label{item:splitting-on-stratum} Define vector fields $\tfrac{\d}{\d \bar{v}_2},\dots,\tfrac{\d}{\d \bar{v}_l}$ as dual to $d\bar{v}_2,\dots,d\bar{v}_l$ in the above basis. % with respect to the above basis. 
			We have a splitting
			\begin{equation*}
				T U_{I}^{\circ}=\cV_{I}\oplus \Theta_{I}\oplus \cZ_{I},
				\quad
				\mbox{where }
				\cV_{I}=\sspan\{\tfrac{\d}{\d \bar{v}_{2}},\dots, \tfrac{\d}{\d \bar{v}_l}\},\  \Theta_{I}=\sspan\{\tfrac{\d}{\d \theta_{1}},\dots, \tfrac{\d}{\d \theta_l}\},
			\end{equation*}
%			where
%			$\cV_{I}=\sspan\{\tfrac{\d}{\d \bar{v}_{2}},\dots, \tfrac{\d}{\d \bar{v}_l}\}$; $\Theta_{I}=\sspan\{\tfrac{\d}{\d \theta_{1}},\dots, \tfrac{\d}{\d \theta_l}\}$, 
			and $\cZ_{I}$ is spanned by the vector fields dual to $ds_{l+1},\dots,ds_{k}$, $d\theta_{l+1},\dots, d\theta_{k}$ and to the (real and imaginary parts of) $dz_{i_{k+1}},\dots, dz_{i_{n}}$, with respect to the basis of $T^{*}U_{I}$ given by the chart  \eqref{eq:coordinates-on-stratum}. \\
			%
			%\item\label{item:splitting-on-stratum} 
			We have $\pi_{*}|_{\cV_{I}\oplus \Theta_{I}}=0$, and $\pi_{*}|_{\cZ_{I}}$ maps $\cZ_{I}$ isomorphically onto the tangent bundle to $X_{I}^{\circ}\cap \pi(U_1)$.
		\end{enumerate}
	\end{lema}
	\begin{proof}
		\ref{item:coordinates-on-stratum} On the stratum $U_{X}\cap X_{I}^{\circ}$, we have a smooth coordinate system $(z_{l+1},\dots,z_{k},z_{i_1},\dots,z_{i_{n-k}})$. Since the coordinates $z_{l+1},\dots,z_{k}$ do not vanish on $U_{X}\cap X_{I}^{\circ}$, we can replace each of them by polar coordinates $(r_i,\theta_{i})$, or $(s_i,\theta_i)$. Since the map $\pi$ is smooth, the pullback of each $s_i$ and $\theta_i$ to $U \cap A_{I}^{\circ}$ is smooth, too. Thus the map \eqref{eq:coordinates-on-stratum} is smooth. To see that it defines a system of coordinates, it is enough to see that the block $[\frac{\d s_{i}}{\d \bar{v}_{j}}]_{l<i,j\leq k}$ of its Jacobian matrix is invertible. By definition \eqref{eq:def-r_i-t_i} of $t_{i}$, we have $s_{i}=-(m_it_i)^{-1}$, and by Lemma \ref{lem:Ui-simple}\ref{item:Ui-cap-dA}, on $U_{I}^{\circ}$ we have $t_{i}=v_i$. Thus it is enough to show that the matrix  $[\frac{\d v_{i}}{\d \bar{v}_{j}}]_{l<i,j\leq k}$ is invertible. By Lemma \ref{lem:dvbar}\ref{item:tame_1-bar}, as we approach $x_0$ we have $d(v_{i}-\bar{v}_{i})\rightarrow 0$, so our matrix approaches the identity; in particular it is invertible once $U_X$ is small enough.
	
		\ref{item:coordinate-vector-fields-on-stratum} %We need to show that the vector fields $\tfrac{\d}{\d \theta_1},\dots,\tfrac{\d}{\d \theta_{k}}$, defined above using coordinates \eqref{eq:AC-chart}, are dual to $d\theta_1,\dots,d\theta_k$ with respect to the coordinates \eqref{eq:coordinates-on-stratum}, too. 
		%Denote the maps \eqref{eq:AC-chart} and \eqref{eq:coordinates-on-stratum} by $\varphi$ and $\varphi_{I}$, respectively. 
		The coordinates of $\tfrac{\d}{\d \theta_{i}}$ in our basis of $TU_{I}^{\circ}$ is given by the $(k+i)$-th column of the above Jacobian matrix.  
		%of $\varphi_{I}\circ \varphi^{-1}$. 
		Its entries are $\frac{\d \bar{v}_j}{\d \theta_i}=0$, $\frac{\d s_j}{\d \theta_i}=0$, $\frac{\d \theta_{j}}{\d \theta_{i}}=\delta_{i}^{j}$ and $\frac{\d z_j}{\d \theta_{i}}=0$, as needed.
		
		%\ref{item:splitting-on-stratum} Follows directly from the definition.
		
		\ref{item:splitting-on-stratum} %\sloppy
		 The map $\pi$ collapses $A_{I}^{\circ}$ onto $X_{I}^{\circ}$. In local coordinates \eqref{eq:coordinates-on-stratum} on $U_{I}^{\circ}$ and $(s_{l+1},\dots,s_{k},\allowbreak\theta_{l+1},\dots,\theta_{k},\allowbreak z_{i_1},\dots,z_{i_{n-k}})$ on its image in $X_{I}^{\circ}$, this map is a projection
		 \begin{equation*}
		 	(\bar{v}_2,\dots,\bar{v}_{l},s_{l+1},\dots,s_{k},\theta_{1},\dots,\theta_k,z_{i_1},\dots,z_{i_{n-k}})\mapsto
		 	(s_{l+1},\dots,s_{k},\theta_{l+1},\dots,\theta_{k}, z_{i_1},\dots,z_{i_{n-k}}).		 \end{equation*}
		 Thus $\pi_{*}|_{\cV_{I}\oplus \Theta_{I}}=0$, and $\pi_{*}|_{\cZ_{I}}$ is an isomorphism onto its image, as needed.
	\end{proof}

\begin{lema}\label{lem:TdA}
	On $U_{I}^{\circ}$, the following hold.
	\begin{enumerate}
		\item\label{item:ort-Z} For every $\beta\in \Omega^{*}(U_X \cap X_{I}^{\circ})$ we have $\pi^{*}\beta|_{\mathcal{V}_{I}\oplus \Theta_{I}}=0$. 
		\item\label{item:ort-Z-vj} For every $j>l$ we have $\bar{v}_{j}=\bar{t}_{j}$ for some $\bar{t}_{j}\in \cC^{\infty}(U_X\cap X_{I}^{\circ})$. In particular, $d\bar{v}_j|_{\mathcal{V}_I\oplus \Theta_I}=0$.
		\item\label{item:ort-alpha} For every $i\in \{1,\dots, k\}$, $j\in \{1,\dots, l\}$ we have $\alpha_{i}|_{\mathcal{V}_{I}}=0$ and $\alpha_{i}(\tfrac{\d}{\d \theta_{j}})=\delta_{i}^{j}$.
		\item\label{item:dv1} We have $d\bar{v}_{1} =-\sum_{i=2}^{l}\frac{\zeta(u_i)}{\zeta(u_1)}\, d\bar{v}_{i}$, where $\zeta=(\eta^{-1})'\colon [0,\infty)\to[0,\infty)$ is as in \eqref{eq:monodromy-vector-field}.
		\item\label{item:dvi} For every $i\in \{1,\dots, l\}$ we have $d\bar{v}_{i}|_{\Theta_{I}\oplus \cZ_{I}}=0$. 
		\item\label{item:omega-dtheta} 
		For every $i\in \{1,\dots, l\}$ we have 
		$\omega\AC^{\epsilon}(\sdot, \frac{\d}{\d\theta_{i}})=\epsilon\, d\bar{v}_{i}$.
	\end{enumerate}
\end{lema}
\begin{proof}
	\ref{item:ort-Z} By Lemma \ref{lem:coordinates-on-stratum}\ref{item:splitting-on-stratum}, we have $\pi_{*}|_{\mathcal{V}_{I}\oplus \Theta_{I}}=0$, so $\pi^{*}\beta|_{\mathcal{V}_{I}\oplus \Theta_{I}}=0$. 
	
	\ref{item:ort-Z-vj} Recall that $\bar{v}_{j}=\sum_{p}\tau^{p}v_{j}^{p}$, where each $\tau^{p}$ is a  smooth function on $X$; and $v_{j}^{p}=0$ for those $p\in R$ such that $U_X^p\cap D_j=\emptyset$. Take $p\in R$ such that $U_X^p\cap D_j\neq \emptyset$. Since $j\not\in I$, Lemma \ref{lem:intro}\ref{item:intro-w-zero} and definition \eqref{eq:def-v_i} of $v_{j}^{p}$ imply that $v_{j}^{p}|_{U_{I}^{\circ}\cap U^{p}}=t_{j}^{p}|_{U_{I}^{\circ}\cap U^{p}}$. Again since $j\not \in I$, the restriction $t_{j}^{p}|_{U_{I}^{\circ}\cap U^p}$ is a pullback of a smooth function on $(U_{X}\cap X_{I}^{\circ})\cap U_{X}^{p}$. We conclude that $\sum_{p}\tau^{p}t_{j}^{p}$ is a smooth function on $U_X\cap X_{I}^{\circ}$, whose pullback to $U_{I}^{\circ}$ equals $\bar{v}_{j}$. This proves the first statement of \ref{item:ort-Z-vj}; the second follows from \ref{item:ort-Z}.
	
	\ref{item:ort-alpha} By Lemma \ref{lem:d-theta}\ref{item:alpha-dtheta}, we have $\alpha_{i}=d\theta_{i}+\pi^{*}\beta_{i}$ for some $\beta_{i}\in \Omega^{1}(U_X)$. Hence part  \ref{item:ort-Z} implies that $\alpha_{i}|_{\mathcal{V}_{I}\oplus \Theta_{I}}=d\theta_{i}|_{\mathcal{V}_{I}\oplus \Theta_{I}}$. Now \ref{item:ort-alpha} follows from the definition of the coordinate vector fields $\tfrac{\d}{\d \theta_j}$.
	
	\ref{item:dv1} Note first that the restriction $u_{1}|_{U_1}$ is nonzero by definition \eqref{eq:Ui} of $U_1$, so each summand in \ref{item:dv1} is well defined. By continuity, it is enough to prove \ref{item:dv1} at each point $x\in \Int_{\d A} U_{I}^{\circ}$. 
	
	Fix $i\in \{1,\dots, l\}$. Since $x\in \Int_{\d A} U_{I}^{\circ}$, we have $t_{i}=0$ on some neighborhood of $x$ in $\d A$. There, $\bar{v}_{i}=-\bar{u}_i$ by definition \eqref{eq:def-vbar-ubar-mu} of $\bar{v}_{i}$; so $\bar{v}_i=-u_i$ by Lemma \ref{lem:intro}\ref{item:intro-u=ubar}. The formula \eqref{eq:def-w_i-u_i} gives $u_{i}=\eta(w_i)$, so $w_{i}=\eta^{-1}(u_i)$, and therefore $dw_{i}=\zeta(u_i)du_i$ by definition of $\zeta$. It follows that on $T_{x}U_{I}^{\circ}$ we have $dw_{i}=-\zeta(u_i) d\bar{v}_{i}$.%, for all $i\in \{1,\dots, k\}$. 
	
	In particular, $d\bar{v}_{1}=-\tfrac{1}{\zeta(u_1)}dw_{1}$. By Lemma \ref{lem:intro}\ref{item:intro-sum} we have $w_{1}=-\sum_{i=2}^{l}w_{i}$, so
	\begin{equation*}
		d\bar{v}_{1}=-\frac{1}{\zeta(u_1)}dw_{1}=\frac{1}{\zeta(u_1)}\sum_{i=2}^{l}dw_{i}=-\sum_{i=2}^{l}\frac{\zeta(u_i)}{\zeta(u_1)}d\bar{v}_{i},
	\end{equation*}
	as needed. %The formula extends to the whole $A_{S}^{\circ}$ by continuity.
	
	\ref{item:dvi} For $i\in \{2,\dots, l\}$ we have $d\bar{v}_{i}|_{\Theta_{I}\oplus \cZ_{I}}=0$ by definition of $\Theta_{I}\oplus\cZ_{I}$, see Lemma \ref{lem:coordinates-on-stratum}\ref{item:splitting-on-stratum}. The remaining equality $d\bar{v}_1|_{\Theta_{I}\oplus \cZ_{I}}=0$ follows from \ref{item:dv1}.
	
	\ref{item:omega-dtheta} By \ref{item:ort-Z}, we have $\pi^{*}\omega_{X}(\sdot,\tfrac{\d}{\d \theta_i})=0$, so $\omega\AC^{\epsilon}(\sdot,\tfrac{\d}{\d \theta_i})=\epsilon \omega_{E}(\sdot,\tfrac{\d}{\d \theta_i})$. By definition \eqref{eq:omegaE} of $\omega_{E}$, we have $\omega_{E}|_{\partial A}=\sum_{j=1}^{N}d\bar{v}_{j}\wedge \alpha_{j}+\sum_{j=1}^{N}\bar{v}_{j}d\alpha_{j}$. By Lemma \ref{lem:d-theta}, the $1$-forms $\alpha_j$ for $j>l$, and the $2$-forms $d\alpha_{j}$ for all $j$ are pullbacks of smooth forms from $U_X$. Thus parts \ref{item:ort-Z} and \ref{item:ort-Z-vj} imply that $(d\bar{v}_{j}\wedge \alpha_{j})(\sdot, \tfrac{\d}{\d \theta_{i}})=0$ for $j>l$ and $d\alpha_{j} (\sdot,\tfrac{\d}{\d \theta_i})=0$ for all $j$. We conclude that 
	\begin{equation*}
		\omega\AC^{\epsilon}\left(\sdot,\frac{\d}{\d \theta_i}\right)=
		\epsilon \sum_{j=1}^{l}\alpha_{j}\left(\frac{\d}{\d\theta_i}\right)\cdot d\bar{v}_{j}-d\bar{v}_{j}\left(\frac{\d}{\d\theta_{i}}\right)\cdot \alpha_j=
		%\epsilon d\bar{v}_{i}-\epsilon d\bar{v}_{1}(\tfrac{\d}{\d \theta_{i}})\alpha_1=
		\epsilon d\bar{v}_i,
	\end{equation*}
	where in the second equality we have used identities $\alpha_{j}(\tfrac{\d}{\d\theta_i})=\delta_{i}^{j}$ and $d\bar{v}_{j}(\tfrac{\d}{\d \theta_{i}})=0$ from \ref{item:ort-alpha}, \ref{item:dvi}.
\end{proof}

		\begin{lema}\label{lem:Z}
		We can shrink $U_{X}$ so that the following holds. There is $\epsilon_0>0$ such that for every $\epsilon\in [0,\epsilon_0]$, the restriction of the form $\omega_{A}^{\epsilon}$ to the subbundle $\cZ_{I}$ is nondegenerate.
	\end{lema}
	\begin{proof}
		By Lemma \ref{lem:coordinates-on-stratum}\ref{item:splitting-on-stratum}, the differential of $\pi$ maps $\cZ_{I}$ isomorphically onto the tangent bundle to the stratum $X_{I}^{\circ}\cap \pi(U_1)$. In particular, the standard almost complex structure $J$ on $X$ pulls back to an almost complex structure on $\cZ_{I}$. We denote this pullback by $J$, too. Now it is enough to prove that, after possibly shrinking $U_{X}$, there is an $\epsilon_0>0$ such that for all $\epsilon\in [0,\epsilon_0]$, the symmetric part of the bilinear form $\cZ_{I}\times \cZ_{I}\ni (v,w)\mapsto \omega_{A}^{\epsilon}(v,Jw)\in \R$ is positive definite. 
		
		By Lemma \ref{lem:TdA}\ref{item:dvi}, we have $d\bar{v}_{i}|_{\cZ_{I}}=0$ for all $i\in \{1,\dots, l\}$. Hence %the restriction of $\omega_{E}$ to $\cZ_{I}$ equals:
		\begin{equation*}
			\omega_{E}|_{\cZ_{I}}=\sum_{i=l+1}^{k}d\bar{v}_{i}\wedge \alpha_{i}+\check{\omega}_{1},
		\end{equation*}
		where $\check{\omega}_{1}\de \sum_{j=k+1}^{N}d\bar{v}_{j}\wedge \alpha_{j}+\sum_{j=1}^{N}\bar{v}_{j}d\alpha_{j}$. By Lemma \ref{lem:d-theta}\ref{item:dvbar-rest-fiberwise},\ref{item:alpha-rest},\ref{item:dalpha}, the $1$-forms $d\bar{v}_{j}$, $\alpha_{j}$ for $j>k$; and the $2$-forms $d\alpha_j$ for all $j$, are pullbacks of smooth, bounded forms on $U_{X,I}^{\circ}\de U_X\cap X_{I}^{\circ}$. By Lemma \ref{lem:TdA}\ref{item:ort-Z-vj}, the functions $\bar{v}_{j}$ for $j>l$ are pullbacks of smooth functions on $U_{X,I}^{\circ}$. Clearly the functions $\bar{v}_{j}$ are bounded. We conclude that the form $\check{\omega}_1$ is bounded in the natural coordinates of $U_{X,I}^{\circ}$.
		
		By Lemma \ref{lem:dvbar}\ref{item:tame_1-bar}, for every $i\in \{l+1,\dots, k\}$, the restriction of $d\bar{v}_{i}$ to $T U_{I}^{\circ}$ equals $(1+c_{i}t_{i})dv_{i}+\sigma_{i}\gamma_{i}$, where $c_{i}$ (respectively, $\gamma_{i}$) is a pullback of a smooth, bounded function (respectively, $1$-form) on $U_{X,I}^{\circ}$; and $\sigma_{i}$ is the function defined in \eqref{eq:rho-sigma}. Moreover, since $u_{i}|_{U_{I}^{\circ}}=0$, we have $dv_{i}=dt_{i}=t_{i}^{2}m_{i}\, ds_{i}$ by Lemma \ref{lem:nu-j}\ref{item:dt-ds}, and $t_{i}^{2}=\sigma_{i}$ by definition \eqref{eq:rho-sigma} of $\sigma_{i}$. Eventually, by Lemma \ref{lem:d-theta}\ref{item:alpha-dtheta}, we have $\alpha_{i}=d\theta_{i}+\pi^{*}\beta_{i}$ for some bounded $\beta_{i}\in \Omega^{1}(U_{X})$. As in the proof of Lemma \ref{lem:omegaAC-sym}, we conclude that
		\begin{equation*}
			\omega_{E}|_{\cZ_{I}}=\sum_{i=l+1}^{k} \sigma_{i}(m_{i}(1+c_i t_i)ds_{i}\wedge d\theta_{i}+\gamma_{i}\wedge d\theta_{i}+m_{i}(1+c_i t_i)ds_{i}\wedge \check{\beta}_{i})+\omega_{1},
		\end{equation*}
		for some smooth, bounded $c_i\in \cC^{\infty}(U_{X,I}^{\circ})$, $\gamma_{i},\check{\beta}_{i}\in \Omega^{1}(U_{X,I}^{\circ})$ and $\omega_{1}\in \Omega^{2}(U_{X,I}^{\circ})$. 
		
		This is precisely formula \eqref{eq:omegaAC-mu} with two modifications. First, the sum starts at $i=l+1$ instead of $1$. Second, the functions and forms which in \eqref{eq:omegaAC-mu} are bounded from $X$, i.e.\ bounded in the natural coordinates of $U_X\cap X_{\emptyset}^{\circ}$, are now bounded in the natural coordinates of $U_X\cap X_{I}^{\circ}$.
		
		The remaining part of the proof follows by repeating the proof of Proposition \ref{prop:positivity}, given in Section \ref{sec:positivity}, starting from formula \eqref{eq:omegaAC-mu}. The only difference is that now the computations take place over the stratum $X_{I}^{\circ}$ instead of $X_{\emptyset}^{\circ}$. In particular, instead of vector fields $\nu_{1},\dots,\nu_{2k}$ on $U_{\emptyset}^{\circ}$, which were dual to $ds_1,\dots,d s_k$, $d\theta_1,\dots,d\theta_k$ with respect to the coordinates \eqref{eq:coordinates-sx}, we now use vector fields which are dual to $ds_{l+1},\dots,d s_{k}$, $d\theta_{l+1},\dots,d\theta_{k}$ with respect to the coordinates \eqref{eq:coordinates-on-stratum}. 
\end{proof}

\begin{proof}[Proof of Proposition \ref{prop:lifts}]
	The point $x_0\in D$ lies in a stratum $X_{S}^{\circ}$ for some nonempty $S\subseteq \{1,\dots, N\}$, see \eqref{eq:stratification}. Choose a fine chart $U_X$ around $x_0$, whose associated index set \eqref{eq:index-set} equals $S$. The preimage $U=\pi^{-1}(U_X)$ admits a finite covering $U=\bigcup_{i\in S}U_i$, where each $U_i$ is the domain of a smooth chart \eqref{eq:AC-chart}. Therefore, it is enough to prove the following. For every $i_0\in S$, we can shrink the chart $U_{X}$ around $x_0$, and find a  positive number $\epsilon_0>0$ such that for every $\epsilon\in (0,\epsilon_0]$, the form $\omega_{A}^{\epsilon}$ is fiberwise nondegenerate on $U_{i_0}$; and the symplectic lift of $\frac{\d}{\d \theta}$ to $U_{i_0}$ is given by the formula \eqref{eq:monodromy-vector-field}.

	As before, we reorder the components of $D$ so that $S=\{1,\dots, k\}$ and $i_0=1$. We can assume that $U_X$ and $\epsilon_0>0$ are small enough so that the statement of Lemma \ref{lem:Z} holds for all subsets $I\subseteq U_X$ containing $1$. We fix one such subset, say $I=\{1,\dots, l\}$, and use coordinates \eqref{eq:coordinates-on-stratum} on $U_{I}^{\circ}=U_1\cap A_{I}^{\circ}$. Also, we fix an $\epsilon\in (0,\epsilon_0]$.

	Fix a point $x\in U_I^{\circ}$. Let $F$ be the fiber of $\theta$ passing through $x$. Since the fine chart $U_{X}$ is by definition adapted to $f$, Definition \ref{def:adapted-chart}\ref{item:f-locally} together with the formula \eqref{eq:def-theta_i} for $\theta_{i}$ show that
	\begin{equation*}
		\theta|_{U_I^{\circ}}=m_1\theta_{1}+\dots+m_{k}\theta_k,
	\end{equation*}
	where we use additive notation \ref{not:S1} on $\S^1$. Hence 
	\begin{equation}\label{eq:TF}
		T_{x}F=\ker \left[\sum_{i=1}^{k} m_{i} d\theta_{i}\right]\subseteq T_{x}U_I^{\circ}.
	\end{equation}
	Assume that a nonzero vector $\nu\in T_{x}U_I^{\circ}$ satisfies $\omega\AC^{\epsilon}(\nu,\nu')=0$ for all $\nu'\in T_{x}F$. To prove non-degeneracy of $\omega\AC^{\epsilon}|_{T_{x}F}$, we need to show that $\nu\not\in T_{x}F$.
	
	Let $\nu^{\theta}$ be the vector given by the formula \eqref{eq:monodromy-vector-field}. We note that for $i\in \{l+1,\dots, k\}$, we have $u_i=0$, so the sum in \eqref{eq:monodromy-vector-field} actually runs from $i=1$ to $l$. For every $j\in \{1,\dots, k\}$, we have $d\theta_{j}(\nu^{\theta})=(\sum_{i=1}^{l}m_{i}\zeta(u_i))^{-1}\cdot \zeta(u_j)\geq 0$, with strict inequality for $j=1$. Hence formula \eqref{eq:TF} gives $\nu^{\theta}\not\in T_{x}F$. Thus to show that $\nu\not\in T_{x}F$, it is enough to show that $\nu$ is proportional to \eqref{eq:monodromy-vector-field}.
	
	Using coordinate vector fields $\tfrac{\d}{\d \bar{v}_i}$ introduced in Lemma \ref{lem:coordinates-on-stratum}\ref{item:splitting-on-stratum}, we can write 
	\begin{equation}\label{eq:ort}
		\nu=\sum_{i=2}^{l} a_{i}\frac{\d}{\d \bar{v}_{i}}+\sum_{i=1}^{l}b_{i}\frac{\d}{\d \theta_{i}}+\xi,\quad\mbox{for some } a_i,b_i\in \R\mbox{ and } \xi\in \mathcal{Z}_{I}.
	\end{equation}
	First, we claim that all coefficients $a_2,\dots, a_{l}$ in \eqref{eq:ort} are zero. 
	
	For $i\in \{2,\dots,l\}$ put $\nu^\theta_{i}=\tfrac{1}{m_{i}}\tfrac{\d}{\d \theta_{i}}-\tfrac{1}{m_{1}}\tfrac{\d}{\d \theta_{1}}\in T_{x}U_{I}^{\circ}$. Formula \eqref{eq:TF} shows that 
	$\nu_{i}^{\theta}\in T_{x}F$, so $\omega\AC^{\epsilon}(\nu,\nu^{\theta}_i)=0$ by definition of $\nu$. Using Lemma \ref{lem:TdA}\ref{item:omega-dtheta}, we compute
	\begin{equation*}
		0=\omega\AC^{\epsilon}(\nu,\nu^\theta_{i})= \tfrac{1}{m_{i}} \omega\AC^{\epsilon}(\nu, \tfrac{\d}{\d \theta_{i}})- \tfrac{1}{m_{1}}\omega\AC^{\epsilon}(\nu,\tfrac{\d}{\d \theta_{1}})=\tfrac{1}{m_{i}}\cdot \epsilon d\bar{v}_i(\nu)-\tfrac{1}{m_{1}}\cdot  \epsilon d\bar{v}_{1}(\nu)=\epsilon(\tfrac{a_i}{m_i}-\tfrac{1}{m_{1}}d\bar{v}_{1}(\nu)),
	\end{equation*}
	where the last equality follows from definition of the coordinate vector fields. Thus for all $i\in \{2,\dots, l\}$, the numbers $\tfrac{a_i}{m_{i}}$ are equal to the same number $a\de \tfrac{1}{m_{1}}d\bar{v}_{1}(\nu)$. 
	
	Suppose $a\neq 0$. Then $\tfrac{a_i}{a}=m_i$. Since $d\bar{v}_{1}|_{\Theta_{I}\oplus \cZ_{I}}=0$ by Lemma \ref{lem:TdA}\ref{item:dvi}, the formula \eqref{eq:ort} for $\nu$ gives
	\begin{equation*}
		m_1=\frac{1}{a}d\bar{v}_1(\nu)=\frac{1}{a}\sum_{i=2}^{l}a_{i} d\bar{v}_{1}\left(\frac{\d}{\d \bar{v}_i}\right)=-\sum_{i=2}^{l} m_{i}\frac{\zeta(u_i)}{\zeta(u_1)}\leq 0,
	\end{equation*}
	where for the last equality we used the fact that $\tfrac{a_i}{a}=m_i$, and the formula for $d\bar{v}_1$ from Lemma \ref{lem:TdA}\ref{item:dv1}. This is a contradiction. Therefore, $a=0$, so $a_{i}=am_i=0$ for all $i\in \{2,\dots, l\}$, as claimed.
	\smallskip
	
	Now, we claim that the vector $\xi\in\cZ_{I}$ in \eqref{eq:ort} is zero. Fix any $\xi'\in \mathcal{Z}_{I}$. Then $\xi'\in T_{x}F$ by \eqref{eq:TF}, so $\omega\AC^{\epsilon}(\nu,\xi')=0$ by definition of $\nu$. By Lemma \ref{lem:TdA}\ref{item:omega-dtheta}, for any $i\in \{1,\dots,l\}$ we have $\omega\AC^{\epsilon}(\tfrac{\d}{\d \theta_{i}},\xi')=-\epsilon d\bar{v}_{i}(\xi')=0$, where the last equality follows from Lemma \ref{lem:TdA}\ref{item:dvi}. Since we have shown that $a_2=\dots=a_{l}=0$, we infer from \eqref{eq:ort} that $0=\omega\AC^{\epsilon}(\nu,\xi')=\omega\AC^{\epsilon}(\xi,\xi')$, for all $\xi'\in \cZ_{I}$. Since by Lemma \ref{lem:Z}, the restriction $\omega\AC^{\epsilon}|_{\cZ_{I}}$ is nondegenerate, we get $\xi=0$, as claimed.
	
	Therefore, $\nu=\sum_{i=1}^{l}b_{i}\tfrac{\d}{\d \theta_{i}}$ for some $b_1,\dots,b_{l}\in \R$. Since for all $j\in \{2,\dots, l\}$ we have $\tfrac{\d}{\d \bar{v}_{j}}\in T_{x}F$, we compute using Lemma \ref{lem:TdA}\ref{item:omega-dtheta} that 
	\begin{equation*}
		0=\omega\AC^{\epsilon}\left(\frac{\d}{\d \bar{v}_{j}},\nu\right)=
		\sum_{i=1}^{l}b_{i}\omega\AC^{\epsilon}\left(\frac{\d}{\d \bar{v}_{j}},\frac{\d}{\d \theta_{i}}\right)=
		\epsilon\sum_{i=1}^{l}b_{i}d\bar{v}_{i}\left(\frac{\d}{\d\bar{v}_{j}}\right)=
		\epsilon\left(b_{j}-b_{1}\frac{\zeta(u_j)}{\zeta(u_1)}\right),
	\end{equation*}
	where the last equality follows from Lemma \ref{lem:TdA}\ref{item:dv1}. We conclude that for all $j\in \{1,\dots, l\}$ we have $b_{j}=\zeta(u_{j})b$ for some $b\in \R$. This shows that $\nu$ is proportional to \eqref{eq:monodromy-vector-field}, as claimed. In particular, the form $\omega_{A}^{\epsilon}$ is non-degenerate at $x$.

	It remains to show that if $\nu$ is normalized by $\theta_{*}(\nu)=\tfrac{\d}{\d \theta}$ then $\nu$ equals \eqref{eq:monodromy-vector-field}. In this case, we have 
	\begin{equation*}
		\frac{\d}{\d\theta}=\theta_{*}(\nu)=\sum_{i=1}^{l}b_{i}m_{i} \cdot \frac{\d}{\d\theta}=b\sum_{i=1}^{l} \zeta(u_{i})m_{i}\cdot \frac{\d}{\d\theta},
	\end{equation*}
	hence $b=(\sum_{i=1}^{l}\zeta(u_{i})m_{i})^{-1}$, as needed.
\end{proof}

\begin{proof}[Proof of Proposition \ref{prop:omega}]
	Fix $x\in \d A$. By Propositions \ref{prop:positivity} and \ref{prop:lifts}, there is an $\epsilon_{x}>0$ and an open neighborhood $V_{x}$ of $x$ in $A$, such that, for any $\delta>0$, each of the forms $\omega\AC^\epsilon$ and $\omega\AC^{\delta,\epsilon}$ is fiberwise nondegenerate on $V_{x}$ for every $\epsilon\in (0,\epsilon_{x}]$, and on $V_{x}\setminus \d A$ for every $\epsilon\in [0,\epsilon_x]$. Since $\d A \cap \bar{W}$ is compact, choosing a finite sub-cover from $\{V_{x}\}$ we get an $\epsilon_0>0$ and a neighborhood $W'$ of $\d A\cap \bar{W}$, such that, again, each of the forms $\omega\AC^\epsilon$ and $\omega\AC^{\delta,\epsilon}$ is fiberwise nondegenerate on $W'$ for all $\epsilon\in (0,\epsilon_0]$. This proves part \ref{item:omega-symplectic}, part \ref{item:omega-lift} follows from Proposition \ref{prop:lifts}.
\end{proof}

We end this section with a simple example illustrating Definition \ref{def:AX} of the A'Campo space, and the  monodromy  at radius zero given by Proposition \ref{prop:omega}\ref{item:omega-lift}. In Proposition \ref{prop:monodromy} we will describe the latter more precisely, in terms of the associated abstract contact open book.

\begin{example}\label{ex:monodromy}
	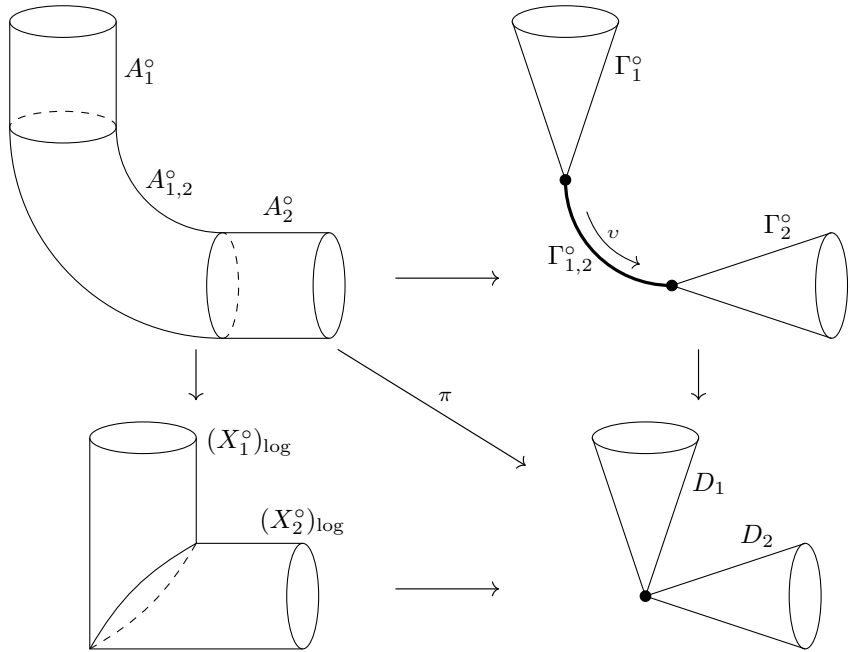
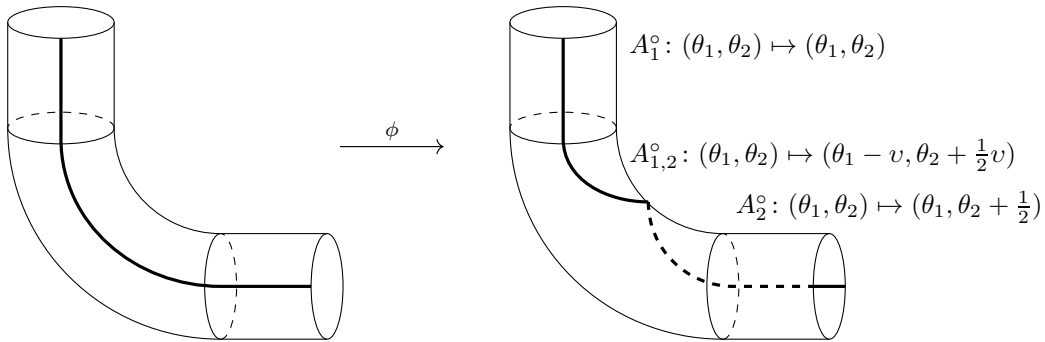
\begin{figure}[ht]
	\begin{subfigure}{\textwidth}
	\centering
	\begin{tikzcd}[column sep=large]
		\begin{tikzpicture}[scale=0.7]
			\path[use as bounding box] (-3,-1) rectangle (4,5.3);
			\draw (-2,5) ellipse (1 and 0.3);
			\draw (-1,5) -- (-1,3) to [out=-90, in=180] (1,1) to (3,1);
			\draw (-3,5) -- (-3,3) to [out=-90, in=180] (1,-1) to (3,-1);
			%\draw[very thick] (-2,4.7) -- (-2,2.8) to [out=-90, in=180] (-0.4,1.6);
			%\draw[very thick, dashed] (-0.4,1.6) to [out=-90, in=180] (1.2,0) -- (2.7,0);
			\draw[dashed] (-1,3) arc(0:180:1 and 0.3);
			\draw (-3,3) arc(180:360:1 and 0.3);
			\draw (1,1) arc(90:270:0.3 and 1);
			\draw[dashed] (1,-1) arc(-90:90:0.3 and 1);
			%\draw[very thick] (2.7,0) -- (3.3,0);
			\draw (3,0) ellipse (0.3 and 1);	
			\node[below right] at (-1.1,4.5) {\small{$A_{1}^{\circ}$}};
			\node[below right] at (-0.7,2.4) {\small{$A_{1,2}^{\circ}$}};
			\node[above right] at (1.5,1) {\small{$A_{2}^{\circ}$}};
		\end{tikzpicture}
		\ar[r]
		\ar[d]
		\ar[rd, "\pi", xshift=10]
		&
		\begin{tikzpicture}[scale=0.7]
			\path[use as bounding box] (-3,-1) rectangle (4,5.3);
			\draw[fill] (0,0) circle [radius=0.1];
			\draw[very thick] (-2,2) to [out=-90,in=180] (0,0);
			\draw[fill] (-2,2) circle [radius=0.1];
			\draw (-3,5) -- (-2,2) -- (-1,5);
			%\draw[very thick] (-2,2) -- (-2,4.7);
			\draw (-2,5) ellipse (1 and 0.3);
			\draw (3,1) -- (0,0) -- (3,-1);
			%\draw[very thick, dashed] (0,0) -- (2.7,0);
			%\draw[very thick] (2.7,0) -- (3.3,0);
			\draw (3,0) ellipse (0.3 and 1);
			\node at (-0.8,4) {\small{$\Gamma_{1}^{\circ}$}};
			\node at (2,1) {\small{$\Gamma_{2}^{\circ}$}};
			%\node [below left] at (0,0) {\small{$[1:0]$}};
			%\node [below left] at (-2,2) {\small{$[0:1]$}};
			\node at (-1.9,0.4) {\small{$\Gamma_{1,2}^{\circ}$}};
			\draw[->] (-1.6,1.4) to [out=-70,in=160] (-0.6,0.4);
			\node at (-1.1,0.9) {\tiny{$\upsilon$}};
			%\node at (1.8,2.3) {\small{$\upsilon=\frac{2\zeta(u_2)}{\zeta(u_1)+2\zeta(u_2)}$}};
			%\draw[->, gray] (0.1,2.4) to[out=180,in=0] (-0.9,1);
		\end{tikzpicture}
		\ar[d]
		\\
		\begin{tikzpicture}[scale=0.7]
		\path[use as bounding box] (-2.5,-1) rectangle (4.5,3.5);	
			\draw (0,3) ellipse (1 and 0.3);
			\draw (-1,3) -- (-1,-1) -- (3,-1);
			\draw (1,3) -- (1,1) -- (3,1);
			\draw (-1,-1) to [out=60, in=210] (1,1);
			\draw[dashed] (-1,-1) to [out=30, in=-120] (1,1);
			%\draw[very thick] (0,2.7) -- (0,0.3);
			%\draw[very thick, dashed] (0.3,0) -- (2.7,0);
			%\draw[very thick] (2.7,0) -- (3.3,0);
			\draw (3,0) ellipse (0.3 and 1);	
			\node at (2,2.8) {\small{$(X_{1}^{\circ})_{\log}$}};
			\node at (3,1.3) {\small{$(X_{2}^{\circ})_{\log}$}};
		\end{tikzpicture}
		\ar[r]
		&
		\begin{tikzpicture}[scale=0.7]
			\path[use as bounding box] (-2.5,-1) rectangle (4.5,3.5);
			\draw[fill] (0,0) circle [radius=0.1];
			\draw (-1,3) -- (0,0) -- (1,3);
			%\draw[very thick] (0,0) -- (0,2.7);
			\draw (0,3) ellipse (1 and 0.3);
			\draw (3,1) -- (0,0) -- (3,-1);
			%\draw[very thick, dashed] (0,0) -- (2.7,0);
			%\draw[very thick] (2.7,0) -- (3.3,0);
			\draw (3,0) ellipse (0.3 and 1);
			\node at (1.2,2) {\small{$D_1$}};
			\node at (2.1,1) {\small{$D_2$}};
%			\node at (0,-0.6) {\small{$X_{1,2}^{\circ}$}};
		\end{tikzpicture}
	\end{tikzcd}
	\caption{Radius-zero fibers in $A$, $\Gamma$, $X_{\log}$ and $X=\C^2$, see diagram \protect\eqref{eq:AX-diagram}.}
	\label{fig:construction}
	\end{subfigure}
%\end{figure}
\medskip

%\begin{figure}[ht]\ContinuedFloat
\begin{subfigure}{\textwidth}
	\centering
	\begin{tikzcd}[column sep=large]
		\begin{tikzpicture}[scale=0.7]
			\path[use as bounding box] (-4,-1) rectangle (3,5.5);
			\draw (-2,5) ellipse (1 and 0.3);
			\draw (-1,5) -- (-1,3) to [out=-90, in=180] (1,1) to (3,1);
			\draw (-3,5) -- (-3,3) to [out=-90, in=180] (1,-1) to (3,-1);
			\draw[very thick] (-2,4.7) -- (-2,2.8) to [out=-90, in=180] (1,0) -- (2.7,0);
			\draw[dashed] (-1,3) arc(0:180:1 and 0.3);
			\draw (-3,3) arc(180:360:1 and 0.3);
			\draw (1,1) arc(90:270:0.3 and 1);
			\draw[dashed] (1,-1) arc(-90:90:0.3 and 1);
			\draw (3,0) ellipse (0.3 and 1);	
%			\node[below right] at (-1,5) {\small{$A_{1}^{\circ}$}};
%			\node[below right] at (-1,3) {\small{$A_{1,2}^{\circ}$}};
%			\node[above right] at (1,1) {\small{$A_{2}^{\circ}$}};
		\end{tikzpicture}
		\ar[r, "\phi", yshift=50]
		&
		\begin{tikzpicture}[scale=0.7]
			\path[use as bounding box] (-4,-1) rectangle (7,5.5);
			\draw (-2,5) ellipse (1 and 0.3);
			\draw (-1,5) -- (-1,3) to [out=-90, in=180] (1,1) to (3,1);
			\draw (-3,5) -- (-3,3) to [out=-90, in=180] (1,-1) to (3,-1);
			\draw[very thick] (-2,4.7) -- (-2,2.8) to [out=-90, in=180] (-0.4,1.6);
			\draw[very thick, dashed] (-0.4,1.6) to [out=-90, in=180] (1.2,0) -- (2.7,0);
			\draw[dashed] (-1,3) arc(0:180:1 and 0.3);
			\draw (-3,3) arc(180:360:1 and 0.3);
			\draw (1,1) arc(90:270:0.3 and 1);
			\draw[dashed] (1,-1) arc(-90:90:0.3 and 1);
			\draw[very thick] (2.7,0) -- (3.3,0);
			\draw (3,0) ellipse (0.3 and 1);	
			\node[below right] at (-1,5) {\small{$A_{1}^{\circ}\colon (\theta_1,\theta_2)\mapsto (\theta_1,\theta_2)$}};
			\node[below right] at (-1,3) {\small{$A_{1,2}^{\circ}\colon (\theta_1,\theta_2)\mapsto (\theta_{1}-\upsilon,\theta_2+\tfrac{1}{2}\upsilon)$}};
			\node[above right] at (1,1) {\small{$A_{2}^{\circ} \colon (\theta_1,\theta_2) \mapsto (\theta_1,\theta_2+\tfrac{1}{2})$}};
		\end{tikzpicture}
	\end{tikzcd}
	\caption{Monodromy at radius zero, see Propositions \ref{prop:omega}\ref{item:omega-lift} and \ref{prop:monodromy}.}
	\label{fig:monodromy}
\end{subfigure}
	\caption{Example \protect\ref{ex:monodromy}: the A'Campo space for the function $f(z_1,z_2)=z_1z_2^2$.}
	\label{fig:example}
\end{figure}	
	Take $f\colon X=\C^{2}\to \C$ given by $f(z_1,z_2)=z_1z_2^2$. 
	Then $f^{-1}(0)=D_{1}+2D_{2}$, where $D_{i}=\{z_{i}=0\}$. In the bottom-right of Figure \ref{fig:construction}, the disks $D_{1},D_{2}$ meeting at the origin are shown as the vertical and horizontal cone, respectively. The fiber $f_{\log}^{-1}(0,1)\subseteq X_{\log}$ is shown in the bottom-left. The restriction of the map $X_{\log}\to X$ to this fiber is an isomorphism in the vertical cylinder, is a double cover in the horizontal one, and collapses the middle circle to the origin. Hence a natural monodromy at radius zero in the Kato--Nakayama space $X_{\log}$ is not continuous: indeed, in this example it is the identity in $(X_{1}^{\circ})_{\log}$, and a  half-twist in $(X_{2}^{\circ})_{\log}$.
	
	The top-right picture in Figure \ref{fig:construction}  is the fiber over $0$ in $\Gamma=\operatorname{graph}(\mu)$: there, the origin of $\C^2$ is replaced by a $1$-simplex $\Gamma_{1,2}^{\circ}$. Eventually, in the fiber $f_{A}^{-1}(0,1)\subseteq A$, shown in the top-left, we multiply this simplex by the middle circle of $f_{\log}^{-1}(0,1)$. The radius-zero monodromy $\phi$ in $A$ interpolates between the identity in $A_{1}^{\circ}$ and half-twist in $A_{2}^{\circ}$, using a parameter of the $1$-simplex $\Gamma_{1,2}^{\circ}$, call it $\upsilon\in [0,1]$. Figure \ref{fig:monodromy} shows how the bold ruling is transformed by $\phi$. A precise formula for the parameter $\upsilon$ can be obtained from equation \eqref{eq:monodromy-vector-field}: in this example, we have  $\upsilon=\frac{2\zeta(u_2)}{\zeta(u_1)+2\zeta(u_2)}$. 
%
%	
%
	%It is shown at the bottom-right of Figure \ref{fig:monodromy}, as follows. The axes $D_{1},D_{2}$ are  disks meeting at the origin, shown as the vertical and horizontal cone, respectively. The bold line is the image of the \enquote{front} ruling under the monodromy: it gets is flipped to the back of $D_1$, and remains untouched on $D_{2}$.
%	
%	The remaining pictures in Figure \ref{fig:construction} are radius-zero fibers in $\Gamma$ (top-right), $X_{\log}$ (bottom-left) and $A$ (top-left). In $\Gamma$, we replace the origin by a $1$-simplex $\Gamma_{1,2}^{\circ}$, with coordinates $[w_1:w_2]$, $w_1+w_2=1$. In $X_{\log}$, we introduce angular coordinates $\theta_{1}$, $\theta_{2}$ satisfying $2\theta_{1}+\theta_{2}=0$, i.e.\  we replace the origin by a circle. There, we can define a the monodromy is \emph{not continuous}.  Eventually, in $A$ we multiply this circle by $\Gamma_{1,2}^{\circ}$ and make the monodromy interpolate between the two pieces, using a parameter $\upsilon\in [0,1]$ on $\Gamma_{1,2}^{\circ}$.
\end{example}

\section{Symplectic monodromy and abstract contact open books}\label{sec:acobs}

We now recall further notions concerning Liouville domains and their fibrations.  For a general introduction, see \cite{CE_from-Stein-to-Weinstein,McDuff_Salamon,Geiges,Giroux_cool-gadget,Oh_1,Seidel_Fukaya-categories} and references therein. Our aim is to establish a natural setting \ref{basic-setting} where we can produce or extend symplectic monodromies.
%\smallskip

\subsection{Hamiltonian vector fields}\label{sec:Hamiltonian}

Let $(M,\omega)$ be a symplectic manifold. A \emph{Hamiltonian} is a smooth function $H\colon M\times \R\ni (x,t)\mapsto H_{t}(x)\in \R$. A \emph{time-independent Hamiltonian} is a smooth function $H\colon M\to \R$, which we view as a Hamiltonian $M\times \R\ni (x,t)\mapsto H(x)\in \R$. In other words, our Hamiltonians are always time-dependent, unless explicitly stated otherwise. 

Fix a Hamiltonian $H\colon M\times \R\to \R$. Since $\omega$ is non-degenerate, there exists a unique vector field $X^{H}$, called a \emph{Hamiltonian vector field} of $H$, satisfying the equality
\begin{equation*}
	\omega(X^{H},\sdot)=dH.
\end{equation*}
The \emph{Hamiltonian flow} of $H$ is a family of symplectomorphisms $\psi_{t}^{H}\colon M\to M$, parametrized by $t\in \R$, defined by the formulas
\begin{equation*}
	\psi^{H}_{0}=\id_{M},\quad \tfrac{d}{dt}\psi_{t}^{H}=X^{H}.
\end{equation*}
Our definition of $X^{H}$ coincides with the one used in \cite{Uljarevic,McDuff_Salamon,Dostoglou_Salamon,Oh_1}. However, \cite{AD_book,McLean} use the opposite one, namely $\omega(\sdot, X^{H})=dH$: thus for a given Hamiltonian $H$, our vector field $X^{H}$ and flow $\psi^{H}_{t}$ are equal to $-X^{H}$ and $(\psi^{H}_{t})^{-1}$ from \cite{AD_book,McLean}. 

So, when we quote \cite{Uljarevic} we need no adaptations. However, we will quote results from \cite{AD_book,McLean} too. To avoid sign mistakes, whenever we do this, we replace $H$ in loc.\ cit.\ by $-H$: this way, we get the same vector field $X^{H}$ and flow $\psi^{H}_{t}$ as in loc.\ cit; so no further modifications are needed. In particular, to get the same Floer Cauchy--Riemann equation \eqref{eq:floereqmclean} as used in \cite[Definition 4.3]{McLean}, we will need to perturb $\phi$ by an opposite Hamiltonian, so we will chose it \emph{negative} near $\d M$, despite the symbol \enquote{$+$}; see Section \ref{sec:perturbed} for details.

\subsection{Liouville and symplectic fibrations}\label{sec:Liouville-fibrations}

\begin{notation}\label{not:fibers}
	Fix a map $f\colon X\to Y$. If $f$ is clear from the context, we will use the following, simplified notation for its fibers. For $y\in Y$, we put $X_{y}=f^{-1}(y)$. Given a form $\lambda\in \Omega^{*}(X)$, we write $\lambda_{y}\de \lambda|_{X_y}$ for the restriction of $\lambda$ to a fiber. For subsets $U\subseteq X$, $V\subseteq Y$ we put $U|_{V}\de U\cap f^{-1}(V)$.
\end{notation}

A subset $B$ of a manifold $M$ is a \emph{codimension zero submanifold with corners} if for every point $p\in B$ there is a coordinate chart $\varphi\colon U\to\RR^d$ of $M$ around $p$ such that $\varphi(B\cap U)$ is an open subset of $\RR_{\geq 0}^k\times\RR^{d-k}$ for some $k\geq 0$. A \emph{manifold with corners} is a topological space which embeds into a smooth manifold as a codimension zero manifold with corners. This definition is compatible with the one given in \cite[Appendice]{corners} by Proposition 3.1 loc.\ cit. We say that a map (or a form) is smooth on $B$ if it extends to a smooth map (or form) on some neighborhood of $B$ in the ambient manifold $M$.

For example, the closure of each $A_{i}^{\circ}$ in Figure \ref{fig:v3} is a codimension zero submanifold with corners of $\d A$, cf.\ Proposition \ref{prop:monodromy}. Another example is the total space of the Milnor fibration
\begin{equation}
	\label{eq:milfib1}
	f\colon \B_{\epsilon}\cap f^{-1}(\D_{\delta}^{*})\to \D_{\delta}^{*},
\end{equation}
where $f\colon \C^n\to \C$ is a holomorphic function with an isolated critical point at $0\in \C^n$.%, and $\epsilon,\delta>0$ are small.
\smallskip

Given a smooth map $f\colon M\to B$ between manifolds with corners, we split the boundary $\partial M$ as:
\begin{equation*}
	\d M=\d\vert M\cup \d\hor M, \quad\mbox{where}\quad \d\vert M=f^{-1}(\d B)\cap \d M\quad\mbox{and}\quad \d\hor M=\overline{\d M\setminus \d\vert M}.
\end{equation*}
We say that $f$ is a \emph{fibration} if it is smoothly locally trivial, i.e.\ every $b\in B$ has a neighborhood $U$ and a diffeomorphism $\varphi\colon U\times f^{-1}(p)\to f^{-1}(U)$ such that $f\circ\varphi=\pr_{U}$. By the Ehressmann lemma, see e.g.\ \cite[6.2.10]{handbook}, a proper map $f$ is a fibration if both $f$ and $f|_{\d\hor M}$ are surjective and submersive.

Let $\lambda\in\Omega^1(M)$ be a $1$-form. The pair $(f,\lambda)$ is called a {\em Liouville fibration} if $f\colon M\to B$ is a fibration and $(M_b,\lambda|_{M_b})$ is a Liouville domain for each $b\in B$, see Section \ref{sec:symplectic-intro}. The similar notion of a \emph{symplectic fibration} is defined replacing $\lambda$ by a closed, fiberwise symplectic form $\omega$, see Section \ref{sec:symplectic-intro}. A Liouville fibration induces a symplectic fibration by taking $\omega=d\lambda$. 

Note that in the definition of a symplectic fibration we require that the form $\omega\in \Omega^{2}(M)$ is \emph{closed}, and consider it a part of the data. This definition, while convenient for our purposes, is \emph{not} a standard one. For example, in  \cite[\sec 4.2]{Oh_1}, what we call a symplectic fibration is called a \emph{Hamiltonian fibration} with a choice of a coupling form, cf.\ \cite[\sec 6.4]{McDuff_Salamon}. We impose this more restrictive condition since it holds in all our applications, and guarantees that the flows of symplectic lifts are symplectomorphisms.%. Meanwhile, with a usual definition additional conditions are needed, see \cite[\sec 4.1]{Oh_1} for a discussion.}

%%%%%%%%%%%%%%%%%%%%%%%%%%%%%%%%%%%%%%%%%%%%%%%%%%%%%%%%%%%%%
%%%%%%%%%%% ! ! ! ! ! ! ! ! ! ! ! %%%%%%%%%%%%%%%%%%%%%%%%%%%
%%%%%%%%%%%%%%%%%%%%%%%%%%%%%%%%%%%%%%%%%%%%%%%%%%%%%%%%%%%%%
%\red{Our \emph{symplectic fibrations are \emph{Hamiltonian fibrations} in \cite[Definition 4.2.1]{Oh_1}. In \cite[Definition 4.1.1]{Oh_1}, a \emph{symplectic fibration} is a fibration with symplectic fibers, see exercise 4.1.3 loc.\ cit. Then, a symplectic connection is the one with symplectic holonomy (Def.\ 4.1.5). By Cartan formula, for a Hamiltonian fibration any symplectic connection in our sense has symplectic holonomy (Proposition 4.1.6). In fact, in this case holonomy is Hamiltonian (Theorem 4.2.2)}}
%%%%%%%%%%%%%%%%%%%%%%%%%%%%%%%%%%%%%%%%%%%%%%%%%%%%%%%%%%%%%
%%%%%%%%%%%%%%%%%%%%%%%%%%%%%%%%%%%%%%%%%%%%%%%%%%%%%%%%%%%%%
%%%%%%%%%%%%%%%%%%%%%%%%%%%%%%%%%%%%%%%%%%%%%%%%%%%%%%%%%%%%%

\subsection{Definition of a symplectic monodromy}

In this section we will modify the Milnor fibration \eqref{eq:milfib1} so that its monodromy becomes a compactly supported symplectomorphism. Note that such a modification is needed: symplectomorphisms preserve volume, and there is no reason why two fibers of \eqref{eq:milfib1} should have the same volume. The following notion, introduced by McLean in \cite[Definition 3.12]{McLean}, abstracts the properties of symplectic monodromy that we need.  

\begin{definition}\label{def:acob}
	Let $(M,\lambda)$ be a Liouville domain. We say that $(M,\lambda,\phi)$ is an \emph{abstract contact open book} if $\phi\colon M\to M$ is a diffeomorphism that is {\em compactly supported}, i.e.\ satisfies $\phi|_{C}=\id_{C}$ for some neighborhood $C$ of $\d M$, and {\em exact}, that is:
	\begin{equation}\label{eq:exact-symplectiomorphism}
		\phi^{*}\lambda-\lambda=-d \ac
	\end{equation}
	for some function $\ac \colon M\to \R$, called \emph{action} of $\phi$. Note that the formula \eqref{eq:exact-symplectiomorphism} defines $\ac$ uniquely, up to an additive constant. Moreover, it implies that $\phi$ is a symplectomorphism.
	
	Let $M$ be a smooth manifold, let $f\colon M\to [0,1]$ be a smooth map, and let $\lambda\in \Omega^{1}(M)$ be a $1$-form such that $(f,\lambda)$ is a Liouville fibration. Let $\phi\colon M\to M$ be a diffeomorphism which preserves each fiber of $f$ and is equal to the identity close to $\d\hor M$. Put $\phi_{t}=\phi|_{M_t}$,  see Notation \ref{not:fibers}. If for every $t\in [0,1]$ the triple $(M_t,\lambda_t,\phi_t)$ is an abstract contact open book, we say that the family $\{(M_t,\lambda_t,\phi_t)\}_{t\in [0,1]}$ is an \emph{isotopy of abstract contact open books}. In this case, we say that the abstract contact open books $(M_t,\lambda_t,\phi_t)$ are \emph{isotopic}, and write $(M_t,\lambda_t,\phi_t)\sim (M_s,\lambda_s,\phi_s)$ for $t,s\in [0,1]$.
\end{definition}

Clearly, if $(M,\lambda,\phi)$ is an abstract contact open book, then so is $(M,\lambda,\phi^m)$ for any iterate $\phi^m$ of $\phi$; and if $(M_0,\lambda_0,\phi_0)\sim (M_1,\lambda_1,\phi_1)$ then $(M_0,\lambda_0,\phi_0^m)\sim (M_1,\lambda_1,\phi_1^m)$ for every $m\geq 1$.
\smallskip

In order to turn monodromies into abstract open books and compare their isotopy classes, we will need the following notion.

\begin{definition}
	\label{def:collartriv}
	Let $f\colon M\to B\times P$ be a fibration from a manifold with corners to a product of %$P$ is a manifold and $B$ is a 
	manifolds with boundary. Fix $b_0\in B$ and let $M_{b_{0}}\de (\pr_{B}\circ f)^{-1}(b_0)$, see Notation \ref{not:fibers}. 
	A \emph{$P$-fiberwise  trivialization} of $f$ is a diffeomorphism
	\begin{equation}
		\label{eq:triv}
		\psi\colon M_{b_0}\times B\to M
	\end{equation}
	such that for all points $x\in M_{b_{0}}$, and all $b\in B$, we have
	\begin{equation*}
		f(\psi(x,b))=(b,\pr_P(f(x))).
	\end{equation*}
	A \emph{$P$-fiberwise collar trivialization} of $f$ is an open neighborhood $C$ of $\d\hor M$ in $M$, together with a $P$-fiberwise trivialization of $f|_C$.
	
	Let $\omega\in \Omega^2(M)$ be a $2$-form such that $(f,\omega)$ is a symplectic fibration. A $P$-fiberwise trivialization \eqref{eq:triv} is called {\em $(B\times P)$-fiberwise symplectic} if for any $(b,p)\in B\times P$ we have %the equality 
	\begin{equation}
		\label{eq:simpcolltriv}
		\psi^*(\omega_{b,p})=\omega_{b_{0},p},
	\end{equation}
	where $\omega_{b,p}$ is the restriction of $\omega$ to the fiber $M_{b,p}=f^{-1}(b,p)$.
\end{definition}

\begin{definition}
	\label{def:sympmonod}
	Let $f\colon M\to \S^1$ and $\omega\in\Omega^2(M)$ be such that $(f,\omega)$ is a symplectic fibration. A {\em symplectic monodromy trivialization} is a smooth map $\Phi:M_1\times [0,1]\to M$ such that  $f(\Phi(x,t))=e^{2\pi\imath t}$ for all $(x,t)\in M\times [0,1]$, and such that for all $t\in [0,1]$, the restriction $\Phi_t:=\Phi|_{M_1\times\{t\}}$ defines a symplectomorphism from $(M_1,\omega_1)$ to $(M_\theta,\omega_\theta)$, where $\theta=e^{2\pi\imath t}$, $\omega_\theta=\omega|_{M_\theta}$, and $\Phi_{0}=\id_{M_1}$. The corresponding {\em symplectic monodromy} is the symplectomorphism $\phi\de\Phi_1$. 
	
	Assume that $\omega=d\lambda$, and that $(f,\lambda)$ is a Liouville fibration. Fix an $\S^1$-fiberwise symplectic collar trivialization $\psi\colon C\times \S^1\to M$ of $f$. A {\em symplectic monodromy trivialization relative to $\psi$} is a symplectic monodromy trivialization $\Phi$ satisfying $\Phi(x,t)=\psi(x, e^{2\pi\imath t})$ for all $(x,t)\in C\times [0,1]$. A {\em symplectic monodromy relative to $\psi$} is the symplectomorphism $\phi:=\Phi_1$ for a symplectic monodromy trivialization relative to $\psi$. Observe that in this case $\phi$ is the identity close to the boundary $\partial M_1$, so if $\phi$ is exact, then $(M_1,\lambda_1,\phi)$ is an abstract contact open book.
\end{definition}

\begin{remark}
	\label{rem:mondunique}
	Fix $f\colon M\to \S^1$ and $\omega\in\Omega^2(M)$ such that $(f,\omega)$ is a symplectic fibration. If a symplectic monodromy exits, then it is unique as an element of the group of symplectomorphisms up to a symplectic isotopy. Indeed,    
	%\proof
	given two symplectic monodromy trivializations $\Phi$ and $\Phi'$, the needed isotopy is $\Phi'_t\comp\Phi_{t}^{-1}\comp\Phi_1$.
	%\endproof
	
	%Similarly, f
	Fix $f\colon M\to \S^1$ and $\lambda\in\Omega^1(M)$ such that $(f,\lambda)$ is a Liouville fibration. Fix an  $\S^1$-fiberwise symplectic collar trivialization $\psi$ of $f$. Then as before, if a symplectic monodromy relative to $\psi$ exists, it is unique in the group of compactly supported symplectomorphisms of $\partial M_1$, up to a compactly supported symplectic isotopy. In particular, if every symplectic monodromy $\phi$ relative to $\psi$ is exact, then the isotopy class of abstract contact open books $(M_1,\lambda_1,\phi)$ does not depend on the choice of $\phi$. %the symplectic monodromy relative to $\psi$. 
\end{remark}

We now distinguish a class of \emph{cohomologically trivial} collar trivializations. In Proposition \ref{prop:sufficientmonod} we will use this technical condition to construct symplectic monodromy via a Moser argument.% \cite[\sec 3.2]{McDuff_Salamon}.}

%The next proposition provides sufficient conditions for constructing a symplectic monodromy. 

Let $f\colon M\to B\times P$ as in Definition~\ref{def:collartriv} and let $\lambda\in\Omega^1(M)$ be such that $(f,\lambda)$ is a Liouville fibration. Put $\omega=d\lambda$. Assume that we are given a $P$-fiberwise collar trivialization
\begin{equation*}
	%\label{eq:coltriv}
	\psi\colon C_{b_0}\times B\to C,
\end{equation*}
that is $(B\times P)$-fiberwise symplectic, i.e.\ satisfies condition \eqref{eq:simpcolltriv}. For any $b\in B$ choose an open neighborhood $U_b$ of $b$ in $B$ and a $P$-fiberwise local trivialization 
\begin{equation*}
	%\label{eq:trivaux}
	\Theta:M_{b}\times U_b\to f^{-1}(U_b\times P)
\end{equation*}
which is induced by $\psi$ on the collar $C$, that is, such that for any $x\in M_{b}\cap C$ and any $b'\in U_{b}$ we have 
\begin{equation*}
	\Theta(x,b')=\psi(\pr_{C_{b_0}}(\psi^{-1}(x)),b').
\end{equation*}
%where $C_{b_0}=(\pr_{B}\circ f)^{-1}(b_0) \cap C$, and $\pr_{C_{b_0}}\colon C_{b_0}\times B\to C_{b_0}$ is the projection, see Notation \ref{not:fibers}.

For every $p\in P$ the form $\Theta^*\omega$ can be interpreted as a family of symplectic forms on the fiber $M_{b,p}$ parametrized by $U_b$; for any $b'\in U_b$, denote by $(\Theta^*\omega)_{b',p}\in\Omega^2(M_{b,p})$ the corresponding form. 
Since the collar trivialization $\psi$ is $(B\times P)$-fiberwise symplectic, for every $b,b'\in U_b$ we have the equality of restrictions 
$(\Theta^*\omega)_{b,p}|_{C_{b}}=(\Theta^*\omega)_{b',p}|_{C_{b}}$, 
hence the difference 
\begin{equation}\label{eq:eta_Theta}
	\eta^{\Theta}_{b,b';p}\de (\Theta^*\omega)_{b,p}-(\Theta^*\omega)_{b',p}
\end{equation}
vanishes in a neighborhood of the boundary $\d M_{b,p}$. Thus we obtain a relative cohomology class $[\eta^{\Theta}_{b,b';p}]\in H^2(M_{b,p},\d M_{b,p};\RR)$. We claim that this cohomology class does not depend on the choice of the trivialization $\Theta$, and hence denote it by  $[\eta_{b,b';p}]$.

To prove the claim, we choose any class $[z]\in H_2(M_{b,p},\d M_{b,p};\RR)$. We have $\eta^{\Theta}_{b,b';p}(z)=\omega(\Theta(z,b))-\omega(\Theta(z,b'))$. Hence, if $\Theta'$ is a different trivialization we have 
$$\eta^{\Theta}_{b,b';p}(z)-\eta^{\Theta'}_{b,b';p}(z)=\omega(\Theta(z,b)-\Theta'(z,b))+\omega(-\Theta(z,b')+\Theta'(z,b')).$$
Since $\Theta$ and $\Theta'$ coincide close to $\d M_{p,b}$, the chain $\Theta(z,b)-\Theta'(z,b)$ is closed. Therefore, since $\omega$ is exact we have the vanishing $\omega(\Theta(z,b)-\Theta'(z,b))=0$. The second summand vanishes for the same reasons. This proves the claim. 

\begin{definition}
	\label{def:cohtriv}
	Fix a Liouville fibration $(f,\lambda)$ and a $(B\times P)$-fiberwise symplectic collar trivialization $\psi$ as in the discussion above. Fix $p\in P$. We say that $(f,\lambda)$ is {\em cohomologically trivial with respect to $\psi$ at $p$} if for every $b\in B$ and a trivialization $\Theta$ as above, the cohomology class $[\eta_{p;b,b'}]$ vanishes in $H^2(M_{b,p},\d M_{b,p};\RR)$ for any $b'\in U_b$.
\end{definition}

\begin{prop}
	\label{prop:sufficientmonod}
	Let $f\colon M\to \S^1\times P$ be as in Definition~\ref{def:collartriv} and let $\lambda\in\Omega^1(M)$ be such that $(f,\lambda)$ is a Liouville fibration. Let $\psi\colon C_{1}\times \S^1\to C$ be its $P$-fiberwise collar trivialization which is $(\S^1\times P)$-fiberwise symplectic. 
	Assume that $(f,\lambda)$  is cohomologically trivial with respect to $\psi$ at every point $p\in P$, see Definition \ref{def:cohtriv}. Then there exists an open covering $\{Q_i\}_{i\in I}$ of $P$, and, for each $i\in I$, a smooth map
	\begin{equation*}
		\Theta^i:M_1|_{Q_i}\times [0,1]\to M|_{Q_i},
	\end{equation*}
	such that for every $(x,t)\in M_1|_{Q_i}\times [0,1]$ we have $(\pr_P\circ f\circ \Theta^i)(x,t)=(\pr_P\circ f)(x)$, and for every $p\in Q_i$ the restriction $\Theta^{i}|_{M_{1,p}\times [0,1]}$ is a symplectic monodromy trivialization relative to the $\S^1$-fiberwise symplectic collar trivialization $\psi|_{C_{1,p}\times\S^1}$.

	For $i\in I$ and $p\in Q_i$, denote by $\phi_{p}^{i}:M_{1,p}\to M_{1,p}$ the associated symplectic monodromy. Assume that $P$ is connected, and for every $i\in I$, $p\in Q_i$ the symplectomorphism $\phi^{i}_{p}$ is exact. Then for every $i\in I$, $p\in Q_i$, the triple $(M_{1,p},\lambda_{1,p},\phi_{p}^{i})$ is an abstract contact open book, whose isotopy type does not depend on the choice of $i\in I$, $p\in Q_i$.
\end{prop}

\begin{proof}
	For the first assertion, we will use a parametrized Moser argument. Put $\omega=d\lambda$.
	
	Since $f$ is a fibration and admits a $P$-fiberwise collar trivialization $\psi$, there exists a smooth map 
	\begin{equation*}
		\Theta':M_1\times [0,1]\to M
	\end{equation*}
	such that we have  $(\pr_P\circ f\circ \Theta')(x,t)=(\pr_P\circ f)(x)$ for any $(x,t)\in M_1\times [0,1]$, and $\Theta'(x,t)=\psi(x,e^{2\pi\imath t})$ for any $(x,t)\in C\cap M_1\times [0,1]$. We view the pullback $\omega':=(\Theta')^*\omega$ as a family of closed $2$-forms in $\Omega^2(M_1)$ parametrized by $t\in [0,1]$ which are symplectic fiberwise over $P$. We denote by $\omega'_t$ the restriction of $\omega'$ to $M_1\times\{t\}$. Since $\psi$ is $(\S^1\times P)$-fiberwise symplectic, $\Theta'$ can be chosen so that the restriction $\omega'_t|_{C_{1,p}}$ is independent of $t$ for any $p\in P$. We now follow the strategy of the proof of \cite[Theorem 3.2.4]{McDuff_Salamon}, taking care of the fact that our construction is parametrized over $P$ and compatible with the collar trivialization. 
	
	Since the assertion is local in the base $P$, we may assume that $M_1$ is of the form $M_1=P\times N_1$. Then we can view  $\omega'_t$ as a family of closed forms $\omega'_{t,p}\in\Omega^{2}(N_1)$ which are independent of $t$ near $\partial N_1$, for each fixed $p\in P$. Since $(f,\lambda)$ is cohomologically trivial with respect to $\psi$, the relative cohomology class $[\omega'_{t,p}-\omega'_{t',p}]$ vanishes in $H^{2}(N_1,\partial N_1,\RR)$ for every fixed $p\in P$, and every $t,t'\in [0,1]$. Therefore, the class  $[\frac{d}{dt}\omega'_{t,p}]$ vanishes in the same group. This means that for any $p\in P$ and any $t\in [0,1]$ there exists a $1$-form $\sigma_{t,p}\in\Omega^1(N_1)$ which vanishes near $\d N_1$ and satisfies $d\sigma_{t,p}=\frac{d}{dt}\omega'_{t,p}$.
	
	We will produce such a family of $\sigma_{t,p}$ smoothly depending on $t$ and $p$. For this, out of the two methods suggested in the proof of \cite[Theorem 3.2.4]{McDuff_Salamon} to construct $\sigma_{t,p}$, we adapt the first one, mimicking the proof of the compactly supported Mayer-Vietoris sequence in de Rham cohomology. 
	
	To be precise: denote by $\Omega_c^*(N_1\times [0,1]\times P)|_{[0,1]\times P}$ the {chain} complex formed by families $\xi_{t,p}$ of compactly supported forms in $N_1\setminus\d N_1$ smoothly depending on $(t,p)$. Let $H_c^*(N_1\times [0,1]\times P)|_{[0,1]\times P}$ be its cohomology. Define a homomorphism 
	\begin{equation*}
		%	\label{eq:deRhamiso}
		\gamma\colon H_c^*(N_1\times [0,1]\times P)|_{[0,1]\times P}\to H_c^*(N_1\setminus\d N_1)\otimes\cC^\infty([0,1]\times P)
	\end{equation*}
	as follows: choose {a basis} $[\Gamma_{1}],\dots,[\Gamma_{r}]$ {of} $H_{*}(N_1,\d N_1)$ and define 
	\begin{equation*}
		\gamma\left([\xi_{t,p}]\right) \de \sum_{j=1}^r\left(\int_{\Gamma_{j}}\xi_{t,p}\right)\cdot \left[\Gamma_{j}\right]^*,
	\end{equation*}
	where $\{[\Gamma_{j}]^*\}_{j=1}^r$ is the {basis of $H_{c}^{*}(N_1\setminus \d N_1)$ dual to} $\{[\Gamma_{j}]\}_{j=1}^r$. 
	
	In order to produce the needed family it is enough to prove that $\gamma$ is an isomorphism. This is proved by induction on the minimal number of open subsets of a good cover of $N_1$, where a good cover is a cover by contractible open subsets such that any finite intersection is contractible. Existence of good covers is established in \cite[Theorem 5.1]{Bott_Tu}. The initial step of the induction follows the proof of the compactly supported Poincar\'e Lemma \cite[Propositions 4.6, 4.7]{Bott_Tu}. The inductive step follows the Mayer-Vietoris argument of \cite[\sec I.5]{Bott_Tu}.
	
	After the smooth family $\sigma_{t,p}$ has been constructed, applying Moser trick to the family $\sigma_{t,p}$ as in \cite[\sec 3.2]{McDuff_Salamon} we produce a family of diffeomorphisms $\Theta''_{t,p}\colon N_1\to N_1$ that are equal to the identity near $\partial N_1$ and such that $(\Theta''_{t,p})^*\omega'_{t,p}= \omega'_{0,p}$ for any $t\in [0,1]$ and $p\in P$. Then the required smooth map $\Theta$ is defined by $\Theta_{t,p}:=\Theta'_{t,p}\comp \Theta''_{t,p}$. 
	
	For the second assertion, we note that for any $p\in Q_i$ the monodromy diffeomorphism $\phi_{p}^{i}$ induced by $\Theta^i$ is compactly supported, so if it is exact, then  $(M_{1,p},\lambda_{1,p},\phi_{p}^{i})$ is an abstract contact open book. Clearly, $(M_{1,p},\lambda_{1,p},\phi_{p}^{i})\sim (M_{1,q},\lambda_{1,q},\phi_{q}^{i})$ if $p,q\in Q_{i}$ and $Q_i$ is connected. If $p\in Q_i\cap Q_{j}$ then since $\phi_{p}^{i}$ and $\phi_{p}^{j}$ are both symplectic monodromies relative to the same collar trivialization, we have $(M_{1,p},\lambda_{1,p},\phi_{p}^{i})\sim (M_{1,p},\lambda_{1,p},\phi_{p}^{j})$ by Remark~\ref{rem:mondunique}.
\end{proof}

\begin{definition}\label{def:monodromy_acob}
	Let $f\colon M\to P\times \S^1$ be a Liouville fibration, with a connected $P$, and let $\psi$ be its fiberwise symplectic collar trivialization.  
	The abstract contact open book from Proposition \ref{prop:sufficientmonod} will be called the \emph{monodromy abstract contact open book} of $f,\psi$. It is unique up to an isotopy, and exists if $\psi$ is cohomologically trivial and the monodromy $\phi$ is exact.
\end{definition}

\subsection{Grading and Conley--Zehnder index}

In order to define a grading of the Floer complex, we need to consider \emph{graded} abstract contact open books. We now recall basic notions leading to their definition, following \cite[Appendix A]{McLean}.

\subsubsection{Graded principal bundles} Let $G$ be a Lie group with universal cover $\tilde{G}$, and let $p\colon F\to M$ be a principal $G$-bundle. A \emph{grading} of $F$ is a principal $\tilde{G}$ bundle $\tilde{F}\to M$ together with an isomorphism $\alpha\colon \tilde{F}\times_{\tilde{G}} G \to F$. Thus there is a natural covering map $\alpha'\colon\tilde{F}\to F$. An isotopy between two gradings $\alpha_0$, $\alpha_1$ is a principal $G$-bundle over $M\times [0,1]$, together with a grading which restricts to $\alpha_i$ over $M\times \{i\}$ for $i\in \{0,1\}$. Choose a base point $f\in F$. Then by \cite[Lemma A.3]{McLean}, there is a one-to-one correspondence
\begin{equation}\label{eq:gradings-subgroups}
	\begin{split}
		&\{\mbox{gradings of }p\colon F\to M\}/_{\mbox{\tiny{isotopy}}} 
		\longleftrightarrow \\
		&\longleftrightarrow
		\{N\vartriangleleft \pi_{1}(F,f): p_{*}|_{N}\colon N\to \pi_{1}(M,p(f))\mbox{ is an isomorphism}\}
	\end{split}
\end{equation}
sending a grading $(\tilde{p}\colon \tilde{F}\to M,\alpha)$ to  $\alpha'_{*}(\pi_{1}(\tilde{F},\tilde{f}))\vartriangleleft \pi_{1}(F,f)$, where $\tilde{f}\in (\alpha')^{-1}(f)$.% is the preimage of $f$ by the isomorphism $\alpha$ given by the grading.

\subsubsection{Graded symplectic manifolds} \label{sec:graded-manifolds}
Let $(M,\omega)$ be a symplectic manifold of dimension $2d$. A \emph{grading} of $M$ is a grading of its symplectic frame bundle $\Fr(M)\to M$, which is a principal $\Sp(2d)$-bundle, see \cite[Definition A.6]{McLean}. Choose an $\omega$-compatible almost complex structure $J$, so that $TM$ becomes a complex vector bundle of rank $d$. Let $K^{*}_{M}\de \Lambda^{d} TM$ be the corresponding anti-canonical bundle. The natural maps
\begin{equation*}
	\begin{tikzcd}
		\Sp(2d) & \ar[l, hook'] U(d) \ar[r, "\det"] & \S^1
	\end{tikzcd}
\end{equation*}
induce isomorphisms between fundamental groups, hence \eqref{eq:gradings-subgroups} gives a one-to-one correspondence between gradings of $M$ and gradings of the $\S^{1}$-bundle of $K^{*}_{M}$. By \cite[Lemma A.5]{McLean}, the latter are in one-to-one correspondence with homotopy classes of $\cC^{\infty}$ trivializations of $K^{*}_{M}$, or, equivalently, of $K_{M}\de \Lambda^{d} T^{*}M$. Recall that $K_M$ is a complex line bundle, so its $\cC^{\infty}$ trivializations are given by nonvanishing $\cC^{\infty}$ sections. To summarize, we have a one-to-one correspondence 
\begin{equation}\label{eq:gradings-trivializations}
	\{\mbox{gradings of } M\}/_{\mbox{\tiny{isotopy}}} 
	\longleftrightarrow
	\{\mbox{nonvanishing smooth sections of }K_{M} \}/_{\mbox{\tiny{homotopy}}}.
\end{equation}
Likewise, given a symplectic fibration $M\to P$ with a fiberwise symplectic form $\omega$ and an $\omega$-compatible almost complex structure $J$ of the relative tangent bundle, let $F\to M$ be the vertical symplectic frame bundle, and let $K\vert$ be the relative canonical bundle. We have a one-to-one correspondence
\begin{equation}\label{eq:gradings-trivializationsfiberwise}
	\{\mbox{gradings of } F\}/_{\mbox{\tiny{isotopy}}} 
	\longleftrightarrow
	\{\mbox{nonvanishing smooth sections of }K\vert \}/_{\mbox{\tiny{homotopy}}}.
\end{equation}

\subsubsection{Graded symplectomorphisms and graded abstract contact open books}\label{sec:graded-acobs}

Let $\phi\colon M\to M$ be a symplectomorphism. Let $\beta\colon \Fr(M)\to \phi^{*}\Fr(M)$ be the induced automorphism of the symplectic frame bundle. A \emph{grading} of $\phi$ is an isomorphism $\tilde{\beta}\colon \tilde{\Fr}(M)\to\phi^{*}\tilde{\Fr}(M)$ which lifts $\beta$ via the grading isomorphism $\alpha\colon \tilde{\Fr}(M)\times_{\tilde{\Sp}(2d)} \Sp(2d) \to \Fr(M)$, i.e.\ satisfies $\tilde{\beta}\times_{\tilde{\Sp}(2d)}\Sp(2d)=\beta$. 

A \emph{grading} of an abstract contact open book $(M,\lambda,\phi)$ is a grading of the symplectic manifold $(M,d\lambda)$ and a compatible grading of the symplectomorphism $\phi\colon M\to M$. An isotopy of abstract contact open books $\{(M_t,\lambda_t,\phi_t)\}_{t\in [0,1]}$ is \emph{graded} if there is a family $\tilde{\beta_t}\colon \tilde{\Fr}(M_t)\to \phi_t^{*}\tilde{\Fr}(M_t)$ of gradings of $\phi_t$ depending smoothly on $t$. 

\begin{lema}
	\label{lem:liftgradings}
	Let $f\colon M\to \S^{1}\times P$ be a symplectic fibration with a closed fiberwise symplectic form $\omega$, and let $F\to M$ be the vertical symplectic frame bundle. Let $\Phi\colon M_1\times [0,1]\to M$ be a map that preserves fibers over $P$, and is a symplectic monodromy trivialization fiberwise over $P$. Any grading $\tilde{F}\to M$ of $F$ induces a grading $\tilde{\beta}_p$ of the symplectic monodromy $\phi_p\colon M_{1,p}\to M_{1,p}$ defined by $\Phi$, which depends smoothly on $p$. 
\end{lema}
\begin{proof}
By the symplectic monodromy construction, given any symplectic monodromy $\phi\colon M_1\to M_1$ there exists a closed fiberwise symplectic form $\omega'\in\Omega^2(M)$ such that $\omega'_{t,p}=\omega_{t,p}$ for all $(t,p)\in \S^1\times P$ and $\phi$ is the time $1$ flow of the lifting of the unit vector field $(0,\frac{\d}{\d\theta})$ on $\S^1\times P$ by the symplectic connection associated with $\omega'$. Indeed, let $\bar{M}$ be the mapping cylinder of $\phi$, i.e.\ $\bar{M}=M_{1}\times [0,1]/_{\sim}$ with $(x,1)\sim (\phi(x),0)$. We have an induced map $\bar{M}\ni (x,t)\mapsto (e^{2\pi\imath t},\pr_{P}\circ f(x))\in \S^1\times P$, and a diffeomorphism $\bar{\Phi}\colon \bar{M}\ni (x,t)\mapsto \Phi(x,t)\in M$. Since $\phi^{*}\omega_{1}=\omega_{1}$, the form $\pr_{M_1}^{*}\omega_{1}$ descends to a fiberwise symplectic form $\bar{\omega}$ on $\bar{M}$. Put $\omega'=(\bar{\Phi}^{-1})^{*}\bar{\omega}$. Since $\Phi$ is fiberwise a symplectic monodromy trivialization, its restrictions $\Phi_{t,p}\colon M_{1,p}\to M_{t,p}$ are symplectomorphisms. It follows that $\omega'_{t,p}=\omega_{t,p}$. The symplectic lift of $(0,\tfrac{\d}{\d \theta})$ from $P\times \S^1$ to $M$ is $\bar{\Phi}_{*}\frac{\d}{\d t}$, and its time one flow is $\phi$, as needed. 

Since the statement only concerns fibers we may assume $\omega=\omega'$ for the rest of the proof. The symplectic connection defines a lift $\mathcal{X}$ of the vector field $(0,\frac{\d}{\d\theta})$ to $F$, and the automorphism  $\beta\colon F\to F$ induced by $\phi$ is the time $1$ flow of $\mathcal{X}$. Since $\tilde{F}$ is a covering space of $F$, there is a unique lifting of $\mathcal{X}$ to a vector field $\mathcal{Y}$ in $\tilde{F}$, and $\tilde{\beta}$ is the time $1$ flow of $\mathcal{Y}$.
\end{proof}

\begin{cor}
	\label{cor:gradedunicitymon}
	Given a symplectic fibration $M\to\S^{1}$ and a grading $\tilde{F}\to M$ of the vertical symplectic frame bundle, any symplectic monodromy inherits a grading, and the graded symplectic monodromy is unique up to graded isotopy. Furthermore, in the situation of Proposition~\ref{prop:sufficientmonod}, a grading of the vertical symplectic frame bundle of $M\to P\times\S^{1}$ induces a grading of each symplectic monodromy $\phi_p$, and of each isotopy connecting them. 
\end{cor}

\subsubsection{Conley--Zehnder index}\label{sec:CZ}

Let $\phi\colon M\to M$ 
be a graded symplectomorphism. Let $x\in M$ be a fixed point of $\phi$. The grading $\tilde{\beta}\colon \tilde{\Fr}(M)\to \phi^{*}\tilde{\Fr}(M)$ sends the fiber $\tilde{\Fr}(M)_{x}$ over $x$ to itself. Its restriction $\tilde{\beta}_{x}$ is thus an element of $\tilde{\Sp}(2d)$, well defined up to a conjugation. It can be viewed as a homotopy class (relative to the ends) of a path $\gamma$ of symplectic matrices, starting from $\id$; well defined up to a conjugation. The \emph{Conley--Zehnder index of $x$}, denoted by $\CZ_{\phi}(x)$, is the Maslov index of $\gamma$, see \cite[\sec 4]{RS_CZ-index} for its definition. This number depends only on the conjugacy class of this path. Indeed, by \cite[Theorem 4.1]{RS_CZ-index} the Maslov index is additive under concatenation of paths, so it is invariant under conjugation. In Section \ref{sec:spectral-sequence} we will need the following, known result.

\begin{lema}\label{lem:CZ-Hamiltonian}
	Let $(M,\omega)$ be a graded symplectic manifold of dimension $2d$. Let $H\colon M\to \R$ be a time-independent Hamiltonian, and let $\psi^{H}_{1}\colon M\to M$ be its time-one flow, see Section \ref{sec:Hamiltonian}. Let $x$ be a Morse point of $H$, of Morse index $\Ind(x)$. Assume furthermore that $H$ is $\cC^{2}$-small at $x$, more explicitly, $||\operatorname{Hess}_{x}(H)||<2\pi$. Then the following hold. 
	\begin{enumerate}
		\item \label{item:CZ_Morse} The point $x$ is a fixed point of $\psi^{H}_{1}$ of Conley--Zehnder index $d-\Ind(x)$.
		\item \label{item:CZ_phi+Morse} Let $\phi\colon M\to M$ be a graded symplectomorphism such that $x$ is a fixed point of $\phi$. Then $x$ is a fixed point of $\phi\circ \psi_{1}^{H}$, of Conley--Zehnder index $\CZ_{\phi}(x)+d-\Ind(x)$.
	\end{enumerate}
\end{lema}
\begin{proof}
	Part \ref{item:CZ_Morse} is proved in \cite[Corollary 7.2.2]{AD_book}. Part \ref{item:CZ_phi+Morse} follows from \ref{item:CZ_Morse} since the Maslov index of a concatenation of paths is a sum of their Maslov indices \cite[Theorem 4.1]{RS_CZ-index}.
\end{proof}
%
%We now recall need to see how the Conley--Zehnder index changes under a small Hamiltonian deformation. Let again 
%
% Then $x$ is a fixed point of $\psi^{H}_1$.  Now by \cite[Corollary 7.2.2]{AD_book}, the Conley--Zehnder index of $x\in \Fix\psi^{H}_{1}$ equals
%\begin{equation}\label{eq:CZ_Morse}
%	\CZ(x)=d-\Ind(x).
%\end{equation}
%Note that the formula in loc.\ cit.\ differs from \eqref{eq:CZ_Morse} by a sign. As explained in Section \ref{sec:Hamiltonian}, this is because we use the opposite sign convention for definition of the Hamiltonian vector field. Therefore, to get the same $\psi_{1}^{H}$ we need to apply loc.\ cit.\ to $-H$, and the Morse index of $-H$ at $x$ is $2d-\Ind(x)$.
%
%Now, , so $x\in \Fix\check{\phi}$, too. Since the Maslov index of a concatenation of paths is a sum of their Maslov indices, see \cite[Theorem 2.3]{RS_CZ-index}, the Conley--Zehnder index $\CZ_{\check{\phi}}(x)$ of $x$ as a fixed point of $\check{\phi}$ is the sum of the Conley--Zehnder indices of $x$ as a fixed point of $\phi$ and $\psi^{H}_{1}$, so by \eqref{eq:CZ_Morse}
%\begin{equation}\label{eq:CZ_phi+Morse}
%	\CZ_{\check{\phi}}(x)=\CZ_{\phi}(x)+\CZ_{\psi_{1}^{H}}(x)=\CZ_{\phi}(x)+d-\Ind(x).
%\end{equation}

\subsection{Constructing symplectic monodromies}\label{sec:basic-setting}

The aim of this section is to give a general construction of symplectic monodromy for families of degenerations of complex manifolds over a punctured disc. Later, we will apply it to a $\mu$-constant family of isolated hypersurface singularities, see Example \ref{ex:mu_constant}; and a log resolution of a single isolated hypersurface singularity, see Section \ref{sec:isotopy-to-radius-zero}. Now, we work in the following general Setting \ref{basic-setting}.

\begin{setting}[see Figure \ref{fig:collars}]\label{basic-setting}
	Let $X$ be a complex variety (possibly singular). Let $P$ be a compact domain in a manifold, and let $(f_{p})_{p\in P}\colon X\to \C$ be a smooth family of regular maps. Put 
	\begin{equation*}
		F:X\times P\ni (x,p)\mapsto (f_{p}(x),p)\in \C\times P.
	\end{equation*}
	Assume that for some $\delta>0$, the restriction of $F$ over $\D_{\delta}^{*}\times P$ is a submersion. In  particular, all fibers $f_{p}^{-1}(z)$ for $z\in \D_{\delta}^{*}$, $p\in P$ are smooth and contained in the smooth part $X\reg$ of $X$.  Let $U\subseteq X$ be an open subset whose closure is compact and whose boundary is a manifold, and such that for every $p\in P$, we have $f_{p}^{-1}(0)\trans \d \bar{U}$. Assume that we have open subsets $V,W\subseteq U$ with the same transversality property, and such that $\overline{V}\subset W\subset\overline{W}\subset U$. 
	Put $C'=((\bar{U}\setminus V)\times P)\cap F^{-1}(\D_{\delta}\times P)$. As usual, we write $C'_{z,p}=(\bar{U}\setminus V)\cap f_{p}^{-1}(z)$ for $z\in \D_{\delta}$, $p\in P$, see Notation \ref{not:fibers}. 
	
	Assume that $\pr_{X}(C')\subseteq X\reg$, and that the restriction $F|_{C'}\colon C'\to \D_{\delta}\times P$ is a submersion. Thus for every $p\in P$, the fiber $C'_{0,p}$ is smooth  and contained in $X\reg$. Furthermore, we assume that the restriction of $F$ gives a locally trivial fibration of pairs
	\begin{equation}
		\label{eq:collarfibrationaux1}
		F\colon (C',(\d \bar{W}\times P) \cap C')   \to\D_\delta\times P.
	\end{equation}
	By Ehressmann Lemma, the latter condition can always be achieved by shrinking $\delta>0$. Now 
	for each $(z,p)\in\D_\delta\times P$, each connected component of $C'_{z,p} \setminus \d \bar{W}$ cannot meet $\d \bar{V}$ and $\d \bar{U}$ at the same time. %We call the ones meeting $\d \bar{V}$ the {\em inner components of the fibration~\eqref{eq:collarfibrationaux1} over $(z,p)$}. 
\smallskip
	
Let $\lambda\in \Omega^{1}(C' \cup F^{-1}(\D_{\delta}^{*}\times P))$ be a $1$-form whose restriction to each collar $C'_{0,p}$ and to each smooth fiber $f_{p}^{-1}(z)\cap \bar{U}$ for $z\in \D_{\delta}^{*}$, $p\in P$ is Liouville. Furthermore, we assume that for some $p_0\in P$, the  Liouville vector field of $\lambda|_{C_{0,p_0}'}$ points outwards $f^{-1}_{p_0}(0)\cap \d \bar{W}$. 
\end{setting}

\begin{figure}
	\begin{tikzpicture}[scale=0.7] 
	%	\path[draw, use as bounding box] (-6,-7.6) rectangle (8,4);  
		\path[draw, name path=c] (-5,2.5) to [out=0,in=180] (-4,2) to [out=0,in=180] (-3,2.5);
		\path[draw, name path=d] (-5,1.5) to [out=0,in=180] (-4,1) to [out=0,in=180] (-3,1.5);
		\tikzfillbetween[ of=c and d ] {pattern=vertical lines, pattern color=black!15};   
		\path[draw, name path=c] (1,2.5) to [out=0,in=180] (2,2) to [out=0,in=180] (3,2.7);
		\path[draw, name path=d] (1,1.5) to [out=0,in=180] (2,1) to [out=0,in=180] (3,1.7);
		\tikzfillbetween[ of=c and d ] {pattern=vertical lines, pattern color=black!15};;   
		\draw (-6,0) -- (8,0);
		\node[below left] at (8,0) {$\d \bar{V}\times P$};
		\draw (-6,2) -- (8,2);
		\node[below left] at (8,2) {$\d \bar{W}\times P$};
		\draw (-6,3.5) -- (8,3.5);
		\node[below left] at (8,3.5) {$\d \bar{U}\times P$};
		%    \draw[very thick] (-4,4) -- (-4,0) to[out=-90,in=0] (-4.7,-2) to[out=0,in=90]  (-4,-4);
		%    \draw[ultra thick] (-3,4) -- (-3,0) to[out=-90,in=90] (-3.4,-2) to[out=-90,in=90]   (-3,-4);
		%    \draw (-5,4) -- (-5,0) to[out=-90,in=90] (-5.4,-2) to[out=-90,in=90]   (-5,-4);
		\draw[->] (-4,-4.4) -- (-4,-5.8);
		\node[right] at (-4,-5.2) {\small{$f_{p_0}$}};
		\draw (-4,-6.6) ellipse (1 and 0.5);
		\filldraw (-4,-6.6) circle (1.5pt);   
		\filldraw (-3,-6.6) circle (2pt);
		\filldraw (-4,2) circle (2pt);
		\filldraw (-4,1) circle (2pt);
		\filldraw (2,2) circle (2pt);
		\filldraw (2,1) circle (2pt);   
		\node[right] at (-3,-6.6) {\small{$\delta$}};
		\draw[->] (2,-4.4) -- (2,-5.8);
		\node[right] at (2,-5.2) {\small{$f_{p}$}};
		\draw (2,-6.6) ellipse (1 and 0.5);
		\filldraw (2,-6.6) circle (1.5pt);   
		\filldraw (3,-6.6) circle (2pt);
		\node[right] at (3,-6.6) {\small{$\delta$}};
		\node at (-5.2,-6.2) {\small{$\D_\delta$}};   
		\draw[line width=2.5] (-4,2)--(-4,1);   
		\draw[line width=2.5] (2,2)--(2,1);   
		\draw[decoration=snake, decorate, segment length=30, dashed] (-3,2.5) -- (1,2.5);
		\draw[decoration=snake, decorate, segment length=30, dashed] (-3,1.5) -- (1,1.5);

		\draw[decoration=snake, decorate, segment length=30, dashed] (-6,2.5) -- (-5,2.5);
		\draw[decoration=snake, decorate, segment length=30, dashed] (-6,1.5) -- (-5,1.5);
		
		\draw[decoration=snake, decorate, segment length=30, dashed] (3,2.7) -- (5,2.5);
		\draw[decoration=snake, decorate, segment length=30, dashed] (3,1.7) -- (5,1.5);
		\node[left] at (-6.3,2) {$C$};
		\draw [decorate, decoration = {calligraphic brace}, very thick] (-6.2,1.5) --  (-6.2,2.5);
		\node[left] at (-6.3,0.8) {\small{$C_{0,p_0}$}};
		\draw [->] (-6.3,0.8) to[out=0,in=180] (-4.2,1.6);
		\draw [decorate, decoration = {calligraphic brace}, very thick] (-2.7,2.5) --  (-2.7,-4);
		\node[right] at (-2.7,-0.9) {$N_{\delta,p_0}$};
		\draw [decorate, decoration = {calligraphic brace}, very thick] (3.3,2) --  (3.3,-4);
		\node[right] at (3.4,-1.1) {$\bar{W}_{\delta,p}$}; 
		\draw [decorate, decoration = {calligraphic brace}, very thick] (8.2,3.5) --  (8.2,0);
		\node[right] at (8.3,1.75) {\small{$C'$}};  
		
		\draw[very thick] (2,4) -- (2,0) to[out=-90,in=0] (1.7,-1) to[out=0,in=90] (2,-2) -- (2,-2.5) to[out=-90,in=0] (1.7,-3) to[out=0,in=90]  (2,-4);
		\draw[ultra thick] (3,4) -- (3,0) to[out=-90,in=90] (2.6,-1.2) to[out=-90,in=90] (2.8,-2) to[out=-90,in=90] (2.6,-2.8) to[out=-90,in=90] (3,-4);
		\draw (1,4) -- (1,0) to[out=-90,in=90] (0.6,-1.2) to[out=-90,in=90] (0.8,-2) to[out=-90,in=90] (0.6,-2.8) to[out=-90,in=90] (1,-4);
		
		\draw[very thick] (-4,4) -- (-4,0) to[out=-90,in=0] (-4.3,-1.3) to[out=0,in=90] (-4,-2) -- (-4,-2.3) to[out=-90,in=0] (-4.3,-2.7) to[out=0,in=90]  (-4,-4);
		\draw[ultra thick] (-3,4) -- (-3,0) to[out=-90,in=90] (-3.4,-1.4) to[out=-90,in=90] (-3.2,-2) to[out=-90,in=90] (-3.4,-2.6) to[out=-90,in=90] (-3,-4);
		\draw (-5,4) -- (-5,0) to[out=-90,in=90] (-5.4,-1.4) to[out=-90,in=90] (-5.2,-2) to[out=-90,in=90] (-5.4,-2.6) to[out=-90,in=90] (-5,-4);

		\draw (-6,-7.5)-- (8,-7.5);
		\node[above left] at (8,-7.5) {$P$};
		\draw (-4,-7.6) -- (-4,-7.4);
		\node[below] at (-4,-7.5) {\small{$p_0$}};
		\draw (2,-7.6) -- (2,-7.4);
		\node[below] at (2.2,-7.5) {\small{$p\neq p_0$}};       
		
	\end{tikzpicture}
	\caption{Producing collar trivializations in Setting \protect\ref{basic-setting}.}
	\label{fig:collars}
\end{figure}
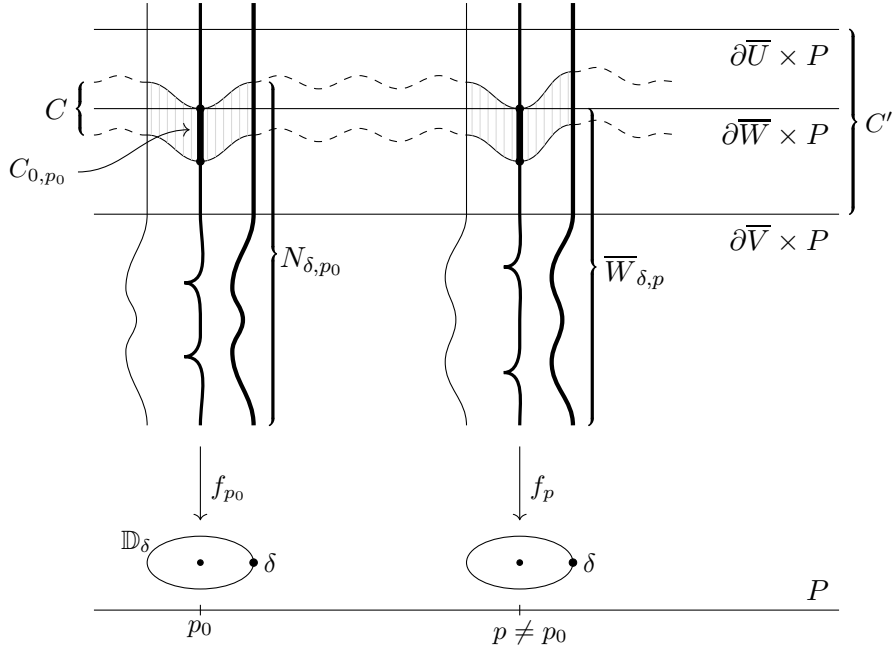

The reader may have in mind the following example: the ambient space is $X=\C^{n}$, the parameter space $P$ is, say, $\D_{\eta}$; the map $F=(f_t,t)\colon \C^n\times \D_\eta\to\CC\times\D_\eta$ is a family of holomorphic functions with isolated singularities and $V\subset W\subset U\subset \C^n$ are concentric balls, such that the only critical value of $f_t$ is $0$, and $f_t^{-1}(0)$ is smooth in $\bar{U}\setminus V$; and $\lambda$ is the standard form $\lambda=\pi \cdot \sum_{i}r_{i}^{2}\, d\theta_{i}$ from Example \ref{ex:J}.
\smallskip

The remaining part of Section \ref{sec:basic-setting} is organized as follows. In Section \ref{sec:basic-Liouville} we show how to adjust $\d W$ within $U\setminus V$ so that $F|_{W\times P}$ becomes a Liouville fibration over $\D^{*}_{\delta}\times P$, with a cohomologically trivial collar trivialization, see Definitions \ref{def:collartriv} and \ref{def:cohtriv}. This may require shrinking $\delta>0$ or the parameter space $P$. Next, in Section \ref{sec:basic-monodromy} we observe that under a mild assumption \eqref{eq:ass-exact}, the monodromy of this Liouville fibration, see Definition \ref{def:sympmonod}, makes the fibers into isotopic abstract contact open books, see Definition \ref{def:acob}. Next, in Section \ref{sec:basic-grading}, we endow those abstract contact open books with a grading, under another mild assumption \eqref{eq:vanishing-c1}. In Remarks \ref{rem:exact} and \ref{rem:c1} we comment that these additional  assumptions \eqref{eq:ass-exact} and \eqref{eq:vanishing-c1} are satisfied in most interesting cases. Eventually, in Section \ref{sec:basic-radius-zero} we extend each fibration to the A'Campo space constructed in Section \ref{sec:ACampo}, and,  under yet another assumption \eqref{eq:lambda_extends}, we use forms introduced in Section \ref{sec:symplectic} to get isotopic abstract contact open books at radius zero. In Section \ref{sec:basic-applications} we gather basic applications of this construction which will be used later.% relevant later in the article.}
	
\subsubsection{Constructing a Liouville fibration with a collar trivialization in Setting \ref{basic-setting}}\label{sec:basic-Liouville} \ 

We keep the notation and assumptions of Setting \ref{basic-setting}. 
Let $(r,\theta)$ be the polar coordinates in the disc $\D_\delta$. The radial vector field $\frac{\d}{\d r}$ in $\D_\delta^{*}$ does not extend to a smooth vector field at the origin, but it does extend if restricted to a ray  $R_{\theta}\subseteq \D_{\delta}$ of constant argument $\theta\in \S^1$. %Recall that $F|_{C'}\colon C'\to \D_{\delta}\times P$, with a form $d\lambda$, is a symplectic fibration. 
Let $\mathcal{X}_{\theta}$ be the symplectic lift of $(\frac{\d}{\d r})|_{R_{\theta}}$ to $F^{-1}(R_{\theta}\times P)\cap C'$.

Let $C_0$ be a compact neighborhood of $F^{-1}(\{0\}\times P)\cap(\d \bar{W}\times P)$ in $F^{-1}(\{0\}\times P)\cap (\bar W\setminus \bar V\times P)$. Its slice $C_{0,p}$ over $p\in P$ is a compact collar neighborhood of $\d \bar{W}\cap f_{p}^{-1}(0)$ in $\bar{W}\cap f_{p}^{-1}(0)$. Let $P'\subseteq P$ be a compact subdomain such that all  integral curves of $\mathcal{X}_\theta$ with initial points in $C_{0}\cap F^{-1}(\{0\}\times P')$ exist for all times $t\in [0,\delta]$ and are contained in $C'$; one can always take $P'=\emptyset$. 
Then after shrinking $P$ around $P'$, or after shrinking $\delta>0$ in case $P'=\emptyset$, we get the same property for $C_0\cap F^{-1}(\{0\}\times P)$. Here we use compactness of $C_0$ and $\D_\delta$ in the first case; and of $C_0$ and $P$ in the second case. This way, for each $\theta\in \S^1$ we obtain a smooth flow
\begin{equation*}
	\varphi_{\theta}\colon C_0\times [0,\delta]\ni ((x,p),t)\mapsto (\varphi_{t\cdot \theta,p}(x),p)\in X\times P.
\end{equation*}
The above formula defines, for each $z\in \D_{\delta}$, $p\in P$, a map $\varphi_{z,p}\colon C_{0,p}\to C_{z,p}'$. More explicitly, we put $\varphi_{0,p}=\id_{C_{0,p}}$ and $\varphi_{z,p}(x)=\varphi_{z/|z|}((x,p),|z|)$ for $z\in \D_{\delta}^{*}$. Since the flow $\varphi_{\theta}$ is defined by a symplectic connection, each map $\varphi_{z,p}$ is a symplectomorphism onto its image. Put
\begin{equation*}
	C\de \bigcup_{\theta\in\S^1}\varphi_{\theta}(C_0\times [0,\delta]),\quad
	\d_{\mathrm{out}}C\de \bigcup_{\theta\in\S^1}\varphi_{\theta}((C_0\cap (\d\bar{W}\times P))\times [0,\delta])
\end{equation*}
Note that $C_{0}=F^{-1}(\{0\}\times P)\cap C$. Now $F|_{C}\colon C\to \D_{\delta}\times P$ is a symplectic fibration, and
\begin{equation}\label{eq:collar-triv-center}
	\psi_{0}\colon C_{0}\times \D_{\delta}\ni ((x,p),z)\mapsto (\varphi_{z,p}(x),p)\in C
\end{equation}
is a $P$-fiberwise trivialization of $F|_{C}$ which is $(\D_{\delta}\times P)$-fiberwise symplectic. 
%
%Then, for a fixed $z_0\in \D_{\delta}^{*}$, the map
%\begin{equation}
%	\label{eq:asscolltriv1}
%	\psi\colon C_{z_0}\times\D_{\delta}^{*}\ni ((x,p),z)\mapsto (\varphi_{p,z}(\varphi_{p,z_0}^{-1}(x)),p)\in C
%\end{equation}
%is a $P$-fiberwise trivialization which is $(\D_\delta^{*}\times P)$-fiberwise symplectic. %On the other hand, after possibly shrinking $\delta>0$ the map
The map 
\begin{equation}
	\label{eq:collarfibrationaux1red}
	F\colon (C', \d_{\mathrm{out}}C )   \to\D_\delta\times P
\end{equation}
is a locally trivial fibration of pairs, which coincides with~\eqref{eq:collarfibrationaux1} over $\{0\}\times P$. In particular, it also has the property that for each $(z,p)\in\D_\delta\times P$, each connected component of $C'_{z,p}\setminus \d_{\mathrm{out}}C_{z,p}$ cannot meet $\d \bar{V}$ and $\d \bar{U}$ at the same time. We call those meeting $\d \bar{V}$ the {\em inner components of the fibration~\eqref{eq:collarfibrationaux1red} over $(z,p)$}. 
Define
\begin{equation}
	\label{eq:Liouv_fib1}
	F|_N\colon N\to \D_\delta^*\times P
\end{equation}
as the symplectic fibration whose fiber over $(z,p)$ is the union of $\overline{V}_{z,p}$ and the inner components of the fibration~\eqref{eq:collarfibrationaux1red} over $(z,p)$. 
Note that, by Ehresmann Lemma, the fiber $N_{z,p}=F^{-1}(z,p) \cap N$ is diffeomorphic to $\bar{W}_{z,p}\de f_{p}^{-1}(z)\cap \bar{W}$, and its diffeomorphism type does not depend on $(z,p)\in \D_{\delta}^{*}\times P$. 

Using \eqref{eq:collar-triv-center}, we get a collar trivialization of \eqref{eq:Liouv_fib1}, as follows. Fix a point $z_{0}\in \D_{\delta}^{*}$, let $C_{z_0}=F^{-1}(\{z_0\}\times P)\cap C$, and let $\hat{\varphi}_{z_0}\colon C_{z_0}\ni (x,p)\mapsto (\varphi_{z_0,p}^{-1}(x),p)\in C_0$ be the fiberwise identification of $C_{z_0}$ with $C_{0}$ via the trivialization \eqref{eq:collar-triv-center}. Now, the map 
\begin{equation}
	\label{eq:asscolltriv1}
	\psi\de \psi_0\circ (\hat{\varphi}_{z_0}\times \id_{\D_{\delta}^{*}})\colon C_{z_0}\times\D_{\delta}^{*}\to C
	%\ni ((x,p),z)\mapsto (\varphi_{p,z}(\varphi_{p,z_0}^{-1}(x)),p)\in C
\end{equation}
is a $P$-fiberwise collar trivialization of \eqref{eq:Liouv_fib1}, which is $(\D_\delta^{*}\times P)$-fiberwise symplectic. 

	\begin{lema}\label{lem:cohtrivial}
		The collar trivialization \eqref{eq:asscolltriv1} is  cohomologically trivial at every $p\in P$.
	\end{lema}	
	\begin{proof}
		By Definition \ref{def:cohtriv}, \eqref{eq:asscolltriv1} is cohomologically trivial at $p\in P$ if for every $z\in \D_{\delta}^{*}$, and every $z'\in \D_{\delta}^{*}$ sufficiently close to $z$, the class $[\eta_{z,z';p}]\in H^{2}(N_{z,p},\d N_{z,p};\R)$ is zero. This class was defined in the formula \eqref{eq:eta_Theta}, as follows. Let $U_{z}\subseteq \D_{\delta}^{*}$ be a neighborhood of $z$ such that the restriction  $F|_{N\cap F^{-1}(U_z\times P)}$ admits a $P$-fiberwise trivialization $\Theta$ which agrees with $\psi$ in the collar. For $z'\in U_{z}$ let $\Theta_{z',p}\colon N_{z,p}\to N_{z',p}$ be the corresponding map identifying the fibers. Then $[\eta_{z,z';p}]$ is the class of $\omega_{z,p}-\Theta_{z',p}^{*}\omega_{z',p}$.
		
		We need to show that, for every $\Delta\in H_{2}(N_{z,p},\d N_{z,p};\R)$, the integral
		%\begin{equation*}
		$
			\int_{\Delta}\omega_{z,p}-\Theta^{*}_{z',p}\omega_{z',p}
		$
		vanishes. By Stokes theorem, this is equivalent to vanishing of 
		$	
			 \int_{\Gamma} \lambda_{z,p}-\Theta_{z',p}^{*}\lambda_{z',p}
		$
		%\end{equation*}
		for every $1$-cycle $\Gamma$ on $\d N_{z,p}$.
		
		The condition that $\Theta$ agrees with $\psi$ in the collar means that $\Theta_{z',p}|_{C_{z',p}}= \psi_{z,p}\circ \psi^{-1}_{z',p}$, where $\psi_{z,p}\colon C_{z_0,p}\to C_{z,p}$ is the symplectomorphism induced by $\psi$. Using the formula \eqref{eq:asscolltriv1} defining $\psi$, we see that $\psi_{z,p}=\varphi_{z,p}\circ \varphi_{z_0,p}^{-1}$, so $\Theta_{z',p}|_{C_{z',p}}=\varphi_{z',p}\circ \varphi_{z,p}^{-1}$. In consequence, to prove the lemma it is enough to show the vanishing of the integral $\int_{\Gamma} \lambda_{0,p}-\varphi_{z,p}^{*}\lambda_{z,p}$ for every $1$-cycle $\Gamma$ on $\d \bar{W}_{0,p}$.
	
		Fix such a $1$-cycle $\Gamma$, and an argument $\theta\in \S^1$. Consider a function $\tau\colon [0,\delta]\ni r \mapsto \int_{\Gamma} \lambda_{0,p}-\varphi_{r\theta,p}^{*}\lambda_{r\theta,p}$. We need to prove that $\tau\equiv 0$. Since  $\varphi_{0,p}=\id_{\bar{W}_{0,p}}$, we have $\tau(0)=0$. It remains to prove that $\tau'(r)=0$ for all $r\in [0,\delta]$. Because $\varphi_{t\theta,p}$ is a time $t$ flow of $\cX_{\theta}$, we have $\tau'(r) =-\int_{\Gamma}\frac{d}{d t}\varphi_{t\theta,p}^{*}\lambda_{t\theta,p}|_{t=r}=-\int_{\varphi_{r\theta,p}(\Gamma)}L_{\cX_{\theta}}\lambda_{r\theta,p}$. By Cartan formula we have $L_{\cX_{\theta}}\lambda_{r\theta,p}=d(\lambda_{r\theta,p}(\cX_{\theta}))+\omega_{r\theta,p}(\cX_{\theta},\sdot)$. The restriction of $\omega_{r\theta,p}(\cX_{\theta},\sdot)$ to $\d N_{r\theta,p}$ is zero, because $\cX_{\theta}$ is $\omega$-orthogonal to the fibers of $F$. Hence the form $L_{\cX_{\theta}}\lambda_{r\theta,p}$ is exact, so its integral along a closed cycle is zero by Stokes theorem. Thus $\tau'(r)=0$, as needed.
	\end{proof}
	
	\begin{remark}\label{rem:cohtriv}
		Consider a special case when $W_{z,p}$ is Stein manifold of dimension $n-1$. This happens e.g.\ if $W$ is a ball in $\C^n$, see Example \ref{ex:Milnor-fibration}. For $n\geq 4$ the group 
		 $H^{2}(N_{z,p},\d N_{z,p})=H_{2n-4}(N_{z,p})=H_{2n-4}(W_{z,p})$ vanishes, so the cohomological triviality proved in Lemma \ref{lem:cohtrivial} holds automatically. 
		 
		 In case $n=2$ (e.g.\ for plane curves), Lemma \ref{lem:cohtrivial} reduces to the observation that all fibers $N_{z,p}$ have the same symplectic area. The latter holds because, by construction, the integral curves of $-\cX_{z/|z|}$ starting in $N_{z,p}$ do not escape the collar $C$, hence in time $|z|$ they map $N_{z,p}$ to the central fiber.
	\end{remark}

We now continue the general construction, and upgrade the symplectic fibration \eqref{eq:Liouv_fib1} to a Liouville one. Assume that there is a nonempty compact subdomain $P_0\subseteq P$, and a number $\delta_0\in [0,\delta]$, such that the Liouville vector field $X_{\lambda}$ of $\lambda$ points outwards $\d_{\mathrm{out}}C_{z,p}$ for all $p\in P_0$ and $z\in \D_{\delta_0}$. Recall that we have assumed it for $\d_{\mathrm{out}}C_{0,p_0}=f_{p_0}^{-1}(0)\cap \d\bar{W}$, so we can always take $P_0=\{p_0\}$, $\delta_0=0$; nonetheless, sometimes other choices will be convenient, too. Since $P_0$ and $\D_{\delta_0}$ are compact, after possibly 
%replacing $P$ with some neighborhood of $P_0$, and $\delta$ with some number in $(\delta_0,\delta]$, 
shrinking $P$ around $P_0$ and $\delta>\delta_0$, we can assume that $X_{\lambda}$ points outwards $ \d_{\mathrm{out}}C_{z,p}$  for all $(z,p)\in \D_{\delta}\times P$. Now the restriction \eqref{eq:Liouv_fib1} is  a Liouville fibration, with a cohomologically trivial collar trivialization \eqref{eq:asscolltriv1}.

\begin{remark}\label{rem:Seidel}
	In the language of \cite[\sec 1.1]{Seidel_exact-sequence}, cf.\ \cite[Remark 15.2]{Seidel_Fukaya-categories}, the restriction $F|_{N}$ is %a Liouville fibration admitting a collar trivialization: such fibrations are called 
	an \emph{exact symplectic fibration}. If $X$ is smooth then the restriction $F|_{\bar{N}}$, i.e.\ $F|_{N}$ together with the singular fiber, is an \emph{exact symplectic fibration with singularities} in the sense of \cite[\sec 15a]{Seidel_Fukaya-categories}.
\end{remark}

\subsubsection{Constructing symplectic monodromy in Setting \ref{basic-setting}}\label{sec:basic-monodromy}
%
%Having constructed a cohomologically trivial collar trivialization, we proceed to describe the associated monodromy abstract contact open books, see Definition \ref{def:monodromy_acob}.
%\smallskip
Our next goal is to introduce a monodromy making the fibers of \eqref{eq:Liouv_fib1} into an abstract contact open books. 

 To this end, we fix a family of loops $\Upsilon$ in $\D_{\delta}^{*}$. More precisely, we choose a domain $Q$ and a smooth map $\Upsilon\colon Q\times\S^1 \to \D_{\delta}^{*}$. The reader may keep in mind the example $\Upsilon\colon (0,\delta]\times\S^1\ni (r,\theta)\mapsto r\theta\in \D_{\delta}^{*}$. Define $N'$ as the fibered product of the Liouville fibration $N\to \D_{\delta}^{*}\times P$ given by \eqref{eq:Liouv_fib1} and the map $\Upsilon\times \id_{P}\colon Q\times \S^1\times P\to \D_{\delta}^{*}\times P$. Then the induced map $N'\to Q\times \S^{1}\times P$ together with the pullback of $\lambda|_{N}$ form a Liouville fibration, which inherits the fiberwise symplectic collar trivialization \eqref{eq:asscolltriv1} constructed above. By Lemma \ref{lem:cohtrivial}, the latter is cohomologically trivial. 
  
 By Proposition \ref{prop:sufficientmonod}, for every $q\in Q$, $p\in P$, and $\theta\in \S^1$ we have a symplectic monodromy trivialization $\Theta_{q,p}\colon N'_{q,\theta,p}\times [0,1]\to N_{q,p}'$, see Definition \ref{def:sympmonod}. %For $\theta\in \S^{1}$, the fiber $N'_{q,\theta,p}$ can be identified with $N_{z,p}$, where $z=\Upsilon(q,\theta)\in \D_{\delta}^{*}$. %Let $\phi_{z,p}$ be the associated symplectic monodromy. 
Now, we assume that
\begin{equation}\label{eq:ass-exact}
	\mbox{each associated symplectic monodromy } \phi_{q,\theta,p}\colon N'_{q,\theta,p}\to N'_{q,\theta,p} \mbox{ is exact}.
\end{equation}
If the condition \eqref{eq:ass-exact} holds then each triple
\begin{equation} \label{eq:basic-acob}
		(N'_{q,\theta,p},\lambda_{q,\theta,p},\phi_{q,\theta,p})
\end{equation}
is a \emph{monodromy abstract contact open book} of $\Theta_{q,p}$, see Definition \ref{def:monodromy_acob}. Its isotopy type does not depend on the choice of $q,\theta,p$. The Liouville domain $(N'_{q,\theta,p},\lambda_{q,\theta,p})$ is isomorphic to the fiber $(N_{z,p},\lambda_{z,p})$ for $z=\Upsilon(q,\theta)$. Its underlying manifold $N_{z,p}\subseteq f_{p}^{-1}(z)$ is diffeomorphic to $\bar{W}_{z,p}=f_{p}^{-1}(z)\cap \bar{W}$. %Clearly, the monodromy $\phi_{q,\theta,p}$ depends on the chosen loop. 

\begin{remark}\label{rem:exact} We now gather some natural examples when condition \eqref{eq:ass-exact} is satisfied, so the construction above yields monodromy abstract contact open books \eqref{eq:basic-acob}.
	\begin{parlist}
		\item\label{item:exact_H1} Clearly, \eqref{eq:ass-exact} holds if for some (hence, any) $z\in \D_{\delta}^{*}$, $p\in P$ we have $H^{1}(\bar{W}_{z,p},\d \bar{W}_{z,p};\R)=0$.
		\item\label{item:exact_Stein} Assume that $W_{z,p}$ is a Stein manifold of dimension $n-1\geq 2$. Then by Lefschetz duality  $H^{1}(\bar{W}_{z,p},\d \bar{W}_{z,p})=H_{2n-3}(W_{z,p})=0$ since $2n-3>n-1$, so condition \eqref{eq:ass-exact} holds by \ref{item:exact_H1}.
		\item \label{item:exact_link} Let $L_{p}\de \Int (C_{p}\cup N_{p})\subseteq X\reg\setminus \Sing f_{p}^{-1}(0)$ be the \enquote{link} of $\Sing f_{p}^{-1}(0)\subseteq X$. We claim that if $H^{1}(L_{p};\R)=0$ then condition \eqref{eq:ass-exact} holds.
		
		Indeed, the symplectic monodromies $\phi_{q,\theta,p}$ constructed above, together with $\id_{C_p}$, glue to a symplectomorphism $\phi\colon L_{p}\to L_{p}$. By assumption, the closed form $\phi^{*}\lambda_{p}-\lambda_{p}$ is exact. Hence its restricted pullbacks $\phi_{q,\theta,p}^{*}\lambda_{q,\theta,p}-\lambda_{q,\theta,p}$ are exact, too, i.e.\ the condition \eqref{eq:ass-exact} holds, as claimed.
		
		\item \label{item:exact_link-example} Part \ref{item:exact_link} can be applied in the following situation. Let $X$ be an algebraic surface with an isolated singularity, whose link $L$ is a rational homology sphere. Take $U$ small enough so that each $L_{p}$ defined above is homotopically equivalent to $L$. Now, we have $H^{1}(L_{p};\R)=0$, so condition \eqref{eq:ass-exact} holds by \ref{item:exact_link}.
		
		\item \label{item:exact_more-links} More generally, assume that $X$ is an algebraic surface whose all singularities have rational homology sphere links, and that $\lambda_{X}$ is defined everywhere on $X\reg$: for example, $X=\C^2$ and $\lambda=\lambda\std$. Using \ref{item:exact_link-example} we first construct compactly supported, exact monodromies around each singularity of $X$ and $f^{-1}(0)$ as above; and then we extend them by identity to the remaining part of the fiber. %Thus once again we obtain monodromy abstract contact open books \eqref{eq:basic-acob}.
		
		In case $P=\{\pt\}$, the same construction works for $f$ replaced by $f\circ h$, where $h$ is a log resolution restricting to an isomorphism away from the zero fiber and in the collar.
	\end{parlist}
\end{remark}

\subsubsection{Grading in Setting \ref{basic-setting}.}\label{sec:basic-grading}
Recall that under the assumptions of Setting \ref{basic-setting} and \eqref{eq:ass-exact}, we have constructed monodromy  abstract contact open books \eqref{eq:basic-acob}. Now, we endow them with a grading. 

Fix $p\in P$. Let $T_{p,\mathrm{vert}}\de \ker (f_{p})_{*} \subseteq T(X\setminus f_{p}^{-1}(0))$ be the vertical tangent bundle. Fix an almost complex structure $J_{p}$ on $T_{p,\mathrm{vert}}$ which is compatible with the fiberwise symplectic form $d\lambda_{p}$, and assume 
\begin{equation}\label{eq:vanishing-c1}
	c_{1}(T_{p,\mathrm{vert}}^{*})=0. 
\end{equation}
This first Chern class depends only on $d\lambda_{p}$, not on $J_{p}$, see \cite[Defintion 4.1.4]{McDuff_Salamon}.  
Let  $K_{p,\mathrm{vert}}$ be the top exterior power of $T_{p,\mathrm{vert}}^{*}$. Condition \eqref{eq:vanishing-c1} means that $K_{p,\mathrm{vert}}$ admits a nonvanishing $\cC^{\infty}$ section. By the correspondence \eqref{eq:gradings-trivializationsfiberwise}, its pullback induces a grading of the vertical symplectic frame bundle of the Liouville fibration $N'\to Q\times \S^1\times P$. By Corollary~\ref{cor:gradedunicitymon}, this grading makes each triple \eqref{eq:basic-acob} a graded abstract contact open book, whose graded isotopy class does not depend on the choice of $q,\theta$. Moreover, grading for one $p\in P$ determines uniquely the grading for all $p\in P$.

\begin{remark}\label{rem:c1} Like in Remark \ref{rem:exact}, we gather some examples where condition \eqref{eq:vanishing-c1} holds, so the monodromy abstract contact open books \eqref{eq:basic-acob} are naturally graded.
	\begin{parlist}
		\item\label{item:c1_H1} Clearly, \eqref{eq:vanishing-c1} holds if for some (hence, any) $z\in \D_{\delta}^{*}$, $p\in P$ we have $H^{2}(W_{z,p};\R)=0$.
		\item\label{item:c1_Kahler} Assume that $d\lambda_{p}$ is K\"ahler. Then we can take for $J_{p}$ the standard complex structure. Denoting by $K_{X_{p}^{*}}$ the canonical bundle of the complex manifold $X_{p}^{*}\de X\setminus f_{p}^{-1}(0)$, we get an isomorphism
		\begin{equation}\label{eq:K-vert}
			K_{X_{p}^{*}}=K_{p,\mathrm{vert}} \otimes f_{p}^{*}(T^{*}\D_{\delta}^{*}).
		\end{equation}
		Thus if $c_{1}(X_{p}^{*})=0$ then condition \eqref{eq:vanishing-c1} holds, and every $\cC^{\infty}$ section of $K_{X_{p}^{*}}$ gives rise to a grading of the monodromy abstract contact open books \eqref{eq:basic-acob}. %In this case, Conley--Zehnder index of their fixed points can be explicitly computed using \cite[Lemma A.8]{McLean}; an example of such computation will be given in Proposition \ref{prop:monodromy}\ref{item:B_i-CZ}.
	\end{parlist}
\end{remark} 

\subsubsection{Summary} Assume that in Setting \ref{basic-setting}, conditions \eqref{eq:ass-exact} and \eqref{eq:vanishing-c1} hold. Then (after possibly shrinking $\delta>0$ or $P$) the above procedure, say with $\Upsilon\colon (0,\delta]\times\S^1\ni (r,\theta)\mapsto r\theta\in \D_{\delta}^{*}$ defines, for all $(r,\theta,p)\in (0,\delta]\times \S^1\times P$, \emph{monodromy graded abstract open books} \eqref{eq:basic-acob}, all graded isotopic to each other. The underlying space $N_{r,\theta,p}$ of \eqref{eq:basic-acob} is diffeomorphic to the fiber $\bar{W}_{r\cdot \theta,p}$.

 Moreover, conditions \eqref{eq:ass-exact} and \eqref{eq:vanishing-c1} hold e.g.\ if $X$ is a Stein manifold with trivial canonical bundle, and $d\lambda$ is an exact K\"ahler form on $X$; e.g.\ $(X,\lambda)=(\C^{n},\lambda\std)$, see Remarks \ref{rem:exact}\ref{item:exact_Stein},\ref{item:exact_more-links} and \ref{rem:c1}\ref{item:c1_Kahler}.

\subsubsection{Extending monodromies to radius zero.}\label{sec:basic-radius-zero}

To get monodromy abstract contact open book \enquote{at radius zero}, we need to perform the above construction in Setting \ref{basic-setting2}, which slightly modifies Setting \ref{basic-setting}, replacing base $\C$ with  $\C_{\log}$. Below, we only consider the case which will be needed later, leaving straightforward generalizations (e.g.\ to more general families of loops) for the reader. 

For $\delta>0$ we denote by $\D_{\delta,\log}$ the preimage of $\D_{\delta}$ in $\C_{\log}$, and identify $\D_{\delta,\log}\setminus \d \C_{\log}$ with $\D_{\delta}^{*}$.

	\begin{setting}\label{basic-setting2}
		Let $X$ be a smooth complex manifold, let $f\colon X\to \C$ be a holomorphic function, and let $U_X,W_X,V_X\subseteq X$ and $C'_{X}=\bar{U}_{X}\setminus V_{X}$ be as in Setting \ref{basic-setting} for $P=\{\pt\}$.  	Let $\lambda_X\in \Omega^{1}(C'_{X}\cup f^{-1}(\D_{\delta}^{*}))$ be a $1$-form such that $d\lambda_{X}$ extends to a K\"ahler form on $X$. 
		
		Assume that the singular fiber $f^{-1}(0)$ is snc, and let $A$ be A'Campo space for $f$. It is endowed with a smooth structure introduced in Section~\ref{sec:AX-smooth} using some adapted atlas and a partition of unity. Let $\pi\colon A\to X$ and $f_A\colon A\to \C_{\log}$ be the smooth maps as in the diagram \eqref{eq:AX-diagram}, see Proposition \ref{prop:AXsmooth}. Put $U=\pi^{-1}(U_X)$, $W=\pi^{-1}(W_X)$, $C'=\pi^{-1}(C'_X)$. Assume that
		\begin{equation}\label{eq:lambda_extends}
			\mbox{the pullback }\pi^{*}\lambda_{X}\in \Omega^{1}(A\setminus \d A)\mbox{ extends to a $1$-form on }A.
		\end{equation} 
		Then for every $\delta',\epsilon>0$, formula \eqref{eq:lambdaACdelta} defines a $1$-form  $\lambda_{A}^{\delta',\epsilon}\in \Omega^{1}(A)$. Put $\omega_{A}^{\delta',\epsilon}=d\lambda_{A}^{\delta',\epsilon}$. 
	\end{setting}
	
	Our goal is to define, in the above Setting \ref{basic-setting2}, symplectic monodromy over $\d \C_{\log}$ with respect to $\omega_{A}^{\delta',\epsilon}$, whose associated abstract contact open book will be graded isotopic with monodromy abstract contact open books \eqref{eq:basic-acob} constructed above over $\d\D_{\delta}$; for $\epsilon,\delta>0$ small enough.
	\smallskip
	
	In Section \ref{sec:basic-Liouville} we have seen that for small enough $\delta>0$, formula \eqref{eq:collar-triv-center} defines a $\D_{\delta}$-fiberwise symplectic collar trivialization
		\begin{equation*}
			\psi_{X}\colon C_{X,0} \times \D_{\delta}\to C_{X},
		\end{equation*}
		where $C_{X,0}=C_{X}\cap f^{-1}(0)$. Let $f|_{N_{X}}\colon N_{X}\to \D_{\delta}^{*}$ be the corresponding Liouville fibration \eqref{eq:Liouv_fib1}. 
		
		In Section \ref{sec:basic-monodromy}, we took a family of loops $\Upsilon$. Now, for simplicity, we take  $\Upsilon\colon (0,\delta]\times \S^{1}\ni (r,\theta)\mapsto r\theta\in \D_{\delta}^{*}$. Assume that the induced monodromies are exact, i.e.\ condition \eqref{eq:ass-exact} holds. Then formula \eqref{eq:basic-acob} gives, for every $r\in (0,\delta]$ and $\theta\in \S^1$, isotopic monodromy abstract contact open books
		\begin{equation}\label{eq:basic-acob-standard}
			(N_{r,\theta},\lambda_{r,\theta},\phi_{r,\theta})
		\end{equation}
		The Liouville domain $(N_{r,\theta},\lambda_{r,\theta})$ is naturally identified with the fiber $N_{X}\cap f^{-1}(r\theta) $ equipped with the restriction of $\lambda_{X}$. Its  underlying space $N_{r,\theta}$ is diffeomorphic to $\bar{W}_X\cap f^{-1}(z)$ for any $z\in \D_{\delta}^{*}$.
		\smallskip
		
		By the assumptions of Setting \ref{basic-setting}, the special fiber $f^{-1}(0)$ is smooth along the collar $C_X$, so the preimage $C\de \pi^{-1}(C_{X})$ is a fibered product of the restriction  $f|_{C_X}\colon C_{X}\to \D_{\delta}$ with the natural map $\D_{\delta,\log}\to \D_{\delta}$. Thus the trivialization $\psi_{X}$ pulls back to a trivialization 
		\begin{equation*}
			\psi_{0}\colon C_{0}\times \D_{\delta,\log}\to C,
		\end{equation*}
		where $C_0\de C\cap f_{A}^{-1}(0,1)$ is the intersection of $C$ with a radius-zero fiber.
		
		By Proposition \ref{prop:omega}\ref{item:omega-symplectic} there is a neighborhood $U'$ of $\bar{U}\cap \d A$, and an $\epsilon_0>0$ such that for every $\delta'>0$ and $\epsilon\in (0,\epsilon_0]$, the form $\omega_{A}^{\delta',\epsilon}$ is fiberwise symplectic in $U'$. Thus after  shrinking $\delta>0$ %(hence $X$) 
		if needed, the restriction $f_{A}|_{\bar{U}\cap f_{A}^{-1}(\D_{\delta,\log})}\colon \bar{U}\cap f_A^{-1}(\D_{\delta,\log})\to \D_{\delta,\log}$, together with the form $\omega_{A}^{\delta',\epsilon}$, is a symplectic fibration for every $\delta'>0$ and $\epsilon\in (0,\epsilon_0]$. Fix one $\delta'>0$ and define a family of $1$-forms $\lambda\in \Omega^{1}(A\times (0,\epsilon_0] \cup C\times [0,\epsilon_0])$ by $\lambda_{\epsilon}=\lambda_{A}^{\delta',\epsilon}$. %We fix one $\epsilon\in (0,\epsilon_0]$ and $\delta'\in (0,\delta)$.
		
		By Remark \ref{rem:vanishing_restriction}, there is a $1$-form $\beta_X\in \Omega^{1}(C_{X,0})$ such that $\lambda_{\epsilon}|_{C\cap \d A}=\pi^{*}(\lambda_{X}|_{C_{X,0}}+\epsilon\beta_{X})$. Hence
		\begin{equation*}%\label{eq:collar-triv-pullback}
			%\psi_0|_{C_0\times \d\D_{\delta,\log}}\times \id_{[0,\epsilon_0]}\colon 
			C_0 \times \d\D_{\delta,\log}\times [0,\epsilon_0]%\to 
			\ni (x,\theta,\epsilon)\mapsto 
			(\psi_{0}(x,0,\theta),\epsilon)
			\in
			C\times [0,\epsilon_0]
		\end{equation*}
		is an $[0,\epsilon_0]$-fiberwise trivialization of $f_{A}|_{C\cap \d A}\times \id_{[0,\epsilon_0]}$, which is  $(\d\D_{\delta,\log}\times [0,\epsilon_0])$-fiberwise symplectic with respect to the form $d\lambda$. 
		
		Consider the symplectic lift $\mathcal{X}$ of the radial vector field $\tfrac{\d}{\d r}$ from $\D_{\delta,\log}$ to $C\times[0,\epsilon_0]$. Note that, unlike in Section \ref{sec:basic-Liouville}, now we do not need to restrict it to rays. By definition of $C_X$ in Section \ref{sec:basic-Liouville}, the flow lines of $\cX$ starting in $(C\cap \d A)\times\{0\}$ exist in $C'\times \{0\}$ for all times $t\in [0,\delta]$. Since $C\cap \d A$ and $[0,\delta]$ are compact, shrinking $\epsilon_0>0$ if needed, we infer that the same is true for the flow lines starting in $(C\cap \d A)\times [0,\epsilon_0]$. This means that the vector field $\mathcal{X}$ defines a flow
		\begin{equation*}
			%\varphi \colon 
			(C\cap \d A)\times [0,\epsilon_0]\times [0,\delta] \ni (x,\epsilon,t) \mapsto (\varphi_{\epsilon,t}(x),\epsilon)\in C' \times [0,\epsilon_0]
		\end{equation*}
		such that for every $\theta\in \S^{1}$, the restriction of $\varphi_{\epsilon,t}$ to $C \cap f_{A}^{-1}(0,\theta)$ is a symplectomorphism onto its image; with respect to the form $d\lambda_{\epsilon}$. The image of this restriction is contained in $C'\cap f_{A}^{-1}(t,\theta)$.
		
		Thus for some $\tilde{C}\subseteq C'\times [0,\epsilon_0]$, the restriction $(f_{A}\times \id_{[0,\epsilon_0]})|_{\tilde{C}}$ admits a $[0,\epsilon_0]$-fiberwise trivialization
		\begin{equation}\label{eq:colltriv2}
			C_0\times \D_{\delta,\log}\times [0,\epsilon_0]
			\ni
			(x,t,\theta,\epsilon)
			\mapsto
			(\varphi_{\epsilon,t}(\psi_0(x,0,\theta)),\epsilon)
			\in \tilde{C}.
		\end{equation}
	 	which is  ($\D_{\delta,\log}\times [0,\epsilon_0]$)-fiberwise symplectic. Note that this trivialization restricts to $\psi_0$ over $\epsilon=0$.
		
		%Recall that for $\epsilon>0$, the forms $\lambda_{A}^{\epsilon,\delta'}$ are fiberwise symplectic in $\bar{U}\cap \d A$, too. 
		Exactly as in Section \ref{sec:basic-Liouville}, we define $\d_{\mathrm{out}}\tilde{C}$ and, using inner components of the locally trivial fibration $f_{A}\times \id_{[0,\epsilon_0]}\colon (\tilde{C},\d_{\mathrm{out}}\tilde{C})\to \D_{\delta,\log}\times [0,\epsilon_0]$, we define a locally trivial fibration 
		\begin{equation}\label{eq:Liouv_fib2}
			(f_{A}\times \id_{[0,\epsilon_0]})|_{\tilde{N}}\colon \tilde{N}\to \D_{\delta,\log}\times [0,\epsilon_0]
		\end{equation}
		which agrees with $f_{A}|_{\bar{W}\cap \d A}\times \id_{[0,\epsilon_0]}$ over $\d\D_{\delta,\log}\times [0,\epsilon_0]$, and with the Liouville fibration \eqref{eq:Liouv_fib1} over $\D_{\delta}^{*}\times\{0\}$. Moreover, the fibration \eqref{eq:Liouv_fib2} is symplectic over $B\de (\D_{\delta,\log}\times (0,\epsilon_0])\cup(\D_{\delta}^{*}\times [0,\epsilon_0])$. %Denote the corresponding restrictions by $F'\colon \tilde{N}'\to \D_{\delta}^{*}\times [0,\epsilon_0]$ and $F''\colon \tilde{N}''\to \D_{\delta,\log}\times (0,\epsilon_0]$.
		\smallskip
		
		By definition of $C_{X}$, the Liouville vector field of $\lambda_{X}$ points outwards $\d_{\mathrm{out}} C_{X}$, hence the one for $\lambda_{0}=\pi^{*}\lambda_{X}$ points outwards $\tilde{N}_{r,\theta,0}$ for every $(r,\theta)\in \D_{\delta,\log}$. Since $\D_{\delta,\log}$ is compact, we can shrink $\epsilon_0>0$ so that the same is true for all $\epsilon\in [0,\epsilon_0]$. Now, the restriction of \eqref{eq:Liouv_fib2} over $B$ is a Liouville fibration, with a  fiberwise collar trivialization obtained from \eqref{eq:colltriv2} by moving the base point inside $\D_{\delta}^{*}$ like in the formula \eqref{eq:asscolltriv1}. Over $\D_{\delta}^{*}\times \{0\}$, this collar trivialization agrees with the one constructed for \eqref{eq:Liouv_fib1} in Section \ref{sec:basic-Liouville}. Exactly the same argument as in Lemma \ref{lem:cohtrivial} shows that the above collar trivialization is cohomologically trivial. 
		\smallskip
		
		As before, we assume that the symplectic monodromy associated to its restriction over each loop $\S^1\ni \theta\mapsto (r,\theta,\epsilon)\in B$ is exact. This condition is satisfied in the situations described in Remark \ref{rem:exact}. Eventually, we get monodromy abstract contact open books 
		\begin{equation}\label{eq:basic-acob-Clog}
		(\tilde{N}_{r,\theta,\epsilon},\lambda_{r,\theta,\epsilon},\phi_{r,\theta,\epsilon})	
		\end{equation}
		 for all $(r,\theta,\epsilon)\in B$, all isotopic to each other. Note that for $r>0$ and $\epsilon=0$, the abstract contact open book \eqref{eq:basic-acob-Clog} can be identified with \eqref{eq:basic-acob-standard}, since both are obtained using the same collar trivialization. In turn, for $(r,\theta)=(0,1)\in \d\D_{\delta,\log}$ and $\epsilon=\epsilon_0>0$, the abstract contact open book \eqref{eq:basic-acob-Clog} is equal to
		 \begin{equation}\label{eq:basic-radius-zero-acob}
		 	(F,\lambda,\phi)
		 \end{equation}
	 	where $F=\bar{W}\cap f_{A}^{-1}(0,1)$ is a radius-zero fiber, $\lambda=\lambda_{A}^{\delta',\epsilon_0}|_{F}$, and $\phi\colon F\to F$ is the time one flow of the monodromy vector field given locally by the formula \eqref{eq:monodromy-vector-field}.
%	 	\smallskip
%	 	
%	 	Moreover, if $c_{1}(X\setminus f^{-1}(0))=0$ then the natural grading of \eqref{eq:basic-acob} defined in Setting \ref{basic-setting} by a nonvanishing section of $K_{X\setminus f^{-1}(0)}$ induces a grading of  \eqref{eq:basic-radius-zero-acob}, too.
%\end{setting}
		
		We summarize the above construction in the following corollary, which will be used in Section \ref{sec:isotopy-to-radius-zero}.
		
		\begin{cor}\label{cor:isotopyisotopy-to-radius-zero}
			The abstract contact open books \eqref{eq:basic-acob} defined in Section \ref{sec:basic-monodromy} using a family of loops $\Upsilon\colon (0,\delta]\times \S^{1}\ni (r,\theta)\mapsto r\theta\in \D_{\delta}^{*}$, are isotopic to the \enquote{radius-zero} abstract contact open book \eqref{eq:basic-radius-zero-acob}. 
		\end{cor}
			
			Moreover, if $c_{1}(X\setminus f^{-1}(0))=0$ then the natural grading of \eqref{eq:basic-acob} defined in Remark \ref{rem:c1}\ref{item:c1_Kahler} by a nonvanishing section of $K_{X\setminus f^{-1}(0)}$ induces a grading of  \eqref{eq:basic-radius-zero-acob}, by Corollary \ref{cor:gradedunicitymon}.

\subsubsection{Applications} \label{sec:basic-applications} 
A typical situation where the above constructions apply is the following.

\begin{example}
	\label{ex:typical}
	Let $Z$ be a Stein manifold and let $\varrho\colon Z\to \R$ be an exhaustive strictly plurisubharmonic function. %, see Section \ref{sec:symplectic-intro}. %Consider the Liouville form $\lambda_Z:=-d^{c}\varrho$, see Section \ref{sec:symplectic-intro}. 
	Let $X\subseteq Z$ be a closed analytic subset whose all singularities are isolated. In case $\dim_{\C}X=2$, assume further that the links of all singularities of $X$ are rational homology spheres, so the assumptions of Remark \ref{rem:exact}\ref{item:exact_Stein} or \ref{item:exact_more-links} hold. Choose regular values $\xi_{V}<\xi_{W}<\xi_{U}$ of $\varrho|_{X}$ such that $\varrho^{-1}[\xi_{V},\xi_{U}]\cap X\subseteq X\reg$, and define $V=\varrho^{-1}(-\infty,\xi_V)$, $W=\varrho^{-1}(-\infty,\xi_W)$, $U=\varrho^{-1}(-\infty,\xi_U)$.% as the corresponding sub-level sets of $\varrho$. 
	
	Consider a family of $1$-forms $\lambda_{t}\in \Omega^{1}(X\reg)$ smoothly dependent on $t\in[0,\epsilon]$ for some $\epsilon\geq 0$, such that $\lambda_{0}=-d^{c}\varrho|_{X\reg}$. Recall from Section \ref{sec:symplectic-intro} that the Liouville vector field of $\lambda_0$ is the gradient field of $\varrho$, so it points outwards $\d \bar{W}$.
	\smallskip
	
	Let $f_s\colon X\to\CC$ be a family of holomorphic functions with isolated singularities in $X$, holomorphically depending on a parameter $s\in \D_\eta$ for some $\eta\geq 0$. Assume that 
	\begin{equation}\label{eq:assumption-only-0}
		\mbox{the only critical value of }f_s%\mbox{ restricted to } %\varrho^{-1}(-\infty,c)
		\mbox{ is }0. 
	\end{equation}
	In particular, all fibers $f_{s}^{-1}(z)$ for $z\neq 0$ are smooth and contained in $X\reg$.
	\smallskip

	Note that, since the form $d\lambda_{0}=-dd^{c}\varrho$ is K\"ahler, it restricts to a symplectic form on each fiber $f_{s}^{-1}(z)\cap \bar{U}$ for $z\neq 0$. Thus for some $\delta>0$, the restrictions $f_{s}|_{\bar{U}\cap f^{-1}(\D_{\delta}^{*})}$, for all $s\in \D_{\eta}$, are symplectic fibrations with respect to $d\lambda_{0}$. In case $\epsilon>0$, we can shrink $\epsilon>0$ so that the same is true for all $d\lambda_{t}$, $t\in [0,\epsilon]$: here we use compactness of $\bar{U}$, $\D_{\eta}$ and $\D_{\delta}$. Now, our situation fits into Setting \ref{basic-setting}, with $P=\D_{\eta}\times [0,\epsilon]$ and $F\colon X\times P\ni (x,s,t)\mapsto (f_{s}(x),s,t)\in \C\times P$. 
	
	The construction in Section \ref{sec:basic-Liouville} shows that we can shrink $\delta>0$ and $\epsilon,\eta>0$ (in case they are nonzero) so that the map $F$ restricts to a Liouville fibration $F|_{N}\colon N\to \D_{\delta}^{*}\times P$ defined in \eqref{eq:Liouv_fib1}, with a collar trivialization \eqref{eq:asscolltriv1}. By construction, the special fiber $\bar{N}_{0,s,t}$ equals $\bar{W}_{0,s,t}$ for all $s\in \D_{\eta}$, $t\in [0,\epsilon]$. 
	
	In the construction of Section \ref{sec:basic-monodromy}, we take the standard family of loops $\Upsilon\colon (0,\delta]\times \S^1\ni (r,\theta)\mapsto r\theta\in \D_{\delta}^{*}$. Recall that condition \eqref{eq:ass-exact} holds by Remark \ref{rem:exact}\ref{item:exact_Stein},\ref{item:exact_more-links}. Thus for every $z\in \D_{\delta}^{*},s\in \D_{\eta},t\in[0,\epsilon]$ formula \eqref{eq:basic-acob} defines isotopic abstract contact open books
	\begin{equation}\label{eq:basic-acob-example}
		(N_{z,s,t},\lambda_{z,s,t},\phi_{z,s,t}).
	\end{equation}
	Here $N_{z,s,t}$ is a subset of $\bar{U}\cap f_{s}^{-1}(z)$ diffeomorphic to $\bar{W}\cap f_{s}^{-1}(z)$, and $\lambda_{z,s,t}$ is the restriction of $\lambda_{t}$. %Note that in case $\dim_{\C}X=2$ we have used the assumption on links of $X$ to apply Remark \ref{rem:exact}, which guarantees exactness of the monodromy.
	
	If furthermore $c_{1}(X\reg)=0$ then by Remark \ref{rem:c1}\ref{item:c1_Kahler}, a choice of a nonvanishing section of $K_{X\reg}$ endows all abstract contact open books \eqref{eq:basic-acob-example} with a grading, making them graded isotopic.% to each other.
\end{example}

\begin{remark}
	\label{rem:muconstant}
	In the situation of Example~\ref{ex:typical}, if $X$ is smooth and near each critical point of $f_0$ the family $f_s$ has constant Milnor number, then condition \eqref{eq:assumption-only-0} is satisfied.
\end{remark}

Example \ref{ex:typical} covers the case of $\mu$-constant families considered in Theorem \ref{theo:Zariski}. We will make it precise in Example \ref{ex:mu_constant} below. To do this, we first recall some known facts about the Milnor fibration.

Let $f\colon (\C^n,0)\to (\C,0)$ be a holomorphic germ with an isolated singularity. By \cite[Corollary 2.9]{Milnor}, there is a radius $\epsilon_0>0$ such that the ordinary sphere $\d \B_{\epsilon}$ of radius $\epsilon$ centered at $0$ is transversal to $f^{-1}(0)$ for any $\epsilon\in (0,\epsilon_0]$. Any $\epsilon \in (0,\epsilon_0]$ is called {\em a Milnor radius}, and $\B_\epsilon$ is called a {\em Milnor ball}. 

Fix a Milnor radius $\epsilon>0$ for $f$. There exists $\delta>0$ such that $f|_{\B_\epsilon\cap f^{-1}(\D_\delta^*)}\colon \B_\epsilon\cap f^{-1}(\D_\delta^*)\to \D_\delta^*$ is a locally trivial fibration. Its restriction over $\d \D_\delta$ is, up to a diffeomorphism, independent on the choices \cite[Theorem 2.10]{Milnor}. It is called the \emph{Milnor fibration} of $f$, see formula \eqref{eq:milfibtbintro} in the introduction. 

\begin{example}
	\label{ex:Milnor-fibration}
	Let $f\colon (\C^{n},0)\to (\C,0)$ be a germ of isolated hypersurface singularity. Put $\varrho(z)=\frac{\pi}{2}||z||^2$, and let $\lambda=-d^{c}\varrho$ be the standard Liouville form, see Example \ref{ex:J}. Let $\xi_V<\xi_W<\xi_U$ be Milnor radii for $f$. Then Example \ref{ex:typical} (with $\epsilon=\eta=0$) gives a subset $N\subseteq \B_{\xi_U}$ and a $\delta>0$ such that the restriction $f|_{N}\colon N\to \D_{\delta}^{*}$, is a Liouville fibration, diffeomorphic to $f|_{\B_{\xi_U}\cap f^{-1}(\D_\delta^*)}$. Moreover, for each $z\in \D_{\delta}^{*}$, we get a monodromy abstract contact open book $(N_{z},\lambda_{z},\phi_z)$. Using the standard holomorphic volume on $\C^{n}$, we endow each $(N_z,\lambda_z,\phi_z)$ with a grading.%It is graded by the  standard holomorphic volume form on $\C^{n}$.}
	
	We remark that, if $n\geq 3$, then by \cite[Theorem 6.4]{Milnor} the Milnor fiber is simply connected, so  by the correspondence \eqref{eq:gradings-subgroups} the above grading of $(N_{z},\lambda_{z},\phi_{z})$ %introduced above using a holomorphic volume form on $\C^n$ %, introduced above using triviality of $K_{\C^n}$, 
	is %in fact 
	the unique one.
\end{example}

\begin{example}
	\label{ex:mu_constant}
	More generally, let $\C^{n}\times \D_{\eta}\ni (x,s)\mapsto f_{s}(x)\in \C$ be a $\mu$-constant family of holomorphic functions with isolated singularities. As before, take $\varrho(z)=\frac{\pi}{2}||z||^2$, and $\lambda=-d^{c}\varrho$. 
	
	Let $\xi_V<\xi_W<\xi_U$ be Milnor radii for $f_0$. Then for $\delta,\eta>0$ small enough, formula \eqref{eq:basic-acob-example} in  Example \ref{ex:typical} (with $\epsilon=0$) defines graded isotopic abstract contact open books $(N_{z,s},\lambda_{z,s},\phi_{z,s})$ for $z\in \D_{\delta}^{*}$, $s\in \D_{\eta}$. For $s=0$, they agree with the ones constructed in Example \ref{ex:Milnor-fibration} for $f_0$.
\end{example}

%For $s=0$, the graded abstract contact open book $(N_{z,0},\lambda_{z,0},\phi_{z,0})$ is graded isotopic to the one constructed in Example \ref{ex:Milnor-fibration} for the Milnor fibration of $f_0$. However, this might no longer be the case for $s\neq 0$. Indeed, as we have explained in the introduction, the Milnor radii chosen in Example \ref{ex:mu_constant} for $f_0$ are not necessarily Milnor radii for $f_s$, $s\neq 0$. In fact, it is a long-standing, open question if a Milnor radius for $f_s$ can be chosen independently of $s$; even modifying the distance function $\varrho$, see \cite{OShea-vanishing-folds}. 
Note that the Milnor radii chosen in Example \ref{ex:mu_constant} for $f_0$ are not necessarily Milnor radii for $f_s$, $s\neq 0$. In fact, it is a long-standing, open question if a Milnor radius for $f_s$ can be chosen independently of $s$; even modifying the distance function $\varrho$, see \cite{OShea-vanishing-folds}. 

In particular, the graded abstract contact open books produced in Example \ref{ex:mu_constant} for $f_0$, and the analogous ones produced using a Milnor radius for $f_s$, might not be isotopic. 
To prove Theorem \ref{theo:Zariski}, we will compare their Floer homology. A crucial ingredient of this comparison will be the fact that they differ by a homologically trivial cobordism. More precisely, the following result is known, see Figure \ref{fig:family}.

\begin{prop}[\cite{Le-Ramanujam}]
	\label{prop:homologtriv}
	Let $\C^{n}\times \D_{\eta}\ni (x,s)\mapsto f_{s}(x)\in \C$ be a $\mu$-constant family of isolated hypersurface singularities. Let $\epsilon_0$ a Milnor radius for $f_0$. For any $s$ small enough and any Milnor radius $\epsilon_s<\epsilon_0$ for $f_s$, the cobordism 
	$\overline{f_s^{-1}(0)\cap \B_{\epsilon_0}\setminus \B_{\epsilon_s}}$
	is homologically trivial, that is, we have the vanishing 
	\begin{equation*}
		H_*(\overline{f_s^{-1}(0)\cap \B_{\epsilon_0}\setminus \B_{\epsilon_s}},f_s^{-1}(0)\cap\d \B_{\epsilon_0};\ZZ)=0.
	\end{equation*}
\end{prop}

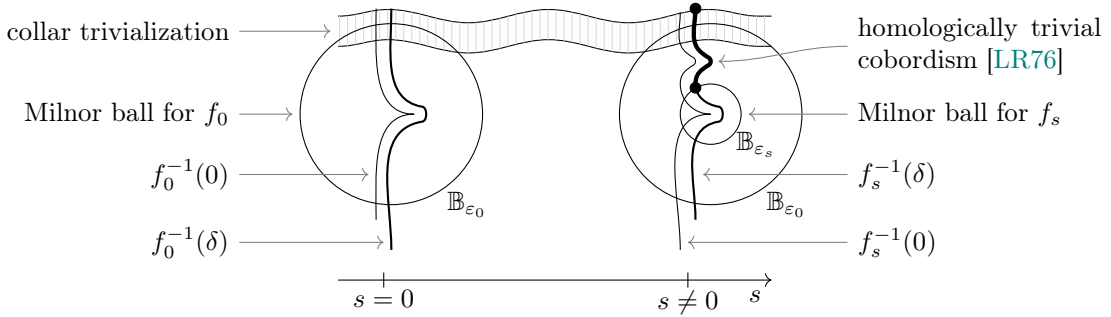
\begin{figure}[htbp]
	\begin{tikzpicture}
\path[use as bounding box] (-6,-2.3) rectangle (10,1.5);
		\draw[->] (-.7,-2.2) -- (5,-2.2); 
		\draw (-0.1,-2.1) -- (-0.1,-2.3);
		\node[below] at (-0.1,-2.2) {\small{$s=0$}};
		\draw (3.9,-2.1) -- (3.9,-2.3);
		\node[below] at (3.9,-2.2) {\small{$s\neq 0$}};
		\node[below left] at (5,-2.2) {\small{$s$}};		
		\draw[decoration=snake, decorate, segment length=60, name path=a] (-.7,1.3) -- (5,1.3);
		\draw[decoration=snake, decorate, segment length=60, name path=b] (-.7,0.9) -- (5,0.9);
		\tikzfillbetween[ of=a and b ] {pattern=vertical lines, pattern color=black!15};
		\draw (-0.2,1.4) to [out=-90,in=180] (0.3,0) to [out=180,in=90] (-0.2,-1.4);
		\draw[thick] (0,1.4) to [out=-90,in=170] (0.4,0.1) to [out=0,in=0] (0.4,-0.1) to [out=-170,in=90] (0,-1.8);
		\draw (0,0) circle (1.2);
		\draw (3.8,1.35) to [out=-90,in=145] (3.9,.8) 
		to [out=-45,in=90] (4,.7) to [out=-90,in=45] (3.9,.6) to [out=-145,in=90] (3.8,.4) 
		to [out=-90,in=180] (4.2,0) to [out=180,in=90] (3.8,-1.8);
		\draw[thick] (4,1.4) to [out=-90,in=145] (4.1,.8) 
		to [out=-45,in=90] (4.2,.7) to [out=-90,in=45] (4.1,.6) to [out=-145,in=90] (4,.4) 
		to [out=-90,in=170] (4.3,0.1) to [out=0,in=0] (4.3,-0.1) to [out=-170,in=90] (4,-1.4);
		\draw[ultra thick] (4,1.4) to [out=-90,in=145] (4.1,.8) 
		to [out=-45,in=90] (4.2,.7) to [out=-90,in=45] (4.1,.6) to [out=-145,in=90] (4,.4); 
		\filldraw (4,1.4) circle (2pt);	
		\filldraw (4,.35) circle (2pt);				
		\draw  (4.2,0) circle (.4);
		\draw (4.2,0) circle (1.2);
		\node at (1,-1.2) {\small{$\B_{\epsilon_0}$}};
		\node at (5.2,-1.2) {\small{$\B_{\epsilon_0}$}};
		\node at (4.8,-0.4) {\small{$\B_{\epsilon_s}$}};
		\draw [ ->, gray] (6,0) to [out=180,in=0] (4.7,0);
		\node[right] at (6,0) {\small{Milnor ball for $f_{s}$}};
		\draw [ ->, gray] (6,.9) to [out=180,in=0] (4.3,.7);
		\node[right] at (6,.9) {\parbox{3.2cm}{\small{homologically trivial cobordism \cite{Le-Ramanujam}}}};
		\draw [ ->, gray] (6,-0.8) to [out=180,in=0] (4.1,-.8);
		\node[right] at (6,-0.8) {\small{$f_{s}^{-1}(\delta)$}};
		\draw [ ->, gray] (6,-1.7) to [out=180,in=0] (3.9,-1.7);
		\node[right] at (6,-1.7) {\small{$f_{s}^{-1}(0)$}};
		\draw [ ->, gray] (-2,1.1) to [out=0,in=180] (-0.8,1.1);
		\node[left] at (-2,1.1) {\small{collar trivialization}};		
		\draw [ ->, gray] (-2,-0.8) to [out=0,in=180] (-0.3,-0.8);
		\node[left] at (-2,-0.8) {\small{$f_{0}^{-1}(0)$}};
		\draw [ ->, gray] (-2,-1.7) to [out=0,in=180] (-0.1,-1.7);
		\node[left] at (-2,-1.7) {\small{$f_{0}^{-1}(\delta)$}};
		\draw [ ->, gray] (-2,0) to [out=0,in=180] (-1.3,0);
		\node[left] at (-2,0) {\small{Milnor ball for $f_{0}$}};					
	\end{tikzpicture}
	\caption{$\mu$-constant family, see Example \ref{ex:mu_constant}}
	\label{fig:family}
\end{figure}

\subsection{Codimension zero families of fixed points}\label{sec:0codim}
We end this section by introducing a notion of \emph{codimension zero families of fixed points}, slightly extending the one introduced by McLean in \cite[p.\ 978]{McLean}. In Proposition \ref{prop:spectral-sequence} we will see that if all fixed points of $\phi$ come in such families, then their topology gives information about Floer homology groups $\HF_{*}(\phi,+)$. In Section \ref{sec:monodromy} we will see that this is the case for the \enquote{radius zero} monodromy on the boundary of the A'Campo space.
%
%For definition of manifolds with corners we refer to Section \ref{sec:Liouville-fibrations} or \cite[Appendice]{corners}.

\begin{definition}%[{\cite[p.\ 978]{McLean}}]
	\label{def:0codim}
	Let $(M,\lambda,\phi)$ be an abstract contact open book. 
	A {\em codimension zero family of fixed points} $B\subseteq M$ is a compact, connected codimension zero submanifold of $M$ with corners, such that there exist an open neighborhood $N_B$ of $B$ in $M$ and a time-independent Hamiltonian $H_{B}\colon N_{B}\to \R$ satisfying
	\begin{equation*}
		B=H_{B}^{-1}(0)=\Fix \phi|_{N_{B}}\quad\mbox{and}\quad \phi|_{N_{B}}=\psi^{H_{B}}_{1}|_{N_{B}}\colon N_{B}\to N_{B}.
	\end{equation*}
	The action function $\ac$ associated with $(M,\lambda,\phi)$ is constant in $B$, and we denote its value by $\ac(B)$. If $(M,\lambda,\phi)$ is graded then the function $\CZ$ is also constant in $B$, and we write $\CZ(B)$ for its value.
\end{definition}

Put $\d^{M}B=B\cap \d M$. Since $\phi=\id$ near $\d M$, the set $\d^{M}B$ is either empty, or a union of some connected components of $\d M$. 

Since $B=H_{B}^{-1}(0)$ is a codimension zero submanifold with corners, at every point $p\in \d B\setminus \d^{M}B$ there is a chart $U_{p}\subseteq N_{B}$ such that $H_{B}|_{U_{p}\setminus B}$ is either positive or negative. Denote by $\d^{+}B$ (resp.\ $\d^{-}B$) the set of those points $p$ such that $H_{B}|_{U_{p} \setminus B}>0$ (resp.\ $<0$). This way, we get a decomposition 
\begin{equation}
	\label{eq:boundarysplitting1}
	\partial B=\partial^M B\sqcup\partial^+ B \sqcup \partial^- B.
\end{equation}
where each factor is a union of connected components of $\d B$. 

\begin{remark}\label{rem:McLean-has-no-d+}
	Definition \ref{def:0codim} generalizes the one given in \cite[p.\ 978]{McLean}. Namely, a codimension zero family $B$ of fixed points satisfies the definition of loc.\ cit.\ if $\d^{-}B=\emptyset$, i.e.\ if the Hamiltonian $H_{B}$ is nonnegative (we warn the reader that, due to the opposite sign convention explained in Section \ref{sec:Hamiltonian}, the Hamiltonian in loc.\ cit.\ is nonpositive). As a result of this more general definition,  our version of McLean's axiom (HF3), stated in Proposition \ref{prop:spectral-sequence}, involves relative homology. This more general setting is needed to treat the additional fixed point component $B_0$, discussed in Section \ref{sec:intro_summary}.
\end{remark} 

In practice, codimension zero families of fixed points will have corners. For example, in Figure \ref{fig:v3} they are the quadrants $\bar{A_{j}^{\circ}}$, see Proposition \ref{prop:monodromy}. The following lemma allows to get rid of  these corners.%Nonetheless, we can avoid them using the following lemma.

\begin{lema}
	\label{lem:cornerelimination}
	Let $(M,\lambda,\phi)$ be an abstract contact open book, and let $B\subseteq M$ be a codimension zero family of fixed points. Then there is an abstract contact open book $(M,\lambda,\phi_1)\sim (M,\lambda,\phi)$ such that $\Fix \phi_1=B_1\sqcup (\Fix\phi\setminus B)$, where $B_1$ is a codimension zero family of fixed points, which is a manifold with boundary (but no corners), and the triple $(B_1,\d^{+}B_1,\d^{-}B_1)$ is homeomorphic to $(B,\d^{+}B,\d^{-}B)$.
\end{lema}
\begin{proof}
	Let $H_{B}\colon N_{B}\to \R$ be the time independent Hamiltonian from Definition \ref{def:0codim}. Then $H_{B}|_{N_B\setminus B}$ has no critical points, so $H_{B}\colon N_B\setminus B\to (-\delta,\delta)$ is a submersion for some $\delta>0$ (if $\d^{-}B$ or $\d^{+}B$ is empty, the image $(-\delta,\delta)$ should be replaced by $[0,\delta)$ or $(-\delta,0]$, respectively). By assumption, the Hamiltonian flow $\psi_{s}^{H_B}$ has no $1$-periodic orbits in $N_B$. Reducing $\delta$ (hence shrinking $N_{B}$) we can assume that it has no $t$-periodic orbits for any $t\in (0,4]$. 
	
	Define a smooth family of functions $\rho_t:[-\delta,\delta]\to [-\delta,\delta]$, parametrized by $t\in [0,1]$, such that $\rho_0=\id$, $\rho_t$ vanishes identically in $[-\frac{1}{4}t\delta,\frac{1}{4} t\delta]$, equals the identity outside $[-\frac{1}{2}\delta,\frac{1}{2}\delta]$ and is strictly increasing away from $[-\frac{1}{4}t\delta,\frac{1}{4} t\delta]$ with derivative bounded by $4$. For each fixed $t\in [0,1]$ define a time independent Hamiltonian $H^{t}\de  \rho_t\comp H_{B}\colon N_{B}\to (-\delta,\delta)$.
	\smallskip
	
	Let $\phi_{t}\de \psi_{1}^{H^t}\colon N_{B}\to N_{B}$ be its time one flow. Recall that $\phi|_{N_B}$ is the time one flow of the Hamiltonian $H_{B}=H^{0}$, so $\phi_0=\phi|_{N_{B}}$, and $\phi_{t}$ agrees with $\phi$ off  $H_{B}^{-1}([-\frac{1}{2}\delta,\frac{1}{2}\delta])\subsetneq N_B$, so we can extend $\phi_{t}$ to an exact symplectomorphism $\phi_{t}\colon M\to M$, equal to $\phi$ away from $N_{B}$. In particular, $\phi_{t}=\id$ near $\d M$. Thus $(M,\lambda,\phi_{t})$ is an isotopy of abstract contact open books, and $\Fix\phi_{t}\setminus N_{B}=\Fix\phi \setminus B$.
	
	Put $B_t=H_{B}^{-1}([-\frac{1}{4}t\delta,\frac{1}{4} t\delta])$. Because $\pm\frac{1}{4}t\delta$ are regular values of $H_{B}$, the subset $B_t\subseteq M$ is a submanifold with boundary, of codimension zero. We claim that $B_t=\Fix\phi_{t}|_{N_B}$. Clearly, $B_t\subseteq \Fix\phi_t$. Outside $B_t$, the Hamiltonian vector field associated with $H^t$ coincides with the Hamiltonian vector field of $H_{B_i}$, multiplied by the derivative of $\rho_t$, which is bounded by $4$. The claim follows because, by assumption, the time $s$ Hamiltonian flow $\psi^{H_{B}}_s$ has no fixed points in $N_{B}\setminus B$ for $s\in (0,4]$.
	
	Therefore, $\Fix \phi_{t}=B_{t}\sqcup (\Fix \phi\setminus N_{B})$, and $B_{t}$ is a codimension zero family of fixed points, without corners. It remains to see that $(B_{t},\d^{+}B_{t},\d^{-}B_{t})$ is homeomorphic to $(B,\d^{+}B,\d^{-}B)$, for $t>0$ small enough. Write $N_{B}=N_{B}^{+}\sqcup N_{B}^{-}\sqcup B$, where $H_{B}>0$ on $N_{B}^{+}$ and $H_{B}<0$ on $N_{B}^{-}$. Define a smooth vector field $X_{1}$ as $-\nabla H_{B}$ on $N_{B}^{+}$, $\nabla H_{B}$ on $N_{B}^{-}$, and $0$ on $B$. Fix a vector field $X_0$ on $N_{B}$, pointing inwards $\d^{+} B \sqcup \d^{-}B$, see \cite[Appendice, \sec 2]{corners}. Since $X_{1}$ vanishes on $B$, we can choose $t>0$ small enough so that $X\de X_0+X_1$ does not vanish on a neighborhood $N$ of $\bar{B_{t}\setminus B}$, and points inwards both $\d^{+}B_{t}\sqcup \d^{-}B_{t}$ and $\d^{+}B\sqcup \d^{-}B$. As in \cite[Appendice, Proposition 6.1]{corners}, we conclude that both $N\cap B_{t}$ and $N\cap B$ are homeomorphic to $(\d^{+}B_{t}\sqcup \d^{-}B_{t})\times [0,\infty)$, and we can extend this homeomorphism to the required homeomorphism of triples.
\end{proof}

\section{The McLean spectral sequence in fixed point Floer homology}\label{sec:spectral-sequence}

In this section, we introduce fixed point Floer homology $\HF_{*}(\phi,+)$ of an abstract contact open book $(M,\lambda,\phi)$, following Uljarevic and McLean \cite{Uljarevic,McLean}. This definition is a variant of \cite{Dostoglou_Salamon}, cf.\ \cite{Her}. It was first introduced in a slightly different way in Seidel's thesis \cite{Seidel_thesis}, see \cite[\sec 4]{Seidel_more}.

After some preparations in Section \ref{sec:completions}, we introduce Floer complexes in Section \ref{sec:floer}. In Proposition \ref{prop:spectral-sequence}, we construct a spectral sequence converging to $\HF_{*}(\phi,+)$ in case $\Fix\phi$ is a disjoint union of codimension zero families of fixed points. This spectral sequence is a slight generalization of the axiom \cite[(HF3)]{McLean}, proved in Appendix C of loc.\ cit. Our arguments are essentially the same, but we give a less compressed presentation for convenience of those readers who are not experts in Floer theory.

\subsection{Completions of Liouville domains}\label{sec:completions}

A \emph{contact manifold} $(N,\alpha)$ is a compact manifold $N$ of dimension $2k-1$ with a $1$-form $\alpha\in \Omega^{1}(N)$ such that $\alpha\wedge (d\alpha)^{k-1}\neq 0$. The \emph{Reeb vector field} $R_{\alpha}$ is defined by $\alpha(R_{\alpha})=1$, $d\alpha(R_{\alpha},\sdot)=0$. 
\smallskip

Let $(M,\lambda)$ be a Liouville domain. Then $(\d M,\lambda|_{\d M})$ is a contact manifold. For $\eta>0$, the backwards flow of the Liouville vector field $X_\lambda$ defines a diffeomorphism onto its image
\begin{equation}
	\label{eq:necktriv}
	\Psi\colon (-\eta,0]\times\partial M\to M
\end{equation}
such that $\Psi^*\lambda=e^r\lambda|_{\partial M}$, where $r$ is a coordinate on $(-\eta,0]$. The {\em completion} of $(M,\lambda)$, see \cite[p.\ 239]{CE_from-Stein-to-Weinstein} or \cite[\sec 1.1]{Uljarevic}, is a manifold $\hat{M}=(M\sqcup (-\eta,\infty)\times\partial M)/_{\sim}$, where $\sim$ identifies $(-\eta,0]\times\partial M$ with its image under $\Psi$, together with a Liouville form  defined as  $\lambda$ in $M$ and  $e^r\lambda|_{\partial M}$ in the \emph{cylindrical end} $(-\eta,\infty)\times\partial M$. Abusing notation, we denote this $1$-form by $\lambda$, too. 

Let $\frac{\d}{\d r}$ be the vector field corresponding to the coordinate $r$ of the cylindrical end $(-\eta,\infty)\times \d M\subseteq \hat{M}$. The restriction $\lambda|_{\{r\}\times\partial M}$ is a contact form, and we denote by $R_\lambda$ its Reeb vector field. Therefore, $R_\lambda$ is a vector field on $(-\eta,\infty)\times\partial M\subseteq \hat M$, which is tangent to the level sets of the coordinate $r$. 

An almost complex structure $J$ on $M$ or $\hat{M}$ is {\em cylindrical} if there is $\epsilon>0$ such that for any $r\geq -\epsilon$, the restriction $J|_{\{r\}\times \d M}$ leaves invariant the contact distribution $\ker(\lambda|_{\{r\}\times M})$ and satisfies $J(\frac{\d}{\d r})=R_\lambda$. This notion coincides with \cite[Definition 4.1]{McLean} and \cite[(2.7), (2.8)]{Uljarevic}.

Fix $a\in \R$. A (time-dependent) Hamiltonian $H_t\colon M\to\RR$ is of {\em slope} $a$ if there exists $\epsilon>0$ such that $H_t\comp\Psi(r,y)=ar$ for $r\geq -\epsilon$, where $\Psi$ is the trivialization~(\ref{eq:necktriv}), see \cite[p.\ 977]{McLean} or \cite[(2.6)]{Uljarevic}. A Hamiltonian of slope $a$ is extended to the completion $\hat{M}$ by the formula $H_t(r,y)=ar$ at $(-\epsilon,\infty)\times \partial M$. At the locus where $H_t(r,y)=ar$ the Hamiltonian flow at time $1$ coincides with the Reeb flow of $R_{\lambda}$ at time $a$. A nonzero slope $a$ is {\em small} if $|a|$ is strictly smaller than the minimal period of a Reeb orbit. In this case, $\psi^{H}_{1}$ has no fixed points in the cylindrical end of $\hat{M}$.

\subsection{The Floer complexes}\label{sec:floer}

Let $(M,\phi,\lambda)$ be a graded abstract contact open book, see Definition \ref{def:acob} and Section \ref{sec:graded-acobs}. We now define the \emph{$\phi$-twisted loop Floer complex} $\CFlp_{*}$ and the \emph{$\check{\phi}$-fixed point Floer complex} $\CFpt_{*}$, using the Floer Cauchy--Riemann equation in a perturbed  \eqref{eq:floerequlj} and unperturbed \eqref{eq:floereqmclean} version, respectively. The first setting is adapted from \cite{Uljarevic}, and mirrors \cite{Dostoglou_Salamon} in the compact case, which in turn specifies to the classical Floer theory for $\phi=\id$. The second one is used in \cite{McLean}: here $\check{\phi}$ is a small Hamiltonian perturbation of $\phi$. It is known how to pass from one setting to the other: we outline the method in Section \ref{sec:floerequivalence}. In particular, both complexes have the same homology $\HF_{*}(\phi,+)$.  Nonetheless, to derive the spectral sequence in Proposition \ref{prop:spectral-sequence}, it will be convenient to use both approaches at different parts of the argument. %This is the reason why we explain both approaches.

\subsubsection{Periodic almost complex structures and Hamiltonians}

Let $(M,\lambda,\phi)$ be a graded abstract contact open book. Extending $\phi$ to the identity in $[0,\infty)\times \d M$ we obtain a compactly supported exact symplectomorphism $\hat M\to\hat M$, which we denote by $\phi$, too. 

A family of almost complex structures $J_t$ on $M$ or $\hat{M}$, smoothly parametrized by $t\in \R$, is {\em $\phi$-periodic} if $J_{t+1}=\phi^*J_t$, see \cite [(2.5)]{Uljarevic}. A (time dependent) Hamiltonian $H_{t}$ is {\em $\phi$-periodic} if it satisfies the equality $H_{t+1}=H_t\comp\phi$, see \cite [(2.4)]{Uljarevic}. A {\em $\phi$-twisted loop} in $M$ is a path $\gamma:\RR\to M$ such that $\phi(\gamma(t+1))=\gamma(t)$, see \cite[\sec 2.2]{Uljarevic}. 

We warn the reader that in \cite[\sec 2.2]{Dostoglou_Salamon} the periodicity conditions are defined in an opposite way, i.e.\ $\gamma(t+1)=\phi(\gamma(t))$, etc. Also, the symbols $t,s$  play the opposite roles in \cite{Dostoglou_Salamon} than in the discussion below, which follows \cite{Uljarevic}.
%\smallskip

%\subsubsection{The Conley--Zehnder index of a Hamiltonian twisted loop}
Let $H$ be a $\phi$-periodic Hamiltonian. It satisfies a basic equality
\begin{equation}\label{eq:basic}
	\psi_{t+1}^{H}=\phi^{-1}\circ \psi_{t}^{H}\circ \phi\circ \psi_{1}^{H},
\end{equation}
cf.\  \cite[p.\ 586]{Dostoglou_Salamon}. 
To prove \eqref{eq:basic}, note that $\psi_{t+1}^{H}=\psi_{t}^{H'}\circ \psi_{1}^{H}$, where $H_{t}'(x)=H_{t+1}(x)$. By $\phi$-periodicity, $H_{t}'=H_{t}\circ \phi$, so $\psi_{t}^{H'}=\phi^{-1}\circ \psi_{t}^{H}\circ \phi$, see \cite[Exercise 3.1.14(iii)]{McDuff_Salamon}. This proves \eqref{eq:basic}.
\smallskip

 We denote by $\cP_{\phi,H}$ the space of $\phi$-twisted loops which are \emph{Hamiltonian}, i.e.\ of the form $t\mapsto \psi_{t}^{H}(x)$ for some $x\in M$. %satisfy $\gamma(t)are integral curves of $H$. 
We have a one-to-one correspondence
\begin{equation}\label{eq:corresponence_Fix-P}
	\cP_{\phi,H}\longleftrightarrow \Fix(\phi \circ \psi_{1}^{H})
	\quad\mbox{given by}\quad
	\gamma\mapsto \gamma(0).
\end{equation}
Indeed, if $\gamma\in \cP_{\phi,H}$ then $(\phi\circ\psi_{1}^{H})(\gamma(0))=\phi(\gamma(1))=\gamma(0)$, so $\gamma(0)\in \Fix \phi\circ\psi_{1}^{H}$. For the converse, let $x\in \Fix (\phi\circ \psi_{1})$ and 
define $\gamma(t)=\psi_{t}^{H}(x)$. We have %Using \eqref{eq:basic}, we compute
\begin{equation*}
	\phi(\gamma(t+1))=\phi\circ \psi_{t+1}^{H}(x)\overset{\mbox{\tiny{\eqref{eq:basic}}}}{=}\phi\circ\phi^{-1}\circ \psi_{t}^{H}\circ \phi \circ \psi_{1}^{H}(x)=
	\psi_{t}^{H}\circ \phi\circ \psi_{1}^{H}(x)=\psi_{t}^{H}(x)=\gamma(t),
\end{equation*}
where the first and last equality follow from the definition of $\gamma$, and the second-to-last holds because $x\in \Fix (\phi\circ \psi_1^H)$. Thus $\gamma$ is a $\phi$-twisted loop, as needed.
\smallskip

Note that if a $\phi$-periodic Hamiltonian $H$ has small negative (or positive) slope, then all fixed points of $\phi\comp \psi^{H}_1$ are away from $\partial M$, and away from the cylindrical end of of $\hat{M}$.

\subsubsection{The Conley--Zehnder index of a Hamiltonian twisted loop}\label{sec:CZ-loop}

Fix a Hamiltonian $\phi$-twisted loop $\gamma\in \cP_{\phi,H}$. By \eqref{eq:corresponence_Fix-P}, the point $\gamma(0)$ is fixed by the symplectomorphism $\phi\circ \psi_{1}^{H}$. The latter has a unique grading, induced from the grading of $\phi$ by the isotopy $(\psi_{t}^{H})$. We define $\CZ(\gamma)$ as the Conley--Zehnder index of $\gamma(0)\in \Fix(\phi\circ\psi_{1}^{H})$, see Section \ref{sec:CZ}.

This definition lifts the relative Conley--Zehnder index introduced in \cite[Definition 2.10]{Uljarevic} to an absolute one. To see this, let us recall how $\CZ$ was defined in Section \ref{sec:CZ}. Since $\gamma$ is an integral curve of the Hamiltonian flow $\psi^{H}_{t}$, for any time $t\in [0,1]$ the differential of $\psi^{H}_{t}$ at time $t$ induces a mapping $\Fr(M)_{\gamma(0)}\to \Fr(M)_{\gamma(t)}$, which lifts to a mapping $\tilde{\psi}^{H}_t\colon \tilde{\Fr}(M)_{\gamma(0)}\to\tilde{\Fr}(M)_{\gamma(t)}$. The grading of $\phi$ defines a mapping $\tilde{\beta}_{\gamma(1)}\colon \tilde{\Fr}(M)_{\gamma(1)}\to\tilde{\Fr}(M)_{\gamma(0)}$. So, composing we obtain a mapping 
\begin{equation*}
	\tilde\beta_{\gamma(1)}\circ \tilde{\psi}^{H}_{1}\colon \tilde{\Fr}(M)_{\gamma(0)}\to \tilde{\Fr}(M)_{\gamma(0)}.
\end{equation*}
which gives rise to a homotopy class of paths in the symplectic group relative to its ends as in Section \ref{sec:CZ}. The Conley--Zehnder index $\CZ(\gamma)$ is now the Maslov index of this path. 

Given two frames $f_0\in \Fr(M)_{\gamma(0)}$ and $f_1\in \Fr(M)_{\gamma(1)}$ such that %$\tilde{\beta}_{\gamma(1)}(f_1)=f_0$, 
$(\psi_{1}^{H})_{*}f_0=f_1$, 
any homotopy class (relative to its ends) of frames $f_t\in \Fr(M)_{\gamma(t)}$ connecting $f_0$ and $f_1$ determines a map $\tilde{\Fr}(M)_{\gamma(1)}\to\tilde{\Fr}(M)_{\gamma(0)}$. Let us call the homotopy class {\em admissible} if the map equals $\tilde{\beta}_{\gamma(1)}$.  

Fix $\gamma_1\in \cP_{\phi,H_{1}}$, $\gamma_2\in \cP_{\phi,H_{2}}$ for two Hamiltonians $H_1$, $H_2$, and fix $u\colon \RR\times\RR\to M$ satisfying $\phi(u(s,t+1))=u(s,t)$, $u(s,t)=\gamma_1(t)$ for $s\leq 0$ and $u(s,t)=\gamma_2(t)$ for $s\geq 1$. Then the grading  induces, via the choice of a homotopy of admissible classes of framings along $u(s,\sdot)$ for any $s$, a symplectic trivialization $T_{s,t}$ of $u^*TM$ such that $\phi^*T_{s,t+1}=T_{s,t}$. Now, the relative grading of $(\gamma_1,\gamma_2)$ is defined in \cite[Definition 2.10]{Uljarevic} as the difference between the Maslov indices $\mu(\Psi^2)-\mu(\Psi^1)$, where $\Psi^{i}\colon t\mapsto T_{i,t}^{-1}\circ (\psi_{t}^{H})^{*}T_{i,0}\in \Sp(2n)$. The path  $\Psi^{i}$ represents the homotopy class used above to define $\CZ(\gamma_{i})$, as claimed.

\subsubsection{The $\phi$-twisted loop Floer complex}\label{sec:perturbed} Now, we introduce the fixed point Floer homology $\HF_{*}(\phi,+)$ using $\phi$-twisted loops, following \cite{Uljarevic}.

Recall from Definition~\ref{def:acob} that associated with an abstract contact open book $(M,\lambda,\phi)$ we have an action function $\ac\colon M\to\RR$, defined (up to a constant) by $\phi^{*}\lambda-\lambda=-d\ac$. We fix a $\phi$-periodic Hamiltonian $H$ of small \emph{negative} slope. Taking $H$ of small \emph{positive} slope would lead to a definition of another Floer homology group, $\HF_{*}(\phi,-)$, which is not relevant for us: hence for the sake of clarity, we focus on $\HF_{*}(\phi,+)$, and leave the analogous discussion of $\HF_{*}(\phi,-)$ to the reader. The relation between $\HF_{*}(\phi,+)$ and $\HF_{*}(\phi,-)$ is outlined in \cite[Remark 4.1]{Seidel_more}. 
	
The \emph{action functional} on the space of $\phi$-twisted loops is given in \cite[Definition 2.6]{Uljarevic} by the formula
\begin{equation}
	\label{eq:actiondef}
	\cA_{\phi,H}(\gamma)\de -\int_0^1(\gamma^*\lambda+H_t(\gamma)\, dt) + \ac(\gamma(1)).
\end{equation}
Note that if $\gamma\in \cP_{\phi,H}$ is a constant path, and $H$ is time-independent, then the image of $\gamma$ is a fixed point of $\phi$, say $p$, and $\cA_{\phi,H}(\gamma)=-H(p)+\ac(p)$.

The set of critical points of $\cA_{\phi,H}$ equals exactly $\cP_{\phi,H}$, see \cite[Lemma 2.7]{Uljarevic}. A $\phi$-periodic Hamiltonian $H$ is {\em nondegenerate} if $\det(D(\phi\comp\psi^{H}_1)(x)-\id)\neq 0$ at any fixed point $x$ of $\phi\comp\psi^{H_t}_1$, see \cite[Definition 2.5]{Uljarevic}. If a $\phi$-periodic Hamiltonian $H$ is nondegenerate then the set $\cP_{\phi,H}$ is finite.
\smallskip

Let $J_t$ be a family of cylindrical, $(d\lambda)$-compatible and $\phi$-periodic almost complex structures on $\hat{M}$.  Fix $\gamma_{-},\gamma_{+}\in \cP_{\phi,H}$. A \emph{Floer trajectory} between $\gamma_{-}$ and $\gamma_{+}$ is a smooth map $u\colon \RR^2\to \hat M$ satisfying
\begin{equation}
	\label{eq:boundulj}
	\phi(u(s,t+1))=u(s,t),\quad 
	\lim_{s\rightarrow -\infty} u(s,t)=\gamma_{-}(t),\quad 
	\lim_{s\rightarrow +\infty} u(s,t)=\gamma_{+}(t)
\end{equation}
and the \emph{perturbed Floer Cauchy--Riemann equation}:
\begin{equation}
	\label{eq:floerequlj}
	\frac{\partial u}{\partial s}+J_t(u)\left(\frac{\partial u}{\partial t}-X^{H}(u)\right)=0,
\end{equation}
see \cite[(2.10)]{Uljarevic}. Whenever additional precision is needed, we refer to such a Floer trajectory as a \emph{$\CFlp_{*}(\phi,H_t,J_t)$-Floer trajectory from $\gamma_{-}$ to $\gamma_{+}$}. Note that if $u(s,t)$ is such a Floer trajectory then so is $u(c+s,t)$ for any $c\in\RR$. So, the additive group $\RR$ acts on the Floer trajectories.

\begin{definition}[{\cite[2.5]{Uljarevic}}]
	\label{def:perturbedregularpair}
	A pair $(H_t,J_t)$ is {\em regular for $\phi$} if $H_t$ is nondegenerate and the linearization of the operator \eqref{eq:floerequlj} for every Floer trajectory is surjective.
\end{definition}
It is known that a generic pair $(H_t,J_t)$ is regular. In particular, regular pairs exist, cf.\ \cite[Theorem 8.1.1]{AD_book} or \cite[p.\ 316]{Seidel_Fukaya_structures_II}.

The \emph{$\phi$-twisted loop Floer complex} associated with a graded abstract compact open book $(M,\lambda,\phi)$ and a regular pair $(H_t,J_t)$ is the $\ZZ_2$-vector space $\CFlp_{*}(\phi,H_t,J_t)$ generated by the Hamiltonian $\phi$-twisted loops, graded by $-\CZ$, whose differential $\d$ is defined as follows. Given $\gamma_+,\gamma_-\in \cP_{\phi,H}$ such that $\CZ(\gamma_+)=\CZ(\gamma_-)+1$, the coefficient of $\gamma_+$ in $\partial(\gamma_-)$ is the number (mod $2$) of Floer trajectories between $\gamma_{-}$ and $\gamma_{+}$, up to the additive action of $\RR$. By \cite[Theorem~2.18]{Uljarevic}, this differential is well defined and its square is zero, so $\CFlp_{*}$ is indeed a chain complex. By \cite[Corollary 2.22 and Lemma 2.25]{Uljarevic}, the homology groups of this complex are independent of the choice of a regular pair $(H_t,J_t)$, up to unique isomorphism, as long as $H$ remains of small negative slope. The reader should note that we have improved the relative grading used by Uljarevic to the absolute grading defined in Section \ref{sec:CZ-loop} using the grading of $\phi$ as in \cite{McLean}; the proofs of \cite{Uljarevic} work in this setting.  

As a consequence of the above independence, we can associate to each graded abstract contact open book $(M,\lambda,\phi)$ a $\Z$-graded Floer homology group $\HF_{*}(\phi,+)$, given by $H$ of small negative slope. The sign \enquote{$+$} comes from the notation of \cite{McLean}, where, as explained in Section \ref{sec:Hamiltonian}, Hamiltonians have opposite signs. Taking $H$ of small \emph{positive} slope would lead to $\HF_{*}(\phi,-)$.

\subsubsection{The action filtration} \label{sec:action-filtration}
Let $\CFlp_{*}(\phi,H_t,J_t)$ be the $\phi$-twisted loop Floer complex introduced in Section \ref{sec:perturbed} for a regular pair $(H_t,J_t)$. For any $a\in \R$, let $F^{\leq a}  \CFlp_{*}(\phi,H_t,J_t)$ be the subspace of $\CFlp_{*}(\phi,H_t,J_{t})$ spanned by those Hamiltonian $\phi$-twisted loops $\gamma\in \cP_{\phi,H}$ whose action satisfies $\cA_{\phi,H}(\gamma)\leq a$. It is a subcomplex of $\CFlp_{*}(\phi,H_t,J_{t})$, see \cite[\sec 2.6]{Uljarevic}. To see this, recall that for $\gamma_{+},\gamma_{-}\in \cP_{\phi,H}$, the coefficient of $\gamma_{+}$ in the expression of $\d(\gamma_{-})$ is the count (mod $2$) of the Floer trajectories connecting $\gamma_{-}$ with $\gamma_{+}$. By  \cite[p.\ 867]{Uljarevic}, these trajectories are anti-gradient flow lines of the action functional $\cA_{\phi,H}$. In particular, the differential $\d$ decreases the value of the action, so it maps $F^{\leq a}  \CFlp_{*}(\phi,H_t,J_t)$ to itself, as claimed.

\subsubsection{The $\check{\phi}$-fixed point Floer complex}
\label{sec:unperturbed}
Now we review the definition of Floer homology given in \cite{McLean}. In Section \ref{sec:floerequivalence} we will see that it is equivalent to the one recalled in Section \ref{sec:perturbed}.

Let again $(M,\lambda,\phi)$ be a graded abstract compact open book. Consider a $\cC^{2}$-small $\phi$-periodic Hamiltonian $H$ which is of small negative slope. Assume that all fixed points of $\check{\phi}\de \phi\comp\psi^{H}_1$ are nondegenerate (hence there is finitely many such points). Let $J_t$ be a family of cylindrical, $(d\lambda)$-compatible and $\check{\phi}$-periodic almost complex structures on $M$. For $p_{-},p_{+}\in \Fix \check{\phi}$, a \emph{Floer trajectory} (or, more precisely: a \emph{$\CFpt_{*}(\check{\phi},J_{t})$-Floer trajectory}) from $p_{-}$ to $p_{+}$ is a smooth map $v\colon \RR^2\to M$ satisfying
\begin{equation}
	\label{eq:boundmclean}
	\check{\phi}(v(s,t+1))=v(s,t),\quad 
	\lim_{s\rightarrow -\infty} v(s,t)=p_{-},\quad 
	\lim_{s\rightarrow +\infty} v(s,t)=p_{+}
\end{equation}
and the \emph{unperturbed Floer Cauchy--Riemann equation}:
\begin{equation}
	\label{eq:floereqmclean}
	\frac{\partial v}{\partial s}+J_t(v)\left(\frac{\partial v}{\partial t}\right)=0,
\end{equation}
see \cite[Definition 4.3]{McLean}.

We define the \emph{$\check{\phi}$-fixed point Floer complex $\CFpt_{*}(\check{\phi},J_t)$} as the $\ZZ_2$-vector space freely generated by the fixed points of $\check{\phi}$, graded by the minus Conley--Zehnder index, and equipped with the following differential: given $p_+,p_-\in \Fix \check{\phi}$ such that $\CZ(p_+)=\CZ(p_-)+1$, the coefficient of $p_+$ in $\partial(p_-)$ is the number of Floer trajectories between $p_{-}$ and $p_{+}$, up to the additive action of $\RR$ in the $s$ variable of $u$. As we will see in Section \ref{sec:floerequivalence} below, for generic choice of $J_t$ and $H$, this differential is well defined, its square is zero and the resulting homology groups coincide with  $\HF_{*}(\phi,+)$ introduced in Section \ref{sec:perturbed} above. In particular, they do not depend on the choice of the almost complex structure $J_t$, and of the Hamiltonian $H$ of small negative slope, cf.\ \cite[p.\ 980]{McLean}.

\subsubsection{Equivalence of the two Floer complexes}\label{sec:floerequivalence} Here we explain the equivalence of Uljarevic's and McLean's definitions of the Floer complex, recalled in Sections \ref{sec:perturbed} and \ref{sec:unperturbed} above. 

Let $H$ be a $\phi$-periodic Hamiltonian and let $(J_t)$ be a family of cylindrical, $(d\lambda)$-compatible, $\phi$-periodic almost complex structures such that $(H_t,J_t)$ is regular for $\phi$, see Definition \ref{def:perturbedregularpair}. As in  Section~\ref{sec:unperturbed} put $\check{\phi}\de\phi\comp\psi^{H}_1$. We claim that there is an isomorphism of chain complexes $\CFlp_{*}(\phi,H_{t},J_{t})\cong \CFpt(\check{\phi},\tilde{J}_{t})$ for a suitable cylindrical, $(d\lambda)$-compatible and $\check{\phi}$-periodic family $\tilde{J}_{t}$. 
	
The correspondence \eqref{eq:corresponence_Fix-P}, together with our definition of $\CZ(\gamma)$, give a natural identification of the underlying graded $\Z_{2}$-vector spaces. For the equality of differential we use the following, known argument from \cite[Proposition 8.6.1]{AD_book}.

Fix $\gamma_-,\gamma_+\in \cP_{\phi,H}$, and let $u:\RR^2\to \hat{M}$ be a $\CFlp_{*}(\phi,H_{t},J_{t})$-Floer trajectory from  $\gamma_{-}$ to $\gamma_{+}$. Let $p_{-},p_{+}$ be the corresponding fixed points of $\check{\phi}$. By \cite[Lemma 2.15]{Uljarevic}, the image of $u$ is contained in $M$. Define $v(s,t):=(\psi^{H}_t)^{-1}(u(s,t))$ and 
\begin{equation}
	\label{eq:Jtilde}
	%v(s,t)\de (\psi^{H}_t)^{-1}(u(s,t))\quad\mbox{and}\quad  
	\tilde{J}_t\de (\psi^{H}_t)^*J_t=D(\psi^{H}_t)^{-1}\circ J_t\circ D(\psi^{H}_t).
\end{equation}
We claim that $v$ is a $\CFpt_{*}(\check{\phi},\tilde{J}_{t})$-Floer trajectory from $p_{-}$ to $p_{+}$. The proof of \cite[Proposition 8.6.1]{AD_book} precisely shows that $v$ satisfies the unperturbed Cauchy--Riemann Floer equation~\eqref{eq:floereqmclean}, with $J_t$ replaced by $\tilde{J}_t$. Clearly, $\lim_{s\rightarrow \pm\infty} v(s,t)=p_{\pm}$. It remains to show that $\check{\phi}(v(s,t+1))=v(s,t)$ and that $\tilde{J}_{t}$ is a $\check{\phi}$-periodic family of cylindrical almost complex structures. This follows by elementary computation, using the basic equality \eqref{eq:basic}. Indeed, \eqref{eq:basic} and definition of $\check{\phi}$ give
\begin{equation*}
	\check{\phi}\circ (\psi_{t+1}^{H})^{-1}=\phi\circ \psi_1^H\circ( \phi^{-1}\circ \psi_t^H\circ\phi\circ\psi_1^H)^{-1}=
	(\psi_{t}^{H})^{-1}\circ\phi
\end{equation*}
Hence by definition of $v$, we have
\begin{equation*}
	\check{\phi}(v(s,t+1))=\check{\phi}\circ (\psi_{t+1}^{H})^{-1}(u(s,t+1))=(\psi_{t}^{H})^{-1}\circ \phi(u(s,t+1))
	\overset{\mbox{\tiny{\eqref{eq:boundulj}}}}{=}
	(\psi_t^{H})^{-1}(u(s,t))=v(s,t),
\end{equation*}
which proves the periodicity condition \eqref{eq:boundmclean}. To see that $\tilde{J}_t$ is $\check{\phi}$-periodic, we write
\begin{equation*}
	\tilde{J}_{t+1}=(\psi_{t+1}^{H})^{*}J_{t+1}
	\overset{\tiny{\mbox{\eqref{eq:basic}}}}{=}
	(\phi^{-1}\circ \psi_{t}^{H}\circ \phi\circ \psi_{1}^{H})^{*}J_{t+1}
	\overset{\tiny{(*)}}{=}
	(\psi_{t}^{H}\circ \phi\circ \psi_{1}^{H})^{*}J_{t}
	=
	(\phi\circ \psi_{1}^{H})^{*}\tilde{J}_{t}
	=
	\check{\phi}^{*}\tilde{J}_t,
\end{equation*}
as needed. The equality $(*)$ above holds by the $\phi$-periodicity condition $J_{t+1}=\phi^{*}J_{t}$.

Since the pair $(H_t,J_t)$ is regular for $\phi$, the linearized Floer equation for $u$ is surjective, so the same happens for $v$ and the family of almost complex structures $\tilde{J}_t$. This way we get a one to one correspondence between the $\CFlp_{*}(\phi,H_t,J_t)$-Floer trajectories from $\gamma_{-}$ to $\gamma_{+}$, and $\CFpt_{*}(\check{\phi}_t,\tilde{J}_t)$-Floer trajectories from $p_{-}$ to $p_{+}$, modulo the $\R$-action. Hence the differentials in both $\CFlp_{*}(\phi,H_t,J_t)$ and $\CFpt_{*}(\check{\phi},\tilde{J}_t)$ coincide. In particular, the results of \cite{Uljarevic} quoted in Section \ref{sec:perturbed} can be used to prove that both differentials are well defined and square to zero. We conclude that the complexes $\CFlp_{*}(\phi,H_t,J_t)=\CFpt_{*}(\check{\phi},\tilde{J}_t)$ constructed in Sections \ref{sec:perturbed} and \ref{sec:unperturbed}, respectively, are naturally isomorphic. In particular, both constructions give the same Floer homology groups $\HF_{*}(\phi,+)$.% and $\HF_{*}(\phi,-)$.

\subsubsection{Invariance by isotopy of abstract contact open books}
The following result, crucial for us, was proved in \cite[Lemma B.3]{McLean}. We include a proof for the convenience of the reader.
\begin{prop}
	\label{prop:isotopy_invariance}
	Let $\{(M_t,\lambda_t,\phi_t)\}_{t\in [0,1]}$ be a graded isotopy of abstract contact open books. Then the groups $\HF_*(\phi_t,+)$ and $\HF_{*}(\phi_t,-)$ are independent of $t$.
\end{prop}
\begin{proof}
	Let $\pi\colon M\to [0,1]$ be the total space of the isotopy, and let $\phi\colon M\to M$ be the diffeomorphism which restricts to $\phi_t$ on each fiber, see Definition \ref{def:acob}. 
	
	Given such an isotopy $(M,\pi,\lambda)$ of Liouville domains, in \cite[\sec 9.1]{Keating}, an isotopy of completions $(\hat{M},\pi,\lambda)$ is constructed. Since $\phi$ is fiberwise compactly supported, it extends to a diffeomorphism $\phi\colon \hat{M}\to\hat{M}$ which respects fibers and at each fiber is a compactly supported exact symplectomorphism. By \cite[\sec 9.1]{Keating}, $\pi\colon \hat{M}\to [0,1]$ admits a trivialization by an exact symplectomorphism, i.e.\ there is  a diffeomorphism $\Phi:\hat{M_0}\times [0,1]\to\hat{M}$ such that $\Phi^*\lambda_t=\lambda_0+dG_t$, where $G_t$ is the restriction to $M_0\times\{t\}$ of a smooth function $G\colon M\times [0,1]\to\RR$. So, given an isotopy of abstract contact open books $(M,\pi,\lambda,\phi)$, we have an isotopy $(\hat{M}_0\times [0,1],\pr_{[0,1]},\lambda,\phi)$ of their completions, such that $\lambda_t=\lambda_0+dG_t$, and $\phi$ is fiberwise an exact compactly supported symplectomorphism. If the isotopy $(M,\pi,\lambda,\phi)$ is graded, then this grading is inherited to $(\hat{M}_0\times [0,1],\pr_{[0,1]},\lambda,\phi)$. 
	
	By construction, the Floer homology of $(\hat{M}_0,\lambda_0+dG_t,\phi_t)$ coincides with the Floer homology of $(\hat{M}_0,\lambda_0,\phi_t)$. The result now follows from \cite[Theorem 2.34]{Uljarevic}.
\end{proof}

\subsection{The McLean spectral sequence}

The remaining part of this section is devoted to the proof of the following result, which is a (slightly more general) version of the axiom (HF3) from \cite[p.\ 980]{McLean}, proved in Appendix C of loc.\ cit; see Remark \ref{rem:McLean-HF3} for comparison. Our proof follows the same path, but we provide a more detailed exposition for the convenience of the reader.

\begin{prop}\label{prop:spectral-sequence}
	Let $(M,\lambda,\phi)$ be a graded abstract contact open book. Assume that 
	\begin{equation*}
		\Fix \phi=\bigsqcup_{i=1}^{N} B_{i},
	\end{equation*}
	where each $B_{i}$ is a codimension zero families of fixed points, see Definition \ref{def:0codim}. Choose an action $\ac\colon M\to \R$, see \eqref{eq:exact-symplectiomorphism}, and a function $\iota\colon \{1,\dots, N\}\to \Z$ with the following property: the function $\hat{\iota}$ from the discrete set $\ac(\bigsqcup_{i=1}^{N}B_{i})\subseteq \R$ to $\Z$, defined by $\hat{\iota}(\ac(B_{i}))=\iota(i)$, is strictly increasing. Assume further that one of the following holds
	\begin{enumerate}[(i)]	
		\item\label{item:ss-H1} $H^{1}(M,\d M;\R)=0$, or
		\item\label{item:ss-bd} for every $i$, the codimension zero family $B_i$ has either $\d^{-}B_{i}=\emptyset$ or $\d^{+}B_i=\emptyset$, see \eqref{eq:boundarysplitting1}.
	\end{enumerate} 
Then there is a spectral sequence $E_{p,q}^{r}$ converging to $\HF_{*}(\phi,+)$, whose first page equals
	\begin{equation}\label{eq:spectral-sequence-first-page}
		E_{p,q}^{1}=\bigoplus_{\{i: \iota(i)=p\}} H_{d+p+q+\CZ(B_i)} (B_{i},\d^{+}B_{i};\Z_2),
	\end{equation}
where $d=\tfrac{1}{2}\dim M$. Moreover, there is an analogous spectral sequence converging to the Floer homology $\HF_{*}(\phi,-)$, whose first page is as above with $\d^{+}B_i$ replaced by $\d^{+}B_i\sqcup \d^{M}B_i$.
\end{prop}

	\begin{remark}\label{rem:graph}
		The technical assumptions \ref{item:ss-H1} or \ref{item:ss-bd} are introduced to simplify the construction below, and most likely can be omitted. Nonetheless, working under either of these assumptions is sufficient for our purposes. Indeed, as we will see in Proposition \ref{prop:monodromy}\ref{item:B_i-boundary}, if $\phi$ is the \enquote{radius zero} symplectic monodromy, given by the flow of the vector field \eqref{eq:monodromy-vector-field}, then all its codimension zero families of fixed points satisfy $\d^{-}B_i=\emptyset$, so condition \ref{item:ss-bd} holds. Moreover, in the proof of Theorem \ref{theo:Zariski} the Liouville domain $M$ will be (diffeomorphic to) the Milnor fiber of an isolated hypersurface singularity $f\colon (\C^{n},0)\to (\C,0)$, so for $n\geq 3$ it will also satisfy $H^1(M,\d M;\R)=0$, i.e.\ condition \ref{item:ss-H1}.%, see Examples \ref{ex:Milnor-fibration} and \ref{ex:mu_constant}. 
		
		We note that each of these assumptions can be replaced by the following, weaker one. Consider a directed graph $\Gamma$ whose vertices are connected components of $M\setminus \bigsqcup_{i}B_i$, and for two such connected components, say $C_1$, $C_2$, there is an edge from $C_1$ to $C_2$ whenever there is an $i$ such that the closure $\bar{C}_1$ meets $\d^{-} B_{i}$, and the closure $\bar{C}_2$ meets $\d^{+}B_{i}$. Now, we assume that the graph $\Gamma$ has no loops. 
	\end{remark}

\begin{remark}\label{rem:McLean-HF3}
	The original McLean spectral sequence introduced in  \cite[(HF3)]{McLean} converges to Floer cohomology. Its 	
	first page is $E_{1}^{p,q}=\bigoplus_{\{i: \iota(i)=p\}} H_{d-p-q-\CZ(B_i)} (B_{i};\Z)$. In the setting of \cite{McLean}, this first page (tensored by $\Z_2$) is dual to the first page \eqref{eq:spectral-sequence-first-page}.
	
	To see this, we recall from Remark \ref{rem:McLean-has-no-d+} that in \cite{McLean} each codimension zero family $B_{i}$ has by assumption $\d^{-}B_{i}=\emptyset$. Moreover, \cite[(HF3)]{McLean} is formulated for a small slope deformation of $\phi$, which forces $\d^{M}B_i=\emptyset$. Thus each $B_i$ in loc.\ cit.\ has $\d B_{i}=\d^{+}B_{i}$, so the group $H_{d-p-q-\CZ(B_i)}(B_{i})$ is isomorphic, by Lefschetz duality, to $H^{d+p+q+\CZ(B_i)}(B_{i},\d^{+}B_{i})$. By the universal coefficient theorem, the latter is dual to $H_{d+p+q+\CZ(B_i)}(B_i,\d^{+}B_i)$, which is the group in \eqref{eq:spectral-sequence-first-page}%, as in \eqref{eq:spectral-sequence-first-page}.}%
	, as claimed.
	
	We have chosen the $\Z_2$ coefficients for convenience, since the same choice is made in our references \cite{Uljarevic,AD_book}. It is enough for most applications, including the proof of Theorem \ref{theo:Zariski}. 
\end{remark}

For the proof of Proposition~\ref{prop:spectral-sequence}, the general idea is the following: choose a Hamiltonian perturbation of $\phi$ such that the fixed points become non-degenerate and contained in the union $\bigsqcup_i B_i$. Prove that any Floer trajectory connecting two fixed points contained in the same component $B_i$ is contained in $B_i$ too. Then, since $\phi|_{B_i}$ equals the identity, following the technique of the classical proof that Hamiltonian Floer homology coincides with ordinary homology, it is possible to identify the pieces of the first page of the spectral sequence. For the precise implementation of this scheme a couple of technicalities need to be taken into account. 

A confinement lemma \cite[Lemma C.2]{McLean} of pseudo-holomorphic curves due to McLean is used: it allows to confine Floer trajectories in a small neighborhood $N_{B_i}$ of $B_i$ rather than in $B_i$. Since the almost complex structure $J_t$ needs to be $\phi$-periodic, and $\phi|_{N_{B_i}\setminus B_i}$ is not the identity, it is not possible to choose $J_t$ independent of time in $\phi|_{N_{B_i}\setminus B_i}$. On the other hand, for the proof of the classical fact that Hamiltonian Floer homology coincides with ordinary homology, an almost complex structure independent of time is used in order to prove that Floer trajectories associated with small time independent Hamiltonians are in fact Morse trajectories. In order to deal with this we will produce first a Hamiltonian deformation of $\phi$ whose fixed points is a disjoint union of codimension zero families of fixed points $B'_i$, and each $B'_i$ is strictly contained in the interior of $B_i$. Then a comparably smaller second deformation will produce non-degenerate fixed points inside $\bigsqcup_i B'_i$, and such that any Floer trajectory connecting two fixed points in the same $B'_i$ is contained in $B'_i$. This way, a $\phi$-periodic complex structure independent of $t$ in $\bigsqcup_i B_i$ will suffice to identify the groups appearing in the first page of the spectral sequence. In the rest of this section we carry out the proof following the program just hinted.

\subsubsection{Technical preparations} Our first aim is to provide a small Hamiltonian perturbation of $\phi$, which is well suited to apply the confinement lemma \cite[Lemma C.2]{McLean}; and a similar perturbation of the identity, which we will later use to compare the Floer complex with pieces of the Morse complex.

\paragraph{Step 1}\label{step:1}
By Lemma \ref{lem:cornerelimination}, we can and do assume that each $B_i$ is a connected codimension zero submanifold of $M$ with boundary, i.e.\ it does not have corners. 
\smallskip

Choose a collar structure $\Psi:(-\eta,0]\times \partial M\to M$ near the boundary of $M$ such that  $\Psi^*\lambda=e^r\lambda|_{\partial M}$, see \eqref{eq:necktriv}, and $\phi=\id$ in the image of $\Psi$. Choose a small negative slope $a$ and consider a smooth decreasing function $\sigma\colon [-\eta,0]\to \RR$ that is constant near $-\eta$ and equal to the linear function $r\mapsto ar$ near $0$. Define a time-independent Hamiltonian $H$ by  $H=\sigma\circ \pr_{(-\eta,0]}\circ \Psi^{-1}$ in the image of $\Psi$, and $H=\sigma(-\eta)$ otherwise. It has small negative slope; and is $\phi$-periodic since it is constant whenever $\phi\neq \id$. Let $\phi'\de \phi\circ \psi_{1}^{H}$ be the corresponding small negative slope deformation $\phi'$ of $\phi$. Now, $\Fix(\phi')=\bigsqcup_{i=1}^{N}B_{i}''$, where each $B_{i}''$ is a codimension zero family of fixed points, contained in $B_{i}$ in such  way that $\d B_{i}''$ does not intersect $\partial M$ anymore. In terms of the decomposition \eqref{eq:boundarysplitting1}, we have
\begin{equation*}
%	\label{eq:boundarysplitting2}
	\partial B_i''=\partial^+ B_i''\sqcup\partial^- B_i'',
\end{equation*}
$\d^{+} B_{i}''=\d^{+} B_{i}$, and the triple $(B_{i}'',\d^{+}B_{i}'',\d^{-}B_{i}'')$ is homeomorphic to $(B_{i},\d^{+}B_{i},\d^{-}B_{i}\sqcup \d^{M} B_i)$. 

\begin{notation}\label{not:B=B''}
	In the remaining part of the proof we will abuse notation and rename $B''_i=B_i$. This way, $\partial B_i=\partial^+ B_i\sqcup\partial^- B_i$.
\end{notation}

\begin{remark}\label{rem:d-B}
	In terms of Notation \ref{not:B=B''}, %the technical 
	assumption \ref{item:ss-bd} of Proposition \ref{prop:spectral-sequence} implies the following. Take $i\in \{1,\dots, N\}$ such that $\d^{+}B_{i}\neq \emptyset$. Then there is a collar structure $\Psi\colon (-\eta,0]\times \d M \to M$, and a level $c\in (-\eta,0)$, such that every connected component of $\d^{-}B_{i}$ is a connected component of $\Psi(\{c\}\times \d M)$; and $\d^{+}B_i$ is disjoint from the collar $\Psi((-\eta,0]\times \d M)$. Hence if some connected component of $\bar{M\setminus B_i}$, say $C$, meets $\d^{-}B_i$, then $C$ is contained in a collar near $\d M$; in particular $C$ does not meet $\d^{+}B_i$.
%	Assume $H^{1}(M,\d M;\R)\neq 0$. Then in terms of Notation \ref{not:B=B''}, the technical assumption of Proposition \ref{prop:spectral-sequence} implies the following: for all $i\in \{1,\dots, N\}$, either $\d^{+}B_{i}=\emptyset$, or each connected component of $\d^{-}B_{i}$ meets exactly one connected component of $M\setminus B_i$.
\end{remark}

For the next steps we need the following notion.

\begin{definition}
	\label{def:growcompat}
	Let $F,G\colon U\to\RR$ be smooth functions on an open subset $U\subset M$. We say that $F$ and $G$ {\em grow compatibly in $U$} if the level sets of $F$ coincide with the level sets of $G$, and there is a smooth positive function $\sigma\colon U\to (0,\infty)$, constant on each of those level sets, such that $dF=\sigma\cdot d G$.
\end{definition}

\paragraph{Step 2}\label{step:2}
Let $H_{B_i}\colon N_{B_i}\to\RR$ be the time independent Hamiltonian associated to a codimension zero family $B_i$ of fixed points of $\phi'$, see Figure \ref{fig:FGH}. Choose a compact collar neighborhood $C_{B_i}$ of $\partial B_i$ contained in $N_{B_i}$ such that the restriction 
$$H_{B_i}|_{\overline{C_{B_i}\setminus B_i}}:\overline{C_{B_i}\setminus {B_i}}\to\RR$$
is topologically locally trivial. We claim that there is a smooth function $H'_{i}\colon M\to \RR$ such that     
\begin{enumerate}
	\item\label{item:H_i-top-triv} the restriction $H_{i}'|_{C_{B_{i}}}$ is topologically locally trivial and a submersion at the interior of $C_{B_i}$, 
	\item\label{item:H_i-compatible} 
	the functions $H'_{i}$ and $H_{B_{i}}$ grow compatibly in the interior of $C_{B_{i}}\setminus B_{i}$, see Definition \ref{def:growcompat},
	\item\label{item:H_i-extends} the restriction of $H_{i}'$ to each connected component of $M\setminus C_{B_i}$ is constant,
	\item\label{item:H_i-action} the function $H'_i$ is constant in $B'_i\de\overline{B_i\setminus C_{B_i}}$ and in each connected component of $\partial B_i$. Furthermore, we have $\min\{H'_i(x):x\in\d^+B_i\}>H'_i(B'_i)>\max\{H'_i(x):x\in\d^-B_i\}$.
\end{enumerate}
It is easy to produce a function $\tilde{H}_{i}\colon U\to \R$ satisfying properties \ref{item:H_i-top-triv}--\ref{item:H_i-extends} in some neighborhood $U$ of $C_{B_i}$. In order to obtain $H_{i}'$ satisfying \ref{item:H_i-top-triv}--\ref{item:H_i-extends} and defined everywhere on $M$, we consider two cases.

First, we assume that condition \ref{item:ss-H1} holds, that is, $H^{1}(M,\d M;\R)=0$. By property \ref{item:H_i-extends}, the $1$-form  $d\tilde{H}_{i}$ extends by zero to a closed $1$-form on $M$, supported away from $\d M$. Since $H^{1}(M,\d M;\R)=0$, this $1$-form is exact, i.e.\ there is a smooth function $H_{i}'\colon M\to \R$ such that $dH_{i}'=d\tilde{H}_{i}'$ on $U$, and $dH_{i}'=0$ on $M\setminus U$. Thus $H_{i}'$ satisfies properties \ref{item:H_i-top-triv}--\ref{item:H_i-extends}, as needed. 

Now, we assume that condition \ref{item:ss-bd} holds. We can choose $\tilde{H}_{i}$ so that it is constant on $B_i\cap \d C_{B_{i}}$. This way, $\tilde{H}_{i}$ extends to a smooth function $B_{i}\cup C_{B_{i}}\to \R$, constant on $B_i'$. Write $\d(B_i\cup C_{B_i})=\d^{+}\sqcup \d^{-}$, where $\d^{+}$ (respectively, $\d^{-}$) lies in the union of connected components of $C_{B_{i}}$ containing $\d^{+}B_i$ (respectively, $\d^{-}B_i$). We can furthermore choose $\tilde{H}_{i}$ so that it is constant on $\d^{+}$ and on $\d^{-}$. If $\d^{+}B_{i}=\emptyset$ then $\d^{+}=\emptyset$, so $\tilde{H}_{i}$ is constant on $\d(B_i\cup C_{B_i})$, and therefore it extends to a smooth function $H_{i}'\colon M\to \R$ satisfying  \ref{item:H_i-top-triv}--\ref{item:H_i-extends}, as needed. Assume $\d^{+}B_{i}\neq \emptyset$. Then by Remark \ref{rem:d-B}, no connected component of $M\setminus (B_{i}\cup \Int C_{B_{i}})$ meets both $\d^{+}$ and $\d^{-}$.  Hence $\tilde{H}_{i}$ is constant on the boundary of each connected component $M\setminus (B_i\cup \Int C_{B_i})$, so again it extends to a smooth function $H_{i}'\colon M\to \R$ satisfying  \ref{item:H_i-top-triv}--\ref{item:H_i-extends}, as required. Property \ref{item:H_i-action} now follows from \ref{item:H_i-top-triv}--\ref{item:H_i-extends}.
\smallskip

\begin{figure}[htbp]
	\begin{tikzpicture}[scale=0.6]
		\path[use as bounding box] (-1,-3.5) rectangle (25,3.5); 
		\begin{scope}
			\fill[black!5] (3,-2.5) -- (3,2.5) -- (7,2.5) -- (7,-2.5) -- (3,-2.5);
			\fill[black!10] (4,-2.5) -- (4,2.5) -- (6,2.5) -- (6,-2.5) -- (4,-2.5);
			\draw[->] (-1,0) -- (10,0);
			\draw[->] (-0.5,-3) -- (-0.5,3);
			\draw[black!50, dotted] (2,-2.5)--(2,2.5);
			\draw[dotted] (3,-2.5)--(3,2.5);
			\draw[black!50, dotted] (4,-2.5)--(4,2.5);
			\draw[black!50, dotted] (6,-2.5)--(6,2.5);
			\draw[dotted] (7,-2.5)--(7,2.5);
			\draw[black!50, dotted] (8,-2.5)--(8,2.5);
			\draw [decorate, decoration = {calligraphic brace}, very thick] (3,2.7) --  (7,2.7);
			\node[above] at (5,2.7) {\small{$B_{i}$}};
			\draw [decorate, decoration = {calligraphic brace}, very thick] (6,-2.7) --  (4,-2.7);
			\node[below] at (5,-2.7) {\small{$B_{i}'$}};
			\draw [decorate, decoration = {calligraphic brace}, very thick] (4,-2.7) --  (2,-2.7);
			\node[below] at (3,-2.7) {\small{$C_{B_{i}}$}};
			\draw [decorate, decoration = {calligraphic brace}, very thick] (8,-2.7) --  (6,-2.7);
			\node[below] at (7,-2.7) {\small{$C_{B_{i}}$}};
			\draw[very thick] (3,0.2) -- (3,-0.2);
			\node[above] at (3,0.2) {\small{$\d^{-}B_{i}$}};
			\draw[very thick] (7,0.2) -- (7,-0.2);
			\node[below] at (7,-0.2) {\small{$\d^{+}B_{i}$}};
			\draw[thick] (0,-3) to[out=45,in=180] (3,0) -- (7,0) to[out=0,in=-135] (10,3);
			\node[right] at(10,3) {\small{$H_{B_{i}}$}};
			\draw[thick] (0,-2) -- (2,-2) to[out=0,in=180] (4,0) -- (6,0) to[out=0,in=180] (8,2) -- (10,2); 
			\node[right] at(10,2) {\small{$H_{i}'$}};
		\end{scope}
		\begin{scope}[shift={(15,0)}]
			\fill[black!5] (3,-2.5) -- (3,2.5) -- (7,2.5) -- (7,-2.5) -- (3,-2.5);
			\fill[black!10] (4,-2.5) -- (4,2.5) -- (6,2.5) -- (6,-2.5) -- (4,-2.5);
			\draw[->] (-1,0) -- (10,0);
			\draw[->] (-0.5,-3) -- (-0.5,3);
			\draw[black!50, dotted] (2,-2.5)--(2,2.5);
			\draw[dotted] (3,-2.5)--(3,2.5);
			\draw[black!50, dotted] (4,-2.5)--(4,2.5);
			\draw[black!50, dotted] (6,-2.5)--(6,2.5);
			\draw[dotted] (7,-2.5)--(7,2.5);
			\draw[black!50, dotted] (8,-2.5)--(8,2.5);
			\draw [decorate, decoration = {calligraphic brace}, very thick] (3,2.7) --  (7,2.7);
			\node[above] at (5,2.7) {\small{$B_{i}$}};
			\draw [decorate, decoration = {calligraphic brace}, very thick] (6,-2.7) --  (4,-2.7);
			\node[below] at (5,-2.7) {\small{$B_{i}'$}};
			\draw [decorate, decoration = {calligraphic brace}, very thick] (4,-2.7) --  (2,-2.7);
			\node[below] at (3,-2.7) {\small{$C_{B_{i}}$}};
			\draw [decorate, decoration = {calligraphic brace}, very thick] (8,-2.7) --  (6,-2.7);
			\node[below] at (7,-2.7) {\small{$C_{B_{i}}$}};
			\draw[very thick] (3,0.2) -- (3,-0.2);
			\node[above] at (3,0.2) {\small{$\d^{-}B_{i}$}};
			\draw[very thick] (7,0.2) -- (7,-0.2);
			\node[below] at (7,-0.2) {\small{$\d^{+}B_{i}$}};
			\draw[thick] (0,-1) -- (2,-1) to[out=0,in=180] (4,0) -- (6,0) to[out=0,in=180] (8,1) -- (9,1); 
			\node[right] at(9,1) {\small{$H_{i}''$}};	
			\draw[thick] (0,0) -- (3,0) to[out=0,in=180] (4.5,1) to[out=0,in=180] (5.5,-1) to[out=0,in=180] (7,0) -- (9,0);
			\node[below] at(9,0) {\small{$F_{i}$}};
		\end{scope}
	\end{tikzpicture}
	\caption{Auxiliary functions $H_{i}'$, $H_{i}''$ and $F_{i}$ introduced in Steps \protect\hyperref[step:2]{2} and  \protect\hyperref[step:3]{3}.}
	\label{fig:FGH}
\end{figure}
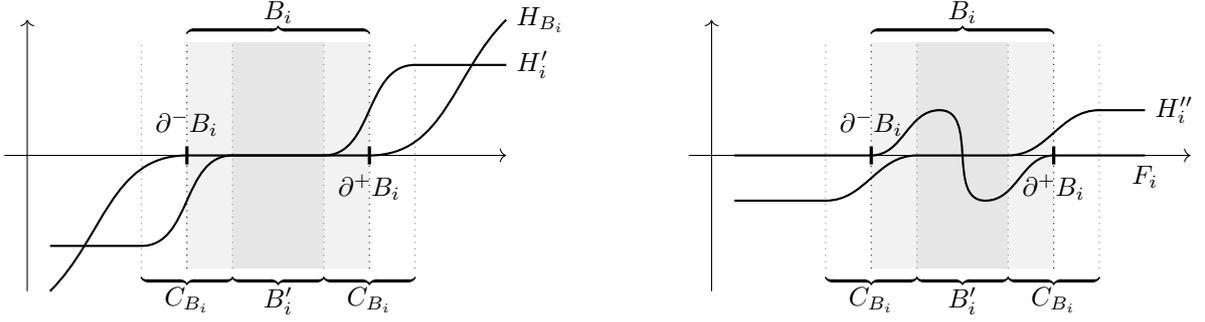

Since the restrictions of $H'_{i}$ and $H_{B_i}$ to $\Int C_{B_i}\setminus B_i$ grow compatibly, there is a function $g\colon \Int C_{B_i}\setminus B_{i}\to (0,\infty)$, constant on the level sets of $H'_{i}$, such that $gX^{H_i'}=X^{H_{B_i}}$. Since $X^{H_{B_i}}$ is zero on $B_i$, and $X^{H_i'}$ does not vanish on $\d B_i$, the function $g$ extends by zero to a smooth function $N_{B_{i}}'\to [0,\infty)$, where $N_{B_{i}}'\de B_{i}\cup \Int C_{B_{i}}$.

Fix $r_i>0$ and put $H''_i=(r_i+g)H'_{i}$. Then $H''_i\colon N_{B_{i}}'\to\RR$ is a time independent Hamiltonian such that we have the equality $\psi^{H''_i}_1=\psi^{H_{B_i}}_1\comp\psi^{r_iH'_i}_1$. Indeed, at a level set of $H'_i$ the composition on the right hand side moves a point first time $r_i$ along an integral curve of $X^{H'_i}$, and then time $1$ along an integral curve of $X^{H_{B_i}}$. Moving time $1$ along an integral curve of $X^{H_{B_i}}$ is the same than moving time $g$ along an integral curve of $X^{H'_i}$ (recall that $g$ is constant on those integral curves), so the total effect is moving time $r_i+g$ along an integral curve of $X^{H'_i}$. 

We claim that if $r_i$ is positive and small enough the set of fixed points of $\psi^{X^{H''_i}}_1$ in $N_{B_{i}}'$ is equal to $B'_i$. Indeed, $B'_i$ is obviously contained in the fixed point set of $\psi^{H''_i}_1$. To check the opposite inclusion we notice that, since the time $1$ flow of $X^{H_{B_i}}$ coincides with the time $1$ flow of $X^{gH'_i}$, and has no fixed points in $C_{B_i}\setminus B_i$, there exists a positive $r_i$ such that the time $1$ flow of $X^{(r_i+g)H'_i}$ has no fixed points in $C_{B_i}$, either. We fix such $r_i$ for the rest of the section.

%Let $B_1$, ..., $B_N$ be the collection of codimension zero family of fixed points of $\phi'$ (see Step 1).  
Define $H'\de \sum_{i=1}^Nr_iH'_i$. The composition $\phi'':=\phi'\comp\psi_1^{H'}$ is an exact symplectomorphism whose fixed point set is the disjoint union of codimension zero families of fixed points $\bigsqcup_{i=1}^N B'_i$.  For every $i$ we have constructed a Hamiltonian $H''_i\colon N_{B_{i}}'\to\RR$ such that $\phi''|_{N_{B_{i}}'}$ is equal to the time $1$ flow $\psi^{H''_i}_1$.
\smallskip

We note that, under the weaker assumption of Remark \ref{rem:graph}, it might not be possible to construct individual functions $H_{i}'\colon M\to \R$. However, it is equally easy to construct a global function $H'\colon M\to \R$ and local $H_{i}''\colon N_{B_{i}}\to \R$; and this is all we need in the sequel. We leave the details to the reader.

\paragraph{Step 3}\label{step:3}
For each $i$ choose a function $F_{i}:B_i\to\RR$ such that
\begin{enumerate}
	\item\label{item:F_i-extends} $F_i$ admits a smooth extension $F_i:M\to\RR$ by a constant function equal to $0$ in $M\setminus B_i$,
	\item\label{item:F_i-no-crit} the restriction of $F_i$ to $\Int(C_{B_i}\cap {B}_i)$  has no critical points,
	\item\label{item:F_i-compatible} the functions $F_i$ and $H_{i}''$ grow compatibly in $\Int(C_{B_i}\cap {B}_i)$,
	\item\label{item:F_i-Morse} the critical points of $F_i|_{B_{i}}$ are of Morse type,
	\item\label{item:F_i-Fix} for every $t\in (0,1]$, the only $t$-periodic orbits of the Hamiltonian flow of $F_{i}$ at the interior of $B_i$ are constant and equal to the Morse points. 
\end{enumerate}
It is easy to produce a function with properties \ref{item:F_i-extends}--\ref{item:F_i-compatible}. A small generic perturbation of such a function which does not modify it at $\Int(C_{B_i}\cap {B}_i)$ achieves property \ref{item:F_i-Morse}. Property \ref{item:F_i-Fix} is achieved multiplying any function with properties \ref{item:F_i-extends}--\ref{item:F_i-Morse} by a sufficiently small positive constant, as a consequence of \cite[Proposition 6.1.5]{AD_book}.

Define $F:=\sum_{i=1}^N F_{i}$. For $\delta>0$ small enough define $\phi_\delta''' \de \phi''\comp\psi_1^{\delta F}$. In the same way that we proved above the existence of a time independent Hamiltonian $H''_i$ satisfying  $\psi^{H''_i}_1=\psi^{H_{B_i}}_1\comp\psi^{r_iH'_i}_1$, here we prove that there exists a time independent Hamiltonian $H_\delta$ such that $\phi'''_\delta=\phi\comp\psi_1^{H_\delta}$. So, we have that $\phi_\delta'''$ is a negative slope Hamiltonian perturbation of $\phi$. Moreover, if $\delta>0$ is small enough, that slope is small, too, and the fixed point set of $\phi_\delta'''$ is the disjoint union of the Morse points of $F_i$ at the interior of $B'_i$. Indeed, by property \ref{item:F_i-Fix} of $F_i$ the fixed points of $\phi_\delta'''$ contained in $B_i'$ are the Morse points of $F_i$. Outside $\bigsqcup_{i} N_{B_i}'$, $\phi_\delta'''$ coincides with $\phi'$, so it has no fixed points there. It remains to show that $\phi_\delta'''$ has no fixed points in $\bigsqcup_{i} C_{B_i''}$. This follows, if $\delta>0$ is small enough, by the same arguments that proved above that if $r_i>0$ is small enough the set of fixed points of $\psi^{H''_i}_1$ is equal to $B'_i$. 

\paragraph{Step 4}\label{step:4}
In the three steps described up to now we have provided a Hamiltonian perturbation of $\phi$ that is well suited to analyze McLean's spectral sequence. In the next one we run a similar procedure for the identity symplectomorphism in $M$, with the later aim of producing a spectral sequence similar to McLean's. To prove Proposition \ref{prop:spectral-sequence}, we will compare the first pages of both spectral sequences.

Consider a smooth function $G\colon M\to\RR$ with the following properties:
\begin{enumerate}
	\item we have $G|_{N_{B_i}'}=H''_i$,
	\item\label{item:G-aviods} the restriction of $G$ to $M\setminus (\bigsqcup_{i} N_{B_i}')$ is a Morse function, and moreover the set of values of $G$ at the Morse points of $G|_{M\setminus (\bigsqcup_{i} N_{B_i}')}$ is disjoint from the finite set $\{G(B'_i)\}_{i=1}^{N}$.
\end{enumerate}

We consider the exact symplectomorphism $\psi'''_{\delta}:=\psi^{\delta G}_1\comp \psi_1^{\delta F}$ for $\delta>0$. As before, we prove that there is a time independent Hamiltonian $K_{\delta}$ such that $\psi'''_{\delta}=\psi^{K_{\delta}}_1$, so that $\psi'''_{\delta}$ is a small negative slope Hamiltonian deformation of $\id_M$. Moreover, $K_{\delta}$ is a Morse function in $M$ which satisfies $K_{\delta'}=\frac{\delta'}{\delta}K_{\delta}$. Therefore, by choosing $\delta$ small enough we may assume that the fixed points $\psi'''_{\delta}$ coincide with the Morse points of $K_{\delta}$ \cite[Proposition 6.1.5]{AD_book}. The function $K_{\delta}$ satisfies also the following properties:
\begin{enumerate}\setcounter{enumi}{2}
	\item\label{item:K-boundaries} $\min\{K_\delta(x):x\in\d^+B_i\}>\max\{K_\delta(x):x\in\d^-B_i\}$.
	\item\label{item:K-compatible} The functions $K_\delta$ and $H''_i$ grow compatibly in $\Int(C_{B_i}\cap B_i)$. 
\end{enumerate}
For \ref{item:K-boundaries} notice that since $F_i|_{\partial B_i}\equiv 0$, we have $K_\delta|_{\d B_i}=\delta G|_{\d B_i}=\delta H''_i|_{\d B_i}=\delta (r_i+g)H'_i |_{\d B_i}=\delta r_iH'_i |_{\d B_i}$. Hence property \ref{item:H_i-action} of Step \hyperref[step:2]{2} implies \ref{item:K-boundaries}. Property \ref{item:K-compatible} is a consequence of property \ref{item:F_i-compatible} of Step \hyperref[step:3]{3}.

\paragraph{Step 5}\label{step:5}
Let $J_t,J'$ be cylindrical $\omega$-compatible almost complex structures such that $J_t$ is $\phi$-periodic, $J'$ is time-independent and they have the following additional properties:
\begin{enumerate}
	\item\label{item:J'=Jt} $J_t$ is independent of $t$ in $\bigsqcup_{i} B_i$ and we have the equality %coincidence 
	\begin{equation}
		\label{eq:coincidenceJ}
		J_t|_{\bigsqcup_{i} B_i}=J'|_{\bigsqcup_{i} B_i}.
	\end{equation}
	\item\label{item:J'-Morse} the pair $(K_{\delta},J')$ is Morse-Smale for any $\delta>0$ (that is, the linearized Morse operator at each gradient line connecting two Morse points of $K_{\delta}$ is surjective, see \cite[Theorem 10.1.5]{AD_book}),% for more details),
	\item\label{item:J-Morse} for each $i$ the pair $(F|_{B'_i},J|_{B'_i})$ is Morse-Smale at the interior of $B'_i$  
\end{enumerate}
In order to construct $J_t$ and $J'$ we choose, following \cite[\sec 10]{AD_book}, a time independent $J'$ satisfying property \ref{item:J'-Morse} and then choose any cylindrical $\omega$-compatible and $\phi$-periodic $J_t$ satisfying property \ref{item:J'=Jt}. Property \ref{item:J-Morse} holds because $F|_{B'_i}$ is proportional by a non-zero constant factor to $K_{\delta}|_{B'_i}$. 

\subsubsection{The two Floer complexes}

Having fixed $H_\delta$ (for $\delta>0$ small enough) and $J_t$ with the properties of the previous section, following \cite[\sec 8]{AD_book} we find that there exists a time dependent perturbation $\check{H}_{\delta,t}$ of the Hamiltonian $H_\delta$, which is as close as we wish to $H_\delta$, does not modify $H_\delta$ in a neighborhood of the fixed points of $\phi'''_\delta$ and in a neighborhood of $\partial M$, and such that $(\check{H}_{\delta,t},J_t)$ is regular for $\phi$, see Definition~\ref{def:perturbedregularpair}. Then the $\phi$-twisted loop Floer complex $\CFlp_{*}(\phi,\check{H}_{\delta,t},J_t)$ introduced in Section \ref{sec:perturbed} is well defined and, by Proposition~\ref{prop:isotopy_invariance}, its homology is $\HF_*(\phi,+)$. 

We define $\check{\phi}_\delta:=\phi\comp\psi_1^{\check{H}_{\delta,t}}$ and $\tilde{J}_{t}\de (\psi_1^{\check{H}_{\delta,t}})^{*}J_{t}$, see formula~\eqref{eq:Jtilde}. Then the $\check{\phi}_{\delta}$-fixed point Floer complex $\CFpt_{*}(\check{\phi}_\delta,\tilde{J}_t)$ introduced in Section \ref{sec:unperturbed} is well defined and, as we have seen in Section~\ref{sec:floerequivalence}, its homology is $\HF_{*}(\phi,+)$, too.

Similarly, there is a time dependent perturbation $\check{K}_{\delta,t}$ of $K_\delta$ which does not modify $K_{\delta}$ near the critical points of $G$, such that $(\check{K}_{\delta,t},J')$ is regular for the identity, see Definition \ref{def:perturbedregularpair}. Then the $\id$-twisted loop Floer complex $\CFlp_{*}(\id,\check{K}_{\delta,t},J')$ is well defined, and its homology is $\HF_*(\id,+)$.%, see \cite{Uljarevic}. 

As before, putting $\tilde{J}_t'=(\psi_1^{\check{K}_{\delta,t}})^{*}J'$, we see that the $\psi_1^{\check{K}_{\delta,t}}$-fixed point Floer complex $\CFpt_{*}(\psi_1^{\check{K}_{\delta,t}},\tilde{J}'_t)$ is well defined and computes the same homology as $\CFlp_{*}(\id,\check{K}_{\delta,t},J')$.% by Section~\ref{sec:floerequivalence}. Here $\tilde{J}'_t$ is the time-dependent almost complex structure derived from $J'$ by the procedure of Section~\ref{sec:floerequivalence}.

The perturbation procedure described in \cite[\sec 8]{AD_book} makes clear that the perturbations $\check{H}_{\delta,t}$ and $\check{K}_{\delta,t}$ can be chosen so that they coincide in $\bigsqcup_{i} B_i$, that is, 
\begin{equation}
	\label{eq:coincidenceH}
	\check{K}_{\delta,t}|_{\bigsqcup_{i}  B_i}=\check{H}_{\delta,t}|_{\bigsqcup_{i}  B_i}.
\end{equation}
The idea is to replace the space of perturbations of the Hamiltonian, denoted by $\cC^\infty_\epsilon(H_0)$ in \cite[\sec 8]{AD_book}, by a Banach space of perturbations $\cC^\infty_\epsilon(H_\delta,K_\delta)$ whose elements are pairs $(h,k)$ where $h:M\times\RR\to \RR$ is a $\phi$-periodic time dependent Hamiltonian, $k:M\times\RR\to \RR$ is a periodic time dependent Hamiltonian and $h$ and $k$ coincide in  $\bigsqcup_{i}B_i$. Then it is possible to prove like in \cite[\sec 8]{AD_book} that the locus of $(h,k)$ such that $(H_\delta+h,J_t)$ and $(K_\delta+k,J')$ are both regular pairs is dense in $\cC^\infty_\epsilon(H_\delta,K_\delta)$.

\subsubsection{The McLean spectral sequence} We now use the action filtration introduced in Section \ref{sec:action-filtration} to define the McLean spectral sequence.

Recall that the action function $\ac$ of $\phi$ is constant in each codimension zero family $B_i$  of fixed points of $\phi$. We have denoted its value by $\ac(B_i)$. Choose a positive number $\xi$ such that $\xi<\tfrac{1}{3}\min\{|\ac(B_i)-\ac(B_j)|:\ac(B_i)\neq \ac(B_j)\}$. 
Let $\iota$ be the indexing function from the statement of Proposition \ref{prop:spectral-sequence}. For an integer $p$ put $I_{p}=\iota^{-1}(p)$ and 
$a_p\de \max\{\ac(B_{i}):\iota(i)\leq p\}$. Then for every $i\in I_{p}$ we have $\ac(B_{i})=a_{p}$. Put $b_{p}=a_{p}+\xi$. Clearly, $b_p\leq b_q$ if $p\leq q$. 

Put $F_{p}\CFlp_{*}(\phi,\check{H}_{\delta,t},J_t)\de F^{\leq b_p}\CFlp_{*}(\phi,\check{H}_{\delta,t},J_t)$, where the right hand side is defined in  Section \ref{sec:action-filtration}. This is an increasing filtration  by subcomplexes of $\CFlp_{*}(\phi,\check{H}_{\delta,t},J_t)$. The \emph{McLean spectral sequence} $E_{p,q}^{r}$ is the spectral sequence associated with this filtration. We have $E_{p,q}^{r}\Rightarrow \HF_{p+q}(\phi,+)$.
\smallskip

Recall from Section \ref{sec:perturbed} that the complex $\CFlp_{*}(\phi,\check{H}_{\delta,t},J_t)$ is generated as a graded vector space by the set $\cP_{\phi,\check{H}_{\delta,t}}$ of Hamiltonian $\phi$-twisted loops. Since we have chosen $\delta$ and $\check{H}_{\delta,t}$ small, the image of every $\gamma\in\cP_{\phi,\check{H}_{\delta,t}}$ is contained in some $B'_{i}$.

We have that $\phi'''_\delta|_{B'_i}$ is equal to the time $1$ Hamiltonian flow of $H_\delta|_{B'_i}$, which is equal to $\delta F_i|_{B'_i}$. Then, by property \ref{item:F_i-Fix} of $F_i$, the only $1$-periodic orbits of the Hamiltonian flow of $H_\delta|_{B'_i}$ are constant and equal to the Morse points of $F_i$. According to the procedure used to define the small perturbation $\check{H}_{\delta,t}$ following \cite[\sec 8]{AD_book}, we have that the only $\phi$-twisted Hamiltonian loops for the Hamiltonian $\check{H}_{\delta,t}$ which are contained in $B'_i$ are the same as those of $H_\delta|_{B'_i}$. Hence each $\gamma\in \cP_{\phi,\check{H}_{\delta,t}}$ is constant and equal to a Morse point of $F_i$ in $B_{i}'$, for some index $i$. Formula \eqref{eq:actiondef} shows that the action $\mathcal{A}_{\phi,\check{H}_{\delta,t}}$ of such point  belongs to $(\ac(B_{i})-\xi,\ac(B_{i})+\xi)$. We thus get an isomorphism of graded $\Z_{2}$-vector spaces
\begin{equation}\label{eq:first-page}
	F_{p}\CFlp_{*}(\phi,\check{H}_{\delta,t},J_t)/F_{p-1}\CFlp_{*}(\phi,\check{H}_{\delta,t},J_t)\cong\bigoplus_{i\in I_{p}}C^i_{*},
\end{equation}
where $C^i_*$ is a subspace of $\CFlp_{*}(\phi,\check{H}_{\delta,t},J_t)$ generated by the Morse points of $F_i|_{B'_i}$. Lemma \ref{lem:confinement} below implies that each $C^{i}_{*}$ is in fact a subcomplex of % and whose differential is induced by the Floer differential of 
$F_{p}\CFlp_{*}(\phi,\check{H}_{\delta,t},J_t)/F_{p-1}\CFlp_{*}(\phi,\check{H}_{\delta,t},J_t)$, hence \eqref{eq:first-page} is an isomorphism of chain complexes. 
\smallskip

In order to compute the first page of McLean's spectral sequence we will compare it with the first page of a similar spectral sequence associated with a filtration in the $\id$-twisted loop Floer complex $\CFlp_{*}(\id,\check{K}_{\delta,t},J')$ which we introduce now. Let $a'_1<\dots <a'_{m}$ be the critical values of $G$; notice that some of these values are equal to $G(B_{i}')$. Let $I_{j}'$ be the set of those indices $i$ such that $a_{j}'=G(B_{i}')$.  

Since $\delta$ is as small as we wish, we have that $K_\delta$ and $G$ are as close as we wish in the maximum norm. Therefore we may assume that the critical values of $K_\delta$ belong to $\bigsqcup_{i=1}^{m}(a'_i-\xi',a'_i+\xi')$, for a fixed $\xi'>0$ such that $\xi'<\tfrac{1}{3}\min\{|a'_i-a'_j|:i\neq j\}$. 

Fix $b'_1<\dots <b'_{m}$ such that $b'_{m}>a'_{m}+\xi'$ and $b'_i\de \frac{1}{2}(a'_i-a'_{i-1})$ for $i<m'$. We define an increasing filtration by subcomplexes of $\CFlp_{*}(\id,\check{K}_{\delta,t},J')$ by $F_j\CFlp_{*}(\id,\check{K}_{\delta,t},J')\de F^{\leq b'_j}\CFlp_{*}(\id,\check{K}_{\delta,t},J_t)$, where as in Section \ref{sec:action-filtration}, $F^{\leq b'_j}\CFlp_{*}(\id,\check{K}_{\delta,t},J')$ is the subcomplex generated by Hamiltonian $\id$-twisted loops whose action is bounded from above by $b'_j$.

As before, since the perturbations $K_{\delta}$ and $\check{K}_{\delta,t}$ are small, we conclude that the set of Hamiltonian $\id$-twisted loops $\cP_{\id,\check{K}_{\delta,t}}$ coincides with the set of Morse points of $K_{\delta}$. Since the action function for the identity symplectomorphism can be taken vanishing identically, Formula~\eqref{eq:actiondef} shows that the action $\mathcal{A}_{\id,\check{K}_{\delta,t}}$ of such a point $x$ belongs to $(a'_{j}-\xi',a'_{j}+\xi')$ for a certain index $j$. If $I'_{j}\neq \emptyset$, property \ref{item:G-aviods} of $G$ shows that $x$ is a critical point of $F_{i}|_{B_{i}'}$ for some $i\in I_{j}'$. We conclude that, whenever $I_{j}'\neq\emptyset$, we have an isomorphism of graded $\Z_{2}$-vector spaces 
\begin{equation}\label{eq:first-page_D}
	F_{j}\CFlp_{*}(\id,\check{K}_{\delta,t},J')/F_{j-1}\CFlp_{*}(\id,\check{K}_{\delta,t},J')\cong\bigoplus_{i\in I_{j}'}D^i_{*},
\end{equation}
where $D^i_*$ is a subspace of $\CFlp_{*}(\id,\check{K}_{\delta,t},J')$ generated by the Morse points of $F_i|_{B'_i}$. As before, Lemma \ref{lem:confinement} shows that \eqref{eq:first-page_D} is an isomorphism of chain complexes.

We claim that we have an isomorphism of chain complexes
\begin{equation}\label{eq:C=D}
	D_{*}^{i}=C_{*-\CZ(B_i)}^{i}.
\end{equation}
By definition, both underlying $\Z_2$-vector spaces are spanned by the Morse points of $F_{i}|_{B_{i}'}$, so they coincide. The grading of such a point, say $x$, in $C_{*}^{i}$ is $-\CZ_{\phi\circ \psi}(x)$, where $\psi$ is the time one flow of $\check{H}_{\delta,t}$. By \eqref{eq:coincidenceH}, $\psi$ coincides in $B_i$ with the time one flow of $\check{K}_{\delta,t}$. Hence the grading of $x$ in $D_{*}^{i}$ is $-\CZ_{\psi}(x)$. 
Lemma \ref{lem:CZ-Hamiltonian}\ref{item:CZ_phi+Morse} shows that  $\CZ_{\phi\circ \psi}(x)=\CZ_{\phi}(x)+\CZ_{\psi}(x)=\CZ(B_i)+\CZ_{\psi}(x)$. Thus the grading of $x$ in $C_{*}^{i}$ equals the grading of $x$ in $D_{*}^{i}$ minus $\CZ(B_i)$, which gives the required degree shift.

It remains to  show that the differentials in complexes $D^{i}_{*}$ and $C^{i}_{*}$ are equal. They are computed by counting, respectively, $\CFlp_{*}(\id,\check{K}_{\delta,t},J')$- and $\CFlp_{*}(\phi,\check{H}_{\delta,t},J_t)$- Floer trajectories connecting the Morse points of $F_{i}|_{B_{i}'}$. By the confinement Lemma \ref{lem:confinement} below, these trajectories are contained in $B_{i}$. There, we have equalities $\phi|_{B_i}=\id_{B_i}$, $J_t|_{B_i}=J'|_{B_i}$, see~\eqref{eq:coincidenceJ}, and $\check{K}_{\delta, t}|_{B_i}=\check{H}_{\delta, t}|_{B_i}$, see~\eqref{eq:coincidenceH}. Hence the Floer trajectories counted by both differentials coincide, and we get isomorphism \eqref{eq:C=D}, as needed.
\smallskip

In the above proof, we have used the following confinement lemma for pseudo-holomorphic curves due to McLean, see \cite[Lemma C.2]{McLean}.

\begin{lema}
	\label{lem:confinement}
	Fix any index $p$. If $\delta>0$ and the perturbation $\check{H}_{\delta,t}$ are chosen small enough, for any two Morse points $q_1,q_2$ of $F$ which are contained in $\bigsqcup_{i\in I_p}B'_i$, we have that any  $\CFlp_{*}(\phi,\check{H}_{\delta,t},J_t)$-Floer trajectory from $q_1$ to $q_2$ %for the pair $(\check{H}_{\delta,t},J_t)$ 
	has image completely contained in $\bigsqcup_{i\in I_p}B_i$.
	
	Similarly, if $\delta>0$ and the perturbation $\check{K}_{\delta,t}$ are chosen small enough, for any two Morse points $q_1,q_2$ of $F$ which are contained in $\bigsqcup_{i\in I'_{p}}B'_i$, any $\CFlp_{*}(\id,\check{K}_{\delta,t},J')$-Floer trajectory from $q_1$ to $q_2$ %for the pair $(\check{K}_{\delta,t},J')$ 
	has image completely contained in $\bigsqcup_{i\in I'_{p}}B'_i$. %Notice that the confinement property is stronger in this case.
\end{lema}
\begin{proof}
For the first case the Lemma coincides exactly with \cite[Lemma C.2]{McLean}, applied to $\phi''$, where $\bigsqcup_{i\in I_{p}}B'_i$ plays the role of codimension zero family $B$ of fixed points and $\bigsqcup_{i\in I_{p}}B_i$ is the neighborhood where $\phi'$ is Hamiltonian, except that loc.\ cit.\ is formulated for $\CFpt_{*}$-Floer trajectories. The case of $\CFlp_{*}$-Floer trajectories is reduced to McLean's setting by the procedure described in Section \ref{sec:floerequivalence}. Moreover, in loc.\ cit.\ $B$ is assumed to be connected, which is not the case for $\bigsqcup_{i\in I_{p}}B'_i$. However, what is actually used in the proof is not the connectedness of $B$, but the fact that the action function $\ac$ is constant on $B$: this is satisfied by $\bigsqcup_{i\in I_{p}}B'_i$ by definition of $I_p$.

For the second case notice that by \cite[Proposition 10.1.9]{AD_book}, if $\delta$ is small enough the solutions of the perturbed Floer Cauchy--Riemann equation \eqref{eq:floerequlj} for $(K_\delta,J')$, i.e.\  $\CFlp_{*}(\id,K_{\delta},J')$-Floer trajectories, are time-independent, and hence equal to gradient trajectories for $-K_\delta$. By Property \ref{item:K-boundaries} of Step \hyperref[step:4]{4} a gradient trajectory starting and ending in $B'_i$ can not escape $B_i$. Moreover, by Property \ref{item:K-compatible} of Step \hyperref[step:4]{4}, a gradient trajectory cannot escape $B'_i$ either, and since the critical points of $K_\delta$ are contained at the interior of $B'_i$ we conclude that all gradient trajectories are contained at the interior of $B'_i$. This proves the confinement  of $\CFlp_{*}(\id,K_{\delta},J')$-Floer trajectories. Now if the perturbation $\check{K}_{\delta,t}$ is chosen small enough, confinement of $\CFlp_{*}(\id,\check{K}_{\delta,t},J')$-Floer trajectories follows from Gromov compactness.
\end{proof}

\begin{remark}
	The proof of \cite[Lemma C.2]{McLean} uses a version of Gromov compactness due to Fish \cite[Theorem 3.1]{Fish}. The latter is a substantial generalization of \cite{Ivashkovich-Shevchishin}. However, the proof of \cite[Lemma C.2]{McLean} remains valid with the reference to \cite{Fish} replaced by \cite{Ivashkovich-Shevchishin}.
\end{remark}

\subsubsection{Computation of the first page}\label{sec:computation}
 To prove Proposition \ref{prop:spectral-sequence}, we need to compute the homology of the complex $C_{*}^{i}$ appearing in the formula \eqref{eq:first-page}. Using the isomorphism \eqref{eq:C=D},  it is enough to compute the homology of the complex $D_{*}^{i}$, defined in \eqref{eq:first-page_D}.

We need to show that the homology of $D_{*}^{i}$ is the relative homology of $(B_{i}, \d^{+}B_{i})$, up to a degree shift. For this we follow the proof that $\HF_*(\id_M,+)=H_{*-d}(M,\ZZ_2)$ via continuation maps in Floer theory. In \cite{Uljarevic} it is remarked that our case (in which $M$ is a Liouville domain) is treated like in the classical proof that Hamiltonian Floer homology for compact $M$ with vanishing second homotopy group, equals ordinary homology shifted by the dimension. This proof is fully explained in \cite{AD_book}.

We recall the needed details here. In \cite[\sec 10]{AD_book} it is proved that since the pair $(K_{\delta}, J')$ is Morse-Smale, the Morse complex of $(K_{\delta}, J')$ is well defined. Since we have $\xi K_\delta=K_{\xi\delta}$, if $\delta$ is chosen small enough, the $\CFlp_{*}(\id,K_{\delta}, J')$-Floer trajectories are independent of $t$, see \cite[Proposition 10.1.9]{AD_book}. Hence the pair $(K_{\delta}, J')$ is regular for $\id$. The $\id$-twisted Floer complex $\CFlp_{*}(\id,K_{\delta}, J')$ is well defined and coincides with the Morse complex of $(K_{\delta}, J')$.

In \cite[\sec 11]{AD_book}, using Floer's continuation maps, a quasi-isomorphism
\begin{equation}
	\label{eq:continuationmap}
	\Phi\colon \CFlp_{*}(\id,K_{\delta}, J')\to  \CFlp_{*}(\id,\check{K}_{\delta, t},J')
\end{equation}
is defined. The action filtration $F_{j} \CFlp_{*}(\id,K_{\delta}, J')$ is defined for the complex $ \CFlp_{*}(\id,K_{\delta}, J')$ in a similar way as above. We claim that the quasi-isomorphism~(\ref{eq:continuationmap}) can be defined in such a way that it respects the action filtration in both sides. To prove the claim we recall some ingredients involved in the construction of the continuation quasi-isomorphism. First we define continuation data $L:M\times\RR\times\RR\to\RR$, where $L$ is a function such that ${L}(p,t,s)=K_{\delta}(p)$ for $s\leq -1$ and  ${L}(p,t,s)\de \check{K}_{\delta,t} (p)$ for $s\geq -1$. Then a small generic perturbation $\tilde{L}:M\times\RR\times\RR\to\RR$ is constructed, without modifying outside $s\in [-1,1]$, and from it $\Phi$ is constructed by counting pseudo-holomorphic curves $u\colon \R\times\S^1\to M\times\R$. The point is that the pseudo-holomorphic curves satisfy the energy bound provided in \cite[Proposition 11.1.2]{AD_book}. If a pseudo-holomorphic curve $u$ connects $q_1$ and $q_2$ then its energy is bounded by
%$$\text{Action}(q_1)-\text{Action}(q_2)+C,$$
\begin{equation*}
	\cA_{\id,K_{\delta}}(q_1)-\cA_{\id,\check{K}_{\delta, t}}(q_2)+c,
\end{equation*}
where $c$ is a constant determined in \cite[p.\ 387]{AD_book}: it is enough that $c$ bounds $\partial\tilde{L}/\partial s$. Choosing $K_{\delta}$ and its perturbation $\check{K}_{\delta, t}$ small, the perturbation $\tilde{L}$ can be chosen small, too. So we can set $c$ as small as we wish, and doing so makes it  clear that $\Phi$ preserves the action filtration.

As a consequence we obtain a quasi-isomorphism
\begin{equation*}%\label{eq:graded-pieces}
	F_{j}\CFlp_{*}(\id,K_{\delta},J')/F_{j-1}\CFlp_{*}(\id,K_{\delta},J')\to F_{j}\CFlp_{*}(\id,\check{K}_{\delta,t},J')/ F_{j-1}\CFlp_{*}(\id,\check{K}_{\delta,t},J').
\end{equation*}
Assume $I'_{j}\neq\emptyset$. The complex on the right-hand side equals $\bigoplus_{i\in I_{j}'} D^{i}_{*}$, see \eqref{eq:first-page_D}. The complex on the left-hand side is defined analogously, as a vector space spanned by those critical points of $K_{\delta}$ whose action lies in the interval $(a_j'-\xi,a_j'+\xi)$. They are precisely the Morse points of $K_{\delta}$ contained in $\bigsqcup_{i\in I_{j}'}B_{i}'$. Since $K_{\delta}$ and $J'$ are time-independent, the Floer equation \eqref{eq:floerequlj} used to define the differential in $\CFlp_{*}$ reduces to the Morse equation used to define the differential in the Morse complex for $(-K_{\delta},J')$, see \cite[\sec 10.1]{AD_book} or \cite[Theorems 7.1 and 7.3]{Salamon_Zehnder}: note that the sign of the Hamiltonian gets switched as explained in Section \ref{sec:Hamiltonian}. Therefore, $D_{*}^{i}$ corresponds to a piece of the Morse complex for $(-K_{\delta},J')$ generated by the Morse points contained in $B_{i}$. This piece computes the relative homology of $(B_{i},\d^{+}B_i)$. Indeed, in the decomposition \eqref{eq:boundarysplitting1}, the gradient of $-K_{\delta}$ points inwards $\d^{+}B$ and outwards $\d^{-}B_i$, so %, and outwards $\d^{M}B_i$ since $K_{\delta}$ has negative slope: 
the claim follows from \cite[(2.2)]{Rot-thesis}. Using Lemma \ref{lem:CZ-Hamiltonian}\ref{item:CZ_Morse} to compute the grading, we get
\begin{equation*}
	H_{*}(D^{i}_{*})=H_{*+d}(B_{i},\d^{+}B_{i}),
\end{equation*}
cf.\ \cite[p.\ 834]{Seidel_Dehn-twist}. 
Together with the formulas \eqref{eq:first-page} and \eqref{eq:C=D}, this gives
\begin{equation*}
	E_{p,q}^{1}=\bigoplus_{i\in I_p} H_{p+q}(C_{*}^{i})=
	\bigoplus_{i\in I_p} H_{p+q+\CZ(B_i)}(D^{i}_{*})=
	\bigoplus_{i\in I_p} H_{p+q+\CZ(B_i)+d}(B_{i},\d^{+}B_{i}),
\end{equation*}
as claimed in Proposition \ref{prop:spectral-sequence}.

To get the analogous spectral sequence converging to $\HF_{*}(\phi,-)$, we need to replace the small negative slope by a small positive one. This way, in Step \hyperref[step:1]{1} $\d^{M}B_i$ becomes a part of $\d^{+}B_i$, see Notation \ref{not:B=B''}, so %the gradient of $K_{\delta}$ points outwards $\d^{M}B_i$, and 
at the end of the day we get relative homology of $B_i$ modulo $\d^{+}B_{i}\sqcup \d^{M}B_i$, as claimed.

\section{The A'Campo abstract contact open book}\label{sec:monodromy}

In this section, we generalize McLean spectral sequence \cite[Theorem 1.2]{McLean}, which converges to the fixed point Floer homology of the monodromy. We will formulate it in Proposition \ref{prop:spectral-sequence-monodromy} after some preparations. To prove it, we will phrase Proposition \ref{prop:omega} in the language of graded abstract contact open books. This way, in Section \ref{sec:isotopy-to-radius-zero} we will get an isotopy between the natural abstract contact open book, constructed at (any small) positive radius in Section \ref{sec:basic-monodromy} (see Example \ref{ex:Milnor-fibration}), and the one at radius zero, given by the flow of \eqref{eq:monodromy-vector-field}, which has good dynamical properties just like the topological A'Campo model \cite{A'Campo}, or McLean model resolution \cite{McLean}, see Remark \ref{rem:McLean-5.41}. We summarize those properties in Proposition \ref{prop:monodromy}, which is the main result of this section. In particular, we will see that the radius zero monodromy satisfies the assumptions of Proposition \ref{prop:spectral-sequence}, which provides the required spectral sequence.

\subsection{The generalized McLean spectral sequence for isolated singularities}\label{sec:monodromy-spectral-sequence}
%\subsection{The statement of the spectral sequence}\label{sec:monodromy-dynamics}

We work with the monodromy graded abstract contact open book constructed in Example \ref{ex:typical} for a single function (without parameters). Let us briefly recall its definition. 

Let $Z$ be a Stein manifold and let $\varrho\colon Z\to \R$ be an exhaustive strictly plurisubharmonic function. Consider the Liouville form $\lambda_Z:=-d^{c}\varrho$, see Section \ref{sec:symplectic-intro}. Let $Y\subseteq Z$ be a closed analytic subset whose all singularities are isolated, and, in case $\dim_{\C}Y=2$, their links are rational homology spheres. Let 
\begin{equation*}
	f\colon Y\to \C
\end{equation*}
be a holomorphic function such that $\Sing Y\subseteq f^{-1}(0)$. Assume that $0\in \C$ is the only critical value of $f$, and all singularities of  $f^{-1}(0)$ are isolated. Then $f|_{f^{-1}(\D_{\delta}^{*})}$ is a submersion for some $\delta>0$.

We choose regular values $\xi_{V}<\xi_{W}<\xi_{U}$ of $\varrho|_{Y}$, and put $U_Y=\varrho^{-1}(-\infty,\xi_U)\cap Y$, $V_Y= \varrho^{-1}(-\infty,\xi_V)\cap Y$ and $W_Y= \varrho^{-1}(-\infty,\xi_W)\cap Y$. Now,  Example \ref{ex:typical} (with $\epsilon=\eta=0$) provides a subset $N\subseteq Y$ such that $f^{-1}(0)\cap \bar{N}=f^{-1}(0)\cap \bar{W}_Y$, and the restriction $f|_{N}\colon N\to \D_{\delta}^{*}$ is a Liouville fibration, whose monodromy yields abstract contact open books 
%Note that $\psi$ satisfies the assumptions \eqref{eq:ass-cohtrivial} and \eqref{eq:ass-exact} of Setting \ref{basic-setting}. Indeed, if $n=2$, these conditions hold by Remark \ref{rem:plane-curves}. If $n\geq 4$ then they hold because $N_{z}\setminus \d N_{z}$ is a Stein manifold of dimension $n-1\geq 3$, so we have the vanishing $H^{i}(N_{z},\d N_{z};\R)=H_{2n-2-i}(N_{z}\setminus \d N_{z};\R)=0$ for $i\in \{1,2\}$ required by equation \eqref{eq:vanishing}. Therefore, choosing  in Setting \ref{basic-setting} a standard family of loops $\Upsilon\colon (0,\delta]\times \S^1\ni (r,\theta)\mapsto r\theta\in \D_{\delta}^{*}$, we obtain, for each $z\in \D_{\delta}^{*}$, a monodromy abstract contact open book 
\begin{equation}\label{eq:monodromy-acob}
	(N_{z},\lambda_Z,\phi_{z}),
\end{equation}
see \eqref{eq:basic-acob-example}. Note that the fiber $N_{z}$ is diffeomorphic to $f^{-1}(z)\cap \bar{W}_Y$. If $Y=\C^n$, $\varrho(z)=\tfrac{\pi}{2}||z||^2$, and $W_Y$ is a Milnor ball, then \eqref{eq:monodromy-acob} is the abstract contact open book constructed in Example \ref{ex:Milnor-fibration}. 

Note also that in the formula \eqref{eq:monodromy-acob} we abuse the notation and write $\lambda_{Z}$ instead of its restriction $\lambda_{Z}|_{N_{z}}$. We will do so several times in this section, to keep the formulas concise.%, we will abuse notation this way several times in this section.
\smallskip

We assume furthermore that
\begin{equation}\label{eq:c1-assumption}
	c_{1}(Y\setminus f^{-1}(0))=0.
\end{equation}

Then, as explained in Remark \ref{rem:c1}\ref{item:c1_Kahler}, a nonvanishing section of $K_{Y\setminus f^{-1}(0)}$ endows the abstract contact open book \eqref{eq:monodromy-acob} with a grading. Its graded isotopy type does not depend on $z\in \D_{\delta}^{*}$. 
%key technical ingredient in the proof of Theorem \ref{theo:Zariski}.
\smallskip

To state our results, we need to fix some additional data. Let $h'\colon Z'\to Z$ be a log resolution of $(f,Y)$. That is, $h'$ is a proper modification such that, if $X$ denotes the  proper transform of $Y$, then $h\de h'|_X\colon X\to Y$ is a resolution of $f$. Put $D=(f\circ h)^{-1}(0)$. Then $X$ is a smooth complex manifold, and $D\redd$ is snc. We assume that $h$ is an isomorphism away from $D$ and near the preimage of the collar $\bar{U}_Y\setminus V_Y$. We write the irreducible decomposition of $D$ as
\begin{equation*}
	D=\sum_{i\in \cP} D_{i}+\sum_{i\in \cE} m_{i}D_{i},
\end{equation*}
where the first sum runs through the components of the proper transform of $f^{-1}(0)$, and the second one through exceptional components of $h$. %In other words, $h^{-1}_{*}(f^{-1}(0))=\sum_{i\in \cP} D_{i}$ and $\Exc h=\sum_{i\in \cE} D_{i}$. 
For $i\in \cP$ we put $m_{i}=1$. 

Fix an integer $m\geq 1$ and assume that the log resolution $h$ is \emph{$m$-separating}, that is, 
\begin{equation}\label{eq:m-separatedness}
	\mbox{if } i\neq j \mbox{ and } D_{i}\cap D_{j}\neq \emptyset \mbox{ then } m_{i}+m_{j}>m.
\end{equation} 
It can always be achieved by composing $h'$ with further blowups, see \cite[Lemma 2.9]{BBLN_contact-loci}.

Since $Y$ is Stein, we can fix an ample divisor $H$ on $X$ such that
\begin{equation}\label{eq:H}
	H=\sum_{i\in \cE}b_{i}D_{i}
\end{equation}
for some negative integers $b_{i}$. For $i\in \cP$ we put $b_i=0$.
\smallskip

We will also need some additional notation. First, by assumption \eqref{eq:c1-assumption}, we can write the canonical bundle of $X$ as
\begin{equation}\label{eq:K}
	K_{X}=\sum_{i\in \cP\cup \cE} a_{i}D_{i},
\end{equation}
for some \emph{discrepancies} $a_{i}\in \Z$. Note that if $c_{1}(Y)=0$ then we have $a_{i}=0$ for $i\in \cP$. We put
\begin{equation}\label{eq:Sm}
	S_{m}=\{i\in \cP\cup \cE : m_{i}|m\},\quad 
	S_{m,p}=\{i\in S_{m}: p=\frac{m}{m_i}b_{i}\}.
\end{equation}
%so $S_{m}$ indexes those components of $D_i$ whose multiplicities divide the fixed number $m$, and we have a decomposition $S_{m}=\bigsqcup_{p}S_{m,p}$. 
Note that $\cP\subseteq S_{m,0}$.

Fix $i\in \cP\cup \cE$. Recall from \eqref{eq:stratification} the notation $X_{i}^{\circ}=D_{i}\setminus (D-D_{i})$. Put $D_{i}^{\circ}= X_{i}^{\circ}\cap \bar{W}_X$, where $W_X=h^{-1}(W_Y)$. We now introduce a natural $m_i$-fold covering 
\begin{equation*}
	\nu_{i}\colon B_{i}^{\circ}\to D_{i}^{\circ},
\end{equation*}
following \cite[\sec 2.3]{Denef_Loeser-Lefshetz_numbers}. It can be thought of as a piece of the fiber $f^{-1}(\delta)\cap \bar{W}_Y$ which \enquote{corresponds to} the component $D_i$. To define $\nu_i$, fix any holomorphic chart $G_{X}\subseteq X$ around a point of $X_{i}^{\circ}$, and let $\{z_{i}=0\}$ be a local equation of $D_{i}$ in $G_{X}$. Then  $f|_{G_{X}}=\mu\cdot z_{i}^{m_i}$ for some $\mu\in \cO_{X}^{*}(G_{X})$. Define $\tilde{G}_{X}=\{(z,x)\in \C\times (G_{X}\cap D_{i}^{\circ}): \mu(x)\cdot  z^{m_{i}}=1 \}$. Gluing these charts, one gets a topological covering $\nu_{i}\colon B_{i}^{\circ}\to D_{i}^{\circ}$ with Galois group $\Z_{m_i}$, as needed. 

Note that if $i\in \cE$, i.e.\ $D_i$ is contracted by $h$, then $\nu_i$ is exactly the $m_i$-fold covering of $X_{i}^{\circ}$ considered in \cite[\sec 2.3]{Denef_Loeser-Lefshetz_numbers} or \cite[Theorem 1.2]{McLean}. If $i\in \cP$, i.e.\ $D_{i}$ is a proper transform of some component of $f^{-1}(0)$, then $\nu_i$ is an isomorphism; and $B_{i}^{\circ}$ can be identified with a submanifold $X_{i}^{\circ}\cap \bar{W}_{X}$ of $X_{i}^{\circ}$, with boundary $D_i\cap \d \bar{W}_{X}$. In other words, if $i\in \cP$ then $B_{i}^{\circ}$ is diffeomorphic to $(h(D_i)\cap \bar{W}_Y)\setminus \operatorname{Bs} h^{-1}$, where $\operatorname{Bs} h^{-1}$ is the locus where $h$ is not an isomorphism.

With this notation at hand, we can formulate the first result of this section. We denote by $H^{BM}_{*}$ the Borel--Moore homology with coefficients in $\Z_2$.

\begin{prop}\label{prop:spectral-sequence-monodromy}
	There is a spectral sequence $E_{p,q}^{r}$ converging to $\HF_{*}(\phi_{z}^{m},+)$ whose first page equals
	\begin{equation}\label{eq:spectral-sequence-monodromy}
		E_{p,q}^{1}=\bigoplus_{i\in S_{m,p}} H^{BM}_{n-1+p+q+2\frac{m}{m_i}(a_i+1)-2m}(B_{i}^{\circ}).
	\end{equation}
\end{prop}

Proposition \ref{prop:spectral-sequence-monodromy} will be proved in two steps. First, in Section \ref{sec:isotopy-to-radius-zero}, we will use the lift of the radial vector field on $\C_{\log}$ to isotope the abstract contact open book \eqref{eq:monodromy-acob} to the abstract contact open book $(F,\lambda,\phi)$ at radius zero, i.e.\ at the boundary of the A'Campo space; where $\phi$ is the time one flow of the vector field \eqref{eq:monodromy-vector-field}. Next, in Section \ref{sec:monodromy-dynamics} we will identify $B_{i}^{\circ}$ with the intersections $F\cap A_{i}^{\circ}$, see \eqref{eq:stratification-pullback}, which, as we will see, are codimension zero families of fixed points of $\phi$ whenever $m_i|m$. The $m$-separating condition \eqref{eq:m-separatedness} will guarantee that $\phi$  has no more fixed points. %In Proposition \ref{prop:monodromy}\ref{item:B_i-action},\ref{item:B_i-CZ} we will compute the action and the Conley--Zehnder index for these families. 
Once this is done, Proposition \ref{prop:spectral-sequence-monodromy} will follow from Proposition \ref{prop:spectral-sequence}.

	\begin{remark}\label{rem:chosen_H}
		The spectral sequence \eqref{eq:spectral-sequence-monodromy} depends on the choice of the $m$-separated resolution $h$ and on the ample divisor $H$. Since the limit $\HF_{*}(\phi^{m},+)$ does not depend on $h$ and $H$, flexibility of these choices suggests that some Floer differentials in \eqref{eq:spectral-sequence-monodromy} should be zero.
		
		One reason for such vanishing can be seen as follows. Since the action functional decreases along the Floer trajectories, see Section \ref{sec:action-filtration}, if we have an inequality of actions  $\ac(B_{i}^{\circ})<\ac(B_{j}^{\circ})$ then there are no Floer trajectories from $B_{i}^{\circ}$ to $B_{j}^{\circ}$. Hence any differential from (a sub-quotient of) $H^{BM}_{*}(B_{i}^\circ)$ to $H^{BM}_{*}(B_{j}^{\circ})$ vanishes. By Proposition \ref{prop:monodromy}\ref{item:B_i-action}, the inequality $\ac(B_{i}^{\circ})<\ac(B_{j}^{\circ})$ is equivalent to $\frac{b_i}{m_i}<\frac{b_j}{m_j}$. 
		
		For example, let $f$ be an irreducible plane curve germ. Then the dual graph of $D$ is a tree, with the proper transform of $f^{-1}(0)$ as a root. Following the resolution process it is easy to see that, for any $H$, the quotients $\frac{b_i}{m_i}$ decrease as we move away from the root. Thus if $D_i$ lies closer to the root than $D_j$, all differentials in \eqref{eq:spectral-sequence-monodromy} from $H_{*}^{BM}(B_i^{\circ})$ to $H^{BM}_{*}(B_{j}^{\circ})$ vanish.  However, in the opposite situation, that is if $D_{j}$ lies closer to the root than $D_i$, the ampleness condition is restrictive enough so that we cannot guarantee the vanishing of differentials from $H^{BM}_{*}(B_{i}^{\circ})$ to $H^{BM}_{*}(B_{j}^{\circ})$, see Example \ref{ex:cusp}. 
		
		This result is compatible with the corresponding one concerning topology of the $m$-th contact locus $\mathscr{X}_{m}$ of $f$. Recall that \cite[Theorem 1.1]{BBLN_contact-loci} provides a spectral sequence converging to $H^{*}_{c}(\mathscr{X}_m)$ whose page is dual (up to a shift) to \eqref{eq:spectral-sequence-monodromy}, which  leads to the \emph{Arc-Floer conjecture} $\HF_{*}(\phi^{m},+)\cong H^{*}_{c}(\mathscr{X}_{m})$, see Conjecture 1.5 loc.\ cit. The differentials $H_{*}(B_j)\to H_{*}(B_{i})$ in the above spectral sequence count, roughly speaking, families of arcs lifting to $D_{i}$, whose limit lifts to $D_j$. In case of irreducible plane curves, the proof of \cite[Theorem 3.4]{JaviEdu_conctact-loci} shows that there are no such families if $D_i$ lies closer to the root than $D_j$, so the corresponding differential, again, is zero. But, in the opposite situation, if $D_j$ lies closer to the root than $D_i$, the differential can be non-zero.
	\end{remark}

	\begin{example}\label{ex:cusp}
		Let $f$ be a simple cusp $x^{2}-y^{3}$. By \cite[p.\ 26]{JaviEdu_conctact-loci} we have $\HF_{*}(\phi^{6},+)=H_{*-1}^{BM}(B^{\circ})$, where $B^{\circ}$ is the entire Milnor fiber, i.e.\ a punctured torus. Thus $\HF_{*}(\phi^{6},+)$ equals $\Z_2$ in degrees $2$ and $3$, and $0$ in the remaining degrees. Let $X\to \C^{2}$ be the minimal log resolution of $f$; it is $6$-separated. We now describe the spectral sequence \eqref{eq:spectral-sequence-monodromy} for various choices of ample divisors $H$ on $X$. 
		
		Write $D=D_0+2D_1+3D_2+6D_3$, where $D_0$ is the proper transform of $f^{-1}(0)$, $D_3$ is the component of $\Exc h$ meeting $D_0$, and $D_1,D_2$ are tips of $\Exc h$. Then $B_{1}^{\circ}$ and $B_{2}^{\circ}$ are disjoint unions of $3$ and $2$ disks, respectively; and $B_{3}^{\circ}$ is a torus with six disks removed. Since $K_{X}=D_{1}+2D_{2}+4D_{3}$, by Proposition \ref{prop:monodromy}\ref{item:B_i-CZ} we have $\CZ(B_{1}^{\circ})=\CZ(B_{2}^{\circ})=0$ and $\CZ(B_{3}^{\circ})=-2$. A direct computation shows that a divisor $H=\sum_{i=1}^{3}b_{i}D_{i}$ is ample if and only if $\frac{b_1}{m_1},\frac{b_2}{m_2}<\frac{b_3}{m_3}$, see Remark \ref{rem:chosen_H}, and $2\frac{b_{1}}{m_{1}}+3\frac{b_2}{m_2}>6\frac{b_{3}}{m_{3}}$. 
		
		Choose $H$ so that $\frac{b_{1}}{m_{1}}=\frac{b_{2}}{m_{2}}=\frac{-13}{6}$ and $\frac{b_{3}}{m_{3}}=\frac{-12}{6}$ (strictly speaking, $H$ is a $\Q$-divisor, but it is easy to see that \eqref{eq:spectral-sequence-monodromy} works  as long as $6\cdot \frac{b_i}{m_i}\in \Z$). Then the nonzero entries of the $E^{1}$ page of \eqref{eq:spectral-sequence-monodromy} are $E^{1}_{-13,14}=H^{BM}_{2}(B_{1}^{\circ})\oplus H^{BM}_{2}(B_{2}^{\circ})=\Z_{2}^{5}$, $E^{1}_{-12,14}=H^{BM}_{1}(B_{3}^{\circ})=\Z_{2}^{6}$, and $E^{1}_{-12,15}=H^{BM}_{2}(B_{3}^{\circ})=\Z_{2}$. In particular, in agreement with Remark \ref{rem:chosen_H}, there is no differential from $H_{*}^{BM}(B_{1}^{\circ})$ or $H_{*}^{BM}(B_{2}^{\circ})$ to $H_{*}^{BM}(B_{3}^{\circ})$. On the other hand, there is a non-trivial differential in the opposite direction. Indeed, since we know that the limit $\HF_{*}(\phi^{6},+)$ is $0$ in degree $1$,  we see that the differential $E^{1}_{-12,14}\to E^{1}_{-13,14}$ is surjective, and the spectral sequence degenerates on the second page.
		
		Similarly, choose $H$ so that  $\frac{b_{1}}{m_{1}}=\frac{-13}{6}$, $\frac{b_{2}}{m_{2}}=\frac{-14}{6}$ and $\frac{b_{3}}{m_{3}}=\frac{-12}{6}$. Then the nonzero entries on the $E^{1}$-page are $E^{1}_{-14,15}=H^{BM}_{2}(B_{2}^{\circ})=\Z_{2}^{2}$, $E^{1}_{-13,14}=H^{BM}_{2}(B_{1}^{\circ})=\Z_{2}^{3}$; and as before  $E^{1}_{-12,14}=H^{BM}_{1}(B_{3}^{\circ})=\Z_{2}^6$, $E^{1}_{-12,15}=H^{BM}_{2}(B_{3}^{\circ})=\Z_{2}$. The only nonzero differential on $E^1$ is $E^{1}_{-12,14}\onto E^{1}_{-13,14}$; on $E^{2}$ it is $E^{2}_{-12,14} \onto E^{2}_{-14,15}$ and on $E^{3}$ the spectral sequence degenerates. These differentials count Floer trajectories from $B_{3}^{\circ}$ to $B_{1}^{\circ}$ and $B_{2}^{\circ}$; there are no Floer trajectories from $B_{1}^{\circ}\sqcup B_{2}^{\circ}$ to $B_{3}^{\circ}$.
		
		For $\frac{b_{1}}{m_{1}}=\frac{-14}{6}$, $\frac{b_{2}}{m_{2}}=\frac{-13}{6}$, $\frac{b_{3}}{m_{3}}=\frac{-12}{6}$, we get an analogous spectral sequence, with the roles of $D_1$ and $D_2$ interchanged. 	For $\frac{b_{1}}{m_{1}}=\frac{-15}{6}$, $\frac{b_{2}}{m_{2}}=\frac{-13}{6}$, $\frac{b_{3}}{m_{3}}=\frac{-12}{6}$, the spectral sequence degenerates on $E^{4}$.
		
		% For the point $(-15,-13)$, one gets a spectral sequence degenerating on the fourth page. Still, the only non-trivial differential counts Floer trajectories from $B_{3}^{\circ}$ to $B_{1}^{\circ}$ and $B_{2}^{\circ}$, never the other way round.
	\end{example}

\subsection{Isotopy to radius zero}\label{sec:isotopy-to-radius-zero}

We will now construct the \enquote{radius zero} graded abstract contact open book $(F,\lambda,\phi)$, and show that it is isotopic to our original one \eqref{eq:monodromy-acob}. We keep the notation and assumptions from Section \ref{sec:monodromy-spectral-sequence}. 

Write $\tilde{f}\de f\circ h$. Put $\tilde{\lambda}_{Y}=-d^{c}(\varrho\circ h)\in \Omega^{1}(X)$, i.e.\ $\tilde{\lambda}_{Y}$ is the suitable restriction of the pullback of the $1$-form $\lambda_{Z}$. 
Since $h|_{X\setminus D}$ is an isomorphism, the form $\tilde{\lambda}_{Y}$ makes $\tilde{f}|_{h^{-1}(N)}$ a Liouville fibration. Moreover, since $h$ is an isomorphism near the collar, the collar trivialization chosen in Section \ref{sec:monodromy-spectral-sequence} lifts to $X$. In particular, for each $z\in \D_{\delta}^{*}$ we have a graded abstract contact open book  $(\tilde{N}_{z},\tilde{\lambda}_{Y},\tilde{\phi}_{z})\de (h^{-1}(N_z),\tilde{\lambda}_{Y},h^{-1}\circ\phi_{z}\circ h)$ isomorphic to \eqref{eq:monodromy-acob}. We will now modify the form $\tilde{\lambda}_{Y}|_{X\setminus D}$ to a form $\lambda_{X}$ such that $d\lambda_{X}$ extends to a K\"ahler form on $X$.

Recall that, in formula \eqref{eq:H}, we have fixed an ample divisor $H$. Fix an integer $b>0$ such that $bH$ is very ample, and let $\jmath\colon X\into \P^{\dim |bH|}$ be the embedding given by $|bH|$. Let $\tfrac{1}{4\pi} \omega_{\mathrm{FS}} \in \Omega^{2}(X)$ be the pullback of the Fubini-Study form through $\jmath$. The basis of $H^{0}(\cO_{X}(bH))$ used to define $\jmath$ gives a hermitian metric on $\cO_{X}(bH)$, see \cite[Example 4.1.2(i)]{Huybrechts}.  Let $||\sdot ||$ be the corresponding norm. Choose a section $s$ of $\cO_{X}(bH)$ which is supported on $H$, i.e.\ with a pole of order $-bb_{i}>0$ along $D_{i}$ for each $i\in \cE$, and no zeros or poles elsewhere.  
Then by \cite[Examples 4.3.9(iii) and 4.3.12]{Huybrechts}, we have $\omega_{\mathrm{FS}}=2\imath\cdot \bar{\d}\d\log ||s||$. Since $d^{c}=\imath(\d-\bar{\d})$, see Section \ref{sec:symplectic-intro}, we have %$\bar{\d}\d=\frac{i}{2}dd^{c}$
$dd^{c}=2\imath\cdot \bar{\d}\d$, so
\begin{equation*}
	\omega_{\mathrm{FS}}=d\lambda_{\mathrm{FS}},\quad \mbox{where}\quad
	\lambda_{\mathrm{FS}}\de d^{c}\log ||s||\in \Omega^{1}(X\setminus \Exc h).
\end{equation*}
We remark that the $1$-form $\lambda_{\mathrm{FS}}$ is, up to scaling, the same form as the one used in \cite[Example 5.14]{McLean}: the sign difference results from varying definitions of $d^{c}$, explained in Section \ref{sec:symplectic-intro}.
\smallskip

For $t\geq 0$ put $\lambda_t=\tilde{\lambda}_{Y}+t\cdot \lambda_{\mathrm{FS}}\in \Omega^{1}(X\setminus \Exc h)$, and put  $\omega_{t}\de d\lambda_{t}$. Then $\omega_{t}$ is a K\"ahler form on $X\setminus D$ for every $t\geq 0$, and on $X$ for every $t> 0$. In particular, since $f$ is holomorphic, these forms are fiberwise symplectic over $\D_{\delta}^{*}$. Now, the map $(\tilde{f},\pr_{[0,t_0]})\colon X\times [0,t_0]\to \C\times [0,t_0]$ fits into Setting \ref{basic-setting}, with $P=[0,t_0]$, $P'=\{0\}$. Hence after possibly shrinking $t_0>0$, we get a subset $M\subseteq X\times [0,t_0]$ such that the above map restricted to $M\setminus (D\times [0,t_0])$ is a Liouville fibration over $\D_{\delta}^{*}\times [0,t_0]$, with a collar trivialization which extends over $\D_{\delta}\times [0,t_0]$, and agrees over $\D_{\delta}\times \{0\}$ with the one chosen in Section \ref{sec:monodromy-spectral-sequence} to define \eqref{eq:monodromy-acob}. Taking for $\Upsilon$ the standard family of loops, the construction in Section \ref{sec:basic-setting} provides monodromy graded abstract contact open books $(M_{z,t},\lambda_{t},\phi_{z,t})$, see formula \eqref{eq:basic-acob}. %Note that in case $\dim_{\C}Y=2$ we use the assumption on the links of $Y$ to apply Remark \ref{rem:exact} making $\phi_{z,t}$ exact; and we use the vanishing vanishing \eqref{eq:vanishing} to endow $(M_{z,t},\lambda_{t},\phi_{z,t})$ with a grading. 

Since $(M_{z,0},\lambda_{0},\phi_{z,0})=(\tilde{N}_{z},\tilde{\lambda}_{Y},\tilde{\phi}_{z})$, we infer that for every $z\in \D_{\delta}^{*}$ and every $t\in [0,t_0]$, the abstract contact open book $(M_{z,t},\lambda_{t},\phi_{z,t})$ is graded isotopic to \eqref{eq:monodromy-acob}.

Put $\lambda_{X}=\lambda_{t_0}$, $\omega_{X}=d\lambda_{X}$. Then $\omega_{X}$ extends to a K\"ahler form on $X$. Let $A$ be the A'Campo space associated to the function $\tilde{f}\colon X\to \C$, with natural smooth maps $\pi\colon A\to X$ and $\tilde{f}_{A}\colon A\to \C_{\log}$ as in Proposition \ref{prop:AXsmooth}, see diagram \eqref{eq:AX-diagram}. 

We note that the pullback $\pi^{*}\lambda_{X}$ extends to a $1$-form on $A$, i.e.\ assumption \eqref{eq:lambda_extends} of Setting \ref{basic-setting2} holds. To see this, we give a local description of $\lambda_{X}$, which will be useful in the proof of Proposition \ref{prop:monodromy}\ref{item:B_i-action}, too. Choose a fine chart $U_{X}$ (see Definition \ref{def:finechartz}) with associated index set $S$ (see \eqref{eq:index-set}), and a trivialization of $\cO_{X}(bH)|_{U_X}$. In this trivialization, we can write our section $s$ as $\xi' \cdot \prod_{i\in S}z_{i}^{bb_{i}}$ for some nonvanishing holomorphic function $\xi'$. The formula for $||\sdot||$ from \cite[Example 4.1.2(i)]{Huybrechts} shows that $\log ||s||=\sum_{i\in S}bb_{i}\log r_{i}+\xi$ for some $\xi\in \cC^{\infty}(U_X)$. Thus $\lambda_{\mathrm{FS}}|_{U_X}=d^{c}\log||s||=-\sum_{i\in S}bb_id\theta_i+d^{c}\xi$, see Example \ref{ex:J}. Since $\tilde{\lambda}_{Y}\in \Omega^{1}(X)$, we get
\begin{equation}\label{eq:lambdaA-locally}
	\pi^{*}\lambda_{X}|_{U_X}=-\sum_{i\in S}t_0bb_i\, d\theta_{i}+\pi^{*}\beta\quad\mbox{for some } \beta\in \Omega^{1}(U_X). 
\end{equation}
Because the functions $\theta_{i}$ are coordinates of each chart \eqref{eq:AC-chart} covering $\pi^{-1}(U_X)$, the formula \eqref{eq:lambdaA-locally} shows that $\pi^{*}\lambda_{X}$ extends to a smooth form on $A$, as claimed.

We are now in Setting \ref{basic-setting2}. Let $F=\tilde{f}_{A}^{-1}(0,1)\cap \pi^{-1}(\bar{W}_{X})$ be a radius-zero fiber, and for $\delta',\epsilon>0$ let $\lambda_{A}^{\delta',\epsilon}\in \Omega^{1}(A)$ be the $1$-form defined in \eqref{eq:lambdaACdelta}. By Corollary \ref{cor:isotopyisotopy-to-radius-zero}, there is an arbitrarily small  $\epsilon_0>0$ such that, for any $\delta'>0$, the abstract contact open book $(M_{z,t_0},\lambda_{t_0},\phi_{z,t_{0}})$ constructed above is isotopic to 
\begin{equation}\label{eq:radius-zero-acob}
	(F,\lambda,\phi),
\end{equation}
where $\lambda=\lambda_{A}^{\delta',\epsilon_0}|_{F}$, and $\phi\colon F\to F$ is the time-one flow of the monodromy vector field \eqref{eq:monodromy-vector-field}.
Moreover, this isotopy endows \eqref{eq:radius-zero-acob} with a grading.  We conclude that the natural monodromy abstract contact open book \eqref{eq:monodromy-acob} is graded isotopic to the radius-zero abstract contact open book \eqref{eq:radius-zero-acob}. The properties of the latter are listed in Proposition \ref{prop:monodromy} below.

\subsection{Dynamics at radius zero}\label{sec:monodromy-dynamics}

In Section \ref{sec:isotopy-to-radius-zero} above, we have constructed an isotopy joining the natural monodromy abstract contact open book \eqref{eq:monodromy-acob} with an abstract contact open book $(F,\lambda,\phi)$ defined after the formula  \eqref{eq:radius-zero-acob}. %Here $F$ is a fiber at radius zero, i.e.\ in the boundary of the A'Campo space. 
Dynamical properties of $\phi$ are summarized in the Proposition \ref{prop:monodromy}, which is a key step in the proof of Theorem \ref{theo:Zariski}.  We will state it using notation from Sections \ref{sec:monodromy-spectral-sequence} and \ref{sec:isotopy-to-radius-zero}, which we now briefly recall.% what is needed to state our result.

The fiber $F$ in \eqref{eq:radius-zero-acob} is the intersection of the radius-zero fiber $\tilde{f}_{A}^{-1}(0,1)$ with the preimage $\bar{W}$ of the sub-level set $\bar{W}_Y$ of a strictly plurisubharmonic function $\varrho\colon Y\to \R$ chosen in the beginning of Section \ref{sec:monodromy-spectral-sequence}. This sub-level set should be thought of as a compact \enquote{ball} around the singularities of $f^{-1}(0)$ and $Y$, not necessarily a small one. The subset $A_{i}^{\circ}\subseteq \d A$ is a  piece of the \enquote{radius-zero tube} $\d A$ lying over $D_{i}\setminus (D-D_{i})$, see the formula  \eqref{eq:stratification-pullback}. The manifold $B_{i}^{\circ}$ was introduced in Section \ref{sec:monodromy-spectral-sequence} as follows: if $D_{i}\subseteq \Exc h$ then $\nu_{i}\colon B_{i}^{\circ}\to  D_{i}\setminus (D-D_{i})$ is the $m_i$-fold covering described in \cite[\sec 2.3]{Denef_Loeser-Lefshetz_numbers}; and if $D_{i}$ is a proper transform of a component $C$ of $f^{-1}(0)$ then $B_{i}^{\circ}=\bar{W}_{Y}\cap C\setminus \operatorname{Bs} h^{-1}$, which should be thought of as the smooth part of $C$ intersected with the closed \enquote{ball} $\bar{W}_Y$. The integer $a_{i}$ is the discrepancy of $D_{i}$, see \eqref{eq:K} and $b_{i}$ is the coefficient of $D_i$ in the chosen ample divisor $H$, see \eqref{eq:H}: in particular, $b_i<0$ if $D_i\subseteq \Exc h$ and $b_i=0$ otherwise. The integer $b>0$ is chosen so that $bH$ is very ample; and $t_0,\epsilon_0>0$ are arbitrarily small positive numbers chosen in Section \ref{sec:isotopy-to-radius-zero} in such a way that $\epsilon_0$ can be taken arbitrarily small compared to $t_0$.

\begin{prop}\label{prop:monodromy}
	Let $(F,\lambda,\phi)$ be the graded abstract contact open book \eqref{eq:radius-zero-acob}. Then %the fixed point set of $\phi^{m}$ is
	\begin{equation}\label{eq:fixed-points}
		\Fix \phi^{m}=\bigsqcup_{m_{i}|m}B_{i},
		\quad \mbox{where} \quad
		B_{i}= \overline{A_{i}^{\circ}}\cap F.
	\end{equation}
	%see Section \ref{sec:AX-def} for definition of $A_{i}^{\circ}\subseteq A$. 
	Moreover, for each $i$ such  that $m_{i}|m$, every connected component of $B_{i}$ is a codimension zero family of fixed points of $\phi^m$, see Definition \ref{def:0codim}. Write $\d B_i=\d^F B_i\sqcup \d^{+} B_i \sqcup \d^{-}B_{i}$, where $\d^{F}B_i=F\cap \d B_i$ and $\d^{\pm}B_{i}$ is the disjoint union of $\d^{\pm} B_{i}'$ for all connected components $B_{i}'$ of $B_{i}$, see formula \eqref{eq:boundarysplitting1}. Then the following holds.
	\begin{enumerate}
		\item \label{item:B_i-boundary} We have $\d^{-} B_i=\emptyset$ and $\partial^F B_i= B_{i}\cap\partial \bar{W}$.
		\item \label{item:B_i-covering} There is a diffeomorphism $\kappa_i\colon B_{i}\setminus \d^{+}B_{i}\to B_{i}^{\circ}$ such that $\nu_{i}\circ \kappa_i=\pi|_{F}$.
		\item \label{item:B_i-CZ} The Conley--Zehnder index of each point in $B_{i}$ equals $\CZ(B_i)=2\frac{m}{m_{i}}(a_{i}+1)-2m$.
		\item \label{item:B_i-action} We can choose the action function $\ac\colon F\to \R$, see \eqref{eq:exact-symplectiomorphism}, % i.e.\ a function satisfying $\lambda-\phi^{*}\lambda=d\ac$, such that 
		so that $\ac|_{B_i}=(t_0b\cdot b_{i}+\epsilon_0)\frac{m}{m_i}$.
	\end{enumerate}	
\end{prop}

\begin{proof}
	Recall that $\phi^{m}\colon F\to F$ is the time $m$ flow of the symplectic lift of the unit angular vector field on $\S^1=\d \C_{\log}$. Fix a fine chart $U_{X}\subseteq X$, i.e.\ one satisfying the statement of Lemma \ref{lem:smoothchart}, let $U=\pi^{-1}(U_X)$, and let $S$ be the associated index set \eqref{eq:index-set}. In local coordinates \eqref{eq:AC-chart} on $U$, this lift is given by the formula \eqref{eq:monodromy-vector-field}. Therefore, $\phi^{m}|_{U}$ is given by
	\begin{equation}\label{eq:monodromy}
		\phi^{m}\colon \theta_{j}\mapsto \theta_{j}+m \ell_{j},
		\quad\mbox{where}\quad 
		\ell_{j}=\frac{\zeta(u_j)}{\sum_{i\in S}m_{i}\zeta(u_i)}
	\end{equation}
	In other words, $\phi$ is a rotation by an angle $2\pi m\ell_{j}$ about the $j$-th axis, see Notation \ref{not:S1}. Note that by Lemma \ref{lem:intro}\ref{item:intro-u=ubar}, the values of $\ell_j$ do not depend on the choice of the chart.
	
	To prove the equality \eqref{eq:fixed-points}, we use the $m$-separatedness assumption \eqref{eq:m-separatedness}, cf.\  \cite[p.\ 242]{A'Campo}.
	
	First, we claim that $\Fix \phi^{m}\subseteq \bigsqcup_{i\in \cP\cup\cE} B_{i}$. To prove this, take $x\in F$ such that $x\not\in \bar{A}_{i}^{\circ}$ for any $i\in \cP\cup\cE$. Put $I=\{i\in S: w_i(x)>0\}$. By Proposition \ref{prop:AX-topo}\ref{item:top-S1-bundle}, $x\in \bar{A}_{I}^{\circ}$, so $\# I\geq 2$ by assumption. Choose $j\in I$ such that $w_i(x)\geq w_j(x)$ for all $i\in I$. 
	Since both functions $\eta$ and $\zeta$ are increasing, we infer that $\zeta(u_i(x))\geq \zeta(u_j(x))>0$ for all $i\in I$, so $\frac{\zeta(u_i(x))}{\zeta(u_j(x))}\geq 1$ if $i\in I$ and $\frac{\zeta(u_i(x))}{\zeta(u_j(x))}=0$ if $i\not\in I$. Hence 
	\begin{equation*}
		0<m\ell_{j}(x)=m\left(\sum_{i\in S}m_{i}\frac{\zeta(u_i(x))}{\zeta(u_j(x))}\right)^{-1}\leq m\left(\sum_{i\in I} m_{i}\right)^{-1}\overset{\mbox{\tiny{\eqref{eq:m-separatedness}}}}{<} 1,
	\end{equation*}
	%where the second-to-last equality follows from \eqref{eq:m-separatedness}. 
	In particular, $m\ell_{j}(x)\not\in \Z$. Formula \eqref{eq:monodromy} shows that $x\not\in \Fix\phi^{m}$, as claimed.
	
	Therefore,  $\Fix \phi^{m}\subseteq \bigsqcup_{i\in \cP\cup\cE} B_{i}$. Take $x\in B_{i}$. By Proposition \ref{prop:AX-topo}\ref{item:top-S1-bundle} we have $w_{i}(x)=1$ and $w_{j}(x)=0$ for all $j\neq i$. Thus  $m\ell_{i}=\tfrac{m}{m_i}$ and $\ell_{j}=0$ for $j\neq i$, so formula \eqref{eq:monodromy} shows that $x\in \Fix\phi^{m}$ if and only if $m_i|m$. This ends the proof of \eqref{eq:fixed-points}. 
	\smallskip
	
	Now, we claim that each $B_{i}$ is a codimension zero family of fixed points, see Definition \ref{def:0codim}. Say that $i=1$. First, we claim that $B_{1}\subseteq F$ is a codimension zero submanifold with corners.
	
	Fix a point $x\in B_1$, let $U_{X}$ be a fine chart around $\pi(x)$, let $S$ be its index set \eqref{eq:index-set}, and let $U=\pi^{-1}(U_X)$. Recall that by Lemma \ref{lem:intro}\ref{item:intro-u=ubar}, we have $\bar{u}_{j}|_{U\cap F}=u_{j}|_{U\cap F}$ for $j\in S$, %and therefore $\bar{v}_{j}|_{U\cap F}=v_{j}|_{U\cap F}$ for all $j\in S$, 
	where $\bar{u}_j$ %, $\bar{v}_j$ 
	are global functions introduced in \eqref{eq:def-vbar-ubar-mu}, and $u_j$ are the local ones, introduced in \eqref{eq:def-w_i-u_i}. By Proposition \ref{prop:AX-topo}\ref{item:top-S1-bundle}, we have $\bar{A}_{1}^{\circ}\cap U=\{u_{1}=1\}$. 	In particular, $\bar{A}_{1}^{\circ}\cap U$ is contained in the open subset $U_{1}\subseteq U$ introduced in \eqref{eq:Ui}.  There, we have a smooth coordinate system $(g,(\bar{v}_{i})_{i\in S\setminus \{1\}},\rest)$, see \eqref{eq:AC-chart}, where $U_1\cap \d A=\{g=0\}$. 
	
	By Lemma \ref{lem:intro}\ref{item:intro-sum}, we have $\bar{A}_{1}^{\circ}\cap U=\{u_1=1\}=\{u_{i}=0: i\in S\setminus \{1\}\}$. For $x\in U\cap \d A$, we have $u_{i}(x)=0$ if and only if $\bar{v}_{i}(x)\geq 0$. Indeed, if $u_i(x)=0$ then by Lemma \ref{lem:intro}\ref{item:intro-independence} $u_{i}^{p}(x)=0$ for every $p\in R$ such that $\tau^{p}(x)\neq 0$, which by \eqref{eq:def-vbar-ubar-mu} implies that $\bar{v}_{i}(x)\geq 0$. Conversely, if $\bar{v}_{i}(x)\geq 0$ then $v_{i}^{p}(x)\geq 0$ for some $p\in R$, which by Lemma \ref{lem:Ui-simple}\ref{item:Ui-cap-dA} implies that $u_{i}^{p}(x)=0$, so again by  Lemma \ref{lem:intro}\ref{item:intro-independence} we get $u_{i}(x)=0$. We conclude that the subset $\bar{A}_{1}^{\circ}\cap U_1\subseteq \d A\cap U_1$ is given by the inequalities $\{\bar{v}_{i}\geq 0: i\in S\setminus \{1\}\}$, see Figure \ref{fig:v3}.  	
	The fiber $F\cap U_1$ is given in $\d A\cap U_1$ by $\sum_{i\in S}m_{i}\theta_{i}=0$. Therefore, $B_{1}\subseteq F$ is a codimension zero submanifold with corners, as claimed.
	\smallskip
	
	Now, according to Definition \ref{def:0codim}, we need to find a neighborhood $N_1$ of $B_1$, and a time independent Hamiltonian $H_{1}\colon N_1\to \R$ such that $\phi|_{N_1}=\psi_{1}^{H_1}$. Define $N_{1}\subseteq F$ as the union of all the sets $U_{1}\cap F$, where $U_1$ is the set \eqref{eq:Ui} defined for an adapted chart $U_X$, and $U_X$ ranges through all fine charts meeting $D_1$. We have seen that $N_{1}$ is a neighborhood of $B_{1}$ in $F$. We claim that
	\begin{equation*}
	 H_{1}\de \tfrac{m}{m_1}\epsilon_0(\bar{v}_{1}+1)|_{N_1}\colon N_{1}\to [0,\infty)	
	\end{equation*}
	is the required time-independent Hamiltonian.
	
	To prove this, once again we restrict to a chart $U_{1}$ with coordinates \eqref{eq:AC-chart}. There, the monodromy $\phi^{m}$ is the time one-flow of a vector field $mX_{\theta}$, where $X_{\theta}$ is the symplectic lift of $\tfrac{\d}{\d \theta}$, given by the formula \eqref{eq:monodromy-vector-field}. Recall that $\tfrac{\d}{\d \theta_1}$ is the coordinate vector field of the chart \eqref{eq:AC-chart}. Put
	\begin{equation*}
		X_1\de mX_{\theta}-\tfrac{m}{m_1}\tfrac{\d}{\d \theta_{1}}.
	\end{equation*}
	We claim that $X_1$ is the Hamiltonian vector field $X^{H_1}$. Since $\theta_{*}(X_{\theta})=\tfrac{\d}{\d \theta}$ by definition of $X_{\theta}$, we have $\theta_{*}(X_1)=0$, so $X_1$ is tangent to the fiber $F$. Because $X_{\theta}$ is orthogonal to $F$, in the tangent space to $F$ we have
	\begin{equation*}
		\omega\AC^{\delta',\epsilon_0}(X_1,\sdot)=\omega\AC^{\epsilon_0}(-\tfrac{m}{m_1}\tfrac{\d}{\d\theta_1},\sdot )=\tfrac{m}{m_1}\epsilon_0 d\bar{v}_1=dH_{1},
	\end{equation*}
	where the second equality holds by Lemma \ref{lem:TdA}\ref{item:omega-dtheta}. Thus $X_{1}=X^{H_1}$, as claimed. 
	
	To verify the conditions of Definition \ref{def:0codim}, we need to check that $\phi^{m}|_{N_1}=\psi^{H_1}_{1}$. Recall that $\phi^{m}$ is the time one flow of $mX_{\theta}$, so it differs from $\psi^{H_1}_{1}$ by a translation by $\tfrac{m}{m_1}$ in the $\theta_{1}$-direction. Since $\tfrac{m}{m_1}\in \Z$, these time one flows are equal, as needed.
	\smallskip
	
	It remains to verify the properties \ref{item:B_i-boundary}--\ref{item:B_i-action}. Part \ref{item:B_i-boundary} is clear since the Hamiltonian $H_{i}$ defined above for each codimension zero family $B_i$ satisfies $H_i=\frac{m}{m_1}\epsilon_0(\bar{v}_{i}+1)\geq 0$ and $H_{i}|_{B_{i}}\equiv 0$. 
	\smallskip
	
	\ref{item:B_i-covering} 
	Locally on $A_{i}^{\circ}$, the fiber $F$ is given by $\{m_i\theta_{i}=0\}$, and the restriction of  $\pi$ collapses the angular coordinate $\theta_{i}$. This yields the required covering, cf.\ Remark \ref{rem:vanishing_restriction}.
	\smallskip
	
	\ref{item:B_i-CZ} Say $i=1$. Recall that the Conley--Zehnder index is constant in a codimension zero family of fixed points, so it is enough to compute it for a single point $x\in F\cap A_{1}^{\circ}$. Choose a fine a chart $U_{X}\subseteq X$ around $\pi(x)$ whose associated index set is $\{1\}$, i.e.\ $U_{X}$ does not meet $D\setminus D_1$. The smooth coordinates \eqref{eq:AC-chart} on $U\cap \d A$ are $(\theta_{1},z_{i_2},\dots,z_{i_n})$: we have skipped $g$ since it cuts out $\d A$. We will compute $\CZ(x)$ using \cite[Lemma A.8]{McLean}. To do this, recall that the grading of $(F,\lambda,\phi)$ is induced by isotopy from the one on $(M_{z,t_0},\lambda_z,\phi_{z,t_0})$. The latter is given by a section $s\vert$ of the relative canonical bundle $K\vert$, such that $s\vert\otimes d\tilde{f}$ maps under \eqref{eq:K-vert} to the section $s_{K}$ of $K_{X\setminus D}$ chosen in Section \ref{sec:basic-grading}. In this case, $s_{K}$ is simply the pullback of the holomorphic volume form on $\C^{n}$, see Example \ref{ex:Milnor-fibration}. 
	
	Since $s_{K}$ has zero of order $a_1\in \Z$ along $D_1$, we can choose the coordinates $(z_{i_2},\dots,z_{i_{n}})$ of our adapted chart $U_X$ so that $s_{K}|_{U_X}=z_{1}^{a_1}\, dz_{1}\wedge dz_{i_2}\wedge \dots \wedge dz_{i_n}$. We have $df|_{U_X}=m_1z_{1}^{m_{1}-1}\,dz_{1}$, so $s\vert|_{U_{X}}=\frac{1}{m_1}z_{1}^{a_{1}-m_{1}+1}\, dz_{i_2}\wedge \dots \wedge dz_{i_n}$. %Therefore, the induced grading of $(F,\lambda,\phi)$ corresponds under \eqref{eq:gradings-trivializations} to $\frac{1}{m_1}e^{(a_{1}-m_{1}+1)\cdot 2\pi\imath \theta_{1}}\, dz_{i_2}\wedge \dots \wedge dz_{i_n}$. 
	By \eqref{eq:monodromy}, the symplectomorphism $\phi^{m}$ is given by $\theta_1\mapsto \theta_1+m\ell_1$, where $\ell_1=\frac{1}{m_1}$. It follows that the winding number used in \cite[Lemma A.8]{McLean} equals $-(a_{1}-m_{1}+1)\cdot \frac{m}{m_{1}}=-\frac{m}{m_1}(a_1+1)+m$, so $\CZ(x)=2\frac{m}{m_{1}}(a_{1}+1)-2m$, as needed.
	\smallskip
	
	\ref{item:B_i-action} 	
	For every $i\in \cP\cup \cE$, choose a point $x_{i}\in A_{i}^{\circ}\cap F$. Recall that the function $\ac$ is defined up to an additive constant, and is constant in each $B_i$. Therefore, to prove \ref{item:B_i-action} it is sufficient to show that for every $i,j\in \cP\cup \cE$ we have
	\begin{equation*}
		\ac(x_j)-\ac(x_i)=(t_0bb_{j}+\epsilon_0)\frac{m}{m_j}-(t_0bb_{i}+\epsilon_0)\frac{m}{m_i}.
	\end{equation*}
	By definition, $\ac$ satisfies $(\phi^m)^{*}\lambda-\lambda=-d\ac$. Therefore, it is enough to prove that for some path $\gamma\colon [0,1]\to F$ such that $\gamma(0)=x_i$, $\gamma(1)=x_j$, we have
	\begin{equation}\label{eq:action-computation}
	\int_{\gamma}(\phi^{m})^{*}\lambda-\lambda=
	-(t_0bb_{j}+\epsilon_0)\frac{m}{m_j}+(t_0bb_{i}+\epsilon_0)\frac{m}{m_i}.
	\end{equation}	
	We can choose $\gamma$ in such a way that it avoids triple intersections, i.e.\ $\gamma([0,1])\cap \bigsqcup_{\#I\geq 3} A_{I}^{\circ}=\emptyset$. Then $\gamma$ can be split into segments with images lying over adapted charts whose associated index sets \eqref{eq:index-set} have at most two elements. Thus considering each piece of $\gamma$ separately, it is sufficient to prove \eqref{eq:action-computation} for $\gamma\colon [0,1]\to U$, where $U$ is the preimage of an adapted chart $U_{X}$ with index set $S\de \{i,j\}$.

	%First, we need a local description of the form $\lambda_{X}=\tilde{\lambda}_{Y}+t_0\lambda_{\mathrm{FS}}$, introduced in Section \ref{sec:isotopy-to-radius-zero}. Fix a fine chart $U_X\subseteq X$, and let $S$ be its associated index set, see Definition \ref{def:adapted-chart}. 
%	We will now describe the restriction $\lambda_{X}|_{U_X}$. Recall from Section \ref{sec:isotopy-to-radius-zero} that $\lambda_{X}=\tilde{\lambda}_{Y}+t_0\lambda_{\mathrm{FS}}$, where $\tilde{\lambda}_{Y}$ is a smooth form on $X$, and $\lambda_{\mathrm{FS}}=d^{c}\log ||s||$, where $s$ is a section of $\cO_{X}(bH)$ with poles of order $-bb_l$ along $D_l$ and no zeros or poles elsewhere. 
%	
%	
%	 Thus inside $U_{X}$, we have $\log||s||=\log(\prod_{l\in S}|\xi z_{l}^{bb_{l}}|)=\sum_{l\in S} bb_{l}\log r_{l}+\log |\xi|$ for some nonvanishing holomorphic function $\xi$. Since $d^{c}r_{l}=-r_{l}\, d\theta_{l}$, see Example \ref{ex:J}, we have $\lambda_{\mathrm{FS}}|_{U_X}=d^{c}\log||s||=-\sum_{l\in S} bb_{l}\, d\theta_{l}+\beta'$ for some smooth $1$-form $\beta''\in \Omega^{1}(U_X)$. Therefore, 

By the formula \eqref{eq:lambdaA-locally}, we have $\pi^{*}\lambda_{X}|_{U_{X}}=-\sum_{l\in S} t_0bb_{l}\, d\theta_{l}+\pi^{*}\beta'$ for some $\beta'\in \Omega^{1}(U_X)$. Using Lemma \ref{lem:d-theta} and \ref{lem:TdA}\ref{item:ort-Z-vj}, we conclude that $\lambda|_{U\cap F}=\sum_{l\in S}(-t_0bb_{l}+\epsilon_0 \bar{v}_l)\, d\theta_{l}+\pi^{*}\beta$ for some $\beta\in \Omega^{1}(U_X)$. The formula \eqref{eq:monodromy} shows that $\pi|_{F}\circ \phi=\pi|_{F}$, so $(\phi^{m})^{*}(\pi^{*}\beta)-(\pi^{*}\beta)=0$. Now by \eqref{eq:monodromy} we have
	\begin{equation}\label{eq:action-integral}
		\int_{\gamma} (\phi^{m})^{*}\lambda-\lambda = \int_{\gamma} \textstyle\sum_{l\in S}(-t_0bb_{l}+\epsilon_0 \bar{v}_l)m\, d\ell_{l}.
	\end{equation}
	If $S=\{i\}$ then the image of $\gamma$ is contained in $A_{i}^{\circ}$, where by \eqref{eq:monodromy} the function $\ell_{i}$ is constant. Hence the above integral is zero, which proves \eqref{eq:action-computation} in this case. 
	
	Assume that $S=\{i,j\}$ for $i\neq j$. We claim that, along $\gamma$, we have the equality
	\begin{equation}\label{eq:inntegration-by-parts}
		\bar{v}_i d\ell_i+\bar{v}_j d\ell_j=d(\bar{v}_i\ell_i+\bar{v}_j\ell_j).
	\end{equation}
	Inside $A_{i}^{\circ}$, the functions $\bar{v}_i$, $\ell_i$ are constant, and $\ell_j=0$, so $\bar{v}_{j}\ell_j=0$ is constant, too. Hence both forms are zero in $A_{i}^{\circ}$, and by symmetry the same holds in $A_{j}^{\circ}$. It remains to prove the equality \eqref{eq:inntegration-by-parts} in $A^{\circ}_{i,j}$.
	
	Fix $l\in \{i,j\}$. By Lemma \ref{lem:intro}\ref{item:intro-u=ubar}, in $A_{i,j}^{\circ}$ we have $\bar{v}_l=-u_{l}$. Recall that, by definition \eqref{eq:def-w_i-u_i} of $u_l$ we have $w_{l}=\eta^{-1}(u_{l})$, so $dw_{l}=\zeta(u_l)\, d u_l$ by definition of $\zeta$. We conclude that $\zeta(u_l)d\bar{v}_l=-\zeta(u_l)du_l=-dw_l$. By Lemma  \ref{lem:intro}\ref{item:intro-sum} we have $w_i+w_j=1$, so  $\zeta(u_i)d\bar{v}_i+\zeta(u_j)d\bar{v}_j=-dw_i-dw_j=0$. Hence $\ell_id\bar{v}_i+\ell_j d\bar{v}_j=0$ by definition \eqref{eq:monodromy} of $\ell_i$, $\ell_j$. This proves \eqref{eq:inntegration-by-parts}.
	
	Now, the equality \eqref{eq:inntegration-by-parts} implies that the integral \eqref{eq:action-integral} is equal to the difference of values of $\sum_{l\in S}(-t_0bb_l+\epsilon_0 \bar{v}_l)m\ell_l$ at the endpoints $x_j$, $x_i$ of $\gamma$. Since $x_j\in A_{j}^{\circ}$, we have $\ell_{i}(x_j)=0$, and similarly $\ell_{j}(x_i)=0$. Moreover, $\bar{v}_{j}(x_j)=\bar{v}_i(x_i)=-1$ and $\ell_{j}(x_j)=\frac{1}{m_j}$, $\ell_{i}(x_i)=\frac{1}{m_i}$. Thus
	\begin{equation*}
	\int_{\gamma} (\phi^{m})^{*}\lambda-\lambda = 
	(-t_0bb_j-\epsilon_0)\frac{m}{m_i}-(-t_0bb_i-\epsilon_0)\frac{m}{m_i}
	\end{equation*}	
	as claimed in the formula \eqref{eq:action-computation}.
\end{proof}

\begin{proof}[Proof of Proposition \ref{prop:spectral-sequence-monodromy}]
	The monodromy abstract contact open book \eqref{eq:monodromy-acob} is graded isotopic to \eqref{eq:radius-zero-acob}. Hence by Proposition \ref{prop:isotopy_invariance}, their fixed point Floer homology groups are the same, i.e.\ $\HF_{*}(\phi_z,+)=\HF_{*}(\phi,+)$. By Proposition \ref{prop:monodromy}, $\phi$ satisfies the assumptions of Proposition \ref{prop:spectral-sequence}, so we get a spectral sequence converging to $\HF_{*}(\phi,+)$, whose first page $E_{p,q}^{1}$ is given by \eqref{eq:spectral-sequence-first-page}.
	
	By Proposition \ref{prop:monodromy}\ref{item:B_i-action}, we can choose the action function $\ac$ for $\phi$ in such a way that $\ac|_{B_{i}}=(t_0 bb_i+\epsilon_0)\frac{m}{m_i}$ for every $i\in S_{m}$. Moreover, the number $\epsilon_0>0$ can be chosen arbitrarily small compared to $t_0 b>0$. Therefore, we can choose the indexing function $\iota \colon S_{m} \to \Z$ in Proposition \ref{prop:spectral-sequence} to be $\iota(i)=\frac{m}{m_i}b_{i}$. Hence the indexing set $\{i: \iota(i)=p\}$ in \eqref{eq:spectral-sequence-first-page} is the set $S_{m,p}$ defined in \eqref{eq:Sm}. Substituting to \eqref{eq:spectral-sequence-first-page} the formula $\CZ(B_i)=2\frac{m}{m_i}(a_i+1)-2m$ from Proposition \ref{prop:monodromy}\ref{item:B_i-CZ}, we get %a spectral sequence converging to $\HF_{*}(\phi,+)$, whose first page equals
	\begin{equation*}
		E_{p,q}^{1}=\bigoplus_{i\in S_{m,p}}
		H_{n-1+p+q+2\frac{m}{m_i}(a_i+1)-2m}(B_{i},\d^{+} B_{i};\Z_2).
	\end{equation*}  
	By definition, $B_{i}$ is compact, and by Proposition \ref{prop:monodromy}\ref{item:B_i-covering}, the manifold $B_{i}\setminus \d^{+}B_{i}$ is diffeomorphic to $B_{i}^{\circ}$. Thus $H_{*}(B_i,\d^{+}B_{i};\Z_2)=H_{*}^{BM}(B_{i}^{\circ})$, as required by the formula \eqref{eq:spectral-sequence-monodromy}.
\end{proof}

Remarks \ref{rem:McLean-5.41}, \ref{rem:degree-shift} and \ref{rem:McLean} below compare our results with the ones of  \cite{McLean}.

\begin{remark}%[{Comparison with \cite{McLean}}]
	\label{rem:McLean-5.41}
	Assume that $Y=\C^n$, $h\colon X\to \C^{n}$ is an $m$-separating resolution of an isolated singularity $0\in f^{-1}(0)$, and $W_Y=\B_{\epsilon}$ is a Milnor ball. Then, Proposition \ref{prop:monodromy} shows that the radius zero monodromy satisfies almost the same dynamical properties as McLean's \emph{model resolution}, see \cite[Theorem 5.41]{McLean}.
	
	The key difference is that our monodromy has codimension zero family of fixed points corresponding to the proper transform of $f^{-1}(0)$. If $W_Y$ is a Milnor ball, they can be avoided by a small perturbation near the boundary, as done in loc.\ cit. However, to prove Theorem \ref{theo:Zariski} we will need to apply Proposition \ref{prop:monodromy} in more generality (namely, keeping the radius of the ball fixed within the family, see Example \ref{ex:mu_constant}).
	
	Note that Proposition \ref{prop:monodromy} is much more general: for example, it holds for an embedded resolution of $f$ with multiple isolated singularities in $f^{-1}(0)$. Moreover, with minor modifications one can formulate a result analogous to Proposition \ref{prop:monodromy}\ref{item:B_i-boundary},\ref{item:B_i-covering} for a degeneration of projective varieties. In this setting, the Conley--Zehnder index computation in Proposition \ref{prop:monodromy}\ref{item:B_i-CZ} can be carried out, too, provided $\phi$ admits a grading: this happens when fibers are Calabi--Yau.
	
	Eventually, we note that the abstract contact open book in Proposition \ref{prop:monodromy} is actually isotopic to the symplectic monodromy of $f$ (defined in a natural way in Example \ref{ex:typical}). Meanwhile, the model resolution constructed in \cite[Example 5.14]{McLean} just shares the same associated graded contact pair, which in this case is $(\d\B_{\epsilon},\d \B_{\epsilon} \cap f^{-1}(0))$.% the embedded contact type of the link.}
\end{remark}

\begin{remark}\label{rem:degree-shift}
	The formula for Conley--Zehnder index in Proposition \ref{prop:monodromy}\ref{item:B_i-CZ} differs from the one in  \cite[Theorem 5.41(3)]{McLean} by $2m$. 
	
	This shift most likely arises due to a mistake in the computation carried out in \cite[p.\ 1025]{McLean}. Let us indicate a possible source of this mistake. The Conley--Zehnder index is computed in loc.\ cit.\ from the winding number associated to a trivialization of the vertical canonical bundle $K\vert$, using Lemma A.8 loc.\ cit. At the bottom of p.\ 1025 loc.\ cit., it is claimed that this winding number is the same as the one for the trivialization of $K_{X\setminus D}$. To be more precise, loc.\ cit.\ uses an analogue of canonical bundles on the model resolutions, but these  two coincide locally, when $X$ is a chart adapted to $f$.
	
	The claimed equality is false. To see this, take $X=\C^{2}$ with coordinates $(z_1,z_2)$, let $f(z_1,z_2)=z_1$, and let $dz_1\wedge dz_2$ be a trivialization of $K_{X\setminus D}$. As in the formula \eqref{eq:K-vert},  write  $K_{X\setminus D}=K\vert \otimes f^{*}K_{\C^{*}}$. The trivialization of $f^{*}K_{\C^{*}}$ is given by $df=dz_{1}$, so the compatible trivialization of $K\vert$ is $dz_2$. To compute the winding number, as in the proof Proposition \ref{prop:monodromy}\ref{item:B_i-CZ}, we pull back these forms by $\Phi_{m\theta}^{-1}\colon (z_1,z_2)\mapsto (e^{-2\pi\imath\cdot m\theta}z_1,z_2)$, and trace the path in $\C^{*}$ made by the leading coefficient. For $K_{X\setminus D}$, the pullback is $e^{-2\pi\imath \cdot m\theta} dz_1\wedge dz_2$, so the winding number is $-m$. In turn, for $K\vert$, the pullback is $dz_2$, so the winding number is $0$. Hence \cite[Lemma A.8]{McLean} shows that to get the correct Conley--Zehnder index, one needs to subtract $2m$ from the one computed in loc.\ cit. %This explains the $2m$ shift.
\end{remark}
\begin{remark}\label{rem:McLean}
	Let $Y=\C^n$, let $f\colon \C^n\to \C$ be a holomorphic function with an isolated critical point $0\in \C^n$, and let $\bar{U}=\B_{\epsilon'}$, $\bar{W}=\B_{\epsilon}$ be Milnor balls. Then Proposition \ref{prop:spectral-sequence-monodromy} provides a spectral sequence converging to Floer homology of the symplectic monodromy of the Milnor fibration \eqref{eq:milfibtbintro} in the tube, introduced in Example \ref{ex:Milnor-fibration}. 
	
	In this case, the term of \eqref{eq:spectral-sequence-monodromy} corresponding to the proper transform of $f^{-1}(0)$ vanishes. To see this, note that $\bigsqcup_{i\in \cP} B_{i}^{\circ}=f^{-1}(0) \cap \B_{\epsilon}^{*}$. By the conical structure theorem \cite[2.10]{Milnor}, $f^{-1}(0) \cap \B_{\epsilon}$ is a cone over $f^{-1}(0)\cap \d \B_{\epsilon}$, so $H_{*}^{BM}(B_{i}^{\circ})=0$ for every $i\in \cP$.
	
	Therefore, in this case the first page \eqref{eq:spectral-sequence-monodromy} is similar to the one in  \cite[Theorem 1.2]{McLean}, cf.\ Remark \ref{rem:McLean-HF3}. The spectral sequence in loc.\ cit.\  converges to the Floer cohomology of the symplectic monodromy of the Milnor fibration in the sphere, i.e.\ $f/|f|\colon \B_{\epsilon}\setminus f^{-1}(0)\to \S^1$.
\end{remark}

\begin{remark}
	The exact values of the action and Conley Zehnder indices, computed in Proposition \ref{prop:monodromy}\ref{item:B_i-action},\ref{item:B_i-CZ} were used to write a precise formula for the degree shift in \eqref{eq:spectral-sequence-monodromy}. However, this formula will not be used in the proof of Theorem \ref{theo:Zariski}, where we will only be interested in vanishing of $\HF_{*}(\phi^m,+)$.
\end{remark}

	To conclude this section, we note the following consequence of Proposition \ref{prop:monodromy}. It is not needed in the sequel, but might become of independent interest.
	
	\begin{cor}\label{cor:no-fixed-poins}
		Let $f\colon (\C^{n},0)\to (\C,0)$ be a singular hypersurface germ. Let $(N,\lambda\std,\phi\std)$ be the abstract contact open book associated to the Milnor fibration of $f$ in the tube $f^{-1}(\d \D_{\delta})\cap \B_{\epsilon}$, constructed in Example \ref{ex:Milnor-fibration}. Then $(N,\lambda\std, \phi\std)$ is isotopic to abstract contact open books $(F,\lambda,\phi_1)$ and $(F,\lambda,\phi_2)$ such that $\Fix\phi_1=\emptyset$ and $\Fix\phi_2=\d F$.
	\end{cor}
	\begin{proof}
Choose a  resolution $h\colon X\to \C^{n}$ of $f$ which is an isomorphism away from the origin and factors through a blowup of $\C^{n}$ at $0$. Since the germ $f$ is singular, all components of $\Exc h$ have multiplicity at least $2$. In particular, $h$ is $1$-separated, i.e.\ satisfies condition \eqref{eq:m-separatedness} for $m=1$. 
 Let $(F,\lambda,\phi)$ be the radius-zero abstract contact open book \eqref{eq:radius-zero-acob} associated to this resolution. Since $(N,\lambda\std,\phi\std)$ is isotopic to $(F,\lambda,\phi)$, it is enough to prove the claim for $(F,\lambda,\phi)$. 
 
 Put $B_{0}=\bigsqcup_{i\in \cP} B_{i}$ , where $B_i=\bar{A}_{i}^{\circ}\cap F$ is as in formula \eqref{eq:fixed-points}, and the sum runs over all components of the proper transform of $f^{-1}(0)$. Applying Proposition \ref{prop:monodromy} for $m=1$, we get $\Fix \phi=B_0$ and $\d B_0=\d^{+}B_0\sqcup \d F$, where $\d^{+}B_0=\bigsqcup_{i\in \cP} \d^{+} B_{i}$.  Applying an isotopy from Lemma \ref{lem:cornerelimination}, we can assume that $B_0$ is a smooth codimension zero manifold with boundary (and no corners). 
 
 The above isotopy does not change the diffeomorphism type of $B_{0}^{\circ}\de B_{0}\setminus \d^{+}B_0$. Thus $B_{0}^{\circ}\cong \bigsqcup_{i\in \cP} X_{i}^{\circ} \cong (f^{-1}(0)\cap \B_{\epsilon})\setminus \{0\}$. By the conical structure theorem \cite[Theorem 2.10]{Milnor}, the latter is diffeomorphic to $L\times [-1,0)$, where $L\de f^{-1}(0)\cap \d \B_{\epsilon}$ is the link of $f$. Combining the above diffeomorphisms with some trivialization of a collar neighborhood of $\d^{+}B_0$, we conclude that there is an open neighborhood $U$ of $B_0$ and a diffeomorphism $U\to  L\times [-1,\eta)$ for some $\eta>0$. % such that $\alpha_{2}^{-1}(-1)=\d F$ and $\alpha_{2}^{-1}(0)=\d^{+}B_0$. 
 
 Now, we argue like in \cite[p.\ 1023]{McLean}. By Definition \ref{def:0codim}, the condition  $\d B_0=\d^{+}B_0\sqcup \d F$ means that, after shrinking $U$ if needed, the restriction $\phi|_{U}$ is the time one flow of some Hamiltonian $v\colon U\to [0,\infty)$ satisfying  $v^{-1}(0)=B_0$. Shrinking $U$ further we can assume that $v$ extends to a smooth function on a compact closure $\bar{U}$, such that $dv$ vanishes precisely on $B_0$. %Shrinking $U$ further we can assume that $v|_{U\setminus B_0}\colon U\setminus B_0\to (0,\eta'')$ is a submersion for some $\eta''>0$.

 Using the above diffeomorphism $L\times [-1,\eta)\to U$, we can define a Hamiltonian $H\colon F\to \R$ which is constant on $F\setminus U$, has no critical points in $U$, and grows compatibly with $v$ on $U\setminus B_{0}$, see Definition \ref{def:growcompat}. Put $\check{\phi}_{\tau}\de \psi^{H}_{\tau}\circ \phi$, so $\check{\phi}_{\tau}$ is isotopic to $\phi$. We claim that $\check{\phi}_{\tau}$ has no fixed points for small $\tau>0$.

Since $H$ is constant on $F\setminus U$, we have $\psi^{H}_{\tau}|_{F\setminus U}=\id_{F\setminus U}$ and thus  $\check{\phi}_{\tau}$ has no fixed points in $F\setminus U$. Since $H$ has no critical points in $U$, there is a time $\tau_0>0$ such that $\psi^{H}_{\tau}$ has no fixed points in $U$ for all $\tau\in (0,2\tau_0]$. In particular, since $\phi|_{B_0}=\id_{B_0}$, we see that $\check{\phi}_{\tau}$ has no fixed points in $B_0$. On $U\setminus B_0$, by Definition \ref{def:growcompat} the level sets of $H$ and $v$ agree, and there is a positive function $\sigma$, constant on those level sets, such that $dv=\sigma\cdot dH$, so $\check{\phi}_{\tau}=\psi^{H}_{\tau}\circ \psi_{1}^{v}=\psi_{\tau+\sigma}^{H}$. Since $dv$ vanishes on $B_0$ and $dH$ does not, there is an open neighborhood $V$ of $B_0$ in $U$ such that $\sigma|_{V}<\tau_0$, so $\check{\phi}_{\tau}$ has no fixed points in $V$ for all $\tau\in (0,\tau_0]$. Similarly, since $dv$ does not vanish on the compact set $\bar{U}\setminus V$, the function $\sigma^{-1}$ is bounded on $U\setminus V$, so since $\psi_{1}^{v}=\phi$ has no fixed points on $U\setminus V$, neither does $\check{\phi}_{\tau}=\psi_{\tau}^{H}\circ \psi_{1}^{v}=\psi^{v}_{\tau\sigma^{-1}+1}$ for small enough $\tau>0$; as claimed.
 
This way, we get the required isotopy $(F,\lambda,\phi)\sim (F,\lambda, \phi_1)$. To get $\phi_{2}$, we replace $H$ by a composition $\rho\circ H$ for some smooth, increasing function $\rho\colon \R\to \R$ such that $\rho'$  vanishes precisely at $H(\d F)$. 
	\end{proof}

\section{Proof of Theorem \ref{theo:Zariski}}\label{sec:proof}

In this section, we apply the techniques developed above to prove Theorem \ref{theo:Zariski}, which asserts that any $\mu$-constant family of isolated hypersurface singularities is equimultiple.

Let us recall the definition of a $\mu$-constant family. Consider a family of power series $f=\sum_{\iota\in \N^{n}}a_{\iota}z^{\iota}$, where $a_{\iota}\colon [0,1]\to \C$ are continuous functions. For $t\in [0,1]$ we write $f_{t}=\sum_{\iota\in \N^{n}}a_{\iota}(t)z^{\iota}\in \C[\![z_1,\dots,z_n]\!]$, and let $\mu(f_t)$ be the Milnor number of $f_{t}$. We say that $f$ is a \emph{$\mu$-constant family} if all numbers $\mu(f_t)$ are finite and equal to each other. We say that $f$ is \emph{equimultiple} if the multiplicities $\nu(f_t)$ of $f_{t}$ are equal for all $t$.

Note that the assumption $\mu(f_t)<\infty$, which we include in the definition for convenience, ensures that the singular germ $\{f_{t}=0\}$ is isolated.
\smallskip

We begin the proof by recalling a known reduction.

\begin{lema}\label{lem:reduction}
	Assume that every $\mu$-constant family $f$ satisfying the conditions  \ref{item:mu-determined}--\ref{item:a_I-holo} below, is equimultiple. Then every $\mu$-constant family is equimultiple.
	\begin{enumerate}
		\item\label{item:mu-determined} There is $N\in \N$ such that $a_{\iota}\equiv 0$ for $|\iota|>N$, i.e.\ each $f_{t}$ is a polynomial of degree at most $N$.
		\item\label{item:a_I-holo} The functions $a_{\iota}$ extend to holomorphic maps $\D\to \C^n$
%		\item\label{item:n>=4} $n\geq 4$.
	\end{enumerate}
\end{lema}
\begin{proof}
	\ref{item:mu-determined} Let $f\in\CC\lbr  z_1,...,z_n\rbr $. If $\mu(f)$ is finite then $f$ is $(\mu(f)+1)$-determined \cite[I.2.24]{GLS_intro-to-singularities}. This means the following: for any $g\in\mathfrak{m}^{\mu(f)+2}$ there exists a formal change of variables $\varphi$ such that $(f+g)\comp\varphi=f$; if moreover $f$ and $g$ are holomorphic then $\varphi$ can be chosen to be holomorphic. As a consequence, $\mu(f+g)=\mu(g)$ and $\nu(f+g)=\nu(g)$.
	
	Now if $f_t:=\sum_{I}a_Iz^I$ is a continuous family with constant Milnor number $\mu$, then for $h_t:=\sum_{|I|\leq \mu+1}a_Iz^I$ we have $\mu(f_t)=\mu(h_t)$ and $\nu(f_t)=\nu(h_t)$, so to prove Theorem \ref{theo:Zariski} for $(f_t)$ it suffices to show it for $(h_t)$ such that each $h_t$ is a polynomial.
	
	\ref{item:a_I-holo} Polynomials whose germs at the origin have Milnor numbers precisely $\mu$ form a Zariski-locally-closed subset of the space of polynomials of degree $\mu+1$, called the $\mu$-constant stratum. Denote it by $S$. The family $f_t$ describes a continuous map $\alpha\colon [0,1]\to S$. Choose a resolution of singularities $\pi\colon \tilde{S}\to S$, and let $p_0$, $p_1$ be points which are in the same irreducible component of $\tilde{S}$ and such that $\pi(p_0)=\alpha(0)$ and $\pi(p_1)=\alpha(1)$. Then there exists a continuous path joining $p_0$ and $p_1$. Because $\tilde{S}$ is smooth, there are points $p_0=q_0, q_1, \dots, q_k = p_1$ and holomorphic parametrizations $\varphi_i:\D\to\tilde{S}$ %(where $\D$ denotes the disk)
	such that $q_i$ and $q_{i+1}$ are contained in $\varphi_i(\D)$ for any $i$. The result follows after replacing $\alpha$ by each $\varphi_i$. 
%	
%	
%	\ref{item:n>=4} Write $\tilde{f}_{t}=f_{t}+z_{n+1}^{M}$, where $M>\max\{\nu(f_{t})\}_t$. Then $\mu(\tilde{f}_{t})=M\mu$, so $(\tilde{f}_{t})$ is $\mu$-constant, too; and $\nu(\tilde{f}_{t})=\min \{\nu(f_t),M\}=\nu(f_t)$.
\end{proof}

%We get the following generalization of \cite[Corollary 1.4]{McLean}
%\begin{prop}
%    Fix $n\geq 3$ and let $f\colon \C^{n}\to \C$ be a holomorphic function with an isolated critical value $0\in \C$. Fix an open set $U\subseteq \C^n$ such that $\bar{U}$ is compact, $f^{-1}(0)\trans \d{U}$ and $0\in f^{-1}(0)$ is the only singular point of $f^{-1}(0)\cap U$. Let $\nu$ be the multiplicity of $0\in f^{-1}(0)$. Let $(F,\lambda,\phi^{m})$ be the \acob associated to $(f,\varrho,c)$ in Example \ref{ex:single-function}. Then
%    \begin{equation}
%        \nu=\min\{m\geq 1: \HF^{*}(\phi^m,+)\neq 0\}.
%    \end{equation}
%\end{prop}

\begin{proof}[Proof of Theorem \ref{theo:Zariski}]
	By Lemma \ref{lem:reduction}, we can assume that our $\mu$-constant family of power series defines a family of holomorphic functions, holomorphically parametrized by a disk $\D_{\eta}$. Clearly, it is enough to show that the multiplicity remains constant for $\eta>0$ sufficiently small. 
	
	In Example \ref{ex:mu_constant}, we have chosen three Milnor radii for $f_0$, say $\epsilon_0'<\epsilon_0<\epsilon_0''$, and constructed an isotopy of graded abstract contact open books $(F_{s},\lambda_{s},\phi^{m}_{s})$, parametrized by $s\in \D_{\eta}$, such that $F_{s}$ is diffeomorphic to $f^{-1}_{s}(\delta)\cap \B_{\epsilon_0}$ for small $\delta>0$.
	
	By Proposition \ref{prop:isotopy_invariance}, for each $m\geq 1$, the Floer homology groups $\HF_{*}(\phi_{s}^{m},+)$ do not depend on $s\in \D_{\eta}$. We will now study the first page  \eqref{eq:spectral-sequence-monodromy}  of the spectral sequence converging to $\HF_{*}(\phi_{s}^{m},+)$.
	
	Let $\nu(f_s)$ be the multiplicity of $f_s$. Fix $s\in \D_{\eta}$, a number $m\geq 1$, and an $m$-separating resolution $h$ of $f_{s}^{-1}(0)\subseteq \C^{n}$. We can assume that $h$ factors as $h=h_{1}\circ h_{2}$, where $h_1$ is a blowup at the origin, and $h_2$ is a morphism such that the base locus of $h_{2}^{-1}$ is contained in the singular locus of $(f_s\circ h_{1})^{-1}(0)\redd$. Write the irreducible decomposition of  $(f_s\circ h)^{-1}(0)$ as $\sum_{i\in \cP} D_i+\nu(f_s)D_1+\sum_{i\in \cE\setminus \{1\}} m_{i}D_{i}$, where $\sum_{i\in \cP}D_i$ is the proper transform of $f_s^{-1}(0)$, and $D_1$ is the proper transform of $\Exc h_1$. The assumption on $h$ implies that $m_{i}>\nu(f_s)$ for every $i\in \cE\setminus \{1\}$. Thus the sets $S_{m}$ defined in \eqref{eq:Sm} satisfy:
	\begin{equation}\label{eq:Sm-for-low-m}
		S_{m}=\cP \quad\mbox{if}\quad m<\nu(f_s)\qquad \mbox{and}\qquad S_{\nu(f_s)}=\cP\cup \{1\}.
	\end{equation}
	
	Put $B_{0}^{\circ}=\bigsqcup_{i\in \cP} B_{i}^{\circ}$. We claim that
	\begin{equation}\label{eq:B_0=0}
		H_{*}^{BM}(B_{0}^{\circ})=0, 
		\quad\mbox{and}\quad
		H_{*}^{BM}(B_{1}^{\circ})\neq 0,
	\end{equation}
	 i.e.\ the contribution of the proper transform of $f_{s}^{-1}(0)$ (respectively, of $\Exc h_1$) to the first page \eqref{eq:spectral-sequence-monodromy} is zero (respectively, nonzero). 
	 
	Recall that $B_{0}^{\circ}\cong f^{-1}_s(0)\cap \B_{\epsilon_0}^{*}$, where $\epsilon_0>0$ is a Milnor radius for $f_0$ chosen in Example \ref{ex:mu_constant}. %Thus  $H_{*}^{BM}(B_{0}^{\circ})=H^{BM}_{*}(f^{-1}_{s}(0)\cap \B_{\epsilon_0}^{*})$. 
	Let $\epsilon_s\in (0,\epsilon_0]$ be a Milnor radius for $f_s$. By the Conical Structure Theorem \cite[Theorem 2.10]{Milnor} 
	 we have that $f^{-1}_{s}(0)\cap \B_{\epsilon_s}$ is homeomorphic to  the cone over $f^{-1}_{s}(0)\cap \d\B_{\epsilon_s}$. Therefore,
	\begin{equation*}
		H^{BM}_{*}(B_0^{\circ})=H^{BM}_{*}(f^{-1}_{s}(0)\cap \B_{\epsilon_0}^{*})=H_*(\overline{f^{-1}_{s}(0)\cap \B_{\epsilon_0}\setminus \B_{\epsilon_s}},f^{-1}_{s}(0)\cap\d\B_{\epsilon_s})=0,
	\end{equation*}
	because the cobordism $\overline{f^{-1}_{s}(0)\cap \B_{\epsilon_0}\setminus \B_{\epsilon_s}}$ is homologically trivial by Proposition \ref{prop:homologtriv}. This proves the first claim in \eqref{eq:B_0=0}. Notice that when $s=0$, the above cobordism does not appear, so the vanishing of $H_{*}^{BM}(B_{0}^{\circ})$ reduces to Remark~\ref{rem:McLean}.
	
	For the second claim in \eqref{eq:B_0=0}, recall that $B_{1}^{\circ}$ is a covering of $D_{1}\setminus (D-D_{1})$, which is a manifold without boundary, so its Borel--Moore homology cannot vanish.
	
	Now, the formulas \eqref{eq:Sm-for-low-m} and \eqref{eq:B_0=0} imply that the first page  \eqref{eq:spectral-sequence-monodromy} is zero for $m<\nu(f_{s})$, and for $m=\nu(f_s)$, it has exactly one nonzero column. We conclude that
	\begin{equation}
		\label{eq:deneMcLeanmultz}
		\nu(f_{s})=\min\{m\geq 1: \HF_{*}(\phi^{m}_{s},+)\neq 0\},
	\end{equation}
	which is independent of $s\in \D_{\eta}$ since $\HF_{*}(\phi^{m}_{s},+)$ is independent of $s$.
\end{proof}

\begin{remark}[{cf.\ \cite[Proposition 1.6]{BBLN_contact-loci}}]\label{rem:tangent-cone}
	Let $F_{s}^{\mathrm{in}}\de \{\operatorname{in}(f_s)=1\}$ be the Milnor fiber of the tangent cone to $f_s$. %As we will see below, t
	As we will explain in more detail below, the proof of Theorem \ref{theo:Zariski} shows that 
	\begin{equation}\label{eq:tangent-cone}
		\HF_{*}(\phi_{s}^{\nu},+)=H_{*+3n-1-2\nu}^{BM}(F_{s}^{\mathrm{in}}),
%		\quad\mbox{in particular,}\quad 
%		H_{*}^{BM}(F_{s}^{\mathrm{in}})=H_{*}^{BM}(F_{0}^{\mathrm{in}}),
	\end{equation}
	where $\nu$ is the multiplicity of $f_s$. By Theorem \ref{theo:Zariski}, the number $\nu$ does not depend on $s$, so by Proposition \ref{prop:isotopy_invariance}, the group $\HF_{*}(\phi_{s}^{\nu},+)$ does not depend on $s$, either. Thus the formula \eqref{eq:tangent-cone} implies that
	\begin{equation*}
		H_{*}^{BM}(F_{s}^{\mathrm{in}})=H_{*}^{BM}(F_{0}^{\mathrm{in}}).
	\end{equation*}
	
	The above equality can be seen as an evidence for \cite[Conjecture 1.7]{BBLN_contact-loci}, which asks whether the homotopy type of the Milnor fiber of the tangent cone is a topological invariant. Recall, however, that by \cite[Theorem A(ii)]{Bobadilla_counterexample}, the projectivized tangent cones $\{\operatorname{in}(f_{s})=0\}\subseteq \P^{n-1}$ may not be homotopically equivalent.
	\smallskip
	
	To prove the equality \eqref{eq:tangent-cone}, recall that for $m=\nu$, the first page \eqref{eq:spectral-sequence-monodromy} has only one nonzero column, which corresponds to the unique component of $(h\circ f_{s})^{-1}(0)$ of multiplicity $\nu$. This component is (the proper transform of) the exceptional divisor of the first blowup over $0\in \C^{n}$. Call this exceptional divisor $D_1$, and define $D_1^{\circ}$ as $D_1$ minus the proper transform of $f_{s}^{-1}(0)$. Then, we have  $\HF_{*}(\phi_{s}^{\nu},+)=H_{*+n+1+2a_1-2\nu}^{BM}(B_1^{\circ})$, where $B_{1}^{\circ}\to D_1^{\circ}$ is a $\nu$-fold covering constructed in \cite[\sec 2.3]{Denef_Loeser-Lefshetz_numbers}, see Section \ref{sec:monodromy-spectral-sequence}. We have $D_1\cong \P^{n-1}$, so by adjunction, the discrepancy $a_1$ equals $n-1$. Moreover, $D_1^{\circ}=\P^{n-1}\setminus \{\mathrm{in}(f_s)=0\}$, and the covering $B_1^{\circ}\to D_1^{\circ}$ can be realized as the restriction of the quotient morphism $\C^n\to \P^{n-1}$ to $ F_{s}^{\mathrm{in}}$; so $B_{1}^{\circ}=F_{s}^{\mathrm{in}}$, which proves \eqref{eq:tangent-cone}.
\end{remark}

\begin{remark}[cf.\ Remark \ref{rem:McLean}]
	\label{rem:multcharaccomparison}
	When $s=0$, Formula~\eqref{eq:deneMcLeanmultz} is a version of \cite[Corollary 1.4]{McLean}. While loc.\ cit.\  uses the graded abstract contact open book arising from $f/|f|$ restricted to a Milnor sphere, that is, the Milnor fibration \eqref{eq:milfibsphintro} in the sphere, we use the graded abstract contact open book associated with the Milnor fibration \eqref{eq:milfibtbintro} in the tube, constructed in Example \ref{ex:Milnor-fibration}.
\end{remark}

\begin{remark}\label{rem:algebraic_Zariski}
	Theorem \ref{theo:Zariski} implies an analogous statement in a purely algebraic setting. More precisely, let $\kk$ be a field of characteristic zero, and let $f\in \kk[t][\![ z_1,\dots,z_n ]\! ]$. Write $f_{t}=f(t,\cdot)\in \kk [\![ z_1,\dots,z_n ]\! ]$; so $t\mapsto f_{t}$ is an algebraic family of power series. Assume that the Milnor number of $f_{t}$ is finite and independent of $t$. Then the multiplicity of $f_t$ is independent of $t$, too. 
	
	Indeed, like in the proof of Lemma \ref{lem:reduction}\ref{item:mu-determined}, we first replace $f$ by its truncation $\sum_{|I|<N} a_{I} z^{I}$, for some $a_{I}\in \kk[t]$ and an integer $N$ independent of $t$, without changing the multiplicity and the Milnor number of each $f_t$. Let now $\kk_{0}$ be a field generated over $\Q$ by all coefficients of polynomials $a_{I}$; so $a_{I}\in \kk_{0}[t]$ for all $I$. Since $\kk_{0}$ is a finitely generated extension of $\Q$, we have an embedding $\kk_{0}\subseteq \C$. This way, $(f_{t})$ becomes a $\mu$-constant family of isolated complex hypersurface singularities, and Theorem \ref{theo:Zariski} shows that the multiplicity of $f_{t}$ does not depend on $t$, as claimed.
\end{remark}

	We recall that Question \ref{q:Zariski}, i.e.\ the original Zariski Question A, does not require the hypersurface singularities to be isolated. Corollary \ref{cor:non-isolated} below, which was proved and communicated to us by David Massey, summarizes some direct consequences of our results for non-isolated singularities.
		
	In this setting, a proper analogue of the Milnor number is the sequence of \emph{L\^e numbers}, introduced in \cite[Definition 2.1]{Massey_Le-varieties-I}. Thus a possible generalization of Theorem \ref{theo:Zariski} to this setting is given by Corollary \ref{cor:non-isolated}\ref{item:Le-constant}. However, since L\^e numbers are not a topological invariant \cite[Theorem A(i)]{Bobadilla_counterexample}, Corollary \ref{cor:non-isolated}\ref{item:Le-constant} does not imply the positive answer to (a family version of) Question \ref{q:Zariski}. 
	
	Nonetheless, \cite[Theorem 7.9]{Massey_Lecture-notes} implies a positive answer in the special case of aligned singularities, see Corollary \ref{cor:non-isolated}\ref{item:aligned}. A germ of a hypersurface singularity $f\colon (\C^n,0)\to (\C,0)$ is \emph{aligned}, see \cite[Definition 7.1]{Massey_Lecture-notes}, if its singular locus admits a stratification satisfying the Thom $a_{f}$ condition, such that the closure of each stratum is smooth at the origin. 
%	
%	We can now summarize the above discussion as follows.
	
	\begin{cor}[Massey]\label{cor:non-isolated}
		Let  $f_t\colon (\C^n,0)\to (\C,0)$ be a continuous family of germs of hypersurface singularities, possibly non-isolated. Assume that $(f_t)$ satisfies one of the following conditions. 
		\begin{enumerate}
			\item \label{item:Le-constant} For some coordinate system on $\C^n$, all L\^e numbers of $f_t$ are independent of $t$.
			\item \label{item:aligned}  The embedded topological type of $f_t$ is independent of $t$, and each $f_t$ has aligned singularities.
		\end{enumerate}
		Then the multiplicity of $f_t$ is independent of $t$.
	\end{cor}
\begin{proof}
	\ref{item:Le-constant} Let $(z_1,\dots,z_n)$ be the chosen coordinate system. Since multiplicity is upper semicontinuous, there is $N\geq 0$ such that $\nu(f_t)\leq N$ for all $t$. Applying inductively the uniform L\^e--Iomdine formulas \cite[Theorem 4.5]{Massey_Lecture-notes}, we conclude that for some $j>N$, $k\in \{0,\dots, n\}$, and $a_1,\dots,a_k\in \C^{*}$, the singularity $g_t\de f_{t}+a_1z_1^{j}+\dots+a_kz_{k}^{j}$ is isolated, and its sequence of L\^e numbers -- which in this case consists just of the Milnor number $\mu(g_t)$, see \cite[Example 2.1]{Massey_Lecture-notes} -- is independent of $t$. By Theorem \ref{theo:Zariski}, the multiplicity of $g_t$ is independent of $t$. Since $j>N\geq \nu(f_t)$, we conclude that $\nu(f_t)=\nu(g_t)$, so the family $(f_t)$ is equimultiple, too, as needed.
	
	\ref{item:aligned} Replacing $f_t$ by the family $\tilde{f}_{t}\colon (\C^{n+k},0)\to (\C,0)$ given by $\tilde{f}_{t}(z,w)=f_t(z)$, we can assume that $n\geq 4$. Now, part \ref{item:aligned} follows from \cite[Theorem 7.9]{Massey_Lecture-notes} and Corollary \ref{cor:Zariski}.
\end{proof}

We conclude with an example of two isolated hypersurface singularities $f_0,f_1\in \C [\! [x,y,z,w]\! ]$ which are topologically equivalent, but are not connected by a $\mu$-constant family. The existence of such examples shows that Question \ref{q:Zariski}, i.e.\ the original Zariski Question A, is strictly stronger than Theorem \ref{theo:Zariski}. This example arose during our discussions with Norbert A'Campo and Enrique Artal.

%, and its solution most likely requires development of additional methods. 

\begin{example}[{cf.\ \cite{ACLM_topological-not-topological}}]\label{ex:not-in-a-family}
	Let $g_0,g_1\in \C[\![ x,y]\!]$ be two curve germs with the same integral Seifert form but different embedded topological type. A series of such pairs was constructed in \cite{duBois-Michel}, a particular example is given by
	\begin{equation*}
	\begin{split}
	g_0&=((y^2-x^3)^2-x^{17}-4yx^{10})((x^2-y^5)^2-y^{11}-4xy^{8}),
	\\
	g_1&= ((y^2-x^3)^2-x^{15}-4yx^{9})((x^2-y^5)^2-y^{13}-4xy^{9}).
	\end{split}
	\end{equation*}
	For $i\in \{0,1\}$ let $f_i\de g_i+z^2+w^2\in \C[\![x,y,z,w]\!]$ be the double suspension of $g_i$. By \cite[Theorem 2]{Sakamoto}, the links of $f_0$ and $f_1$ have the same Seifert form. Results of Kervaire and Levine, see \cite{Durfee}, imply that $f_0$ and $f_1$ are topologically equivalent.
	
	Suppose $f_0$ and $f_1$ occur in a $\mu$-constant family $(f_t)_{t\in [0,1]}$. By Theorem \ref{theo:Zariski}, each $f_t$ has multiplicity $2$. Applying Morse lemma in families, cf.\ \cite[Theorem 2.47]{GLS_intro-to-singularities}, we infer that in some local coordinates $(z_1,\dots,z_4)$ we have $f_t=z_1^2+\check{f}_t$ for some $\mu$-constant family $(\check{f}_{t})_{t\in [0,1]}\subset \C[\! [ z_2,z_3,z_4 ]\!]$. Since $\check{f}_0$, $\check{f}_{1}$ have multiplicity two, once again applying Theorem \ref{theo:Zariski} and Morse lemma we get, after a  coordinate change, a $\mu$-constant family $(\check{g}_{t})_{t\in [0,1]}\subset \C[\! [ z_3,z_4 ]\!]$ such that $f_{t}=z_1^2+z_2^2+\check{g}_{t}$ and $\check{g}_{i}=g_{i}$ for $i\in \{0,1\}$. By \cite{Le-Ramanujam}, the family $(\check{g}_{t})$ is topologically trivial, so $g_0$ is topologically equivalent to $g_1$, a contradiction.
\end{example}
%
%Of course, topologically equivalent singularities $f_0,f_1$ from Example \ref{ex:not-in-a-family} both have multiplicity $2$. It is natural to ask whether the Floer homology of their monodromies coincide, too. More precisely, let $\phi_i$, for $i\in \{0,1\}$, be the symplectic monodromy for $f_i$, constructed in Example \ref{ex:Milnor-fibration}. Now, we can ask the following.
%\begin{question}\label{q:HF-for-Artal-example}
%	Is it true that $\HF_{*}(\phi_0^m,+)=\HF_{*}(\phi_1^{m},+)$ for every $m\geq 1$? 
%\end{question}
%Note that $\HF_{*}(\phi_0,+)=\HF_{*}(\phi_1,+)=0$ by \eqref{eq:deneMcLeanmultz}, and  $\HF_{*}(\phi_0^2,+)=\HF_{*}(\phi_1^2,+)=\Z_2\oplus \Z_2$ by \eqref{eq:tangent-cone}; where the summands have degrees $10$ and $11$. If $f_0,f_1$ were connected by a $\mu$-constant family, this equality would follow from Remark \ref{rem:HF-invariant-in-families}; but since they are not, Question \ref{q:HF-for-Artal-example} becomes nontrivial. Note that by \cite[p.\ 51]{ACLM_topological-not-topological}, the Denef--Loeser zeta functions for $f_0$ and $f_1$ are different. %In view of \cite[Conjecture 1.5]{BBLN_contact-loci}, this suggests that the answer to Question \ref{q:HF-for-Artal-example} could be negative.
%
%

\bibliographystyle{amsalpha}
\bibliography{bibl}

\end{document}